%% file: main.tex
\documentclass[aos]{imsart}

\RequirePackage{amsthm,amsmath,amsfonts,amssymb}
\RequirePackage[numbers]{natbib}
\RequirePackage[colorlinks,citecolor=blue,urlcolor=blue]{hyperref}
\RequirePackage{graphicx}

\usepackage{amsmath}
\usepackage{titletoc}
\usepackage{mathrsfs}
\usepackage{dsfont}
\usepackage{enumerate}
\usepackage{enumitem}
\usepackage{subfig}
\usepackage{float}
\usepackage{tikz}
\usepackage{pgfplots}
\usepackage[normalem]{ulem}

\DeclareMathOperator*{\argmax}{arg\,max} 
\DeclareMathOperator*{\argmin}{arg\,min}
\hypersetup{
colorlinks = true,
urlcolor = blue,
linkcolor = blue,
citecolor = blue,
}
\usepackage{lastpage}
\pgfplotsset{compat=1.15}
\usepackage{mathrsfs}
\usetikzlibrary{arrows}
\usepackage{array}
\usepackage{mathtools}

\usepackage{pifont}
\newcommand{\cmark}{\ding{51}}%
\newcommand{\xmark}{\ding{55}}%

\newcommand{\black}{\color{black}}

\startlocaldefs 
\numberwithin{equation}{section}
\theoremstyle{plain}
\newtheorem{theorem}{Theorem}[section] 
\newtheorem{lemma}[theorem]{Lemma}
\newtheorem{definition}[theorem]{Definition}
\newtheorem{proposition}[theorem]{Proposition}
\newtheorem{corollary}[theorem]{Corollary}
\newtheorem{model}{Model}
\endlocaldefs

\usepackage{xcolor}

\usepackage{natbib}
\bibpunct{(}{)}{;}{a}{,}{,}
\usepackage{hyperref}
\hypersetup{
colorlinks = true,
urlcolor = blue,
linkcolor = blue,
citecolor = blue,
}

\usepackage{multibib}
\newcites{proofs}{References}

\begin{document}

\begin{frontmatter}
\title{Wasserstein Generative Adversarial Networks are Minimax Optimal Distribution Estimators
}
\runtitle{Wasserstein GANs are Minimax Optimal Distribution Estimators}

\begin{aug}
    \author[A,C]{\fnms{Arthur} \snm{Stéphanovitch }\ead[label=e1]{stephanovitch@dma.ens.fr}}
    ,
    \\
    \author[A]{\fnms{Eddie} \snm{Aamari}\ead[label=e2]{eddie.aamari@ens.fr}}
    \and
    \author[B]{\fnms{Clément} \snm{Levrard}\ead[label=e3]{clement.levrard@univ-rennes1.fr}}
    \address[A]{\'Ecole Normale Sup\'erieure, Université PSL, CNRS\\
D\'epartement de Math\'ematiques et Applications\\
F-75005 Paris, France\\
    \printead{e1,e2}
    }
    \address[C]{Université Paris Cité, Sorbonne Université, CNRS\\
Laboratoire de Probabilités, Statistique et Modélisation\\
Paris, France
    }

    \address[B]{Université de Rennes, CNRS \\
    Institut de recherche mathématique de Rennes \\ F-35000 Rennes, France\\
    \printead{e3}
    }
\end{aug}

\begin{abstract}
We provide non asymptotic rates of convergence of the Wasserstein Generative Adversarial networks (WGAN) estimator. We build neural networks classes representing the generators and discriminators which yield a GAN that achieves the minimax optimal rate for estimating a certain probability measure $\mu$ with support in $\mathbb{R}^p$. The probability $\mu$ is considered to be the push forward of the Lebesgue measure on the $d$-dimensional torus $\mathbb{T}^d$ by a map $g^\star:\mathbb{T}^d\rightarrow \mathbb{R}^p$ of smoothness $\beta+1$. Measuring the error with the $\gamma$-Hölder Integral Probability Metric (IPM), we obtain up to logarithmic factors, the minimax optimal rate $O(n^{-\frac{\beta+\gamma}{2\beta +d}}\vee n^{-\frac{1}{2}})$ where $n$ is the sample size, $\beta$ determines the smoothness of the target measure $\mu$, $\gamma$ is the smoothness of the IPM ($\gamma=1$ is the Wasserstein case) and $d\leq p$ is the intrinsic dimension of $\mu$. In the process, we derive a sharp interpolation inequality between  Hölder IPMs. This novel result of theory of functions spaces generalizes classical interpolation inequalities to the case where the measures involved have densities on different manifolds.
\end{abstract}



\end{frontmatter}

\startcontents[mainsections]

\input{Principal}

\newpage

\begin{frontmatter}
\title{Supplement to "Wasserstein Generative Adversarial Networks are Minimax Optimal Distribution Estimators"
}

\runtitle{Wasserstein GANs are Minimax Optimal Distribution Estimators}

\begin{aug}
    \author[A,C]{\fnms{Arthur} \snm{Stéphanovitch }\ead[label=e1]{stephanovitch@dma.ens.fr}}
    ,
    \\
    \author[A]{\fnms{Eddie} \snm{Aamari}\ead[label=e2]{eddie.aamari@ens.fr}}
    \and
    \author[B]{\fnms{Clément} \snm{Levrard}\ead[label=e3]{clement.levrard@univ-rennes1.fr}}
    \address[A]{Ecole Normale Sup\'erieure, Université PSL, CNRS\\
D\'epartement de Math\'ematiques et Applications\\
F-75005 Paris, France\\
    \printead{e1,e2}
    }
    \address[C]{Université Paris Cité, Sorbonne Université, CNRS\\
Laboratoire de Probabilités, Statistique et Modélisation\\
Paris, France
    }

    \address[B]{Université de Rennes, CNRS \\
    Institut de recherche mathématique de Rennes \\ F-35000 Rennes, France\\
    \printead{e3}
    }
\end{aug}



\end{frontmatter}

\startcontents[appendices]
\printcontents[appendices]{l}{1}{{\huge{Contents} \vspace{0.2cm}}\setcounter{tocdepth}{2}}


\input{supplementary}

\end{document}

%% file: Principal.tex
\section{Introduction} 
Let $X_1,...,X_n$ be i.i.d. random points drawn from a probability measure $\mu$ with support in $\mathbb{R}^p$. The inference of $\mu$ is a fundamental problem in statistics and machine learning, for which numerous methods have been developed~\citep{tsybakov2004introduction}. Recent years have witnessed the advent of generative methodologies based on Generative Adversarial Networks \citep{GANs}, with outstanding achievements in the fields of image~\citep{karras2021}, video \citep{vondrick2016generating}, and text generation \citep{SeqGANs}. In this paper we focus on the Wasserstein GAN (WGAN) approach of \citep{arjovsky2017wasserstein}, which uses the $1$-Wasserstein distance as an alternative to the Jensen-Shannon divergence implemented in traditional GAN. Over the years, WGAN and their derivatives have gained popularity in the machine learning community. They are today considered as one of the most successful generative techniques, achieving state-of-the art results in difficult problems \citep{karras2021} while improving the stability and getting rid of mode collapse \citep{gulrajani2017improved}. Although WGANs have shown excellent properties
in numerous empirical studies reported in the machine learning literature (\cite{liu2019wasserstein}, \cite{luo2018eeg}, \cite{stanczuk2021wasserstein}), many of
their theoretical properties remain to be studied.

The present work establishes the minimax optimality of the WGAN estimator. To set the scene and notation, the generative problem is to use the data $X_1,...,X_n$ to learn $\mu$ and, simultaneously, be able to sample from a distribution close to it. Aiming to solve this problem, a generative adversarial network consists of a class of generator functions $\mathcal{G}$ and a  class of discriminators $\mathcal{D}$. Given a known easy-to-sample
distribution $\nu$ on a latent space $\mathcal{Z}$, the generator $ \mathcal{G}\ni g:\mathcal{Z}\rightarrow \mathbb{R}^p$ approximates $\mu$ by trying to minimize over $\mathcal{G}$ a certain Integral Probability Metric (IPM) \citep{IPMsMuller}: 
\begin{equation}\label{eq:IPM}
d_{ \mathcal{D}}(\mu,g_{\# \nu}):=\sup \limits_{D\in \mathcal{D}} \mathbb{E}_\mu[D(X)]-\mathbb{E}_\nu[D(g(Z))],
\end{equation}
where $g_{\# \nu}$ stands for the pushforward measure of $\nu$ by $g$. The goal of the discriminator $\mathcal{D}\ni D:\mathbb{R}^p\rightarrow \mathbb{R}$
is to distinguish between the true distribution and the fake one $g_{\# \nu}$, by maximizing over $\mathcal{D}$ the quantity
$$
L(g,D):=\mathbb{E}_\mu[D(X)]-\mathbb{E}_\nu[D(g(Z))].
$$
The Generative Adversarial Networks min-max problem can then be written as 
\begin{equation}\label{WGANs-theo}   
\inf \limits_{g\in \mathcal{G}}\ \sup \limits_{D\in \mathcal{D}}\ \mathbb{E}_\mu[D(X)]-\mathbb{E}_\nu[D(g(Z))].
\end{equation}
One can see the class $\mathcal{D}$ as a subset of a larger class $\mathcal{F}$, with  $d_\mathcal{F}$ being a metric over distributions. Various types of classes $\mathcal{F}$ have been used in the GAN literature. This includes Lipschitz continuous function  (WGAN, 
~\citep{arjovsky2017wasserstein}), Sobolev
functions (Sobolev GAN, \citep{mroueh2017sobolev}) and reproducing kernel Hilbert space
(MMD GAN, \citep{li2017mmd}). These discrepancies are to be contrasted with the more historically used ones in classical non parametric density estimation such as
$L^p$ distance, Hellinger distance and Kullback-Leibler divergence \citep{tsybakov2004introduction}, which are only applicable when the model is dominated. In the present paper, we work in a general setting where the target measure $\mu$ can have a low dimensional structure which makes these usual discrepancy measures irrelevant. We choose the discriminative class $\mathcal{F}$ to be the Hölder class $\mathcal{H}^\gamma_1(\mathbb{R}^p,\mathbb{R})$ corresponding to the unit ball  of functions of smoothness $\gamma\geq 1$. Note that the case $\gamma=1$ is equivalent to the Wasserstein case when $\mu$ has compact support.

Optimal rates of estimation of a measure $\mu$ with support in $\mathbb{R}^p$ are known to decay exponentially as the ambient dimension $p$ increases. To overcome this curse of dimensionality, some low-dimensional structural
assumptions on $\mu$ need to be imposed.  In this work, we suppose that there exists a map $g^\star \in \mathcal{H}^{\beta +1}_K(\mathbb{T}^d,\mathbb{R}^p)$ with  $\mathbb{T}^d=\mathbb{R}^d/\mathbb{Z}^d$   the $d$-dimensional torus, such that $\mu=g^\star_{\# U}$, with $U\sim \mathcal{U}([0,1]^d)$ a uniform random variable on the cube. In particular, there exist $Y_1,...,Y_n$ i.i.d. such that $Y_i\sim \mathcal{U}([0,1]^d)$ and $g^\star(Y_i)=X_i$. Note that the $Y_i$ are unknown, we only have access to the $X_i$.  A similar low-dimensional-assumption is made in \cite{schreuder2021statistical} and \cite{chae2022rates} where the map $g^\star$ belongs to $\mathcal{H}^{\beta +1}_K([0,1]^d,\mathbb{R}^p)$.    In this context, the inference problem consists in trying to find an estimator $\hat{\mu}$ of $\mu=g^\star_{\# U}$ based on the sample $(X_i)_{i=1,...,n}$, such that the expected error $$\mathbb{E}_{X_i\sim g^\star_{\# U}}[d_{\mathcal{H}_1^\gamma}(g^\star_{\# U},\hat{\mu}(X_1,...,X_n))],$$ is as small as possible. The estimator $\hat{\mu}$ is said to be minimax optimal if there is no estimator that achieves a better rate of convergence uniformly over the model. Formally, it means that there exists a constant $C>0$ independent of $n$ such that
\begin{align*}
&\sup \limits_{g^\star \in \mathcal{H}^{\beta +1}_K}\mathbb{E}_{X_i\sim g^\star_{\# U}}[d_{\mathcal{H}_1^\gamma}(g^\star_{\# U},\hat{\mu}(X_1,...,X_n))] \\
&\leq C \inf \limits_{\hat{T}} \sup \limits_{g^\star \in \mathcal{H}^{\beta +1}_K}\mathbb{E}_{X_i\sim g^\star_{\# U}}[d_{\mathcal{H}_1^\gamma}(g^\star_{\# U},\hat{T}(X_1,...,X_n))].
\end{align*}
In \cite{divol2021measure}, the author provides a (non-generative) minimax estimator of densities in the manifold setting for the Wasserstein distance. It uses a local polynomial manifold estimator from \citep{aamari2017nonasymptotic} coupled with a kernel density estimator. However, the local polynomial estimator is very computationally expensive and therefore cannot  be used in high dimension. In \cite{tang2022minimax}, the author also provide a minimax estimator of densities in the manifold setting but for the distance $d_{\mathcal{H}^\gamma_1}$. The estimator uses a regularization of the empirical measure $\mu_n=\frac{1}{n}\sum \limits_{i=1}^n \delta_{X_i}$ with a manifold estimator coupled with a truncated wavelet regularization. This estimator is purely theoretical and would be extremely computationally expensive to be implemented in practice. Therefore, the existence of a tractable minimax estimator is still an open question and is crucial to be answered in order to provide efficient tools for concrete applications.

As the Wasserstein GAN \citep{arjovsky2017wasserstein} approach has been proven to be easily implementable and provide state-of-the-art results in various field,
we focus in this paper on the GAN estimator 
\begin{equation}\label{WGANS}
    \hat{g} \in \argmin \limits_{g\in \mathcal{G}} \sup \limits_{D\in \mathcal{D}} \frac{1}{n} \sum \limits_{i=1}^n D(X_i)-D(g(U_i))
\end{equation}
for $U_i\sim \mathcal{U}([0,1]^d)$ i.i.d., $\mathcal{G}\subset \mathcal{H}^{\beta+1}_K(\mathbb{T}^d,\mathbb{R}^p)$ and $\mathcal{D} \subset \mathcal{H}^{\gamma}_1(\mathbb{R}^p,\mathbb{R})$ with $\gamma \in[1,\beta+1]$. The estimator $\hat{g}$ is to be understood as an empirical approximation of the solution of \eqref{WGANs-theo} based on the data $X_1,...,X_n$. The probability measure $\hat{g}_{\# U}$ is then naturally our estimator of the target $g^\star_{\# U}$. Recall that as  $\mathcal{D}\subset \mathcal{H}^\gamma$, the case $\gamma=1$ corresponds to the classical WGAN estimator. One of the strengths of the GAN estimator \eqref{WGANS} is that it carries out both support and density estimations at the same time, which allows in particular  to avoid the use of any manifold estimator like in \cite{divol2021measure} and \cite{tang2022minimax}. 

 \subsection*{Main contributions}
We build tractable classes of neural networks $\mathcal{G}$ and $\mathcal{D}$, such that
\begin{equation}\label{eq:minimaxrate}
\sup \limits_{g^\star \in \mathcal{H}_K^{\beta +1}}\mathbb{E}_{X_i\sim g^\star_{\# U}}[d_{\mathcal{H}_1^\gamma}(g^\star_{\# U},\hat{g}_{\# U})]\leq C_1(\log n)^{C_2} \left( n^{-\frac{\beta+\gamma}{2\beta +d}}\vee n^{-\frac{1}{2}}\right),
\end{equation}
with $C_1,C_2>0$ some constants independent of $n$.
This rate has been proven to be minimax optimal in \cite{tang2022minimax} up to the logarithmic factors. To the best of our knowledge, this is the first result showing that the GAN estimator attains minimax rates for the Wasserstein/Hölder distances. This result improves, in particular, on the rates obtained in  \cite{chen2020distribution}, 
 \cite{schreuder2021statistical} and  \cite{chae2022rates}. Minimax convergence rates of the vanilla version of the GAN \citep{GANs}, have recently been obtained  by \cite{belomestny2023rates} for the Jensen–Shannon divergence. Their result only deals with the full dimensional setting  $d=p$, as the Jensen–Shannon divergence is only non-trivial when the measures to be compared are not singular with respect to each other.

We obtain the minimax rate of convergence \eqref{eq:minimaxrate} over measures $\mu=g^\star_{\# U}$ that have a density with respect to the volume measure of an unknown submanifold. In the process, minimax rates are also obtained for two intermediate statistical models:
\begin{itemize}
\item (Model~\ref{model1}, Theorem~\ref{theo:minimaxfirstgan}) We first treat a general low dimensional setting where the target measure
may have atoms and its support is not necessarily a manifold. In this case, the class of discriminators $\mathcal{D}$ used is theoretical, meaning it is not computable in practice. This model is studied to discuss the limitation of the assumptions of the general case.
\item (Model \ref{model3}, Proposition~\ref{prop:minimaxfulldimgamma})  We prove minimax rates in the full dimensional setting $p=d$, where the target measure has a density with respect to the $p$-dimensional Lebesgue measure. This case is treated to understand in the easier setting of full dimension, how additional assumptions can help us obtain a tractable estimator. 
\item (Model \ref{model2}, Theorem~\ref{theo:boundhallgammas}) Adapting the method developed in the full dimensional case to the submanifold case, we propose a tractable GAN estimator achieving minimax rates for all $\gamma\in [1,\beta+1]$ simultaneously.
\end{itemize}
The main result is proven using a new  interpolation inequality (Theorem~\ref{theo:theineq}) that bounds the distance $d_{\mathcal{H}_1^\gamma}(g_{\# U},g^\star_{\# U})$ by the quantity $d_{\mathcal{H}_1^{\beta+1}}(g_{\# U},g^\star_{\# U})^{\frac{\beta+\gamma}{2\beta+1}}$ up to logarithmic factors. A detailed overview of the main results is provided in Section \ref{sec:overview}.

\subsection*{Organization of the paper} In Section \ref{sec:preliminaries}, we set the notation, properly state our main results, and compare our contributions to the existing works. In Section \ref{sec:theoreticalGAN}, we introduce premilinary tools that we will use to show the minimax optimality of the GAN. In Section \ref{sec:tractable}, we gain some insights by treating the general low dimensional setting and the full dimensional setting of Models \ref{model1} and \ref{model3}. Finally, in Section \ref{sec:themanifoldcase}, we treat the low dimensional submanifold setting of Model~\ref{model2}. Technical points of the proofs are gathered in the supplementary material.

\section{Overview of the main results and comparison with other works}\label{sec:preliminaries}
\subsection{Notation}
Let $n \geq 1$ be the sample size, $p\geq 2$ be the dimension of the ambient space, $d\in \{1,...,p\}$ the intrinsic dimension of the data, $\beta\geq 1$ the regularity of the data, $\gamma \in [1,\beta+1]$ the regularity of the discriminator and $K>1$ a bounding constant for several parameters to come.

We write $\langle \cdot, \cdot \rangle$ the canonical dot product on $\mathbb{R}^p$, $\|x\|$ the Euclidean norm of a vector $x \in \mathbb{R}^p$ and $B^p(0,K)$ the $p$-dimensional ball of radius $K$. Let $\mathbb{N}_0=\{0,1,...,\}$ be the set of non-negative integers. For $\eta >0$, $\mathcal{X},\mathcal{Y}$ two subsets of Euclidean spaces and $f=(f_1,...,f_p)\in C^{\lfloor \eta \rfloor}(\mathcal{X},\mathcal{Y})$ the set of $\lfloor \eta \rfloor:=\max \{k\in \mathbb{N}_0 |\ k\leq \eta\}$  times differentiable functions, denote $\partial^\nu f_i = \frac{\partial^{|\nu|}f_i}{\partial x_1^{\nu_1}...\partial x_d^{\nu_d}}$ the partial differential operator  for any multi-index $\nu = (\nu_1,...,\nu_d)\in \mathbb{N}_0^d$ with $|\nu|:=\nu_1+...+\nu_d\leq \lfloor \eta \rfloor$. Write $\|f_i\|_{\eta-\lfloor \eta \rfloor}=\sup \limits_{x\neq y} \frac{f_i(x)-f_i(y)}{\min\{1,\|x-y\|^{\eta - \lfloor \eta \rfloor}\}}$ and let 
\begin{align*}
\mathcal{H}^\eta_K(\mathcal{X},\mathcal{Y})=\Big\{& f  \in C^{\lfloor \eta \rfloor}(\mathcal{X},\mathcal{Y}) \ \big| \max \limits_{i} \sum \limits_{|\nu|\leq \lfloor \eta \rfloor} \|\partial^\nu f_i\|_{L^\infty(\mathcal{X},\mathcal{Y})} + \sum \limits_{|\nu|  = \lfloor \eta \rfloor} \|\partial^\nu f_i\|_{\eta-\lfloor \eta \rfloor}   \leq K\Big\}
\end{align*}
denote the ball of radius $K$ of the Hölder space $\mathcal{H}^\eta(\mathcal{X},\mathcal{Y})$, the set of functions $f:\mathcal{X}\rightarrow \mathcal{Y}$ of regularity $\eta$. To lighten the notation we will write $\mathcal{H}_1^\gamma$ instead of $\mathcal{H}^\gamma_1(\mathbb{R}^p,\mathbb{R})$ and $\mathcal{H}^{\beta+1}_K$ instead of $\mathcal{H}^{\beta+1}_K(\mathbb{T}^d,\mathbb{R}^p)$ when the context is clear. For a uniform random variable $U\sim \mathcal{U}([0,1]^d)$ on the $d$-dimensional unit cube and $g\in \mathcal{H}^{\beta+1}_K(\mathbb{T}^d,\mathbb{R}^p)$, we denote by $g_{\# U}$ the push forward measure of $U$ by $g$, i.e. $\forall f \in \mathcal{H}_1^0(\mathbb{R}^p,\mathbb{R})$, 
$$
\int_{\mathbb{R}^p}f(x)dg_{\# U}(x):=\int_{[0,1]^d}f(g(u))d\lambda^d(u),
$$
with $\lambda^d$ being the $d$-dimensional Lebesgue measure. 

For a $d$-dimensional submanifold $\mathcal{M}$ of $\mathbb{R}^p$ and $\epsilon>0$, we denote by $\mathcal{M}^\epsilon:=\{x\in \mathbb{R}^p|\ d(x,\mathcal{M})< \epsilon\}$ the set of points such that their distance to the manifold $d(x,\mathcal{M}):=\inf_{y\in \mathcal{M}}\|x-y\|$, is less than $\epsilon$. The reach $r_{\mathcal{M}}$ \citep{federer1959} of $\mathcal{M}$ corresponds to the largest $\epsilon\geq0$ such that its orthogonal projection $\pi_\mathcal{M}$ is well defined from $\mathcal{M}^\epsilon$ to $\mathcal{M}$.
 The projection on $\mathcal{T}_{x}(\mathcal{M})$ the tangent space at the point $x\in \mathcal{M}$, will be denoted by $P_{\mathcal{T}_{x}(\mathcal{M})}$.
We will write $\int_{\mathcal{M}}f(x)d\lambda_{\mathcal{M}}(x)$ the integral of a function $f:\mathcal{M}\rightarrow \mathbb{R}$ with respect to the volume measure on $\mathcal{M}$, i.e. the $d$-dimensional Hausdorff measure.

Given a differentiable map $f:\mathbb{R}^k\rightarrow \mathbb{R}^l$ we denote by $\nabla f$ its differential and by $\|\nabla f(x)\|$ its operator norm at  $x$. We denote by $(\nabla f(x))^\top$ its transpose matrix. If $k\leq l$, $\lambda_{\min}((\nabla f(x) )^\top \nabla f(x))$ corresponds to the smallest eigenvalue of the matrix $(\nabla f(x) )^\top \nabla f(x)$.

 The distance we use to measure discrepancy between two probability measures $\mu,\nu$ on $\mathbb{R}^p$ is the Hölder Integral Probability Metric \citep{IPMsMuller} :
\begin{equation}\label{eq:holderIPM}
 d_{\mathcal{H}_1^\gamma}(\mu,\nu):= \sup \limits_{D\in \mathcal{H}^{\gamma}_1}\mathbb{E}_{X\sim \mu,Y\sim \nu}[  D(X)-D(Y)].
\end{equation}
In the case where $\gamma=1$,  the Wasserstein and the $d_{\mathcal{H}_1^1}$ distances are equivalent on compact domains. This can be shown using the dual formulation of the 1-Wasserstein distance $W_1$~\citep{villani2009optimal}: for two probability measures $\mu,\nu$ with support in the ball $B^p(0,K)$, we have 
$$d_{\mathcal{H}_1^1}(\mu,\nu)\leq W_1(\mu,\nu)\leq (2K+1)d_{\mathcal{H}_1^1}(\mu,\nu).$$ 
Therefore, in the case where $\gamma=1$ and the discriminator class $\mathcal{D}$ is a subset of $\mathcal{H}^1_1$, the GAN estimator \eqref{WGANS} is called the Wasserstein GAN estimator. We will denote  $|A|$ the cardinality of a set $A$. For $a,b\in \mathbb{R}$, $a\wedge b$ and $a\vee b$ will denote the minimum and maximum value between $a$ and $b$ respectively.
Throughout, letters $C$, $C_2$, $C_3$ ... denote constants that can vary from line to line and that only depend on $p,d,\beta$ and $K$.

\subsection{Overview of the main results}\label{sec:overview} As mentioned above, the goal of the present paper is to build a tractable GAN estimator \eqref{WGANS} that attains minimax rates for the distance $d_{\mathcal{H}_1^\gamma}$ with $\gamma\in [1,\beta+1]$. We cover a total of three different statistical models. The first two models are studied in order to get insights on how to construct a minimax and tractable estimator in the third model, as well as to highlight the key technical points to do so. Having gathered insights from these analysis, we obtain the desired result under the assumptions of the third model. 
\subsubsection{A (too) general low dimensional model}
The first model allows the target measure to have an intrinsic low dimensional structure as well as a non smooth support.
\begin{model}\label{model1}There exists $g^\star \in \mathcal{H}^{\beta+1}_K(\mathbb{T}^d,\mathbb{R}^p)$ such that the target measure $\mu$ verifies 
$\mu=g^\star_{\# U}.$
\end{model}
Supposing that $\mu$ is the push-forward measure of a low dimensional space by a smooth map is a classical assumption in the GAN literature (\cite{chae2022rates}, \cite{schreuder2021statistical}). The latent space is usually the cube $[0,1]^d$ or the whole $\mathbb{R}^d$ space when the target measure is sub-Gaussian. Like in \cite{gonzalez2021dynamics}, we use the torus $\mathbb{T}^d=\mathbb{R}^d/\mathbb{Z}^d$ for the latent space so that the target measure $g^\star_{\# U}$ is compactly supported and we can use harmonic analysis tools such as wavelets \citep{giné_nickl_2015} to describe $g^\star$.
Model~\ref{model1} is very general as it does not have assumptions on the smoothness of support of $\mu$ besides that it is compact and connected. Under these assumptions, we obtain in Section \ref{sec:theoreticalGAN} the following result.

\textbf{Theorem~\ref{theo:minimaxfirstgan}} For $\mu$ verifying the assumptions of Model~\ref{model1} and $\gamma\in [1,\beta+1]$, there exist theoretical classes of functions $\mathcal{G}$ and $\mathcal{D}_\gamma$ such that the GAN estimator $\hat{g}$ of \eqref{WGANS} verifies
\begin{align*}  
\mathbb{E}\Big[d_{\mathcal{H}_1^\gamma}(\hat{g}_{\# U},\mu)\Big] \leq  C\log(n)^{C_2}\left(n^{-\frac{\beta+\gamma}{2\beta+d}}\vee n^{-\frac{1}{2}}\right).
\end{align*}

 The rate obtained in Theorem~\ref{theo:minimaxfirstgan} has been proven to be minimax optimal in \cite{tang2022minimax} up to the logarithmic factors.   To the best of our knowledge, Theorem~\ref{theo:minimaxfirstgan} is the first result showing that the GAN is minimax for the distance $d_{\mathcal{H}_1^\gamma}$ with $\gamma\in [1,\beta+1]$. In particular this proves that the Wasserstein GAN ($\gamma=1$) is also minimax optimal.  

The minimax rate is attained in the setting of general  push forward maps which allows the target measure to have a low dimensional structure and therefore, be singular to the Lebesgue measure of the ambient space. Furthermore, the only assumption being the regularity of the map $g^\star$, it even allows the target measure $g^\star_{\# U}$ to be singular to the $d$-dimensional Hausdorff measure. Indeed, $g^\star_{\# U}$ could have some parts that have finite $k$-dimensional Hausdorff measure for any $k\in \{0,...,d\}$, as for example with $g^\star(x)=(x_1,...,x_k,0,...,0)$. On the other hand, the support of $g^\star_{\# U}$ could have some strong irregularities like corners or intersections.

However, the discriminator class $\mathcal{D}_\gamma$ we use in Theorem~\ref{theo:minimaxfirstgan} to estimate $\mu$ in Model~\ref{model1} is not tractable. To use a tractable class, we will suppose that $\mu$ has a smooth density with respect to the volume measure of a submanifold. To get some insights from a simple setting, we first treat the "full dimensional case" where the submanifold is the whole ambient space $\mathbb{R}^p$.

\subsubsection{A full dimensional density-based model}
In the full dimensional case $d=p$, we focus on measures that satisfy the following assumptions.
\begin{model}\label{model3} The target measure $\mu$ is absolutely continuous with respect to the $p$-dimensional Lebesgue measure, is compactly supported in $B^p(0,K)$ and its density $f^\star$ belongs to $\mathcal{H}^{\beta}_K(\mathbb{R}^p,\mathbb{R})$.
\end{model}
Note that the densities of push forward measures $g_{\# U}$ with $U\sim \mathcal{U}([0,1]^p)$ are discontinuous on the boundary of their support. Therefore,  measures in Model~\ref{model3} cannot be push forward of the Lebesgue measure on the cube $[0,1]^p$. In the density case, we hence use an estimator that is different from the GAN estimator \eqref{WGANS}. Given $n$ i.i.d. data $X_i\in \mathbb{R}^p$ of law $f^\star$, the  adversarial density estimator is defined as 
\begin{equation}\label{eq:gandensity}
\hat{f} \in \argmin \limits_{f\in \mathcal{F}}\ \max \limits_{D\in \mathcal{D}}\ \sum \limits_{i=1}^n D(X_i)-D(U_i)f(U_i),
\end{equation} 
for $U_i\in \mathbb{R}^p$ i.i.d. uniform random variables on the ball $B^p(0,K)$. Note that this is not the typical GAN estimator where usually the samples of candidate densities $f\in \mathcal{F}$, are generated by push forward distributions. Such a GAN-like formulation, which uses densities as generators, has already been proposed by \cite{liang2018generative} and \cite{singh2018nonparametric}. The drawback of this method is that it does not provide a direct way to sample from the approximate distribution $\hat{f}$. We do not develop on the generation from the estimator \eqref{eq:gandensity} as this estimator is only a by-product of the method we developed to solve the low dimensional case of Model~\ref{model2}. We obtain in Section \ref{sec:fulldim} the following result. 

\textbf{Proposition~\ref{prop:minimaxfulldimgamma}:} For $\mu$ verifying the assumptions of Model~\ref{model3}, there exist tractable neural network classes $\mathcal{F}$ and $\mathcal{D}$ such that the density adversarial estimator  $\hat{f}$ of \eqref{eq:gandensity}, verifies for all $\gamma\in[1,\beta]$ simultaneously:
    \begin{align*}  
\mathbb{E}\Big[d_{\mathcal{H}_1^{\gamma}}(\hat{f},\mu)\Big] \leq  C\log(n)^{C_2}\left(n^{-\frac{\beta+\gamma}{2\beta+p}}\vee n^{-\frac{1}{2}}\right).
\end{align*}

Proposition~\ref{prop:minimaxfulldimgamma} asserts that the adversarial density estimator of \ref {eq:gandensity} also attains minimax rates (up to the logarithmic factor) for the classical density estimation problem.
\subsubsection{A non-degenerate manifold model}
We now come back to the low dimensional case and define the assumptions we add to Model~\ref{model1}, so that we can use tractable classes of generators and discriminators for the GAN estimator \eqref{WGANS}.
\begin{definition}\label{defi:manifold regularity}
    Let $g \in \mathcal{H}^{\beta+1}_K(\mathbb{T}^d,\mathbb{R}^p)$.
    \begin{itemize}
        \item[i)] The map $g$ is said to verify the $K$-manifold regularity condition if it is injective and its image $\mathcal{M}_{g}$ is a $d$-dimensional submanifold with reach $r_{g}$ larger than $K^{-1}$.
        \item[ii)] The map $g$ is said to verify the $K$-density regularity condition if the differential of $g$ verifies 
        $$
    \inf \limits_{u\in[0,1]^d} \lambda_{\min}((\nabla g(u) )^\top \nabla g(u))^{1/2}\geq K^{-1}.$$
     \end{itemize}
\end{definition}
 The $K$-manifold regularity condition prevents $g(\mathbb{T}^d)$ to have self-intersections or corners (in contrast to Model \ref{model1}), and allows to define the projection onto $g(\mathbb{T}^d)$. 
On the other hand, the $K$-density regularity condition imposes the map $g$ to be a local $C^{\beta+1}$-diffeomorphism and the density of $g_{\# U}$ (with respect to the volume measure on $g(\mathbb{T}^d)$) to be bounded below and belong to $\mathcal{H}^{\beta}_{C}(g(\mathbb{T}^d),\mathbb{R})$.  
Using these definitions, we define the central model of the paper, under which we obtain a tractable and minimax estimator of measures having a low dimensional structure.

\begin{model}\label{model2}There exists $g^\star \in \mathcal{H}^{\beta+1}_K(\mathbb{T}^d,\mathbb{R}^p)$ that verifies the $K$-manifold and $K$-density regularity conditions such that the target measure $\mu$ verifies 
$\mu=g^\star_{\# U}.$
\end{model}

Model \ref{model2} includes the push-forwards of $\mathcal{U}([0,1]^d)$ via smooth enough diffeomorphisms from the canonical torus to $\mathbb{R}^p$. 
For instance, supposing $K>0$ large enough, writing $\text{Id}_{\mathbb{T}^d}$ for the canonical embedding of $\mathbb{T}^d$ to $\mathbb{R}^p$ and $r_{\mathbb{T}^d}$ the reach of the embedded torus  $\text{Id}_{\mathbb{T}^d}(\mathbb{T}^d)$, for any  $g=\text{Id}_{\mathbb{T}^d}+v$ with $v\in \mathcal{H}^{\beta+1}_{r_{\mathbb{T}^d}/4}(\mathbb{T}^d,\mathbb{R}^p)$, the measure $g_{\# U}$ satisfies the assumptions of Model~\ref{model2}.   
Using the torus $\mathbb{T}^d$ as the latent space constrains the manifold $\mathcal{M}_{g^\star}$ to have empty boundary. 
In particular, Model~\ref{model2} is non-empty only when $d<p$.
This empty boundary assumption is common for manifold-based statistical models~\cite{divol2021measure}, \cite{tang2022minimax}.

  The choice of a fixed latent space (here $\mathbb{T}^d$) allows to bypass any multi-patch support estimation step usually required in  broader settings of densities bounded below on compact submanifolds.
Though less general than arbitrary manifolds without boundary (see e.g.~\cite{tang2022minimax}), Model~\ref{model2} provides a framework in which we can construct a tractable generative method. 
  Indeed, under the assumptions of Model~\ref{model2}, we obtain in Section~ref{sec:tractablehbeta} the main result of the paper.

\textbf{Theorem~\ref{theo:boundhallgammas}} For $\mu$ verifying the assumptions of Model~\ref{model2}, there exist tractable neural network classes $\mathcal{G}$ and $\mathcal{D}$ such that the GAN estimator $\hat{g}$ of \eqref{WGANS} verifies for all $\gamma\in[1,\beta+1]$ simultaneously:
    \begin{align*}  
\mathbb{E}\Big[d_{\mathcal{H}_1^\gamma}(\hat{g}_{\# U},\mu)\Big] \leq  C \log(n)^{C_2}\left(n^{-\frac{\beta+\gamma}{2\beta+d}}\vee n^{-\frac{1}{2}}\right).
\end{align*}

To the best of our knowledge, the GAN estimator \eqref{WGANS} is the first tractable estimator attaining minimax rates up to logarithmic factors in the manifold setting. Furthermore, it is also the first one being minimax for several IPMs simultaneously. Indeed, as in the other comparable results in the literature, Theorem~\ref{theo:minimaxfirstgan} provides a minimax estimator $\hat{g}_\gamma$ that is different for each $\gamma\in[1,\beta+1]$ as it uses a discriminator class $\mathcal{D}_\gamma\subset \mathcal{H}^\gamma_1$.  In contrast, Theorem~\ref{theo:boundhallgammas} provides a single GAN estimator $\hat{g}_{\beta+1}$ computed with a class $\mathcal{D}_{\beta+1}$, that attains minimax rates for all the distances $d_{\mathcal{H}_1^\gamma}, \gamma\in[1,\beta+1]$ simultaneously, up to logarithmic factors. To obtain this uniform minimax optimality, we show an interpolation inequality that allows to compare the different IPMs. That is, we obtain in Section \ref{sec:interpolation} the following result.

\textbf{ Theorem~\ref{theo:theineq}}
 Let $g,g^\star \in \mathcal{H}^{\beta+1}_K(\mathbb{T}^d,\mathbb{R}^p)$ verifying the $K$-manifold and $K$-density regularity conditions. There exists constants $C,C_{2}>0$ such that if $d_{\mathcal{H}_1^{\beta+1}}(g_{\# U},g_{\# U}^\star)\leq C^{-1},
     $
    then for all $\epsilon\in(0,1)$ and $\gamma \in [1,\beta+1]$, we have
    \begin{align*}
        d_{\mathcal{H}_1^\gamma}(g_{\# U},g_{\# U}^\star)\leq C_2 \log(\epsilon^{-1})^4 \left(d_{\mathcal{H}_1^{\beta+1}}(g_{\# U},g_{\# U}^\star)^{\frac{\beta+\gamma}{2\beta+1}} + \epsilon\right).   
    \end{align*}

This result allows to show that if an estimator is minimax for the distance $d_{\mathcal{H}_1^{\beta+1}}$, then it will also be minimax for the distances $d_{\mathcal{H}_1^\gamma}$ with  $\gamma \in [1,\beta+1]$ up to logarithmic factors. Theorem~\ref{theo:theineq} requires the functions $g$ and $g^\star$ to verify the $K$-manifold and $K$-density regularity conditions. This is the reason why we add these extra assumptions in Model~\ref{model2} compared to Model~\ref{model1}. The proof of the Theorem strongly relies on wavelet theory~\citep{giné_nickl_2015}. All along the paper, we try to show how wavelets provide a key tool to describe Hölder regularity and therefore are perfectly suited for the estimation of measures with smoothness properties.

\subsection{Comparison with other works}\label{sec:comparison}
Let us compare more extensively our results with other works on GAN and minimax estimators. Specifically we focus on:
\begin{itemize}
    \item The settings in which the results are obtained
    \item The rates of convergence of the error up to logarithmic factors
    \item The computational tractability of the methods
\end{itemize}
\subsubsection*{Our result}
We treat two different cases: the general setting (Model~\ref{model1}) where $\mu$ is not necessarily supported on a manifold and can have atoms, and the setting where $\mu$ has a density with respect to the volume measure of a submanifold (Models \ref{model3} and \ref{model2}).

In both settings we obtain the minimax optimal rate of $O(n^{-\frac{\beta+\gamma}{2\beta +d}}\vee n^{-\frac{1}{2}})$ where $n$ is the sample size, $\beta$ determines the smoothness of the target measure $\mu$, $\gamma$ is the smoothness of the discriminator class and $d\leq p$ is the intrinsic dimension of $\mu$.
In the first setting the class $\mathcal{D}$ used for the GAN estimator \eqref{WGANS} is theoretical, meaning that it is not computable in practice. In the second setting, we provide tractable classes of neural networks which lead to a single GAN being minimax optimal for all $\gamma\in [1,\beta+1]$ simultaneously.

\subsubsection*{GAN estimators}
Let us first compare our result with the already existing rates obtained for the GAN estimator in the Hölder/Wasserstein discrimination setting.
\begin{itemize}[leftmargin=*]
\item In \cite{chen2020distribution}, assuming that the support of the target measure is convex, the authors    treat the full dimension case $p=d$ and the case where the target measure is supported on a linear subspace case. They obtain the rate $O(n^{-\frac{\gamma}{2\gamma+p}}\vee n^{-\frac{\beta+1}{2(\beta+1)+p}})$ (for the distance $d_{\mathcal{H}^\gamma_1}$ with $\gamma\geq 1$) in the full dimension case and the rate $O(n^{-\frac{1}{2+d}})$ (for the Wasserstein distance $W_1$) in the linear subspace case. These rates are unfortunately always slower than the minimax rates $O(n^{-\frac{\beta+\gamma}{2\beta+p}}\vee n^{-\frac{1}{2}})$ and $O(n^{-\frac{1}{d}})$ respectively. Their method is tractable, and  uses ReLU feedforward neural networks.

\item In \cite{schreuder2021statistical}, the authors treat the case $\beta+1=\gamma$, in noisy settings and suppose that the target measure is supported on a manifold. They provide the first result assessing the influence of the noise and of the contamination on
the error of generative Modeling. They obtain the rate $O(n^{-\frac{\beta+1}{d}}\vee n^{-\frac{1}{2}})$ which is always strictly slower than the minimax rate $O(n^{-\frac{2\beta+1}{2\beta+d}}\vee n^{-\frac{1}{2}})$  when $1<\beta+1<d/2$.  The class $\mathcal{G}$ is a set of Hölder functions   and the class $\mathcal{D}$ is taken as the whole $\mathcal{H}^\gamma_1$, which is not tractable.

\item In \cite{chae2022rates}, the authors treat the Wasserstein case $\gamma=1$. They obtain the rate $O(n^{-\frac{\beta}{2\beta+d}})$ which is always strictly slower than the minimax rate $O(n^{-\frac{\beta+1}{2\beta+d}})$. They use a set of neural networks for the class $\mathcal{G}$, but use the set of optimal discriminators in $\mathcal{H}_1^1$ between functions in $\mathcal{G}$, as the class $\mathcal{D}$. We use the same discriminator class in our first setting but this class of functions is way to expensive to compute in practice.
\end{itemize}
In a setting extending all these 3 works, our results show that the GAN estimator can be minimax for properly chosen classes $\mathcal{G}$ and $\mathcal{D}$.

\subsubsection*{Minimax rates}
Let us now compare our work with estimators that are not the GAN estimator but that attain the minimax optimal rate.
\begin{itemize}
\item In \cite{liang2021generative}, the authors treat the full dimensional case $p=d$ with Sobolev regularity instead of Hölder. They obtain the minimax rate of convergence $O(n^{-\frac{\beta+\gamma}{2\beta+p}}\vee n^{-\frac{1}{2}})$ for the Sobolev discrimination metric $d_{\mathcal{W}^\gamma_1}$, $\gamma\in \mathbb{N}_0$ with $\mathcal{W}^\gamma_1$ the unit ball  of the Sobolev space $\mathcal{W}^\gamma$. They provide the first  formalization of nonparametric
estimation under the adversarial framework. The estimator used is adversarial but is not the GAN estimator as in the GAN loss \eqref{WGANS}, it does not use samples of the generator and uses a regularization of the empirical measure. The discriminator class is the whole Sobolev class $\mathcal{W}^{\gamma}_1$, which is not computable in practice. Furthermore, they compute integrals with respect to the $p$-dimensional Lebesgue measure, which is not tractable for large ambient dimension $p$.

\item In \cite{divol2021measure}, the author treats the Wasserstein case $\gamma=1$ when the measure admits a $\beta$-regular density with respect to the volume measure of a $d$-dimensional closed (compact boundaryless) manifold. They propose the first estimator attaining the minimax rate of convergence $O(n^{-\frac{\beta+1}{2\beta+d}})$ in the manifold case. The estimator uses a  local polynomial estimator from \citep{aamari2017nonasymptotic} that is not tractable for large ambient dimension~$p$.

\item In \cite{tang2022minimax}, the authors treat the case where the measure admits an $\alpha$-regular density with respect to the volume measure of a $d$-dimensional $\beta$-regular closed submanifold ($\beta\geq \alpha-1$). They obtain the minimax rate of convergence $O(n^{-\frac{\alpha+\gamma}{2\alpha+d}}\vee n^{-\frac{\beta\gamma}{d}} \vee n^{-\frac{1}{2}})$ for $\gamma\in (0,\infty)$. In particular they cover the case 
$\gamma\in (0,1)$ which sheds light on the fact that when $\gamma$ is small, the distance $d_{\mathcal{H}^\gamma_1}$ tends to be more sensitive to the misalignment between
the supports of the measures than to the discrepancy between the densities on their respective supports. Although it is adversarial, the estimator is very different from the GAN. Indeed, for all $D\in \mathcal{H}^\gamma_1(\mathbb{R}^p,\mathbb{R})$, they compute an intermediate estimator $\hat{\mathcal{J}}(D)$ of $\mathbb{E}_{\mu}[D(X)]$, such that the quantity $\sup_{D\in \mathcal{H}^\gamma_1} \mathbb{E}_{\mu}[D(X)]- \hat{\mathcal{J}}(D)$, achieves the minimax rate. Then,  an adversarial estimator $\hat{\mu} \in \argmin_{\nu\in \mathcal{S}}\sup_{D\in \mathcal{H}^\gamma_1} \mathbb{E}_{\nu}[D(Y)]- \hat{\mathcal{J}}(D)$ is computed. As they suppose that the target measure $\mu$ belongs to the class $\mathcal{S}$ and the intermediate estimator $\hat{\mathcal{J}}(D)$ already attains the minimax rate, then naturally, $\hat{\mu}$ also attains this rate. However, the real estimation is realised through the computation of the intermediate estimator $\hat{\mathcal{J}}(D)$, not the adversarial estimator $\hat{\mu}$. To compute the estimator $\hat{\mathcal{J}}(D)$, they use a manifold estimator that is as computationally heavy as in \cite{divol2021measure}. Furthermore, the estimator computes several infima over the whole Hölder class $\mathcal{H}^{\beta}_1$, the wavelet expansion of each discriminator $D\in \mathcal{H}^\gamma_1$, and computes all their derivatives up to the order $\gamma$. Unfortunately, all these steps make the estimator not tractable.
\end{itemize}

The fact that we can obtain minimax rates for the GAN estimator \eqref{WGANS}, shows in particular that it is not necessary to regularize the empirical measure in the GAN loss like in \cite{liang2021generative} and \cite{tang2022minimax}.
Taking a discriminator class $\mathcal{D}$ much smaller than the Hölder functions, actually suffices to
implicitly regularize the measures as this class cannot distinguish well between discrete and smooth measures. In particular, it makes our estimator tractable in contrast to these three works.

See Table \ref{tab:1} for a summary of the different results discussed previously.

\newcolumntype{M}[1]{>{\centering}m{#1}}

\begin{table}[h!]
    \centering
\begin{tabular}{|M{2.3cm}| M{3.2cm}| M{2.3cm}| M{1.2cm} |M{1.2cm} |M{1.2cm}| }
 \hline
 Reference & Target measure & Distance & GAN & Optimal & Tractable \tabularnewline
 \hline\hline
 \cite{chen2020distribution} & density with convex support & $d_{\mathcal{H}^\gamma_1},$ $\gamma\geq 1$ & \cmark & \xmark & \cmark \tabularnewline
 \hline
 \cite{schreuder2021statistical} & $g^\star_{\# U}$, $g^\star:[0,1]^d\rightarrow \mathbb{R}^p$ & $d_{\mathcal{H}^\gamma_1},$ $\gamma=\beta$ & \cmark & \xmark & \xmark \tabularnewline
 \hline
  \cite{chae2022rates} & $g^\star_{\# U}$, $g^\star:[0,1]^d\rightarrow \mathbb{R}^p$ & $W_1$ & \cmark & \xmark & \xmark \tabularnewline
 \hline
  \cite{liang2021generative} & density & $d_{\mathcal{W}^\gamma_1},$ $\gamma\in \mathbb{N}_0$&\xmark & \cmark & \xmark \tabularnewline
 \hline
  \cite{divol2021measure} & density on a closed manifold & $W_p$, $p\geq 1$& \xmark  & \cmark & \xmark  \tabularnewline
 \hline
   \cite{tang2022minimax} & density on a closed manifold & $d_{\mathcal{H}^\gamma_1},$ $\gamma > 0$ & \xmark  & \cmark & \xmark  \tabularnewline
 \hline\hline
    Theorem \ref{theo:boundhallgammas} & $g^\star_{\# U}$, $g^\star:\mathbb{T}^d\rightarrow \mathbb{R}^p$ & $d_{\mathcal{H}^\gamma_1},$ $\gamma\in [1,\beta+1]$ simultaneously & \cmark  & \cmark & \cmark  \tabularnewline
 \hline
\end{tabular}
\caption{Comparison of the results on the estimation of a target measure of regularity $\beta$.}
    \label{tab:1}
\end{table}

\section{Preliminary tools}\label{sec:theoreticalGAN} This section gathers preliminary results used to build the minimax GAN estimators of the following sections. These tools are both of statistical nature~(Section~\ref{firstbound}), and of analytical nature (Section \ref{sec:wavelets}). 
\subsection{Bias-variance trade-off for GANs}\label{firstbound}In this Section, we suppose that we are in the setting of  Model~\ref{model1} (which includes the setting of  Model~\ref{model2}). We exhibit a general bound on the objective and give the required values of the approximation errors and covering numbers of the classes $\mathcal{G}$ and $\mathcal{D}$, so that this bound provides minimax rates.
\subsubsection{A general bound} Consider two closed sets of functions: a set of generators $\mathcal{G}\subset \mathcal{H}^{\beta+1}_K(\mathbb{T}^d,\mathbb{R}^p)$ and a set of discriminators $\mathcal{D} \subset \mathcal{H}^{\gamma}_1(B^p(0,K),\mathbb{R})$. The approximation error of the class $\mathcal{G}$ is defined as
\begin{equation}\label{eq:approxerrG}
    \Delta_\mathcal{G}:=\inf \limits_{g\in \mathcal{G}} d_{\mathcal{H}_1^\gamma}(g_{\# U},g^\star_{\# U}).
\end{equation}
The error $\Delta_{\mathcal{G}}$ corresponds to how well the class $\mathcal{G}$ approximates $g^\star_{\# U}$ for the cost $d_{\mathcal{H}^\gamma_1}$. It is crucial for this error to be small since $$\Delta_\mathcal{G}\leq \mathbb{E}_{X_i\sim g^\star_{\# U}}\big[d_{\mathcal{H}_1^\gamma}(\hat{g}_{\# U},g^\star_{\# U})\Big].$$
On the other hand, the approximation error of the class $\mathcal{D}$ is defined as
\begin{equation}\label{eq:approxerrD}
\Delta_\mathcal{D}:=\sup \limits_{g\in \mathcal{G}} \bigl\{ d_{\mathcal{H}_1^\gamma}(g_{\# U},g^\star_{\# U})-d_{\mathcal{D}}(g_{\# U},g^\star_{\# U})\bigl\}.
\end{equation}
It corresponds to how well the class $\mathcal{D}$ approximates the discrimination made by $\mathcal{H}^\gamma_1$ between the measures $g^\star_{\# U}$ and $g_{\# U}$, for all $g\in \mathcal{G}$. As opposed to $\Delta_{\mathcal{G}}$, it is not crucial that this error be small. Indeed the role of the class $\mathcal{D}$ is only to associate a score $d_{\mathcal{D}}(g_{\# U},g^\star_{\# U})$ to each function $g\in \mathcal{G}$, in order to find a good potential minimizer $\hat{g}_{\# U}$. However, this score does not have to be at the same scale as $d_{\mathcal{H}_1^\gamma}(g_{\# U},g^\star_{\# U}).$
For example, we show in Section \ref{sec:themanifoldcase} that under suitable assumptions on $g,g^\star$ we have
$$
d_{\mathcal{H}_1^\gamma}(g_{\# U},g^\star_{\# U})\leq d_{\mathcal{H}_1^{\beta+1}}(g_{\# U},g^\star_{\# U})^{\frac{\beta+\gamma}{2\beta+1}}
$$
up to logarithmic factors. Therefore, having a class $\mathcal{D}$ that well approximates $\mathcal{H}_1^{\beta+1}$ allows to obtain a minimax estimator although $\Delta_\mathcal{D}$ is big in this case. Nevertheless, if $\Delta_{\mathcal{D}}$ is small, it ensures that the discrimination done by the class $\mathcal{D}$ is relevant and therefore with high probability, $\hat{g}$ is a good candidate among the class $\mathcal{G}$.

Let us now define the notion of covering number which is the other feature of $\mathcal{G}$ and $\mathcal{D}$ that intervenes in the bound on the objective. For   $\epsilon>0$, a minimal $\epsilon$-covering $\mathcal{F}_\epsilon$ of a class of functions $\mathcal{F}$ is defined as: 
\begin{equation}\label{eq:coveringnum}
\mathcal{F}_\epsilon:=\argmin \{|A|\ |\ \forall f \in \mathcal{F},\exists f_\epsilon \in A, \|f-f_\epsilon\|_\infty\leq \epsilon\}.
\end{equation}
The covering numbers $|\mathcal{G}_\epsilon|$ and 
$|\mathcal{D}_\epsilon|$ characterise the sizes of the classes of functions. The larger the classes are, the most likely the estimators $\hat{g}$ and $\hat{D}_{\hat{g}}$ (see Definition \ref{defi:L}) will overfit the data $(X_i)_i$. The following result exhibits a classical bias/variance trade off between the approximation errors and covering numbers.

\begin{theorem}\label{theo:boundexpecterror} If $g^\star \in \mathcal{H}^{\beta+1}_K(\mathbb{T}^d,\mathbb{R}^p)$ and $\mathcal{G}\subset \mathcal{H}^{\beta+1}_K(\mathbb{T}^d,\mathbb{R}^p)$, $\mathcal{D} \subset \mathcal{H}^{\gamma}_1(B^p(0,K),\mathbb{R})$, then the GAN estimator \eqref{WGANS} verifies
\begin{align*}  
\mathbb{E}&\Big[ d_{\mathcal{H}_1^\gamma}(\hat{g}_{\# U},g^\star_{\# U})\Big]-\Delta_\mathcal{G}  -\Delta_\mathcal{D}\\
& \leq
C \min \limits_{\delta \in [0,1]}\Biggl\{\sqrt{\frac{(\delta+1/n)^2\log(n|\mathcal{G}_{1/n}| |\mathcal{D}_{1/n}|)}{n}}+\frac{1}{\sqrt{n}}(1+\delta^{(1-\frac{d}{2\gamma})}+\log(\delta^{-1})\mathds{1}_{\{2\gamma= d\}})  \Biggl\}.
\end{align*}
\end{theorem}
The proof of Theorem~\ref{theo:boundexpecterror} can be found in Section \ref{sec:theo:boundexpecterror}. 
It gives a bound on the error of the estimator in terms of covering number and approximation errors of the classes $\mathcal{G}$ and $\mathcal{D}$. 
It is a WGAN version of Theorem 1 in \cite{belomestny2023rates}, designed for Jensen-Shannon GANs.  

\subsubsection{Expected growth of optimal generators and discriminators classes}
Let us discuss the values of the approximation errors and covering numbers that are required by Theorem~\ref{theo:boundexpecterror} to achieve  minimax optimality. From \citet{tang2022minimax}, we know that the minimax convergence rate is $O(n^{-\frac{\beta+\gamma}{2\beta+d}}\vee n^{-\frac{1}{2}})$. Therefore, if $\frac{\beta+\gamma}{2\beta+d}<1/2$, in light of  Theorem~\ref{theo:boundexpecterror} we need to take $\delta\geq C n^{-\frac{\gamma}{2\beta+d}}$ so that $\frac{\delta^{1-\frac{d}{2\gamma}}}{\sqrt{n}}\leq Cn^{-\frac{\beta+\gamma}{2\beta+d}}$. Then, we also need \begin{equation}\label{eq:requiredcovering}
\log(|\mathcal{G}_{1/n}|)\vee \log(|\mathcal{D}_{1/n}|)\leq Cn^{\frac{d}{2\beta+d}}
\end{equation}
so that $\delta\sqrt{\frac{\log(|\mathcal{G}_{1/n}| |\mathcal{D}_{1/n}|)}{n}}\leq C n^{-\frac{\beta+\gamma}{2\beta+d}}\vee n^{-\frac{1}{2}}$. Lastly, we need to take $\mathcal{G}$ and $\mathcal{D}$ such that
\begin{equation}\label{eq:requirederror}
\Delta_\mathcal{G}  \vee\Delta_\mathcal{D}\leq  C\left(n^{-\frac{\beta+\gamma}{2\beta+d}}\vee n^{-\frac{1}{2}}\right).
\end{equation}
On the other hand, we can prove that the covering number of a class of functions is lower bounded by its precision of approximation of the Hölder classes.
\begin{proposition}\label{prop:appvscov}
    Let be $\eta,\epsilon>0$ and $\mathcal{F}\subset L^2(\mathbb{T}^d,\mathbb{R})$ a class of functions such that
    $
    \sup \limits_{f\in \mathcal{H}^\eta_1}\ \inf \limits_{g\in \mathcal{F}} \|f-g\|_\infty\leq \epsilon
    $,
    then for all $ \theta \in(0,\epsilon/3)$ we have $$\log(|\mathcal{F}_\theta|)\geq C \epsilon^{-\frac{d}{\eta}}.$$
\end{proposition}
The proof of Proposition~\ref{prop:appvscov} can be found in Section \ref{sec:prop:appvscov}. From the previous proposition we see that although we need the approximation errors of the classes $\mathcal{G}$ and $\mathcal{D}$ to verify $\Delta_\mathcal{G}  \vee\Delta_\mathcal{D}\leq  Cn^{-\frac{\beta+\gamma}{2\beta+d}}\vee n^{-\frac{1}{2}}$ in order to be minimax, we also need 
\begin{equation}\label{eq:requiredbound}
\sup \limits_{g^\star\in \mathcal{H}^{\beta+1}_K}\ \inf \limits_{g\in \mathcal{G}} \|g^\star-g\|_\infty\geq C n^{-\frac{\beta+1}{2\beta+d}} \quad \text{ and } \quad \sup \limits_{D^\star \in \mathcal{H}^{\gamma}_1}\ \inf \limits_{D\in \mathcal{D}} \|D^\star-D\|_\infty\geq C n^{-\frac{\gamma}{2\beta+d}\frac{d}{p}}
\end{equation}
so that $\log(|\mathcal{G}_{1/n}|)\vee \log(|\mathcal{D}_{1/n}|)\leq Cn^{\frac{d}{2\beta+d}}$ as desired. We are going to build classes of functions verifying these properties using a wavelet approximation produced by neural networks.

\subsection{Wavelets: a key tool to describe regularity}\label{sec:wavelets} In this section, we introduce the wavelet framework \citep{giné_nickl_2015} and use it to build the classes $\mathcal{G}$ and $\mathcal{D}$ of our GAN estimator in the setting of Model~\ref{model1}. 
\subsubsection{Wavelets, Besov spaces and regularity trade-off} Let $\psi,\phi\in \mathcal{H}^{\lfloor \beta \rfloor+3}(\mathbb{R},\mathbb{R})$ be a compactly supported \emph{scaling} and \emph{wavelet} function respectively (see~\citep{daubechies1988orthonormal} on Daubechies wavelets). For ease of notation, the functions $\psi,\phi$ will be written $\psi_0,\psi_1$ respectively. Then for $j\in \mathbb{N}_0,l \in \{1,...,2^p-1\}, w \in \mathbb{Z}^p$, the family of functions 
$$\psi_{0w}(x) = \prod \limits_{i=1}^p \psi_{0}(x_i-w_i) \ \text{ , } \ \psi_{jlw}(x) = 2^{jp/2}\prod \limits_{i=1}^p \psi_{l_i}(2^{j}x_i-w_i)
$$
form an orthonormal basis of $L^2(\mathbb{R}^p,\mathbb{R})$ (with $l_i$ the $i$th digit of the base-2-decomposition of $l$).
Let $q_1,q_2\geq 1,s>0$ and $b\geq 0$ such that $\lfloor \beta \rfloor+3>s$. The Besov space $\mathcal{B}^{s,b}_{q_1,q_2}(\mathbb{R}^p,\mathbb{R})$ consists of functions $f:\mathbb{R}^p\rightarrow \mathbb{R}$ that admit a wavelet expansion in $L^2$:
$$
f(x)=\sum \limits_{w\in \mathbb{Z}^p} \alpha_f(w)\psi_{0w}(x) + \sum \limits_{j=0}^\infty \sum \limits_{l=1}^{2^p-1}\sum \limits_{w\in \mathbb{Z}^p} \alpha_f(j,l,w)\psi_{jlw}(x),
$$
equipped with the norm
\begin{align*}
\|f\|_{\mathcal{B}^{s,b}_{q_1,q_2}}= &\Biggl(\left(\sum \limits_{w\in \mathbb{Z}^p} |\alpha_f(w)|^{q_1}\right)^{q_2/q_1}\\
& +\sum \limits_{j=0}^\infty 2^{jq_2(s+p/2-p/q_1)}(1+j)^{bq_2} \sum \limits_{l=1}^{2^p-1} \Big(\sum \limits_{w\in \mathbb{Z}^p} |\alpha_f(j,l,w)|^{q_1}\Big)^{q_2/q_1}\Biggl)^{1/q_2}.
\end{align*}
Note that for $b=0$, $\mathcal{B}^{s,0}_{q_1,q_2}(\mathbb{R}^p,\mathbb{R})$ coincides with the classical Besov space $\mathcal{B}^{s}_{q_1,q_2}(\mathbb{R}^p,\mathbb{R})$~\citep{giné_nickl_2015}.
The Besov spaces can be generalized for any $s\in \mathbb{R}$ as a subspace of the space of tempered distribution $\mathcal{S}^{'}(\mathbb{R}^p)$. Indeed for $f \in \mathcal{S}^{'}(\mathbb{R}^p)$ and $C>0$, writing $\alpha_f(j,l,w)=\langle f,\psi_{jlw}\rangle$, the Besov space for $s\in \mathbb{R}$ is defined as
$$\mathcal{B}^{s,b}_{q_1,q_2}(\mathbb{R}^p,\mathbb{R},C)=\{f\in \mathcal{S}^{'}(\mathbb{R}^p) |\ \|f\|_{\mathcal{B}^{s,b}_{q_1,q_2}}\leq C\}.$$

In the following we will use intensively the connection between Hölder and Besov spaces.  
\begin{lemma}\label{lemma:inclusions} (Proposition 4.3.23 \citep{giné_nickl_2015}, (4.63)  \citep{haroske2006envelopes}) If $\alpha>0$ is a non integer, then 
$$\mathcal{H}^\alpha(\mathbb{R}^p,\mathbb{R})=\mathcal{B}^\alpha_{\infty,\infty}(\mathbb{R}^p,\mathbb{R})$$
with equivalent norms. If $\alpha\geq 0$ is an integer, then $\forall \epsilon>0$
$$\mathcal{B}^{\alpha,1+\epsilon}_{\infty,\infty}(\mathbb{R}^p,\mathbb{R}) \xhookrightarrow{} \mathcal{H}^\alpha(\mathbb{R}^p,\mathbb{R})\xhookrightarrow{} \mathcal{B}^\alpha_{\infty,\infty}(\mathbb{R}^p,\mathbb{R}),$$
where we write $A\xhookrightarrow{} B$ if the function space $A$ compactly injects in the function space $B$.
\end{lemma}

The $p$-dimensional wavelet basis $(\psi_{jlw})_{jlw}$ will allow us to describe functions in $\mathcal{H}^\gamma_1(\mathbb{R}^p,\mathbb{R})$ that appear in the $d_{\mathcal{H}^\gamma_1}$ IPM. On the other hand, to describe our generators $g\in \mathcal{G}\subset \mathcal{H}^{\beta+1}_K(\mathbb{T}^d,\mathbb{R}^p)$, we will use the periodised wavelet basis (see \citet{giné_nickl_2015}). One can construct a wavelet basis on $L^2(\mathbb{T}^d)$ by using the periodised scaling and wavelet functions on $[0,1]/\mathbb{Z}$:
\begin{equation}\label{eq:periodicwav}
 \phi^{per}_j(s)=2^{j/2}\sum \limits_{k\in \mathbb{Z}}\phi(2^j(s-k)) \text{ and } \psi^{per}_j(s)=2^{j/2}\sum \limits_{k\in \mathbb{Z}}\psi(2^j(s-k)).
\end{equation}
Note that it is a  well known fact that $\sum \limits_{k\in \mathbb{Z}}\phi(s-k)=1$ (Section 4.3.4 in \cite{giné_nickl_2015}).  For ease of notation, $\phi^{per}_j$ and $\psi_j^{per}$ will be written $\psi^{per}_{j0},\psi^{per}_{j1}$. Then 
\begin{align*}
    \Big(1,\big\{\psi^{per}_{jlz}=\prod \limits_{i=1}^d \psi^{per}_{jl_i}(\cdot-2^{-j}z_i)\ | \ j\in \mathbb{N}_0,l\in\{1,...,2^d\},z \in \{0,...,2^j-1\}^d\big\}\Big),
\end{align*}
is a wavelet basis of $L^2(\mathbb{T}^d)$. We will use this basis to describe functions in $\mathcal{H}^{\beta+1}(\mathbb{T}^d,\mathbb{R}^p)$ as Lemma~\ref{lemma:inclusions} still holds for the periodised Besov spaces\\ $\mathcal{B}^{\beta+1}_{\infty,\infty}(\mathbb{T}^d,\mathbb{R}^p)$ (Section 4.3.4 in \cite{giné_nickl_2015}). Let us define an operator that allows to change the regularity of a function by modifying its wavelet coefficients.

\begin{definition}
Let $\mathcal{X}$ be either $\mathbb{R}^\eta$ or the torus $[0,1]^\eta/\mathbb{Z}^\eta$, and $s,b,\gamma,c\in \mathbb{R}$. For $f \in \mathcal{B}^{s,b}_{\infty,\infty}(\mathcal{X},\mathbb{R})$ a tempered distribution, define the tempered distribution $\Gamma^{\gamma,c}(f)\in \mathcal{B}^{s+\gamma,b+c}_{\infty,\infty}(\mathcal{X},\mathbb{R})$ by its wavelets coefficients 
$$\langle\Gamma^{\gamma,c}(f),\psi_{jlz}\rangle:= 2^{j\gamma}(1+j)^c\langle f,\psi_{jlz}\rangle.$$
\end{definition}
The tempered distribution $\Gamma^{\gamma,c}(f)$ is a "regularization" of $f$ if $\gamma<0$ or $\gamma=0$ and $c<0$. For ease of notation, we will write $\Gamma^\gamma$ instead of $\Gamma^{\gamma,0}$. A first use of this operator is in the following lemma.

\begin{lemma}\label{lemma:deltaG} For all $\eta\geq 1$, $g\in \mathcal{H}^{\eta}_K$, $\gamma\in[ 1,\eta+1]$ and $\epsilon \in (0,1)$ we have 
$$d_{\mathcal{H}_1^\gamma}(g_{\# U},g^\star_{\# U})\leq C\left(\log(\epsilon^{-1})^{1/2} \sum \limits_{i=1}^p \|\Gamma^{-(\gamma-1)}(g_i)-\Gamma^{-(\gamma-1)}(g^\star_i)\|_{L^2([0,1]^d,\mathbb{R})}+\epsilon\right).$$
\end{lemma}
The proof of Lemma~\ref{lemma:deltaG} can be found in Section \ref{sec:lemma:deltaG}. This lemma asserts that the distance between $g_{\# U}$ and $g^\star_{\# U}$ for the norm $d_{\mathcal{H}^\gamma_1}$ is controlled by the $L^2$ distance between the $\gamma-1$ regularization of $g$ and $g^\star$. Therefore, to have a small generator error $\Delta_{\mathcal{G}}$, we do not have to control $\|g-g^\star\|_\infty$, but only the quantity $\|\Gamma^{-(\gamma-1),-2}(g)-\Gamma^{-(\gamma-1),-2}(g^\star)\|_{L^2}$.  This is a key result that will help us obtain the right bound on $\Delta_{\mathcal{G}}$. Indeed, from \eqref{eq:requirederror} we know that we need $\Delta_\mathcal{G} \leq  Cn^{-\frac{\beta+\gamma}{2\beta+d}}\vee n^{-\frac{1}{2}}$, but from \eqref{eq:requiredbound} we know that we will have at best $\inf \limits_{g\in \mathcal{G}} \|g^\star-g\|_\infty\geq C n^{-\frac{\beta+1}{2\beta+d}}$. Therefore, we need to take advantage of the $\gamma$-smoothness of the discriminator in order to obtain 
$$\Delta_\mathcal{G}=\inf \limits_{g\in \mathcal{G}}d_{\mathcal{H}_1^\gamma}(g_{\# U},g^\star_{\# U})\leq  Cn^{-\frac{\beta+\gamma}{2\beta+d}}\vee n^{-\frac{1}{2}}.$$
Let us note that in  \cite{tang2022minimax} it is supposed that the target measure $\mu$ belongs to the class $\mathcal{S}$ used in the adversarial training (see Section \ref{sec:comparison}), which allows them to get around this problem as they have $\Delta_\mathcal{S}=\inf_{\nu\in \mathcal{S}} d_{\mathcal{H}_1^\gamma}(\nu,\mu)=d_{\mathcal{H}_1^\gamma}(\mu,\mu)=0$. 
Furthermore, to utilize the $\gamma$-smoothness of the discriminator, they compute a high order estimator of the type
$$\mathcal{J}(D)=\sum \limits_{i=1}^n \sum \limits_{1\leq |j|\leq \lfloor \gamma \rfloor}\frac{1}{j!} \nabla^{j} D(X_i)\big(\hat{g} \circ \hat{q} (X_i)-X_i\big)^j$$
which makes them having to compute all the derivatives up to the order $\gamma$ of every $D\in \mathcal{H}^\gamma_1(\mathbb{R}^p,\mathbb{R})$. Lemma~\ref{lemma:deltaG} gives a simple way to utilize the $\gamma$-smoothness of the discriminator using wavelet theory and that allows us not to have to compute high order estimators.

To obtain Lemma~\ref{lemma:deltaG}, the main idea is to write $$\mathbb{E}\Big[D(g(U))-D(g^\star(U))\Big]=\mathbb{E}\Big[\int_0^1 \Big\langle\nabla D\big(g^\star(U)+t(g(U)-g^\star(U))\big),\ g(U)-g^\star(U) \Big\rangle dt \Big],$$ 
and to transfer the regularity of the discriminator $\nabla D\in \mathcal{H}^{\gamma-1}_1$ to the generators $g,g^\star$ by modifying their wavelets coefficients. This result is a particular case of Proposition~\ref{prop:Hölder} that shows that in a general setting, we can trade the regularity of the distributions in the $L^2$ scalar product. This result will be used multiple times along the paper.
\subsubsection{From GANs to wavelets and back}

Let us now discuss how we are going to use wavelet theory to construct our classes of generators and discriminators. For $\eta \in (0,\beta+1),\delta\in (0,1)$, and $C_\eta>0$ such that $\mathcal{H}^\eta_1(\mathbb{T}^d,\mathbb{R}^p)\subset \mathcal{B}^{\eta}_{\infty,\infty}(\mathbb{T}^d,\mathbb{R}^p,C_\eta)$, define
\begin{equation}\label{eq:fdeltaper}
\mathcal{F}^{\eta,\delta}_{per}:=\{f\in \mathcal{B}^{\eta}_{\infty,\infty}(\mathbb{T}^d,\mathbb{R}^p,C_\eta K) | \ \langle f_i,\psi_{jlz}^{per}\rangle_{L^2}=0,  \forall j\geq \log_2(\delta^{-1})\}
\end{equation}
the set of functions in $\mathcal{B}^{\eta}_{\infty,\infty}$ that do not have frequencies higher than $\log_2(\delta^{-1})$. 
We will later build our class of generators $\mathcal{G}$ as an approximation by neural network of the class $\mathcal{F}^{\eta,\delta}_{per}$. 
From Lemma~\ref{lemma:deltaG},  approximation error and covering numbers of the class $\mathcal{F}^{\eta,\delta}_{per}$ can be derived.
\begin{proposition}\label{prop:Fdelta}
    We have 
    $$\Delta_{\mathcal{F}^{\eta,\delta}_{per}}:=\sup \limits_{f^\star\in \mathcal{H}^\eta_K}\inf \limits_{f\in \mathcal{F}^{\eta,\delta}_{per}} d_{\mathcal{H}_1^\gamma}(f_{\# U},f^\star_{\# U})\leq  C\log(\delta^{-1})^{C_2}\delta^{\eta-1+\gamma},$$
    and $\forall \epsilon>0$,
    $$\log(|\mathcal{F}^{\eta,\delta}_{per}|_\epsilon)\leq C\delta^{-d}\log( \delta^{-1})\log(\epsilon^{-1}).$$
\end{proposition}
The proof of Proposition~\ref{prop:Fdelta} can be found in Section \ref{sec:prop:Fdelta}. Choosing $\delta= \delta_n:=n^{-\frac{1}{2\beta+d}}$ and $\eta=\beta+1$ in Proposition~\ref{prop:Fdelta}, we obtain
\begin{equation}\label{eq:suited}
\Delta_{\mathcal{F}^{\beta+1,\delta_n}_{per}}\leq  C\log(n)^{C_2}n^{-\frac{\beta+\gamma}{2\beta+d}}\quad \text{ and }\quad \log(|\mathcal{F}^{\beta+1,\delta_n}_{per}|_{1/n})\leq Cn^{\frac{d}{2\beta+d}}\log(n)^2.
\end{equation}
 From \eqref{eq:requiredcovering} and \eqref{eq:requirederror}, we see that these properties make $\mathcal{F}^{\beta+1,\delta_n}_{per}$ ideal to be our class of generators $\mathcal{G}$.

To make the method tractable, 
we do not use the class $\mathcal{F}^{\beta+1,\delta_n}_{per}$ in our estimator but rather a neural networks class that allows to approximate it. These neural networks can be either with tanh~\citep{De_Ryck_2021} or  ReQU activation functions \citep{belomestny2022simultaneous}. Fix $L>0$ and let be $\hat{\phi}:\mathbb{R}\rightarrow \mathbb{R}$ such that  
\begin{equation}\label{eq:approxipsipaper2}
\|\phi-\hat{\phi}\|_{\mathcal{H}^{\lfloor \beta \rfloor +2}} \leq C L^{-1},
\end{equation}
for $\phi$ the Daubechies scaling function. The existence of a neural network  (either with tanh or ReQU activation)    $\hat{\phi}$ satisfying~\eqref{eq:approxipsipaper2} is guaranteed by Lemmas ~\ref{lemma:Deryck} and \ref{lemma:Belo} in Section~\ref{sec:appendixnn1}. We also detail in this section how to compute $\hat{\psi}^{per}_{jlz}$ \eqref{eq:defaprroxphiper}, the approximation by neural networks of the periodised Daubechies wavelet $\psi^{per}_{jlz}$ using $\hat{\phi}$. Define the class of function $\hat{\mathcal{F}}^{\eta,\delta}_{per}$ by
\begin{align}\label{eq:hatfper}
    \hat{\mathcal{F}}^{\eta,\delta}_{per}=\Big\{f\in L^2(\mathbb{T}^d,\mathbb{R}^p)| & \ f_i=\sum \limits_{j=0}^{\log_2(\delta^{-1})} \sum \limits_{l=1}^{2^d} \sum \limits_{z\in \{0,...,2^j-1\}^d} \hat{\alpha}(j,l,w)_i \hat{\psi}^{per}_{jlz}\nonumber\\
    & \text{ with } |\hat{\alpha}(j,l,z)_i|\leq C_\eta K2^{-j(\eta+d/2)},\forall i\in \{1,...,p\} \Big\},
\end{align}
for $C_\eta$ the constant such that  $\|\cdot\|_{\mathcal{B}^\eta_{\infty,\infty}} \leq C_\eta\|\cdot\|_{\mathcal{H}^\eta}$. The class  $\hat{\mathcal{F}}^{\eta,\delta}_{per}$ is an approximation of the class  $\mathcal{F}^{\eta,\delta}_{per}$ defined in \eqref{eq:fdeltaper} using the  tanh (or ReQU)   neural network approximation $\hat{\psi}^{per}_{jlz}$ \eqref{eq:defaprroxphiper} of the periodised wavelet \eqref{eq:periodicwav}. Using propositions \ref{prop:Fdelta} we obtain in the next proposition that $\hat{\mathcal{F}}^{\eta,\delta}_{per}$ inherits the same properties as $\mathcal{F}^{\eta,\delta}_{per}$ in terms of approximation error and covering numbers. 
\begin{proposition}\label{prop:propertiesofhatper}
    If $\hat{\phi}:\mathbb{R}\rightarrow \mathbb{R}$ is such that \eqref{eq:approxipsipaper2} holds with $L \geq \delta^{-(\lfloor \beta \rfloor +1+\gamma)}$, then
        $$\Delta_{\hat{\mathcal{F}}^{\eta,\delta}_{per}}:=\sup \limits_{f^\star\in \mathcal{H}^\eta_K}\inf \limits_{f\in \hat{\mathcal{F}}^{\eta,\delta}_{per}} d_{\mathcal{H}_1^\gamma}(f_{\# U},f^\star_{\# U})\leq  C\log(\delta^{-1})^{C_2}\delta^{\eta+\gamma-1}$$
    and  $\forall \epsilon\in ( \delta^{\eta-1}L^{-1},1)$,
    $$\log(|\hat{\mathcal{F}}^{\eta,\delta}_{per}|_\epsilon)\leq C\delta^{-d}\log(\lfloor \delta^{-1}\rfloor)\log(\epsilon^{-1}).$$
\end{proposition}

The proof of Proposition~\ref{prop:propertiesofhatper} can be found in Section \ref{sec:prop:propertiesofhatper}. Using this result, we have from \eqref{eq:suited} that the class $\hat{\mathcal{F}}^{\beta+1,n^{-\frac{1}{2\beta+d}}}_{per}$ is perfectly suited to be our class of generators. 

The choice of the discriminator class is however more delicate as we do not have a result such as Lemma~\ref{lemma:deltaG} enabling to exploit the $(\beta+1)$-smoothness of the generators to bound $\Delta_\mathcal{D}$ \eqref{eq:approxerrD}.
Indeed, we used Lemma~\ref{lemma:deltaG} to show that we can transfer the regularity of the discriminators to the generators in the IPM $d_{\mathcal{H}_1^\gamma}(g_{\# U},g^\star_{\# U})$, and therefore bound $\Delta_{\mathcal{G}}$ by the approximation error of the $\beta+\gamma$-smooth map $\Gamma^{-(\gamma-1)}(g^\star)$, for the $L^2$ distance. We will see in the following sections that it is not straightforward that we can do the same thing for $\Delta_{\mathcal{D}}$ (meaning giving the regularity of generators to  discriminators) as the measures $g_{\# U}$ and $g^\star_{\# U}$ do not have any regularity from the point of view of the $p$-dimensional Lebesgue measure. 

In the next section, we study two different settings to get insights on how to solve this problem. We show that we can either bypass it by taking a non tractable class $\mathcal{D}$, or add more assumptions on $g^\star$ (Model~\ref{model2}) so that the measure $g^\star_{\# U}$ has some regularity from the point of view of  the $p$-dimensional Lebesgue measure.

\section{Model generality vs. tractability: two sides models}\label{sec:tractable}
In this section we treat two different settings to get some insights on how to build a minimax and tractable GAN estimator. We first treat the general case of Model~\ref{model1} and discuss the limits of the assumptions. Then, we treat the case of Model~\ref{model3} to understand in the easier setting of full dimension, how additional assumptions can help us obtain a tractable estimator.
\subsection{Model~\ref{model1}: A tractable generator class in the general low dimensional case}\label{sec:theoreticalgan}
In this section, we fix a $\gamma\in [1,\beta+1]$ and define two classes of function $\mathcal{G}$ and $\mathcal{D}$ such that the associated GAN estimator is minimax for $d_{\mathcal{H}_1^\gamma}$ in the setting of Model~\ref{model1}.

Fix the approximation of the scaling function in \eqref{eq:approxipsipaper2} to be at precision $L^{-1}=n^{-1}$  and take as the generator class 
\begin{equation}\label{G}
    \mathcal{G}:=\hat{\mathcal{F}}^{\beta+1,n^{-\frac{1}{2\beta+d}}}_{per}.
\end{equation}
 It is a neural network architecture (with either tanh or ReQU activation) composed of a core block $\hat{\psi}^{per}$ (trained to approximate the periodised wavelet $\psi^{per}$) and a final layer that outputs a linear combination of the $O(n^d \log (n))$ blocks $\hat{\psi}^{per}(2^j\cdot-w_i)$. 
The core block $\hat{\psi}^{per}$ has width $O(n)$, depth $O(1)$, and the linear layer is of size $O(n^d\log (n))$. 
In particular, \eqref{G} is a subclass of the classical fully-connected architecture with width $O(n^{d+1} \log(n))$ and depth $O(1)$.
  We detail in Section \ref{sec:appendixnnopti} how to implement the class \eqref{G} in practice.


Let us take as the discriminator class 
\begin{equation}\label{D}
\mathcal{D}:=\{D_{g,g^{'}}^\star | \  g,g^{'}\in \mathcal{G}_{1/n}\},
\end{equation}
with $D^\star_{g,g^{'}} \in \argmax \limits_{D\in \mathcal{H}^{\gamma}_1} \mathbb{E}[D(g(U))-D(g^{'}(u))]$. The class of discriminators is constructed specifically to well approximate the IPMs between the measures $g_{\# U}$ with $g\in \mathcal{G}$, without being too massive. This class $\mathcal{D}$ is not really computable in practice, but we use it here to illustrate that we can bypass the problems risen by the lack of regularity of the target measure $g^\star_{\# U}$, at the cost of the tractability of the method. Having defined the classes $\mathcal{G}$ and $\mathcal{D}$, we can state the final bound for the GAN estimator \eqref{WGANS} in the setting of Model~\ref{model1}.
\begin{theorem}\label{theo:minimaxfirstgan} For all $g^\star \in \mathcal{H}^{\beta+1}_K(\mathbb{T}^d,\mathbb{R}^p)$ and $\gamma\in [1,\beta+1]$, the GAN estimator $\hat{g}$ of~\eqref{WGANS} with $\mathcal{G}$ and $\mathcal{D}$ defined in \eqref{G} and \eqref{D}, satisfies
\begin{align*}  
\mathbb{E}\Big[d_{\mathcal{H}_1^\gamma}(\hat{g}_{\# U},g^\star_{\# U})\Big] \leq  C\log(n)^{C_2}\left(n^{-\frac{\beta+\gamma}{2\beta+d}}\vee n^{-\frac{1}{2}}\right).
\end{align*}
\end{theorem}

The proof of Theorem~\ref{theo:minimaxfirstgan} can be found in Section \ref{sec:theo:minimaxfirstgan}. As usual in nonparametric estimation, an extra logarithmic term in Theorem~\ref{theo:minimaxfirstgan} is to be a attributed to a concentration bound (Theorem~\ref{theo:boundexpecterror}). More interestingly, another one arises from functional inequalities when $\beta$ is not an integer (Lemma~\ref{lemma:deltaG}). We do not keep track of the extra exponent $C_2$ to avoid unnecessary heaviness, as the obtained exponent is  very likely suboptimal. 

 As pointed out in remark 4 of \cite{tang2022minimax}, some authors consider a more general setting than ours, by allowing the observed data to contain noise. It usually consists in supposing that instead of observing data $X_i\sim \mu$ i.i.d, we observe $Y_i=X_i+\epsilon_i$ with $\epsilon_i$ i.i.d. errors. In this setting, \cite{Genovese_2012} showed that when the error follows the standard Gaussian distribution, the minimax rate for estimating a manifold under the
Hausdorff distance is lower bounded by $O(\log(n)^{-1})$. To bypass this
slow convergence, previous work like~\cite{divol2021measure}, assume the noise variance 
to decay with the number of data so that it does not impact the rate of convergence. In our setting, supposing we observe noisy data with $\|\epsilon_1\|\leq n^{-\frac{\beta+ \gamma}{2\beta+d}}\vee n^{-\frac{1}{2}}$ almost surely, will not impact Theorem~\ref{theo:minimaxfirstgan} as the noise scales at the minimax rate.

Theorem~\ref{theo:minimaxfirstgan} shows that the GAN estimator is minimax optimal (up to the logarithmic factor) for the distance $d_{\mathcal{H}_1^\gamma}$. Nevertheless, this result does not really translate to practical applications as the class of discriminators $\mathcal{D}$ \eqref{D} we used is not tractable. Due to the lack of smoothness of the target measure $g_{\# U}^{\star}$, we had to use the fact that the class $\mathcal{G}$ is small rather than it is of smoothness $\beta+1$. Supposing that the target measure has a smooth density with respect to the volume measure of a submanifold, we show in the next sections that we can exploit this regularity and therefore, use a tractable class of discriminators.

\subsection{Model~\ref{model3}: A tractable $\mathcal{D}$ in full dimension}\label{sec:fulldim}
In this section we move away from the case of transport maps from the torus and treat the setting of Model~\ref{model3} where the target measure $\mu$ has a density $f^\star$ with respect to the $p$-dimensional Lebesgue measure. The goal is to understand how to build a tractable estimator of densities on submanifolds, in the simple case where the submanifold is the whole $\mathbb{R}^p$ space. 

Let us first state an analog of Theorem~\ref{theo:boundexpecterror} in the case of densities.

\begin{corollary}\label{coro:boundexpectederrorfulldim} If the target $f^\star$ belongs to $\mathcal{H}^{\beta}_K(\mathbb{R}^p,\mathbb{R})$, has support included in $B^p(0,K)$ and $\mathcal{F}\subset \mathcal{B}^{\beta}_{\infty,\infty}(B^p(0,K),\mathbb{R},K)$, $\mathcal{D} \subset \mathcal{B}^{\gamma}_{\infty,\infty}(B^p(0,K),\mathbb{R},1)$ with $\gamma \in [1,\beta]$, then the density adversarial estimator \eqref{eq:gandensity} verifies

\begin{align*}  
& \mathbb{E}\Big[ \sup \limits_{D\in \mathcal{H}^{\gamma}_1} \int_{\mathbb{R}^p}D(x)(\hat{f}(x)-f^\star(x))d\lambda^p(x)\Big]-\Delta_\mathcal{F}  -\Delta_\mathcal{D}\\
& \leq  C \min \limits_{\delta \in [0,1]} \left\{ \sqrt{\frac{(\delta+1/n)^2\log(n|\mathcal{F}_{1/n}| |\mathcal{D}_{1/n}|)}{n}}
 +\frac{1}{\sqrt{n}}(1+\delta^{(1-\frac{p}{2\gamma})}+\log(\delta^{-1})\mathds{1}_{\{2\gamma= p\}})\right\} .
\end{align*}
\end{corollary}
 The proof of Corollary~\ref{coro:boundexpectederrorfulldim} can be found in Section \ref{sec:coro:boundexpectederrorfulldim}.  Let us now build the classes of functions $\mathcal{F}$ and $\mathcal{D}$ by adapting to the non periodic case, the construction of the class of generators \eqref{G} from the previous section. For $\eta\in(0,\beta+1)$ and $\delta\in (0,1)$, define
 $$\mathcal{F}^{\eta,\delta}=\{f\in \mathcal{B}^{\eta}_{\infty,\infty}(\mathbb{R}^p,\mathbb{R},C_\eta K) | \ \langle f,\psi_{jlw}\rangle_{L^2}=0,  \forall j\geq \log_2(\delta^{-1})\text{ and } supp(f)\subset B^p(0,K)\}$$
the set of functions in $\mathcal{B}^{\eta}_{\infty,\infty}$ that do not have high frequencies and have compact support in $B^p(0,K)$.  As in Section \ref{sec:theoreticalgan}, we are going to approximate the class $\mathcal{F}^{\eta,\delta}$ by neural networks. Let $\hat{\phi},\hat{\psi}$ be the approximations by neural networks of the scaling and wavelet Daubechies function from \eqref{eq:approxipsi} and \eqref{eq:approxipsiphi} at precision $L^{-1}=n^{-1}$. For ease of notation $\hat{\phi},\hat{\psi}$ will be written $\hat{\psi}_0,\hat{\psi}_1$ respectively and define the multi-dimensional approximation of the wavelet as 
$$
\hat{\psi}_{0w}(x) = \prod \limits_{i=1}^p \hat{\psi}_{0}(x_i-w_i) \ \text{ , } \ \hat{\psi}_{jlw}(x) = 2^{pj/2}\prod \limits_{i=1}^p \hat{\psi}_{l_i}(2^{j}x_i-w_i).
$$
Using these approximated wavelets, Define the class of function $\hat{\mathcal{F}}$ by
\begin{equation}
    \hat{\mathcal{F}}^{\eta,\delta}:=\Big\{\sum \limits_{j=0}^{\log_2(\delta^{-1})} \sum \limits_{l=1}^{2^p} \sum \limits_{w\in \{-K2^{j},...,K2^j\}^p} \hat{\alpha}(j,l,w) \hat{\psi}_{jlw}\big|\ |\hat{\alpha}(j,l,w)|\leq C_\eta K2^{-j(\eta+p/2)} \Big\}.
\end{equation}
 As in the periodic case, we obtain that $ \hat{\mathcal{F}}^{\eta,\delta}$ and $\mathcal{F}^{\eta,\delta}$ have the same bounds on their approximation errors and covering numbers. 

\begin{proposition}\label{prop:propertiesofhat}
    If $\|\phi-\hat{\phi}\|_{\mathcal{H}^{\lfloor \beta \rfloor +2}} \leq L^{-1}$ with $L\geq \delta^{-(\lfloor \beta \rfloor +2+\gamma)}$, then
        $$\Delta_{ \hat{\mathcal{F}}^{\eta,\delta}}:=\sup \limits_{h^\star\in \mathcal{H}^{\eta}_1}\inf \limits_{ f\in \hat{\mathcal{F}}^{\eta,\delta}} \sup \limits_{D\in \mathcal{H}^{\gamma}_1} \int_{\mathbb{R}^p}D(x)(f(x)-h^\star(x))d\lambda^p(x)\leq  C\delta^{\eta+\gamma}$$
    and  $\forall \epsilon\in (L^{-1}\delta^{\eta-1},1)$,
    $$\log(| \hat{\mathcal{F}}^{\eta,\delta}|_\epsilon)\leq C\delta^{-p}\log(\lfloor \delta^{-1}\rfloor)\log(\epsilon^{-1}).$$
\end{proposition}

The proof of Proposition~\ref{prop:propertiesofhat} can be found in Section \ref{sec:prop:propertiesofhat}. Here, we need our class of generators to approximate at a scale $O(n^{-\frac{\beta+\gamma}{2\beta+p}})$ the density $f^\star \in \mathcal{H}^{\beta}_K(\mathbb{R}^p,\mathbb{R})$, while having a covering number that does not exceed $O(n^{\frac{p}{2\beta+p}})$ (see~\eqref{eq:requiredcovering}) up to logarithmic factors. Therefore, choosing $\eta=\beta$ and $\delta_n=n^{-\frac{1}{2\beta+p}}$ in Proposition~\ref{prop:propertiesofhat}, we have 
$$\inf \limits_{ f\in \hat{\mathcal{F}}^{\beta,\delta_n}} \sup \limits_{D\in \mathcal{H}^{\gamma}_1} \int_{\mathbb{R}^p}D(x)(f(x)-f^\star(x))d\lambda^p(x)\leq  Cn^{-\frac{\beta+\gamma}{2\beta+p}}$$
and 
$$\log(|\left(\hat{\mathcal{F}}^{\beta,\delta_n}\right)_{1/n}|)\leq Cn^{\frac{p}{2\beta+p}}\log(n),$$
which makes $\hat{\mathcal{F}}^{\beta,\delta_n}$ perfectly suited to be our class of generators.  

We now develop a method that allows to exploit the smoothness of the target measure.  We are going to define a class of discriminators $\mathcal{D}$ such that the density adversarial estimator \eqref{eq:gandensity} attains minimax rates for all $d_{\mathcal{H}_1^{\gamma}}$, $\gamma \in [1,\beta]$ simultaneously. Then, the question of what should be the regularity of the class of discriminators $\mathcal{D}$ arises. Let us first show using Proposition~\ref{prop:Hölder}, that IPMs of high regularity control IPMs of lower regularity.

\begin{corollary}\label{coro:ineqfulldim}
   Let be $f,f^\star \in \mathcal{H}^{\beta}_K(\mathbb{R}^p,\mathbb{R})$ compactly supported in $B^p(0,K)$. Then for all $\alpha\geq \gamma>0$, we have
    $$d_{\mathcal{B}^{\gamma,2}_{\infty,\infty}(1)}(f,f^\star)\leq C d_{\mathcal{B}^{\alpha,2}_{\infty,\infty}(1)}(f,f^\star)^\frac{\beta+\gamma}{\beta+\alpha}.$$
\end{corollary}
The proof of Corollary~\ref{coro:ineqfulldim} can be found in Section \ref{sec:coro:ineqfulldim}. This is the key result that will allow us to prove that if $\hat{f}$ attains the minimax rate for $d_{\mathcal{H}^\alpha_1}$, then it will also attain the minimax rate for $d_{\mathcal{H}^\gamma_1}$ for all $ \gamma \in [1,\alpha]$. Indeed, by injection of Besov spaces into Holder spaces (up to logarithmic terms), Corollary~\ref{coro:ineqfulldim} implies that for all $\epsilon \in (0,1)$,
$$d_{\mathcal{H}^{\gamma}_{1}}(f,f^\star)\leq C\log(\epsilon^{-1})^2 d_{\mathcal{H}^{\alpha}_{1}}(f,f^\star)^\frac{\beta+\gamma}{\beta+\alpha}+\epsilon,$$
as detailed in Section \ref{sec:prop:minimaxfulldimgamma}. Therefore, if for a certain $\alpha\in [1, \beta]$ we obtain $d_{\mathcal{H}_1^{\alpha}}(\hat{f},f^\star) \leq  O(n^{-\frac{\beta+\alpha}{2\beta+p}})$, then $\forall \gamma \in [1,\alpha]$ we will have up to logarithmic terms,
\begin{equation}\label{eq:theresonwhy}
d_{\mathcal{H}_1^{\gamma}}(\hat{f},f^\star)\leq Cd_{\mathcal{H}_1^{\alpha}}(\hat{f},f^\star)^\frac{\beta+\gamma}{\beta+\alpha}\leq O(n^{-\frac{\beta+\alpha}{2\beta+p}})^\frac{\beta+\gamma}{\beta+\alpha}=O(n^{-\frac{\beta+\gamma}{2\beta+p}}).
\end{equation}
Hence, by taking the regularity of the discriminator class to be equal to $\beta$, if we obtain $d_{\mathcal{H}_1^{\beta}}(\hat{f},f^\star) \leq  O(n^{-\frac{\beta+\beta}{2\beta+p}})$, we would have minimax rates for $d_{\mathcal{H}_1^{\gamma}}$ for all $ \gamma \in [1,\beta]$. However,
in the case where $\beta>p/2$ the minimax rate for the distance $d_{\mathcal{H}_1^{\beta}}$ is $O(n^{-1/2})$ not $O(n^{-\frac{\beta+\beta}{2\beta+p}})$, so equation \eqref{eq:theresonwhy} does not work anymore. Therefore, we choose the regularity of the discriminator class $\mathcal{D}$ to be
$$\tilde{\beta}:=\beta\wedge p/2.$$
At the end of the day, we choose
\begin{equation}\label{eq:FandD}
\mathcal{F}:=\hat{\mathcal{F}}^{\beta,n^{-\frac{1}{2\beta+p}}} \quad \text{ and } \quad \mathcal{D}:=\hat{\mathcal{F}}^{\tilde{\beta},n^{-\frac{1}{2\beta+p}}}
\end{equation}
as classes of generators and discriminators. 
{
 As for Model 1, $\mathcal{F}$ is a class of neural networks with one layer of $O(n^d \log(n))$ blocks, each block with width $O(n)$ and depth $O(1)$. Contrary to Model 1, the class of discriminator is now a class of neural networks, with the same dimensions as for $\mathcal{F}$.
}

We are first going to show that the density adversarial estimator \eqref{eq:gandensity} with these classes attains minimax rates for the distance $d_{\mathcal{H}_1^{\tilde{\beta}}}$. Then applying  Corollary~\ref{coro:ineqfulldim}, we will obtain that it also attains minimax rates for all $d_{\mathcal{H}_1^{\gamma}}$, $\gamma\in[1,\tilde{\beta}]$. 

Now to obtain the rate for the distance $d_{\mathcal{H}_1^{\tilde{\beta}}}$, we need to bound the approximation error of the class $\mathcal{D}$ which is defined in the density case as
\begin{align*}
\Delta_\mathcal{D} & := \sup \limits_{f\in \mathcal{F}}\sup \limits_{D^\star \in \mathcal{H}^{\tilde{\beta}}_{1}} \inf\limits_{D\in \mathcal{D}}\int_{\mathbb{R}^p}(D^\star(x)-D(x))(f(x)-f^\star(x))d\lambda^p(x).
\end{align*}
Here, in contrast to the assumptions of Model~\ref{model1}, the target measure $f^\star$ has some regularity from the point of view of the $p$-dimensional Lebesgue measure. Therefore, as illustrated in the following lemma, we can give the $\beta$-regularity of the function $f-f^\star$ to the generators $D^\star-D$ and bound $\Delta_\mathcal{D}$ by the approximation error of the $\tilde{\beta}+\beta$-smooth map $\Gamma^{-\beta}(D^\star)$.

\begin{lemma}\label{lemma:petitlemme} For $\mathcal{D}$ defined in \eqref{eq:FandD} we have
$$\Delta_\mathcal{D} \leq C \sup \limits_{h \in \mathcal{B}^{0}_{\infty,\infty}}  \sup \limits_{D^\star \in \mathcal{H}^{\tilde{\beta}}_{1}} \inf\limits_{D\in \mathcal{D}}\int_{\mathbb{R}^p}(\Gamma^{-\beta}(D^\star)(x)-\Gamma^{-\beta}(D)(x))h(x)d\lambda^p(x).$$
\end{lemma}
The proof of Lemma~\ref{lemma:petitlemme} can be found in Section \ref{sec:lemma:petitlemme}. Putting the bounds of Lemma~\ref{lemma:petitlemme} and Proposition~\ref{prop:propertiesofhat} in Corollary~\ref{coro:boundexpectederrorfulldim}, we obtain that the density adversarial estimator is minimax for the distance $d_{\mathcal{H}_1^{\tilde{\beta}}}$.

\begin{proposition}\label{coro:minimaxfulldim} For $\hat{f}$ the density adversarial estimator \eqref{eq:gandensity} using the classes $\mathcal{F}$ and $\mathcal{D}$ defined in \eqref{eq:FandD}, we have that
    \begin{align*}  
\mathbb{E}\Big[d_{\mathcal{H}_1^{\tilde{\beta}}}(\hat{f},f^\star)\Big] \leq  C\log(n)^{C_2}n^{-\frac{\beta+\tilde{\beta}}{2\beta+p}}.
\end{align*}
\end{proposition}
The proof of Proposition~\ref{coro:minimaxfulldim} can be found in Section \ref{sec:coro:minimaxfulldim}. Using Corollary~\ref{coro:ineqfulldim}, we can now show the final bound for the density adversarial estimator \eqref{eq:gandensity} in the setting of Model~\ref{model3}.

\begin{proposition}\label{prop:minimaxfulldimgamma} For all $f^\star \in \mathcal{H}^{\beta+1}_K(\mathbb{R}^p,\mathbb{R})$, the density adversarial estimator $\hat{f}$ of~\eqref{eq:gandensity} using the classes $\mathcal{F}$ and $\mathcal{D}$ defined in \eqref{eq:FandD} verifies for all $\gamma\in[1,\beta]$
    \begin{align*}  
\mathbb{E}\Big[d_{\mathcal{H}_1^{\gamma}}(\hat{f},f^\star)\Big] \leq  C\log(n)^{C_2}\left(n^{-\frac{\beta+\gamma}{2\beta+p}}\vee n^{-\frac{1}{2}}\right).
\end{align*}
\end{proposition}
The proof of Proposition~\ref{prop:minimaxfulldimgamma} can be found in Section \ref{sec:prop:minimaxfulldimgamma}. We can conclude that the density adversarial estimator attains minimax optimal rates up to logarithmic factors in the density case. Note that we had to distinguish whether $\beta\leq p/2$ or not. 
\begin{itemize}
    \item In the case $\beta<p/2$, the estimator $\hat{f}$ is computed using a discriminator class of regularity $\beta$ and attains minimax rates for all $\gamma \in [1,\beta]$ at the same time.
    \item In the case $\beta\geq p/2$, the estimator $\hat{f}$ is computed using a discriminator class of regularity $p/2$ and attains minimax rates for all $\gamma \in [1,\infty)$ at the same time, as for $\gamma\geq p/2$, $d_{\mathcal{H}_1^{\gamma}}(\hat{f},f^\star)\leq d_{\mathcal{H}_1^{p/2}}(\hat{f},f^\star)$.
\end{itemize}

As mentioned in Section \ref{sec:overview}, the density adversarial estimator \eqref{eq:gandensity} does not provide a direct way of sampling from the approximate density $\hat{f}$. However, one could obtain such estimator by supposing that the target density $f^\star$ is the push forward measure of a known easy to sample density $\nu$ on $\mathbb{R}^p$. In this case, defining the GAN estimator as
\begin{equation*}
    \hat{g} \in \argmin \limits_{g\in \mathcal{G}} \sup \limits_{D\in \mathcal{D}} \frac{1}{n} \sum \limits_{i=1}^n D(X_i)-D(g(Y_i))
\end{equation*}
for $Y_i\sim \nu$ i.i.d. and $\mathcal{G}\subset \mathcal{H}^{\beta+1}_K(\mathbb{R}^p,\mathbb{R}^p)$, one could prove (following the same method developed above) that this estimator attains minimax rates of convergence. We do not develop further on this estimator as the main purpose of this paper is to provide results for the manifold case of Model~\ref{model2}. Having developed a method for the full dimensional case, we adapt it in the following section to the case where the submanifold is of dimension $d<p$.

\section{The manifold case}\label{sec:themanifoldcase}
In this section, we suppose that we are in the setting of Model~\ref{model2}.
Using this hypothesis we have that $g^\star_{\# U}$ admits a density with respect to the volume measure of the manifold $\mathcal{M}_{g^\star}=g^\star(\mathbb{T}^d)$ and that its density $f_{g^\star}$ belongs to $\mathcal{H}^{\beta}_{C}(\mathcal{M}_{g^\star},\mathbb{R})$. In order to adapt the method of the preceding section, we first prove an interpolation inequality similar to Corollary~\ref{coro:ineqfulldim} in this setting. Then, we use it to prove minimax optimality of the GAN estimator.
\subsection{Interpolation inequality for smooth densities on manifolds}\label{sec:interpolation}
In this section we want to prove an analog version of Corollary~\ref{coro:ineqfulldim} but in the context of Model~\ref{model2}. The difficulty comes from the fact that although our push forward maps $g,g^\star:\mathbb{T}^d\rightarrow \mathbb{R}^p$ belong to~$\mathcal{H}^{\beta+1}$, the densities they produce on their respective images do not have regularity from the point of view of the $p$-dimensional Lebesgue measure. We will see that in order to exploit the regularity of the push forward maps, we can split the cost $\mathbb{E}[D(g(U))-D(g^\star(U))]$ into two terms: one characterizing the distance between the support of $g_{\# U}$ and $g^\star_{\# U}$, and the other characterizing their difference in terms of density. Let us first state the desired result and then detail its proof. 
\begin{theorem}\label{theo:theineq}
 Let  $g,g^\star \in \mathcal{H}^{\beta+1}_K(\mathbb{T}^d,\mathbb{R}^p)$ that verify the $K$-manifold regularity condition with $g^\star$ that verifies the $K$-density regularity condition of Definition \ref{defi:manifold regularity}. There exists constants $C,C_{2}>0$ such that if 
        $
     d_{\mathcal{H}_1^{\beta+1}}(g_{\# U},g^\star_{\# U})\leq C^{-1}$, then for all $\epsilon\in(0,1)$ and $\gamma \in [1,\beta+1]$,  we have
    \begin{align*}
        d_{\mathcal{H}_1^{\gamma}}(g_{\# U},g^\star_{\# U})\leq C_2\log(\epsilon^{-1})^4\left( d_{\mathcal{H}_1^{\beta+1}}(g_{\# U},g^\star_{\# U})^{\frac{\beta+\gamma}{2\beta+1}} + \epsilon\right).   
    \end{align*}
\end{theorem} 

Compared to Corollary~\ref{coro:ineqfulldim},  $g_{\# U}$ and $g_{\# U}^\star$ need to be already close in $d_{\mathcal{H}^{\beta+1}_1}$ distance for Theorem \ref{theo:theineq} to apply. This assumption will imply that there exists a diffeomorphism between their supports with several key properties (Proposition~\ref{prop:compatibility}). The precise definition of these properties is gathered into the notion of $(g_{\# U},g^\star_{\# U})$-compatible map (Definition~\ref{defi:compatibility}) and can be found in the appendix Section \ref{sec:compatibility}. In a nutshell, for $t>d_{\mathcal{H}^{\beta+1}_1}(g_{\# U},g^\star_{\# U})$, a map $T:\mathcal{M}_{g^\star}^{t}\rightarrow \mathcal{M}_{g^\star}$ being $(g_{\# U},g^\star_{\# U})$-compatible is a projection onto $\mathcal{M}_{g^\star}$ such that its restriction to $\mathcal{M}_{g}$ is a $C^{\beta+1}$-diffeomorphism.

Let us now prove Theorem~\ref{theo:theineq} through $(g_{\# U},g^\star_{\# U})$-compatibility.
For $D\in \mathcal{H}^{\gamma}_1(\mathbb{R}^p,\mathbb{R})$ and $T$ being $(g_{\# U},g^\star_{\# U})$-compatible, write $$\mathbb{E}_{U\sim \mathcal{U}([0,1]^d)}\Big[D(g(U))-D(g^\star(U))\Big]=A_1+A_2,$$ with
$$A_1:=\mathbb{E}_{U\sim \mathcal{U}([0,1]^d)}\Big[D(g(U))-D(T(g(U)))\Big]$$
characterizing the distance between the supports, and 
$$A_2:=\mathbb{E}_{U\sim \mathcal{U}([0,1]^d)}\Big[D(T(g(U)))-D(g^\star(U))\Big]$$
characterizing the distance between the densities.
Splitting the cost this way allows to finely exploit the regularity of $g$ and $g^\star$:
\begin{itemize}[leftmargin=*]
    \item 
The term $A_2$ writes as an integral of a product of functions defined on the same support $\mathcal{M}_{g^\star}$.
Indeed, denoting by $f_{T_g}$ the density of $(T\circ g)_{\# U}$ with respect to the volume measure $\lambda_{\mathcal{M}_{g^\star}}$, we have $$A_2 =\int_{\mathcal{M}_{g^\star}}D(x)(f_{T_g}(x)-f_{g^\star}(x))d\lambda_{\mathcal{M}_{g^\star}}(x).$$
Since the only support involved is $\mathcal{M}_{g^\star}$, wavelet-based expansions allow to transfer the regularity of the densities on $\mathcal{M}_{g^\star}$ to the discriminator, 
as in the the full-dimensional case (Corollary~\ref{coro:ineqfulldim}).
\item
The term $A_1$ writes similarly, through the Taylor expansion \begin{align*}
A_1=\sum \limits_{i=1}^p \int_{[0,1]^d}\int_0^1\partial_iD(g(u)+t(T(g(u))-g(u)))(T(g(u))-g(u))_idtd\lambda^d(u).
\end{align*}
Here, a similar wavelet-like expansion is made possible using the compatibility properties of $T : \mathcal{M}_g \to \mathcal{M}_{g^\star}$. However, extra technical difficulties arise to deal with the evaluation of a regularised version of $D$ on the ambient points $\bigl( g(u)+t(T(g(u))-g(u) \bigr)_{0 \leq t \leq 1}$, while the wavelet construction naturally takes place on the torus $\mathbb{T}^d$.

\end{itemize}
More formally, the bound on $A_1$ comes from the following result.

\begin{lemma}\label{lemma:firstterm}
   Let $g,g^\star \in \mathcal{H}^{\beta+1}_K(\mathbb{T}^d,\mathbb{R}^p)$ that verify the $K$-manifold regularity condition with $g^\star$ that verifies the $K$-density regularity condition. There exist constants $C,C_{2}>0$ such that, if $
     d_{\mathcal{H}_1^{\beta+1}}(g_{\# U},g^\star_{\# U})\leq C^{-1}
     $, then for all $\epsilon\in(0,1)$ and $\gamma \in [1,\beta+1]$, we have
    \begin{align*}
       d_{\mathcal{H}_1^{\gamma}}((T\circ g)_{\# U},g_{\# U})\leq C_2\log(\epsilon^{-1})^4 \left(\sup \limits_{D \in \mathcal{H}^{\beta+1}_1, D\circ T=0}\mathbb{E}_{U\sim \mathcal{U}([0,1]^d)}[D(g(U))]^{\frac{\beta+\gamma}{2\beta+1}} + \epsilon\right). 
    \end{align*}
\end{lemma} 
The proof of Lemma~\ref{lemma:firstterm} can be found in Section \ref{sec:lemma:firstterm}. 
Note that the right hand side only depends on $g^\star$ through the constraint $D \circ T = 0$, i.e. that $D_{|\mathcal{M}_{g^\star}}=0$. As a result, this term is oblivious to the mass distribution of $g^\star_{\# U}$ over $\mathcal{M}_{g^\star}$, and only measures distances between the support $\mathcal{M}_g$ and $\mathcal{M}_{g^\star}$.

To deduce a bound on $A_1$, observe that
\begin{align*}
    \sup \limits_{D \in \mathcal{H}^{\beta+1}_1,D\circ T=0}\mathbb{E}_{U\sim \mathcal{U}([0,1]^d)}[D(g(U))] & =    \sup \limits_{D \in \mathcal{H}^{\beta+1}_1,D\circ T=0}\mathbb{E}_{U\sim \mathcal{U}([0,1]^d)}[D(g(U))-D(g^\star(U))]\\
    & \leq d_{\mathcal{H}_1^{\beta+1}}(g_{\# U},g^\star_{\# U}),
\end{align*}
which finally leads to
$$A_1\leq C_2\log(\epsilon^{-1})^4 \left(d_{\mathcal{H}_1^{\beta+1}}(g_{\# U},g^\star_{\# U})^{\frac{\beta+\gamma}{2\beta+1}} + \epsilon\right).$$

The term $A_2$ may be bounded using the following result.

\begin{lemma}\label{lemma:secondterm}
 Let $g,g^\star \in \mathcal{H}^{\beta+1}_K(\mathbb{T}^d,\mathbb{R}^p)$ that verify the $K$-manifold regularity condition with $g^\star$ that verifies the $K$-density regularity condition. There exist constants $C,C_{2}>0$ such that if $
     d_{\mathcal{H}_1^{\beta+1}}(g_{\# U},g^\star_{\# U})\leq C^{-1},
     $ then for all $\epsilon\in(0,1)$ and $\gamma \in [1,\beta+1]$, we have
    \begin{align*}
        d_{\mathcal{H}_1^{\gamma}}((T\circ g)_{\# U},g^\star_{\# U})\leq C_2\log(\epsilon^{-1}) \sup \limits_{D\in \mathcal{H}_1^{\beta+1}(\mathcal{M}_{g^\star},\mathbb{R})}\mathbb{E}_{U}[D(T(g(U)))-D(g^\star(U))]^{\frac{\beta+\gamma}{2\beta+1}} + \epsilon. 
        \end{align*}
\end{lemma} 
The proof of Lemma~\ref{lemma:secondterm} can be found in Section \ref{sec:lemma:secondterm}. For $D\in\mathcal{H}^{\beta+1}_{1}(\mathcal{M}_{g^\star},\mathbb{R})$, we have $D\circ T\in \mathcal{H}^{\beta+1}_{C}(\mathcal{M}_{g^\star}^t,\mathbb{R})$. Thus, using Theorem~\ref{whitney}, we can show (like for Proposition~\ref{prop:extensionH}) that $D\circ T$ can be extended globally to a map in $\mathcal{H}^{\beta+1}_{C}(\mathbb{R}^p,\mathbb{R})$. In particular, we have
\begin{align*}
    \sup \limits_{D \in \mathcal{H}^{\beta+1}_1(\mathcal{M}_{g^\star},\mathbb{R})}\mathbb{E}_{U\sim \mathcal{U}([0,1]^d)}[D(T(g(U)))-D(g^\star(U))] & \leq C  d_{\mathcal{H}_1^{\beta+1}}(g_{\# U},g^\star_{\# U}).
\end{align*}
Hence, from Lemma \ref{lemma:secondterm} we deduce that
$$A_2\leq C_2\log(\epsilon^{-1}) d_{\mathcal{H}_1^{\beta+1}}(g_{\# U},g^\star_{\# U})^{\frac{\beta+\gamma}{2\beta+1}} + \epsilon,$$
that concludes the proof of Theorem~\ref{theo:theineq}. 

Just as in the full-dimensional case (Section~\ref{sec:fulldim}), the next section builds upon this interpolation bound to show that GAN estimators attain optimal estimation rates for all the distances $\bigl(d_{\mathcal{H}_1^\gamma}\bigr)_{\gamma \in [1,\beta+1]}$ simultaneously.


\subsection{A tractable GAN estimator in the manifold case}\label{sec:tractablehbeta}
We are now in position to define our classes of generators and discriminators for the GAN estimator in the setting of Model~\ref{model2}. We use the class of generators $\mathcal{G}$ from Section \ref{sec:theoreticalgan} and the class of discriminators $\mathcal{D}$ from Section~\ref{sec:fulldim} (both computable in practice). As in Section \ref{sec:fulldim}, we first show that the GAN estimator attains minimax rates for the distance $d_{\mathcal{H}_1^{\tilde{\beta}+1}}$ (for $\tilde{\beta}$ defined in \eqref{eq:tildebetaplus1}). Then applying Theorem~\ref{theo:theineq}, we obtain that it also attains minimax rates for all $d_{\mathcal{H}_1^{\gamma}}$, $\gamma\in[1,\beta+1]$.

In order to apply Theorem~\ref{theo:theineq} to the GAN estimator $\hat{g}$, we need it to verify the manifold regularity condition of Definition \ref{defi:manifold regularity}. We define in the Appendix Section \ref{sec:numericcondi}, the $\chi$-numerical regularity condition that implies the manifold regularity condition. In a nutshell, it is an easy to compute condition that ensures that if two points are far in the torus, then their images by $\hat{g}$ are
not too close. We choose
\begin{equation}\label{eq:G}
\mathcal{G}:=\{ g \in \hat{\mathcal{F}}^{\beta+1,n^{-\frac{1}{2\beta+d}}}_{per}|\ g \text{ verifies the } \chi\text{-numerical regularity condition}\}
\end{equation}
to be the class of generators. 

As mentioned in Section \ref{sec:theoreticalgan},
the approximation error and covering numbers of the class $\hat{\mathcal{F}}^{\beta+1,n^{-\frac{1}{2\beta+d}}}_{per}$ are well suited for the GAN to attain minimax rates (see Proposition \ref{prop:propertiesofhatper}). Proposition \ref{prop:deltagmodel3} shows that requiring the generators to verify the $\chi$-numerical regularity condition does not deteriorate the approximation error of the class $\mathcal{G}$. Define
\begin{equation}\label{eq:tildebetaplus1}
\tilde{\beta}+1:=(\beta+1)\wedge d/2,
\end{equation}
write $\tilde{\delta}_n := n^{-\frac{1}{2\tilde{\beta}+d}}$, and choose 
\begin{equation}\label{eq:DD}
\mathcal{D}:= \hat{\mathcal{F}}^{\tilde{\beta}+1,\tilde{\delta}_n}
\end{equation}
to be the class of discriminators.
{

Up to the $\chi$-numerical regularity condition (that may be practically enforced using a regularization term, see the remark below Proposition \ref{prop:numericalcondi}), $\mathcal{G}$ and $\mathcal{D}$ are classes of neural networks with the same structure as in Model 2.
}

The specific discriminator regularity $\tilde{\beta}+1$ comes naturally in light of Theorem~\ref{theo:theineq}. Using the same argument as above \eqref{eq:theresonwhy} for the full-dimensional case, $\tilde{\beta} + 1$ is a regularity cap for the discriminator class, over which no gain may be expected in terms of convergence of the GAN (recall that the convergence rate in terms of $d_{\mathcal{H}_1^\gamma}$ cannot be faster than $n^{-1/2}$).

However, the use of this discriminator class raises two problems.
\begin{itemize}
    \item The logarithm of its $1/n$-covering is of order $O(n^{\frac{p}{2\tilde{\beta}+d}})$, which is too large as it depends on the ambient dimension $p$.
    \item It only approximates $\mathcal{H}^{\tilde{\beta}+1}_1$ up to an error $O(n^{-\frac{\tilde{\beta}+1}{2\tilde{\beta}+d}})$ in sup norm, so it is not clear whether we can obtain that $\Delta_{\mathcal{D}}$ is of order $O(n^{-\frac{2\tilde{\beta}+1}{2\tilde{\beta}+d}})$.
\end{itemize}To bypass these problems, we sharpen the bound given by Theorem \ref{theo:boundexpecterror} through a finer analysis of the bias and variance terms of the class $\mathcal{D}$. The quantities $\log(|\mathcal{D}_{1/n}|)$ and $\Delta_\mathcal{D}$ will be replaced using the following
finer (but more involved) quantities.
\begin{itemize}
    \item For $\epsilon\in(0,1)$, let
\begin{equation*}
|(\mathcal{D}_{\mathcal{G}})_\epsilon|:=\sup_{g\in \mathcal{G}}\left|\left(\{D\circ g\ |\ D\in \mathcal{D}\}\right)_\epsilon\right|,
\end{equation*} 
be the largest  $\epsilon$-covering number of the class $\{D\circ g\ |\ D\in \mathcal{D}\}$ among $g\in \mathcal{G}$. 
\item For $g\in \mathcal{G}$, define 
$$D_g^\star\in \argmax \limits_{D\in \mathcal{B}^{\tilde{\beta}+1}_{\infty,\infty}(C)} \mathbb{E}[D(g(U))-D(g^\star(U))]\quad \text{ and }\quad \overline{D}_g\in \argmin \limits_{D\in \mathcal{D}}\|D-D_g^\star\|_{\mathcal{B}^0_{\infty,\infty}},$$
as the optimal discriminator within $\mathcal{B}^{\tilde{\beta}+1}_{\infty,\infty}$ and its approximation by $\mathcal{D}$. 
Let us write 
$$\Delta_{\mathcal{D}}^g:=\mathbb{E}[D_g^\star(g(U))-D_g^\star(g^\star(U))-(\overline{D}_g(g(U))-\overline{D}_g(g^\star(U)))]$$
for the discrimination error between $g_{\# U}$ and $g^\star_{\# U}$ of the class $\mathcal{D}$ compared to the class $\mathcal{B}^{\tilde{\beta}+1}_{\infty,\infty}(C)$.

\end{itemize}

Using these two quantities, Theorem~\ref{theo:boundexpecterror} can be sharpened as follows.

\begin{theorem}\label{theo:boundexpecterror2}  For all $g^\star \in \mathcal{H}^{\beta+1}_K(\mathbb{T}^d,\mathbb{R}^p)$ verifying the $K$-manifold and $K$-density regularity conditions of Definition \ref{defi:manifold regularity}, the GAN estimator $\hat{g}$ of \eqref{WGANS}, with $\mathcal{G}$ and $\mathcal{D}$ defined in~\eqref{eq:G} and~\eqref{eq:DD}, satisfies
\begin{align*}  
 \mathbb{E}\Big[d_{\mathcal{H}_1^{\tilde{\beta}+1}}(\hat{g}_{\# U},g^\star_{\# U})\Big]-\Delta_\mathcal{G}  -\mathbb{E}[\Delta^{\hat{g}}_\mathcal{D}]
 \leq &
C \min \limits_{\delta \in [0,1]}\Biggl\{\sqrt{\frac{(\delta+1/n)^2\log(n|\mathcal{G}_{1/n}| |(\mathcal{D}_{\mathcal{G}})_{1/n}|)}{n}}\\
& +\frac{1}{\sqrt{n}}(1+\delta^{(1-\frac{d}{2(\tilde{\beta}+1)})}+\log(\delta^{-1})\mathds{1}_{\{\beta+1= d/2\}})  \Biggl\}.
\end{align*}
    \end{theorem}
The proof of Theorem \ref{theo:boundexpecterror2} can be found in Section \ref{sec:theo:boundexpecterror2}. Contrary to $\Delta_{\mathcal{D}}$, the term $\mathbb{E}[\Delta^{\hat{g}}_\mathcal{D}]$ depends only on the approximation error of $\mathcal{D}$ for the candidate minimizer $\hat{g}$. On the other hand, in contrast to $\log(|\mathcal{D}_{1/n}|)$, the term $\log(|(\mathcal{D}_{\mathcal{G}})_{1/n}|)$ does not depend on the dimension of the ambient space $p$ as shown in the following result.

\begin{lemma}\label{lemma:coveringd}
    For $g\in \text{Lip}_K(\mathbb{T}^d,[0,1]^p)$, there exists an $\epsilon$-covering $(\mathcal{N}_g)_\epsilon$ of the class $\mathcal{N}_g=\{D\circ g\ |\ D\in \hat{\mathcal{F}}^{\tilde{\beta}+1,\delta}\}$, such that 
    $$\log(|(\mathcal{N}_g)_\epsilon|)\leq C\delta^{-d}\log(\delta^{-1})\log(\epsilon^{-1}).$$
\end{lemma}

The proof of Lemma~\ref{lemma:coveringd} can be found in Section \ref{sec:lemma:coveringd}. 
For $\delta = \tilde{\delta}_n = n^{-\frac{1}{2\tilde{\beta}+d}}$, we deduce that $\log( |(\mathcal{D}_{\mathcal{G}})_{1/n}|)\leq C \log(n)^2 n^{\frac{d}{2\tilde{\beta}+d}}$, so it  scales roughly like $\log( |\mathcal{G}_{1/n}|)$.

Let us now explain how to bound the term $\mathbb{E}[\Delta^{\hat{g}}_\mathcal{D}]$. Compared to Model~\ref{model1}, $g^\star_{\# U}$ has some regularity with respect to a smooth submanifold. Therefore, this regularity can be exploited with Theorem~\ref{theo:theineq} in the following manner.

\begin{lemma}\label{lemma:deltaD}
    There exists a constant $C>0$ such that, if $g\in \mathcal{G}$ verifies the $\chi$-manifold regularity condition and
        $
     d_{\mathcal{H}^{\tilde{\beta}+1}}(g_{\# U},g_{\# U}^\star)\leq C^{-1},$ 
     then we have
    $$\Delta_{\mathcal{D}}^g\leq C_2\log(n)^{C_3}\|\nabla D_g^\star-\nabla \overline{D}_g\|_\infty\Big( d_{\mathcal{H}_1^{\tilde{\beta}+1}}(g_{\# U},g^\star_{\# U})^{\frac{\tilde{\beta}+1}{2\tilde{\beta}+1}}+\frac{1}{n}\Big).$$
\end{lemma}
The proof of Lemma \ref{lemma:deltaD} can be found in Section \ref{sec:lemma:deltaD}. 
It shows that $\mathbb{E}[\Delta^{\hat{g}}_\mathcal{D}]$ depends both on the precision of the approximation of $\mathcal{H}_1^{\tilde{\beta}+1}$ by $\mathcal{D}$, and on the $d_{\mathcal{H}_1^{\tilde{\beta}+1}}$ distance between $\hat{g}_{\# U}$ and $g^\star_{\# U}$.
Therefore, if $\hat{g}_{\# U}$ is close to $g^\star_{\# U}$, the term $\mathbb{E}[\Delta^{\hat{g}}_\mathcal{D}]$ is much smaller than $\Delta_{\mathcal{D}}$.

From Lemmas~\ref{lemma:coveringd} and~\ref{lemma:deltaD}, we obtain that the GAN estimator built with  the neural network   classes $\mathcal{G}$ and $\mathcal{D}$ defined in  \eqref{eq:G} and \eqref{eq:DD} is minimax for the distance $d_{\mathcal{H}^{\tilde{\beta}+1}}$.


\begin{theorem}\label{theo:boundhbeta} For all $g^\star \in \mathcal{H}^{\beta+1}_K(\mathbb{T}^d,\mathbb{R}^p)$ verifying the $K$-manifold and $K$-density regularity conditions of Definition \ref{defi:manifold regularity}, the GAN estimator $\hat{g}$ of \eqref{WGANS}, with $\mathcal{G}$ and $\mathcal{D}$ defined in~\eqref{eq:G} and~\eqref{eq:DD}, satisfies
    \begin{align*}  
\mathbb{E}\Big[d_{\mathcal{H}_1^{\tilde{\beta}+1}}(\hat{g}_{\# U},g^\star_{\# U})\Big] \leq  C \log(n)^{C_2}n^{-\frac{2\tilde{\beta}+1}{2\tilde{\beta}+d}}.
\end{align*}
\end{theorem}

The proof of Theorem~\ref{theo:boundhbeta} can be found in Section \ref{sec:theo:boundhbeta}. 
Finally, combining Theorem~\ref{theo:boundhbeta} with the interpolation bound of Theorem~\ref{theo:theineq}, the main result follows:  over Model~\ref{model2}, GAN estimators attain minimax optimal rates for all the distances $\bigl(d_{\mathcal{H}_1^\gamma}\bigr)_{\gamma \in [1,\beta+1]}$ simultaneously. 
\begin{theorem}\label{theo:boundhallgammas} For all $g^\star \in \mathcal{H}^{\beta+1}_K(\mathbb{T}^d,\mathbb{R}^p)$ verifying the $K$-manifold and $K$-density regularity conditions of Definition \ref{defi:manifold regularity}, the GAN estimator $\hat{g}$ \eqref{WGANS}, with $\mathcal{G}$ and $\mathcal{D}$ defined in \eqref{eq:G} and \eqref{eq:DD}, satisfies, for all $\gamma\in[1,\tilde{\beta}+1]$,
    \begin{align*}  
\mathbb{E}\Big[d_{\mathcal{H}_1^\gamma}(\hat{g}_{\# U},g^\star_{\# U})\Big] \leq  C \log(n)^{C_2}\left(n^{-\frac{\beta+\gamma}{2\beta+d}}\vee n^{-\frac{1}{2}}\right).
\end{align*}
\end{theorem}

As in the full dimensional case, we have:
\begin{itemize}
    \item When $\beta+1<d/2$,  $\hat{g}$ is computed using a discriminator class of regularity $\beta+1$ and attains minimax rates up to logarithmic factors for all $\gamma \in [1,\beta+1]$ at the same time.
    \item When $\beta+1\geq d/2$,  $\hat{g}$ is computed using a discriminator class of regularity $d/2$ and attains minimax rates up to logarithmic factors for all $\gamma \in [1,\infty)$ at the same time.
\end{itemize}
To the best of our knowledge, the GAN estimator \eqref{WGANS} is the first tractable estimator attaining minimax rates up to logarithmic factors in the manifold setting. It is also the first one to be proved to be minimax for several IPMs simultaneously. On the other hand, the fact that we attain minimax rates shows in particular that Theorem~\ref{theo:theineq} is sharp up to the logarithmic factors.

In \cite{liang2021generative}, the author shows that in the case $\gamma<d/2$, if we take $\mathcal{D}=\mathcal{H}^\gamma_1$ as the discriminator class for the GAN estimator \eqref{WGANS}, then the rate of error is $O(n^{-\frac{\gamma}{d}})$ which is strictly slower than the minimax. This is due to the fact that the class $\mathcal{H}^\gamma_1$ is so massive that it is able to well discriminate between empirical and smooth measures. Therefore, it will be more sensitive to the distance between the data itself than the underlying measures behind it. This is well illustrated by the fact that for any measure $\nu$ with a smooth density with respect to a closed $d$-dimensional manifold $\mathcal{M}\subset \mathbb{R}^p$, we have (Lemma 13 in \cite{tang2022minimax})
$$C_1\log(n)^{-1}n^{-\frac{\gamma}{d}}\leq \mathbb{E}_{X_i,Y_i \sim \nu}\Big[d_{\mathcal{H}^\gamma_1}( \frac{1}{n}\sum \limits_{i=1}^n \delta_{X_i},\frac{1}{n}\sum \limits_{i=1}^n \delta_{Y_i})\Big]\leq C_2 n^{-\frac{\gamma}{d}}.$$
Although the $X_i$'s and $Y_i$'s are i.i.d. with the same distribution, the $d_{\mathcal{H}^\gamma_1}$ distance between their empirical distribution is strictly greater than the minimax rate. To overcome this problem, \cite{liang2021generative} and \cite{tang2022minimax} use as estimator
$$\hat{\mu} \in \argmin \limits_{\nu \in S}  \sup_{D \in \mathcal{H}^\gamma_1} \mathbb{E}_{\nu}[D(X)]-\hat{\mathcal{J}}(D),$$
where $\hat{\mathcal{J}}$ a regularization of the integration of $D$ under the empirical measure $\mu_n=\frac{1}{n}\sum \limits_{i=1}^n \delta_{X_i}$. Theorem~\ref{theo:boundhallgammas} shows that actually, it is not necessary to regularize $\mu_n$ to obtain minimax rates. Taking a discriminator class $\mathcal{D}$ that is much smaller than the Hölder functions suffices to
implicitly regularize the measures, as this class is not able to distinguish between discrete and smooth measures.
This result gives a theoretical explanation of the well known fact that WGANs work because the discriminator class fails to approximate the Wasserstein distance as studied in \cite{stanczuk2021wasserstein}.

\section*{Discussion}
In this paper, we provided minimax rates of convergence of the GAN estimator in a general setting. This is a step towards a theoretical understanding of why Generative Adversarial Networks obtain state of the art results in various fields. On the other hand,
it shows that it is not necessary to regularize the empirical measure in the GAN estimator \eqref{WGANS} (like in \cite{liang2021generative} and \cite{tang2022minimax}) to obtain minimax rates.

Our analysis suggests that if the target measure has some higher Hölder regularity than the class used in the IPM, one should use discriminators having the same regularity than the target. This is supported by our inequality between the discrimination of different regularities (Theorem~\ref{theo:theineq}). It shows that the Hölder class of the same regularity than the target distinguishes just as efficiently as the classes of weaker regularity (in particular the Wasserstein).

 Our results strongly highlight the use of wavelets to describe function's regularity.  From a theoretical point of view, wavelets are only suitable tools to get interpolation inequalities such as Theorem \ref{theo:theineq}. On a practical side, they allow to easily enforce regularities of both generators and discriminators.
Provided that such high-order regularity constraints can be made tractable on general neural networks, other architectures than the ones we propose in Sections \ref{sec:tractable} and \ref{sec:themanifoldcase}  might be proved optimal.
On the other hand, any class of functions able to approximate Daubechies wavelets well would work. 
Nonetheless, approximating Daubechies wavelets is not an easy task (see, e.g., \cite{daubechies1991two}), a using neural networks for this purpose seems to be a relevant strategy (see, e.g., \cite{Daubechies23}).


Overall, this suggests that a "Generative Adversarial Wavelet" estimator would be perfectly suited to estimate smooth densities on arbitrary manifolds, and that using neural networks may be a relevant way to implement it. Future research will be focused on extending our analysis of this type of estimator and on numerical experiments to see how it performs compared to the state of the art. Furthermore, we could also consider moving away from the classical GAN setting with push-forward measures and extend our results to the case of manifolds with several charts.

\begin{supplement}
\stitle{Supplement to "Wasserstein Generative Adversarial Networks are Minimax Optimal Distribution Estimators"}
\sdescription{
The supplement contains auxiliary results and the neural network construction, together with all the proofs.}
\end{supplement}
\bibliographystyle{plainnat}
\bibliography{adversarial_training}

%% file: supplementary.tex
\appendix
\section{Additional notation and Neural Network construction}\label{sec:appendix}
In this section, we define some additional notation needed for the poofs and we detail the construction of the Neural Networks used by our classes of generators and discriminators.

\subsection{Additional notation} We will write $\text{Lip}_1$ the set of 1-Lipschitz functions. The support of a function $f$ will be denoted by $supp(f)$. For a path $\gamma:[0,1]\rightarrow \mathbb{T}^d$, we write $\dot{\gamma}$ its differential. For $u,v\in \mathbb{T}^d$ in the $d$-dimensional torus, we write $\|u-v\|=\min \{\int_0^1\|\dot{\gamma}(t)\|dt \ | \ \gamma:[0,1]\rightarrow \mathbb{T}^d,\gamma(0)=u,\gamma(1)=v\}$ as the periodised Euclidean norm. The $d$-dimensional ball of radius $\epsilon>0$ in the torus for this distance will be written $B^d(0,\epsilon)$.
The Hausdorff distance between two subset $A,B\subset \mathbb{R}^p$ will be denoted $\mathbb{H}(A,B)$. For $x$ a point belonging either to $\mathbb{R}^p,\mathbb{T}^d$ or $\mathbb{R}$, we will write $O(x)$ to denote a quantity that is bounded by $C\|x\|$ for $C>0$ independent of $x$.

To lighten the notation, for a function $f:\mathbb{R}^p\rightarrow \mathbb{R}$ that admits a wavelet expansion in $L^2$:
$$
f(x)=\sum \limits_{w\in \mathbb{Z}^p} \alpha_f(w)\psi_{0w}(x) + \sum \limits_{j=0}^\infty \sum \limits_{l=1}^{2^p-1}\sum \limits_{w\in \mathbb{Z}^p} \alpha_f(j,l,w)\psi_{jlw}(x),
$$

we will write
$$
f(x)= \sum \limits_{j=0}^\infty \sum \limits_{l=1}^{2^p}\sum \limits_{w\in \mathbb{Z}^p} \alpha_f(j,l,w)\psi_{jlw}(x),
$$
with the convention that $\psi_{02^pw}=\psi_{0w}$ and for all $j\geq 1$, $\psi_{j2^pw}=0$. Likewise, the wavelet expansion of a map $f:\mathbb{T}^d\rightarrow \mathbb{R}$ will be written
$$
f(u)= \sum \limits_{j=0}^\infty \sum \limits_{l=1}^{2^d}\sum \limits_{z\in \{0,...,2^j-1\}^d} \alpha_f(j,l,w)\psi_{jlz}^{per}(u),
$$
with the convention that $\psi_{02^d0}=1$ and for all $z\neq 0$, $j\geq 1$, $\psi^{per}_{02^dz}=\psi^{per}_{j2^dz}=\psi^{per}_{j2^d0}=0$.

Recalling the definition of Besov spaces from Section \ref{sec:wavelets}, in the case $q_1=q_1=\infty$, the norm $\|\cdot\|_{\mathcal{B}^{s,b}_{q_1,q_2}}$ of a function $f:\mathbb{R}^p\rightarrow \mathbb{R}$ value 
\begin{align*}
\|f\|_{\mathcal{B}^{s,b}_{\infty,\infty }}= &\sup \limits_{j\geq 0} 2^{j(s+p/2)}(1+j)^{b} \sum \limits_{l=1}^{2^p} \sup \limits_{w\in \mathbb{Z}^p} |\alpha_f(j,l,w)|.
\end{align*}
Likewise, for a function $f:\mathbb{T}^d\rightarrow \mathbb{R}$ we have 
\begin{align*}
\|f\|_{\mathcal{B}^{s,b}_{\infty,\infty }}= &\sup \limits_{j\geq 0} 2^{j(s+p/2)}(1+j)^{b} \sum \limits_{l=1}^{2^d} \sup \limits_{z\in \{0,...,2^j-1\}^d} |\alpha_f(j,l,z)|.
\end{align*}
\subsection{Approximation of low wavelet frequencies by Neural Networks}\label{sec:appendixnn}

\subsubsection{Theoretical construction of the approximated wavelet}\label{sec:appendixnn1}
Let us describe how to approximate the class $\mathcal{F}^{\eta,\delta}_{per}$ of \eqref{eq:fdeltaper} by neural networks. The first ingredient we need, is to have a good approximation $\hat{\phi}$ of the Daubechies scaling function $\phi\in \mathcal{H}^{\lfloor \beta \rfloor+3}(\mathbb{R},\mathbb{R})$ \citepproofs{daubechies1988orthonormal}, that will then give us a good approximation of the wavelet function \begin{equation}\label{eq:defpsi}
\psi=\sum \limits_{k=-\infty}^{+\infty}\lambda_k\phi(2\cdot-k).
\end{equation} The classical technique to approximate $\phi$ is to use that it is solution to the two-scale difference equation \citepproofs{daubechies1991two}:
\begin{equation}\label{eq:scaling}
    \phi(x)=\sum \limits_{k=-\infty}^{+\infty}h_k\phi(2x-k)
    \end{equation}
for $h_k=2\int \phi(y)\phi(2y-k)dy$. From Theorem 4.2.10 in \citepproofs{giné_nickl_2015}, there exists $\phi$ with compact support in $[0,N]$ that belongs to $\mathcal{H}^{\lfloor \beta \rfloor+3}(\mathbb{R},\mathbb{R})$ with $N\leq 12(\lfloor \beta \rfloor+3)$.  Then~\eqref{eq:scaling} can be rewritten as
\begin{equation}\label{eq:scaling2}
    \phi(x)=\sum \limits_{k=0}^{N}h_k\phi(2x-k).
    \end{equation}
Using this property, the classical methods to approximate $\phi$ are iterative schemes as the "cascade algorithm" \citepproofs{daubechies1991two}.   Here we use neural networks in order to connect with the "Network" part of the GAN. Any neural network architecture that allows to approximate $\phi$ and its derivatives to an arbitrary precision would work. For example we can use the tanh neural
networks construction of \cite{De_Ryck_2021}.
\begin{lemma}[Theorem 5.1, \cite{De_Ryck_2021}]\label{lemma:Deryck}
    Let $k\in \mathbb{N}$ and  $L>0$, there exists a class of neural networks $\mathcal{T}_{L}^{k} \subset C^\infty([0,1]^d, \mathbb{R}^p)$ with tanh activation function that is made of two hidden layers of width $CL$ and $CL^d$ respectively, such that for any $f\in \mathcal{H}_1^{k}([0,1]^d,B^p(0,1))$, there exists a neural network $h_f\in \mathcal{T}_{L}^{k}$ such that $\forall l \in \{0,..., k-1\} $ we have
$$\|f-h_f\|_{\mathcal{H}^l} \leq C \frac{1}{L^{k-l}}.$$
\end{lemma}

The same way, we can also use the ReQU construction of \citepproofs{belomestny2022simultaneous}. \black

\begin{lemma}[Theorem 2, \citeproofs{belomestny2022simultaneous}]\label{lemma:Belo} For all $\alpha > 2$ and $L>0$, there exists a class of neural networks $\mathcal{R}_{L}^{\alpha} \subset \mathcal{H}^{\lfloor \alpha \rfloor }([0,1]^d, \mathbb{R}^p)$ with ReQu activation function that is of width $CL^d$ and depth $C$ such that for any $f\in \mathcal{H}_1^{\alpha}([0,1]^d,B^p(0,1))$, there exists a neural network $h_f\in \mathcal{R}_{L}^{\alpha}$ such that $\forall l \in \{0,..., \lfloor \alpha \rfloor\} $ we have
$$\|f-h_f\|_{\mathcal{H}^l} \leq C \frac{1}{L^{\alpha-l}}.$$
\end{lemma}

In \citeproofs{belomestny2023rates} a bound on the covering number of the neural network class  $\mathcal{R}_{L}^{\alpha}$ is also provided. We do not use it as it depends on the number of parameters of the network which is to large here to obtain minimax rates. We instead give a smaller covering using the wavelet decomposition of the functions (Proposition \ref{prop:propertiesofhatper}).

Fix $L>0$ and define   $\hat{\phi}\in \mathcal{T}_{L}^{\lfloor \beta \rfloor+2}$ (or $\hat{\phi}\in \mathcal{R}_{L}^{\lfloor \beta \rfloor+2}$ if we use ReQU instead of tanh) \black as the neural network from   Lemma \ref{lemma:Deryck} (or  Lemma \ref{lemma:Belo} in the ReQU case) \black such that  
\begin{equation}\label{eq:approxipsi}
\|\phi-\hat{\phi}\|_{\mathcal{H}^{\lfloor \beta \rfloor +2}} \leq C L^{-1}.
\end{equation}
Recalling \eqref{eq:defpsi}, define also
\begin{equation}\label{eq:approxipsiphi}
\hat{\psi}=\sum \limits_{k=-N+1}^{1}\lambda_k\hat{\phi}(2\cdot-k),
\end{equation}
with $\lambda_k=(-1)^{k+1}h_{1-k}$ for the approximation by neural networks of the wavelet function. Using these neural networks, let us approximate the  periodised wavelets from \eqref{eq:periodicwav}.
Recalling that $\phi$ has support in $[0,N]$ and $\psi$ in $[-N/2+1/2,N/2+1/2]$, then for $s\in[0,1]$ and $z\ \in\{0,...,2^j-1\}$ we have 
\begin{equation}\label{eq:defpsiper}
    \phi^{per}_j(s-2^{-j}z)=2^{j/2}\sum \limits_{k=\lceil -1-2^{-j}N\rceil}^{0}\phi(2^j(s-k)-z)
\end{equation}
and 
\begin{equation}\label{eq:defphiper}
\psi^{per}_j(s-2^{-j}z)=2^{j/2}\sum \limits_{k=\lceil -1-2^{-(j+1)}(N+1)\rceil}^{\lfloor 1+2^{-(j+1)}(N-1) \rfloor}\psi(2^j(s-k)-z).
\end{equation}
Note that if $2^j> N$ then 
\begin{equation}\label{eq:periosca}
    \phi^{per}_j(s-2^{-j}z)=\phi(2^js-z)+\phi(2^j(s+1)-z)
\end{equation}
and
\begin{equation}\label{eq:periowav}
    \psi^{per}_j(s-2^{-j}z)=\psi(2^j(s-1)-z)+\psi(2^js-z)+\psi(2^j(s+1)-z).
\end{equation}
Define the periodised wavelet approximations
$$\hat{\phi}^{per}_j(s)=2^{j/2}\sum \limits_{k=\lceil -1-2^{-j}N\rceil}^{0}\hat{\phi}(2^j(s-k))$$ and
$$
\hat{\psi}^{per}_j(s)=2^{j/2}\sum \limits_{k=\lceil -1-2^{-(j+1)}(N+1)\rceil}^{\lfloor 1+2^{-(j+1)}(N-1) \rfloor}\hat{\psi}(2^j(s-k)).$$
for $\hat{\phi}$ and $\hat{\psi}$ from \eqref{eq:approxipsi} and \eqref{eq:approxipsiphi} .  For ease of notation $\hat{\phi}^{per}$ and $\hat{\psi}_j^{per}$ will be written $\hat{\psi}^{per}_{j0},\hat{\psi}^{per}_{j1}$ respectively. We finally define the multi-dimension approximation of the wavelet as 
\begin{equation}\label{eq:defaprroxphiper}
\hat{\psi}^{per}_{02^d0}(u) = 1 \ \text{ , } \ \hat{\psi}^{per}_{jlz}(u) = \prod \limits_{i=1}^d \hat{\psi}^{per}_{jl_i}(u_i-2^{-j}z_i) \ l=1,...,2^d, \ z\in \{0,...,2^j-1\}^d. 
\end{equation} 
 We use the convention that for all $z\neq 0$, $j\geq 1$, $\hat{\psi}^{per}_{02^dz}=\hat{\psi}^{per}_{j2^dz}=\hat{\psi}^{per}_{j2^d0}=0$.   From Lemma 1 in \cite{belomestny2022simultaneous}, the product $(x_1,...,x_d)\mapsto \prod_{i=1}^d x_i$ can be exactly implemented by a REQU neural network with finite width and depth. On the other hand, from Corollary 3.7 in \cite{De_Ryck_2021}, for all $\epsilon>0$ there exists a one layer tanh neural network with finite width that implements the product $(x_1,...,x_d)\mapsto \prod_{i=1}^d x_i$ up to an error $\epsilon$.
\black 

Now taking a function $f=\sum \limits_{j=0}^{\log_2(\delta^{-1})} \sum \limits_{l=1}^{2^d} \sum \limits_{z\in \{0,...,2^j-1\}^d} \hat{\alpha}_f(j,l,z) \hat{\psi}_{jlz}^{per}$, the next proposition states that the true wavelet coefficients $\alpha_f(j,l,z)$ are close to its approximate coefficients $\hat{\alpha}_f(j,l,z)$. 

\begin{proposition}\label{prop:approxhatper}
For any $f\in L^2(\mathbb{T}^d,\mathbb{R})$, if there exists a sequence $(\hat{\alpha}_f(j,l,z))$ with $|\hat{\alpha}_f(j,l,z)|\leq C2^{-j(\eta+d/2)}$ such that $$f=\sum \limits_{j=0}^{\log_2(\delta^{-1})} \sum \limits_{l=1}^{2^d} \sum \limits_{z\in \{0,...,2^j-1\}^d} \hat{\alpha}_f(j,l,z) \hat{\psi}_{jlz}^{per},$$
then $\forall j\leq \log_2(\delta^{-1})$ we have
$$
|\alpha_f(j,l,z)-\hat{\alpha}_f(j,l,z)|\leq CL^{-1}\big(2^{-j(\lfloor \beta \rfloor +2+d/2)}\delta^{-(\lfloor \beta \rfloor +2-\eta)}\log_2(\delta^{-1})+2^{-j(\eta+d/2)}\big)$$
and $\forall j> \log_2(\delta^{-1})$ we have 
$$
|\alpha_f(j,l,z)|\leq CL^{-1}2^{-j(\lfloor \beta \rfloor +2+d/2)}\delta^{-(\lfloor \beta \rfloor +2-\eta)}\log_2(\delta^{-1}).
$$
\end{proposition}

The proof of Proposition \ref{prop:approxhatper} follows the same arguments than the proof of the following proposition.

\begin{proposition}\label{prop:approxhat}
For any $f\in L^2(\mathbb{R}^p,\mathbb{R})$, if there exists a sequence $(\hat{\alpha}_f(j,l,w))$ with $|\hat{\alpha}_f(j,l,w)|\leq C2^{-j(\eta+p/2)}$ such that $$f=\sum \limits_{j=0}^{\log_2(\delta^{-1})} \sum \limits_{l=1}^{2^p} \sum \limits_{w\in \{-K2^{j},...,K2^j\}^p} \hat{\alpha}_f(j,l,w) \hat{\psi}_{jlw},$$
then $\forall j\leq \log_2(\delta^{-1})$ we have
$$|\alpha_f(j,l,w)-\hat{\alpha}_f(j,l,w)|\leq CL^{-1}(2^{-j(\lfloor \beta \rfloor +2+p/2)}\delta^{-(\lfloor \beta \rfloor +2-\eta)}\log_2(\delta^{-1})+2^{-j(\eta+p/2)})$$
and $\forall j> \log_2(\delta^{-1})$ we have 
$$|\alpha_f(j,l,w)|\leq CL^{-1}2^{-j(\lfloor \beta \rfloor +2+p/2)}\delta^{-(\lfloor \beta \rfloor +2-\eta)}\log_2(\delta^{-1}).$$
\end{proposition}

\begin{proof}\label{sec:prop:approxhat}

Let $(j,l,w)\in \{1,...,\log(\delta^{-1})\}\times \{1,...,2^p\}\times \{-K2^{j},...,K2^j\}^p$, we have 
\begin{align*}
    \|\hat{\psi}_{jlw}-\psi_{jlw}\|_{\mathcal{H}^{\lfloor \beta \rfloor +2}} & \geq \|\hat{\psi}_{jlw}-\psi_{jlw}\|_{\mathcal{B}^{\lfloor \beta \rfloor +2}_{\infty,\infty}}\\
    & = \sup \limits_{(j^{'},l^{'},w^{'})}2^{j^{'}(\lfloor \beta \rfloor +2+p/2)}|\alpha_{\hat{\psi}_{jlw}}(j^{'},l^{'},w^{'})-\mathds{1}_{\{(j^{'},l^{'},w^{'})=(j,l,w)\}}|.
\end{align*}
Furthermore
\begin{align*}
    \|& \hat{\psi}_{jlw}-\psi_{jlw}\|_{\mathcal{H}^{\lfloor \beta \rfloor +2}} \\
    &=2^{jp/2}\|\prod \limits_{i=1}^p \hat{\psi}_{l_i}(2^j\cdot_i-w_i)-\prod \limits_{i=1}^p \psi_{l_i}(2^j\cdot_i-w_i)\|_{\mathcal{H}^{\lfloor \beta \rfloor +2}}\\
    & =2^{jp/2}\|\sum \limits_{i=1}^p (\hat{\psi}_{l_i}(2^j\cdot_i-w_i)-\psi_{l_i}(2^j\cdot_i-w_i))\prod \limits_{r>i}^p \hat{\psi}_{l_r}(2^j\cdot_r-w_r)\prod \limits_{s<i}^p \psi_{l_s}(2^j\cdot_s-w_s)\|_{\mathcal{H}^{\lfloor \beta \rfloor +2}}\\
    & \leq 2^{j(\lfloor \beta \rfloor +2+p/2)}\sum \limits_{i=1}^p \| (\hat{\psi}_{l_i}(\cdot_i-w_i)-\psi_{l_i}(\cdot_i-w_i))\prod \limits_{r>i}^p \hat{\psi}_{l_r}(\cdot_r-w_r)\prod \limits_{s<i}^p \psi_{l_s}(\cdot_s-w_s)\|_{\mathcal{H}^{\lfloor \beta \rfloor +2}}\\
    & \leq C 2^{j(\lfloor \beta \rfloor +2+p/2)}L^{-1},
\end{align*}
recalling that $\|\hat{\psi}_{l_i}-\psi_{l_i}\|_{\mathcal{H}^{\lfloor \beta \rfloor +2}} \leq CL^{-1}$ for $l_i\in \{0,1\}$.
We deduce that for all $(j^{'},l^{'},w^{'})\neq(j,l,w)$ we have
$$
|\alpha_{\hat{\psi}_{jlw}}(j^{'},l^{'},w^{'})|\leq C2^{(j-j^{'})(\lfloor \beta \rfloor +2+p/2)}L^{-1}
$$
and 
$$
|\alpha_{\hat{\psi}_{jlw}}(j,l,w)-1|\leq CL^{-1}.
$$
Let $(j^{'},l^{'},w^{'})\in \mathbb{N}_0^\star\times \{1,...,2^p\}\times \mathbb{Z}^p$, as $$supp(\hat{\psi}_{jlw}),supp(\psi_{jlw})\subset \bigotimes \limits_{i=1}^p [2^{-j}(w_i-N),2^{-j}(w_i+N)],$$
we have that if 
$$\bigotimes \limits_{i=1}^p [2^{-j}(w_i-N),2^{-j}(w_i+N)]\bigcap \bigotimes \limits_{i=1}^p [2^{-j^{'}}(w_i^{'}-N),2^{-j^{'}}(w_i^{'}+N)] = \varnothing, $$ then 
$$\langle \hat{\psi}_{jlw},\psi_{j^{'}l^{'}w^{'}}\rangle=0.$$

Let
$$f=\sum \limits_{j=0}^{\log_2(\delta^{-1})} \sum \limits_{l=1}^{2^p} \sum \limits_{w\in \{-K2^{j},...,K2^j\}^p} \hat{\alpha}_f(j,l,w) \hat{\psi}_{jlw} \in \hat{\mathcal{F}}^{\eta,\delta},$$
taking $\hat{\alpha}_f(j,l,w):=0$ for $j\geq \log_2(\delta^{-1})$, we have 
\begin{align*}
    |&\alpha_f(j^{'},l^{'},w^{'})-\hat{\alpha}_f(j^{'},l^{'},w^{'})| \\
    = & | \sum \limits_{j=0}^{\log_2(\delta^{-1})} \sum \limits_{l=1}^{2^p} \sum \limits_{w\in \{-K2^{j},...,K2^j\}^p} \hat{\alpha}_f(j,l,w) \langle\hat{\psi}_{jlw},\psi_{j^{'}l^{'}w^{'}}\rangle-\hat{\alpha}_f(j^{'},l^{'},w^{'})|\\
     \leq & | \sum \limits_{(j,l,w)\neq (j^{'},l^{'},w^{'})} |\hat{\alpha}_f(j,l,w)\langle \hat{\psi}_{jlw},\psi_{j^{'}l^{'}w^{'}}\rangle|\\
    & +|\hat{\alpha}_f(j^{'},l^{'},w^{'})||\langle \hat{\psi}_{j^{'}l^{'}w^{'}} , \psi_{j^{'}l^{'}w^{'}}\rangle-1|\mathds{1}_{\{j^{'}\leq\log_2(\delta^{-1})\}}\\
     \leq & \sum \limits_{\substack{(j,l,w)\neq (j^{'},l^{'},w^{'})\\ supp(\hat{\psi}_{jlw})\cap supp(\psi_{j^{'}l^{'}w^{'}})\neq \varnothing}} C2^{-j(\eta+p/2)}2^{(j-j^{'})(\lfloor \beta \rfloor +2+p/2)}L^{-1}\\
     & + C2^{-j^{'}(\eta+p/2)}L^{-1}\mathds{1}_{\{j^{'}\leq\log_2(\delta^{-1})\}}\\
     = & CL^{-1}2^{-j^{'}(\lfloor \beta \rfloor +2+p/2)}  \sum \limits_{\substack{(j,l,w)\in \{1,...,\log_2(\delta^{-1})\}\times \{1,...,2^p\}\times \{-K2^{j},...,K2^j\}^p\\ supp(\hat{\psi}_{jlw})\cap supp(\psi_{j^{'}l^{'}w^{'}})\neq \varnothing}} 2^{j(\lfloor \beta \rfloor +2-\eta)}\\
     & + C2^{-j^{'}(\eta+p/2)}L^{-1}\mathds{1}_{\{j^{'}\leq\log_2(\delta^{-1})\}}\\
    \leq & CL^{-1}2^{-j^{'}(\lfloor \beta \rfloor +2+p/2)}  \sum \limits_{j\in \{1,...,\log_2(\delta^{-1})\}} 2^{j(\lfloor \beta \rfloor +2-\eta)}\\
     & + C2^{-j^{'}(\eta+p/2)}L^{-1}\mathds{1}_{\{j^{'}\leq\log_2(\delta^{-1})\}}\\
     \leq & CL^{-1}(2^{-j^{'}(\lfloor \beta \rfloor +2+p/2)}\delta^{-(\lfloor \beta \rfloor +2-\eta)}\log_2(\delta^{-1})+2^{-j^{'}(\eta+p/2)}\mathds{1}_{\{j^{'}\leq\log_2(\delta^{-1})\}}).
\end{align*}
\end{proof}

\subsubsection{Implementation of the class $\hat{\mathcal{F}}^{\beta+1,n^{-\frac{1}{2\beta+d}}}_{per}$ \eqref{eq:hatfper}}\label{sec:appendixnnopti}
Let us first give a method to find a   tanh (or ReQU) \black  neural network $\hat{\phi}$ from \eqref{eq:approxipsi} that approximates the true scaling function at a precision $n^{-1}$. This method is largely inspired by the Section 4 of \citeproofs{daubechies1991two}. The idea is to use an iterative scheme based on the two-scale difference equation \eqref{eq:scaling2}. The unique solution $\phi$ of this equation, is a fixed point fixed point $Vf=f$ of the linear operator 
\begin{equation}\label{eq:scaling3}
    Vf(x)=\sum \limits_{k=0}^{N}h_kf(2x-k).
\end{equation}  
Take a function $f_0\in \mathcal{H}^{\lfloor \beta \rfloor +3}(\mathbb{R},\mathbb{R})$ such that $f_0^{(l)}(k)=\phi^{(l)}(k)$ for all  $k\in\{0,1,...,N\}$ and $l\in \{0,...,\lfloor \beta \rfloor +3\}$. Then, writing $V^jf_0$ for the application of the operator $V$ a total of $j$ times on the function $f_0$, Theorem $4.2$ in \citeproofs{daubechies1991two} states that:
$$\text{For }f_j=V^jf_0, \text{ we have } \|f_j^{(l)}-\phi^{(l)}\|_\infty\leq C 2^{-j(\lfloor \beta \rfloor +3-l)}, \forall l\in \{0,...,\lfloor \beta \rfloor +3\}.$$ 
We use   a smooth neural networks (either with tanh or ReQU activation functions) \black to approximate the function $f_0$. Based on the data $(k,\phi^{(l)}(k))_{k\in \{0,...,N\}}$ (\citeproofs{daubechies1992two} give a precise method to compute the values $\phi^{(l)}(k)$), find a   tanh (or ReQU) \black  neural network $\hat{f}_0 \in \mathcal{R}^{\lfloor \beta \rfloor +3}_{CN}$ using classical gradient descent optimization on the weights of the network,  such that
\begin{equation}\label{eq:optirequ}
\sum \limits_{l=0}^{\lfloor \beta \rfloor +3}\sum \limits_{k=0}^N |\hat{f}_0^{(l)}(k)-\phi^{(l)}(k)|\leq n^{-\log_2(C_h)-1},
\end{equation}
for $C_h=\max \limits_{k\in \{0,...,N\}} |h_k| \in (1,2)$. Then, the neural network $\hat{f}_j=V^j \hat{f}_0$, which is a linear combination of scalings and translations of $\hat{f}_0$, verifies 
$$
\|\hat{f}_j^{(l)}-\phi^{(l)}\|_\infty\leq C\left( 2^{-j(\lfloor \beta \rfloor +3-l)}+C_h^jn^{-\log_2(C_h)-1}\right), \forall l\in \{0,...,\lfloor \beta \rfloor +3\}.
$$
Finally, for $j=\log_2(n)$ we obtain a   tanh (or ReQU) \black  neural network $\hat{\phi}=\hat{f}_j$ such that
$$\|\phi-\hat{\phi}\|_{\mathcal{H}^{\lfloor \beta \rfloor +2}} \leq C n^{-1}.$$

Note that as we need a smooth approximation of $\phi$, we cannot use classical ReLU neural networks as it will not give us a function that also approximates the derivatives of $\phi$.
Having the approximation $\hat{\phi}$ of the scaling function, we can compute the approximated periodised wavelets $\hat{\psi}^{per}_{jlz}$ as described in the precedent Section \eqref{eq:defaprroxphiper}. The class $\hat{\mathcal{F}}^{\beta+1,n^{-\frac{1}{2\beta+d}}}_{per}$ \eqref{eq:hatfper} can then be implemented as the set of neural networks being linear combinations of the  $\hat{\psi}^{per}_{jlz}$, for $j\in \{0,...,\log_2(n^{\frac{1}{2\beta+d}})\}$, $l\in \{1,...,2^d\}$ and $z\in \{0,...,2^j-1\}^d$. 

Let us sum up the different steps of the construction and then detail their tractability. 
\begin{itemize}
    \item[1.] Find by gradient descent a   tanh (or  ReQU) \black neural network $\hat{f}_0 \in \mathcal{R}^{\lfloor \beta \rfloor +3}_{CN}$ verifying \eqref{eq:optirequ}.
    \item[2.] Compute the neural network $$\hat{f}_j=V^j \hat{f}_0=\sum \limits_{k_1,...,k_j\in \{0,...,N\}^j}h_{k_1}...h_{k_j}\ \hat{f}_0(2^jx-\sum \limits_{i=1}^{j}2^{i-1}k_i)$$
    for $j=\log_2(n).$
    \item[3.] Build $p$ neural networks $f_i$ of the shape $f_i=\sum \limits_{j=0}^{\log_2(n^{\frac{1}{2\beta+d}})} \sum \limits_{l=1}^{2^d} \sum \limits_{z\in \{0,...,2^j-1\}^d} \hat{\alpha}_f(j,l,z)_i \hat{\psi}_{jlz}^{per}$, for $\hat{\psi}^{per}_{jlz}$ from \eqref{eq:defaprroxphiper}.
\end{itemize}

\textbf{Step 1.} Although we need to compute the derivatives of $\hat{f}_0$ up to the order $\lfloor \beta \rfloor +3$, this is not very costly as $\hat{f}_0$ is a uni-dimensional neural network with width and depth that do not depend on the number of data. Furthermore, we only need to compute the error with $\phi^{(l)}$ at the points $k\in \{1,...,N\}$ which also does not depend on the number of data. Finally we need to obtain an error $O(n^{-\log_2(C_h)-1})$ which is strictly greater than $O(n^{-2})$. To the best of our knowledge, there are not existing results on the number of operations it would take for a   tanh or \black ReQU neural network to obtain such error. Nevertheless, if one would use classical splines (instead of neural network), then it would take a finite number of steps to have an error of $0$ at the points $k\in \{1,...,N\}$. Therefore, one could hope that it would take as most $O(\log(n))$ operation for a   tanh (or ReQU) \black  neural network to obtain an error $O(n^{-2})$.

\textbf{Step 2.} Computing the operator $V$only requires a finite number of products and sums. Therefore this step takes $O(\log(n))$ operations.

\textbf{Step 3.} This step requires the computations of the $\hat{\psi}^{per}_{jlz}$, but for $u\in[0,1]^d$, and  fixed $j$, there are only a finite number of $\hat{\psi}^{per}_{jlz}(u)$ that are not 0. The computation of each $\hat{\psi}^{per}_{jlz}$ using the neural network $\hat{\phi}=\hat{f}_{\log(n)}$ from step 2, is done in a finite time. Therefore this step takes $O(\log(n))$ operations.

Note that once the neural network class $\hat{\mathcal{F}}^{\beta+1,n^{-\frac{1}{2\beta+d}}}_{per}$ has been implemented, it is then very easy to optimize it in the GAN loss \eqref{WGANS} as the neural network depends linearly on the weights $\hat{\alpha}_f(j,l,z)_i$.

  Our choice of neural network architecture is motivated by the need of having generators with a small $\mathcal{H}^{\beta+1}$ norm. Indeed, it is stated in Proposition \ref{prop:reguofG} that the neural network within $\hat{\mathcal{F}}^{\beta+1,n^{-\frac{1}{2\beta+d}}}_{per}$ have their  $\mathcal{H}^{\beta+1}$ norm bounded by $O(\log(n)^{2})$. This is a key property allowing to use the interpolation inequality (Theorem \ref{theo:theineq}) without getting multiplicative constants larger than $O(\log(n)^C)$.  If one were to use a more classical architecture (with prescribed width and depth), they would not have a guarantee that the generator selected by the GAN will have a $\mathcal{H}^{\beta+1}$ norm bounded by $O(\log(n)^C)$. Indeed, although a tanh or ReQu neural network is smooth, the norm of its derivatives can be proportional to the size of its layers. Therefore, one would have to constraint the norm of the derivatives by adding to  the GAN loss \eqref{WGANS} a cost depending on the value of the derivatives which would make the estimator not tractable.\black

\black

\section{Proofs of the technical tools of Section \ref{sec:theoreticalGAN}}
In this section, we gather the proofs of Section \ref{sec:theoreticalGAN}. We start by stating some preliminaries results of interpolation theory that will be used in the proofs. Then, we give the proofs of each result in the order in which they are stated in the paper. 
\subsection{Some useful interpolation inequalities}\label{sec:tenicalinterp}
Let us state a general interpolation inequality that shows that we can trade of the regularity of the distributions in the $L^2$ scalar product.
\begin{proposition}\label{prop:Hölder}
Let $\mathcal{X}$ be either $\mathbb{R}^\eta$ or $[0,1]^\eta/\mathbb{Z}^\eta$, $h_1\in \mathcal{B}^{s_1,b_1}_{\infty,\infty}(\mathcal{X},\mathbb{R}^p,1)$ and $h_2\in \mathcal{B}^{s_2,b_2}_{\infty,\infty}(\mathcal{X},\mathbb{R}^p,1)$ for some $s_1,s_2,b_1,b_2\in \mathbb{R}$. Then for $\tau\in \mathbb{R}$, $t,r\in[0,1]$, $q>1$ and $1/q+1/q^\star=1$ we have
$$
\Big\langle h_1,h_2\Big\rangle_{L^2(\mathcal{X},\mathbb{R}^p)}\leq \Big\langle \tilde{\Gamma}^{t\tau}(h_1),\Gamma^{(1-t)\tau}(h_2)\Big\rangle_{L^2(\mathcal{X},\mathbb{R}^p)}^{\frac{1}{q}}\Big\langle \tilde{\Gamma}^{-r\frac{q^\star}{q}\tau}(h_1),\Gamma^{-(1-r)\frac{q^\star}{q}\tau}(h_2)\Big\rangle_{L^2(\mathcal{X},\mathbb{R}^p)}^{\frac{1}{q^{\star}}}
$$
where $\tilde{\Gamma}^{t\tau}(h_1)\in \mathcal{B}^{s_1+t\tau,b_1}_{\infty,\infty}$ and $\tilde{\Gamma}^{-r\frac{q^\star}{q}\tau}(h_1)\in \mathcal{B}^{s_1-r\frac{q^\star}{q},b_1}_{\infty,\infty}$ are defined by
$$
\langle \tilde{\Gamma}^{t\tau}(h_1)_i,\psi_{jlz}\rangle=S(j,l,w)_i\langle \Gamma^{t\tau}(h_1)_i,\psi_{jlz}\rangle
$$ 
and 
$$
\langle \tilde{\Gamma}^{-t\frac{q^\star}{q}\tau}(h_1)_i,\psi_{jlz}\rangle=S(j,l,w)_i\langle \Gamma^{-r\frac{q^\star}{q}\tau}(h_1)_i,\psi_{jlz}\rangle
$$
with $S(j,l,w)_i \in \{-1,1\}$ the sign of $\langle h_{1_i},\psi_{jlz}\rangle \langle h_{2_i},\psi_{jlz}\rangle$, $i=1,...,p$.
\end{proposition}

\begin{proof}
    The proof of this proposition is a simple application of Hölder's inequality on the wavelets coefficients of $h_1$ and $h_2$. For $i\in \{1,...,p\}$ and $k\in \{1,2\}$, writing $\alpha_{h_k}(j,l,w)_i=\langle h_{k_i},\psi_{j,l,w}\rangle_{L^2(\mathcal{X},\mathbb{R}^p)}$ we have
\begin{align*}
    \Big\langle h_1&,h_2\Big\rangle_{L^2(\mathcal{X},\mathbb{R}^p)}  
=\sum \limits_{i=1}^p\sum \limits_{j=0}^\infty \sum \limits_{l=1}^{2^\eta} \sum \limits_{w \in \mathbb{Z}^\eta}\alpha_{h_1}(j,l,w)_i\alpha_{h_2}(j,l,w)_i\\
 \leq & \sum \limits_{i=1}^p\sum \limits_{j=0}^\infty \sum \limits_{l=1}^{2^\eta} \sum \limits_{w \in \mathbb{Z}^\eta}2^{j\frac{\tau}{q}}|\alpha_{h_1}(j,l,w)_i\alpha_{h_2}(j,l,w)_i|^{1/q}2^{-j\frac{\tau}{q}}|\alpha_{h_1}(j,l,w)_i\alpha_{h_2}(j,l,w)_i|^{1/q^\star}\\
 \leq & \left(\sum \limits_{i=1}^p\sum \limits_{j=0}^\infty \sum \limits_{l=1}^{2^\eta} \sum \limits_{w \in \mathbb{Z}^\eta}2^{j\tau}|\alpha_{h_1}(j,l,w)_i\alpha_{h_2}(j,l,w)_i|\right)^{\frac{1}{q}}\\
& \times \left(\sum \limits_{i=1}^p\sum \limits_{j=0}^\infty \sum \limits_{l=1}^{2^\eta} \sum \limits_{w \in \mathbb{Z}^\eta}2^{-\tau\frac{q^\star}{q}}|\alpha_{h_1}(j,l,w)_i\alpha_{h_2}(j,l,w)_i|\right)^{\frac{1}{q^\star}}\\
 = &\Big\langle \tilde{\Gamma}^{t\tau}(h_1),\Gamma^{(1-t)\tau}(h_2)\Big\rangle_{L^2(\mathcal{X},\mathbb{R}^p)}^{\frac{1}{q}}\Big\langle \tilde{\Gamma}^{-r\frac{q^\star}{q}\tau}(h_1),\Gamma^{-(1-r)\frac{q^\star}{q}\tau}(h_2)\Big\rangle_{L^2(\mathcal{X},\mathbb{R}^p)}^{\frac{1}{q^{\star}}}.
\end{align*}
\end{proof}

This proposition is a key result that we will use numerous times all along the paper. An useful corollary is the following.
\begin{corollary}\label{coro:ineq without reg}
   Let $\mu,\nu$ be two probability measures with compact supports in $\mathbb{R}^p$. Then for all $\theta,\theta_1,\theta_2>0$ with $\theta_1,\theta_2$ non integers such that $\theta_1<\theta<\theta_2$,  we have
    \begin{align*}
        \sup \limits_{D \in \mathcal{H}^{\theta}_1}\mathbb{E}_{\substack{X\sim \mu \\ Y\sim \nu}}[D(X)-D(Y)]\leq & C \sup \limits_{D \in \mathcal{H}^{\theta_1}_1}\left(\mathbb{E}_{\substack{X\sim \mu \\ Y\sim \nu}}[D(X)-D(Y)]\right)^\frac{\theta_2-\theta}{\theta_2-\theta_1}\\
        & \times \sup \limits_{D \in \mathcal{H}^{\theta_2}_1}\left(\mathbb{E}_{\substack{X\sim \mu \\ Y\sim \nu}}[D(X)-D(Y)]\right)^\frac{\theta-\theta_1}{\theta_2-\theta_1}.
        \end{align*}
\end{corollary}

 We will use this result in the proof of minimax optimality of the GAN estimator of Section \ref{sec:tractable}. Before giving the proof of this result, let us first define the notion of envelope of a measure.

\begin{definition}\label{defi:envelope}
Let $\mu$ be a probability measure compactly supported on $\mathbb{R}^p$. For $\epsilon \in (0,1]$, we call the $\epsilon$-envelope of $\mu$, the probability measure $\mu_\epsilon$ defined for any Borel set $A$ by
    \begin{equation}\label{eq:envelopebis}
    \mu_\epsilon(A) = \frac{1}{\lambda^p(B^p(0,\epsilon))}\int_{\mathds{R}^p} \lambda^p(A\cap B(x,\epsilon)){\rm d}\mu(x).
    \end{equation}
\end{definition}

Then for $\gamma>0$ and $D\in \mathcal{H}^\gamma_1$, we have 
\begin{align*}
    \int_{\mathbb{R}^p} D(x)(d\mu-d\mu_\epsilon)(x) & = \int_{\mathbb{R}^p} \left(D(x)-\frac{1}{\lambda^p(B^p(0,\epsilon))}\int_{B^p(x,\epsilon)} D(y)d\lambda^p(y)\right)d\mu(x)\leq \epsilon^{1\wedge \gamma}.
\end{align*}
Furthermore, $\mu_\epsilon$ is absolutely continuous with respect to $\lambda^p$ and its density $f_\mu^\epsilon$ equals to
\begin{align*}
    f_\mu^\epsilon(y) & = \frac{1}{\lambda^p(B^p(0,\epsilon))}\int_{\mathds{R}^p} \mathds{1}_{\{y\in B^p(x,\epsilon)\}}{\rm d}\mu(x)\\
    & \leq \lambda^p(B^p(0,\epsilon))^{-1},
\end{align*}
so $f_\mu^\epsilon$ is bounded and therefore belongs to $L^2$. Let us now give the proof of Corollary \ref{coro:ineq without reg}.

\begin{proof}[Proof of Corollary \ref{coro:ineq without reg}]
For $\epsilon \in (0,1]$, let $\mu_\epsilon,\nu_\epsilon$ be the $\epsilon$-envelope of $\mu,\nu$ respectively and $f^\epsilon_\mu,f_{\nu}^\epsilon$ the densities of $\mu_\epsilon$ and $\nu_\epsilon$. For all $D\in \mathcal{H}^\theta_1$ we have 
\begin{align*}
    \mathbb{E}_{X\sim \mu}[D(X)]-\mathbb{E}_{Y\sim \nu}[D(Y)]&  \leq \mathbb{E}_{X_\epsilon\sim \mu_\epsilon}[D(X_\epsilon)] - \mathbb{E}_{Y_\epsilon\sim \nu_\epsilon}[D(Y_\epsilon)] +2\epsilon^{1\wedge \theta}\\
    & =\int_{\mathbb{R}^p} D(x)(f_\mu^\epsilon(x)-f_{\nu}^\epsilon(x))d\lambda^p(x) +2\epsilon^{1\wedge \theta}.
\end{align*}
Using Proposition \ref{prop:Hölder} for $h_1=D$, $h_2=f_\mu^\epsilon-f_\nu^\epsilon$, $s_1=\theta$,  $\tau=\theta-\theta_1$ , $s_2=b_1=b_2=0$, $t=r=1$
and $q=\frac{\theta_2-\theta_1}{\theta_2-\theta}$ we have 
\begin{align*}
\int_{\mathbb{R}^p} D(x)(f_\mu^\epsilon(x)-f_\nu^\epsilon(x))d\lambda^p(x)\leq & \left(\int_{\mathbb{R}^p} \tilde{\Gamma}^{\theta-\theta_1}(D)(x)(f_\mu^\epsilon(x)-f_\nu^\epsilon(x))d\lambda^p(x)\right)^\frac{\theta_2-\theta}{\theta_2-\theta_1}\\
& \left(\int_{\mathbb{R}^p} \tilde{\Gamma}^{\theta-\theta_2}(D)(x)(f_\mu^\epsilon(x)-f_\nu^\epsilon(x))d\lambda^p(x)\right)^\frac{\theta-\theta_1}{\theta_2-\theta_1}\\
 \leq & C \left(\sup \limits_{D\in\mathcal{B}^{\theta_1}_{\infty,\infty}(1)} \int_{\mathbb{R}^p}  D(x)(f_\mu^\epsilon(x)-f_\nu^\epsilon(x))d\lambda^p(x)\right)^\frac{\theta_2-\theta}{\theta_2-\theta_1}\\
 & \left(\sup \limits_{D\in\mathcal{B}^{\theta_2}_{\infty,\infty}(1)}\int_{\mathbb{R}^p} D(x)(f_\mu^\epsilon(x)-f_\nu^\epsilon(x))d\lambda^p(x)\right)^\frac{\theta-\theta_1}{\theta_2-\theta_1}.
\end{align*}

From Lemma \ref{lemma:inclusions} we have that $\mathcal{B}^{\theta_1}_{\infty,\infty}(1)\subset \mathcal{H}^{\theta_1}_C$ and for any $D\in\mathcal{H}^{\theta_1}_C$ we have
\begin{align*}
\int_{\mathbb{R}^p}  D(x)(f_\mu^\epsilon(x)-f_\nu^\epsilon(x))d\lambda^p(x)
    & \leq 2\epsilon^{1\wedge\theta_1} + \mathbb{E}_{X\sim \mu}[D(X)]-\mathbb{E}_{Y\sim \nu}[D(Y)]. 
\end{align*}
Doing the same thing for $D\in\mathcal{H}^{\theta_2}_C$ and letting $\epsilon$ go to $0$, we get the result.
\end{proof}

Let us now show that we can gain some weak Besov regularity by paying a logarithmic term.

\begin{proposition}\label{prop:logforweakregularity}
    Let $\mathcal{X}$ be either $\mathbb{R}^p$ or $\mathbb{T}^d$, $f\in L^1(\mathcal{X},\mathbb{R},1)\cap L^2(\mathcal{X},\mathbb{R})$ and $g\in \mathcal{H}^\gamma_1(\mathcal{X},\mathbb{R})$ with $\gamma>0$. Then for all $\tau>0$, $\epsilon\in (0,1)$ we have
$$\int_\mathcal{X}f(x)g(x)d\lambda_{\mathcal{X}}(x)\leq C\log(\epsilon^{-1})^\tau\int_\mathcal{X}f(x)\tilde{\Gamma}^{0,-\tau}_\epsilon(g)(x)d\lambda_{\mathcal{X}}(x)+C_\gamma\epsilon$$
for $C_\gamma$ such that $\mathcal{B}^{\gamma/2}_{\infty,\infty}(\mathcal{X},\mathbb{R},1)\subset \mathcal{H}^0_{C_\gamma}(\mathcal{X},\mathbb{R})$ and
$$\tilde{\Gamma}^{0,-\tau}_\epsilon(g)(x)=\sum \limits_{j=0}^{\log(\epsilon^{-2/\gamma})} \sum \limits_{l=1}^{2^{\text{dim}(\mathcal{X})}} \sum \limits_{w \in \mathbb{Z}^{\text{dim}(\mathcal{X})}}(1+j)^{-\tau}\alpha_g(j,l,w)\frac{\alpha_g(j,l,w)\alpha_f(j,l,w)}{|\alpha_g(j,l,w)\alpha_f(j,l,w)|}\psi_{j,l,w}(x).$$
\end{proposition}
\begin{proof}
Let us write $M$ the dimension of the space $\mathcal{X}$. For $k\geq 1$ we have
\begin{align*}
\sum \limits_{j=k}^\infty \sum \limits_{l=1}^{2^{M}} \sum \limits_{w \in \mathbb{Z}^{M}}\alpha_f(j,l,w) \alpha_g(j,l,w)
    & \leq     \sum \limits_{j=k}^\infty \sum \limits_{l=1}^{2^{M}} \sum \limits_{w \in \mathbb{Z}^{M}}C2^{-j(\gamma+M/2)}| \alpha_f(j,l,w)|\\
     & \leq  C2^{-k\gamma/2}   \sum \limits_{j=k}^\infty \sum \limits_{l=1}^{2^{M}} \sum \limits_{w \in \mathbb{Z}^{M}} 2^{-j(\gamma+M)/2}| \alpha_f(j,l,w)|\\
    & = C2^{-k\gamma/2}\int \kappa(x)f(x)d\lambda^p(x)
\end{align*}
for 
$$
\kappa(x)=\sum \limits_{j=k}^\infty \sum \limits_{l=1}^{2^{M}} \sum \limits_{w \in \mathbb{Z}^{M}} 2^{-j(\gamma+M)/2}S(j,l,w)\psi_{j,l,w}(x)
$$
with $S(j,l,w)=\frac{\alpha_f(j,l,w)}{|\alpha_f(j,l,w)|}$.
As $\kappa \in \mathcal{B}^{\gamma/2}_{\infty,\infty}(1)\subset \mathcal{H}^0_{C_\gamma}$, we have 
$$\int \kappa(x)f(x)d\lambda^p(x)\leq \|\kappa\|_\infty \|f\|_{L_1}\leq C_\gamma.$$
Let us write $\gamma_\epsilon= \lfloor \log_2(\epsilon^{-2/\gamma})\rfloor$, we have
\begin{align*}
\int_\mathcal{X}f(x)g(x)d\lambda_{\mathcal{X}}(x) & = \sum \limits_{j=0}^\infty \sum \limits_{l=1}^{2^{M}} \sum \limits_{w \in \mathbb{Z}^{M}}\alpha_g(j,l,w) \alpha_f(j,l,w) \\
    & \leq \sum \limits_{j=0}^{\gamma_\epsilon} \sum \limits_{l=1}^{2^{M}} \sum \limits_{w \in \mathbb{Z}^{M}}\alpha_g(j,l,w) \alpha_f(j,l,w) + C_\gamma\epsilon\\
    & \leq (\gamma_\epsilon+1)^\tau \sum \limits_{j=0}^{\gamma_\epsilon} \sum \limits_{l=1}^{2^{M}} \sum \limits_{w \in \mathbb{Z}^{M}}(1+j)^{-\tau}|\alpha_g(j,l,w) \alpha_f(j,l,w)| + C_\gamma\epsilon
    \\
    & \leq C\log(\epsilon^{-1})^\tau\int_\mathcal{X}f(x)\tilde{\Gamma}^{0,-\tau}_\epsilon(g)(x)d\lambda_{\mathcal{X}}(x)+C_\gamma\epsilon
\end{align*}    
\end{proof}

\subsection{Proof of the bias-variance trade-off for GANs (Section~\ref{firstbound})}
Let us define some useful notations for the GAN estimator.
\begin{definition}\label{defi:L}
For $D\in \mathcal{H}^{\gamma}_1$ and $g\in \mathcal{H}^{\beta+1}_K(\mathbb{T}^d,\mathbb{R}^p)$ define:
\begin{itemize}
\item the discrimination score under $D$ of the push-forward measure $g_{\# U}$ and $g^\star_{\# U}$ as
$$
L(g,D):=\mathbb{E}_{U\sim \mathcal{U}([0,1]^d)}[D(g(U))-D(g^\star(U))],
$$
\item the discretized version of the discrimination score
$$L_n(g,D):=\frac{1}{n}\sum \limits_{i=1}^n D(X_i)-D(g(U_i)),$$
with $U_i\sim \mathcal{U}([0,1]^d)$ and $X_i\sim g^\star_{\# U}$ i.i.d. random variables, $i=1,...,n$, 
\item the optimal discriminator in $\mathcal{H}^{\gamma}_1$ between $g_{\# U}$
$$D^\star_g \in \argmax \limits_{D\in \mathcal{H}^{\gamma}_1} L(g,D),$$  
\item the optimal discriminator in $\mathcal{D}$
of the empirical approximation of the expectation
$$ \hat{D}_g \in \argmax \limits_{D\in \mathcal{D}} L_n(g,D).$$
\end{itemize}
\end{definition}

 When the context is clear we will write $\mathbb{E}[D(g(U))]$ instead of $\mathbb{E}_{U\sim \mathcal{U}([0,1]^d)}[D(g(U))]$.

\subsubsection{Proof of Theorem \ref{theo:boundexpecterror}}\label{sec:theo:boundexpecterror}

Let us state a first proposition that gives a bound on the expected difference between $\mathbb{E}[D(g(U))]$ and its empirical approximation $\frac{1}{n}\sum \limits_{i=1}^n D(g(U_i))$, for all $g\in \mathcal{G}$ and $D\in \mathcal{D}$.
\begin{proposition}\label{prop:meanvsiid} If $\mathcal{G}\subset \mathcal{H}^{\beta+1}_K(\mathbb{T}^d,\mathbb{R}^p)$ and $\mathcal{D} \subset \mathcal{H}^{\gamma}_1(B^p(0,K),\mathbb{R})$,  for i.i.d. random variables $U$ and $U_i$ that are uniform on $[0,1]^d$, we have
\begin{align*}  
& \mathbb{E}_{U_i}\Big[\sup \limits_{\substack{g\in \mathcal{G} \\D\in \mathcal{D}}} \mathbb{E}_{U}[D(g(U)]-\frac{1}{n}\sum \limits_{i=1}^n D(g(U_i))\Big]\\
& \leq  C \min \limits_{\delta \in [0,1]}  \Biggl\{ \sqrt{\frac{(\delta+1/n)^2\log(n|\mathcal{G}_{1/n}| |\mathcal{D}_{1/n}|)}{n}}+\frac{1}{\sqrt{n}}(1+\delta^{(1-\frac{d}{2\gamma})}+\log(\delta^{-1})\mathds{1}_{\{2\gamma= d\}})\Biggl\}.
\end{align*}
    
\end{proposition}

\begin{proof} We are first going to use Dudley's inequality bound (\citeproofs{vaart2023empirical}, chapter 2). For $g\in \mathcal{H}^{\beta+1}_K$ and $D\in \mathcal{H}^\gamma_1$, we have from the Faa di Bruno formula that $D\circ g \in \mathcal{H}^\gamma_{C}$. Let us take $$\Theta=\{D\circ g \in \mathcal{H}^{\gamma}_{C} \ | \ D\in \mathcal{D},g\in \mathcal{G}\}$$ equipped with the metric $M=\|\cdot\|_\infty$. For $\theta=D\circ g \in \Theta$, we define the zero-mean sub-Gaussian process (e.g. \citepproofs{Van2016prob}, chapter 5) as $X_\theta=\frac{1}{\sqrt{n}} \sum \limits_{i=1}^n D(g(U_i))-\mathbb{E}[D(g(U))]$ with $U_i\sim \mathcal{U}([0,1]^d)$ i.i.d. Then as $\sup \limits_{\theta,\theta^{'}\in \Theta} M(\theta,\theta^{'})\leq 2$, using Dudley's bound we obtain $\forall \delta \in [0,2]$,
\begin{align}\label{eq:dudleyf}
\mathbb{E}\Big[\sup \limits_{\substack{g\in \mathcal{G} \\D\in \mathcal{D}}} \mathbb{E}[D(g(U)]-\frac{1}{n}\sum \limits_{i=1}^n D(g(U_i)) \Big]\leq & 2 \mathbb{E}\Big[\sup \limits_{\substack{g\in \mathcal{G},\ D\in \mathcal{D}\nonumber \\ \|D\circ g\|_\infty\leq \delta}} \mathbb{E}[D(g(U))]-\frac{1}{n}\sum \limits_{i=1}^n D(g(U_i))\Big]\\
& +\frac{16}{\sqrt{n}}\int_{\delta/4}^{2}\sqrt{\log(|\Theta_\epsilon|)}d\epsilon.
\end{align}
Let us take care of the first term of \eqref{eq:dudleyf}. Let $\mathcal{G}_{1/n}$ and $\mathcal{D}_{1/n}$ be minimal $1/n$ covering of the class $\mathcal{G}$ and $\mathcal{D}$ for the distance $\|\cdot\|_\infty$. Therefore, for any $ g \in \mathcal{G}$, $D \in \mathcal{D}$, $\exists g_n\in \mathcal{G}_{1/n}$ and $D_n\in \mathcal{D}_{1/n}$ such that $\|g-g_n\|_\infty\leq \frac{1}{n}$ and $\|D-D_n\|_\infty\leq \frac{1}{n}$. As $\mathcal{H}_1^\gamma \subset \text{Lip}_1$, we have for $\alpha>0$ 
\begin{align*}
     & \mathbb{P}\Big(\sup \limits_{\substack{g\in \mathcal{G},\ D\in \mathcal{D}\\ \|D\circ g\|_\infty\leq \delta}} \mathbb{E}[D(g(U))]-\frac{1}{n}\sum \limits_{i=1}^n D(g(U_i))\geq \alpha \Big)\\
     & \leq \mathbb{P}\Big( \sup \limits_{\substack{g\in \mathcal{G},\ D\in \mathcal{D}\\ \|D\circ g\|_\infty\leq \delta}}\mathbb{E}[D_n(g_n(U))]-\frac{1}{n}\sum \limits_{i=1}^n D_n(g_n(U_i))\geq \alpha -\frac{4}{n} \Big)\\
    & \leq \mathbb{P}\Big( \substack{\exists\ g_n \in \mathcal{G}_{1/n}, D_n\in \mathcal{D}_{1/n}\\ \|D_n\circ g_n \|_\infty\leq \delta+\frac{2}{n}} \text{ with } \mathbb{E}[D_n(g_n(U))]-\frac{1}{n}\sum \limits_{i=1}^n D_n(g_n(U_i))\geq \alpha -\frac{4}{n} \Big)\\
    & \leq \sum \limits_{\substack{g_n \in \mathcal{G}_{1/n}, D_n\in \mathcal{D}_{1/n}\\ \|D_n\circ g_n \|_\infty\leq \delta+\frac{2}{n}}}  \mathbb{P}\Big( \mathbb{E}[D_n(g_n(U))]-\frac{1}{n}\sum \limits_{i=1}^n D_n(g_n(U_i))\geq \alpha -\frac{4}{n} \Big)\\
    & \leq |\mathcal{G}_{1/n}| |\mathcal{D}_{1/n}|  \exp\left(-\frac{n(\alpha -\frac{4}{n})^2}{2(\delta+\frac{2}{n})^2}\right),
\end{align*}
applying Hoeffding's inequality to the random variables $Z_i=D_n(g_n(U_i))$, $i=1,...,n$. Then for $\tau \in (0,1]$, solving $|\mathcal{G}_{1/n}| |\mathcal{D}_{1/n}|  \exp\left(-\frac{n(\alpha -\frac{4}{n})^2}{2(\delta+\frac{2}{n})^2}\right)=\tau$ with respect to $\alpha$, we obtain that with probability at least $1-\tau$
\begin{align*}
& \sup \limits_{\substack{g\in \mathcal{G},\ D\in \mathcal{D}\\ \|D\circ g\|_\infty\leq \delta}} \mathbb{E}[D(g(U))]-\frac{1}{n}\sum \limits_{i=1}^n D(g(U_i))\\
& \leq C\left(\sqrt{\frac{(\delta+1/n)^2\log(|\mathcal{G}_{1/n}| |\mathcal{D}_{1/n}|/\tau)}{n} }+\frac{1}{n}\right).
\end{align*}
Taking $\tau=1/n$, we deduce that
\begin{align*}
& \mathbb{E}\Big[\sup \limits_{\substack{g\in \mathcal{G},\ D\in \mathcal{D}\\ \|D\circ g\|_\infty\leq \delta}} \mathbb{E}[D(g(U))]-\frac{1}{n}\sum \limits_{i=1}^n D(g(U_i))\Big]\\
& \leq C\left(\sqrt{\frac{(\delta+1/n)^2\log(|\mathcal{G}_{1/n}| |\mathcal{D}_{1/n}|n)}{n} }+\frac{1}{n}\right).
\end{align*}
For the second term of \eqref{eq:dudleyf}, as $\Theta\subset\mathcal{H}^{\gamma}_C$ we have $\log(|\Theta_\epsilon|)\leq C \log(|(\mathcal{H}^{\gamma}_{1})_\epsilon|)$ and from \citeproofs{donoho1998data} we know that $\log(|(\mathcal{H}^{\gamma}_{1})_\epsilon|)\leq C \epsilon ^{-\frac{d}{\gamma}}$ so
\begin{align*}
\frac{16}{\sqrt{n}}\int_{\delta/4}^{2}\sqrt{\log(|\Theta_\epsilon|)}d\epsilon \leq & \frac{C}{\sqrt{n}}\int_{\delta/4}^{2}\epsilon ^{-\frac{d}{2\gamma}}d\epsilon\\
 = & \frac{C\mathds{1}_{\{2\gamma\neq d\}}}{\sqrt{n}(1-\frac{d}{2\gamma})}\left(2^{1-\frac{d}{2\gamma}}-\left(\frac{\delta}{4}\right)^{(1-\frac{d}{2\gamma})}\right)\\
 & +\frac{C\mathds{1}_{\{2\gamma= d\}}}{\sqrt{n}}(\log(2)-\log(\delta/4)).
\end{align*}
Therefore we have $$\frac{16}{\sqrt{n}}\int_{\delta/4}^2\sqrt{\log(|\Theta_\epsilon|)}d\epsilon\leq  \frac{C}{\sqrt{n}}(1+\delta^{(1-\frac{d}{2\gamma})}+\log(4\delta^{-1})\mathds{1}_{\{2\gamma= d\}}).$$
\end{proof}
We see that having small classes of functions allows the empirical approximation to be close to the expectation. As we would like our estimator $\hat{g}\in \argmin \limits_{g\in \mathcal{G}}\ \max \limits_{D\in \mathcal{D}} \sum \limits_{i=1}^n D(X_i)-D(g(U_i))$ to obtain a small error $\mathbb{E}_{U\sim \mathcal{U}([0,1]^d)}[D(\hat{g}(U))-D(g^\star(U))]$ for all $D\in \mathcal{D}$, it is then crucial that $\mathbb{E}\Big[\sup \limits_{\substack{g\in \mathcal{G} \\D\in \mathcal{D}}} \mathbb{E}[D(g(U)]-\frac{1}{n}\sum \limits_{i=1}^n D(g(U_i))\Big]$ be small. From Proposition \ref{prop:meanvsiid} we can now give the proof of Theorem \ref{theo:boundexpecterror}.

\begin{proof}[Proof of Theorem \ref{theo:boundexpecterror} ]
Let $\overline{D}_g\in \argmax \limits_{D\in \mathcal{D}} L(g,D)$ and $\overline{g} \in \argmin \limits_{g\in \mathcal{G}} L(g,D^\star_g)$, we can decompose the objective as 
\begin{align*}
L(\hat{g},D^\star_{\hat{g}})  = &L(\hat{g},D^\star_{\hat{g}})-    L(\hat{g},\overline{D}_{\hat{g}}) + L(\hat{g},\overline{D}_{\hat{g}}) - L(\overline{g},D^\star_{\overline{g}}) + L(\overline{g},D^\star_{\overline{g}})\\
\leq & \Delta_\mathcal{D} +L(\hat{g},\overline{D}_{\hat{g}}) - L_n(\hat{g},\overline{D}_{\hat{g}}) + L_n(\hat{g},\overline{D}_{\hat{g}})- L(\overline{g},\overline{D}_{\overline{g}}) \\
&+L(\overline{g},\overline{D}_{\overline{g}}) -L(\overline{g},D^\star_{\overline{g}}) +\Delta_\mathcal{G}\\
 \leq & \Delta_\mathcal{D} +L(\hat{g},\overline{D}_{\hat{g}}) - L_n(\hat{g},\overline{D}_{\hat{g}}) + L_n(\hat{g},\hat{D}_{\hat{g}})- L(\overline{g},\overline{D}_{\overline{g}}) +\Delta_\mathcal{G}\\
\leq  &\Delta_\mathcal{D} +L(\hat{g},\overline{D}_{\hat{g}}) - L_n(\hat{g},\overline{D}_{\hat{g}}) + L_n(\overline{g},\hat{D}_{\overline{g}})- L(\overline{g},\overline{D}_{\overline{g}}) +\Delta_\mathcal{G}  
\\
 \leq & \Delta_\mathcal{D} +L(\hat{g},\overline{D}_{\hat{g}}) - L_n(\hat{g},\overline{D}_{\hat{g}}) + L_n(\overline{g},\hat{D}_{\overline{g}})- L(\overline{g},\hat{D}_{\overline{g}}) +\Delta_\mathcal{G}
\end{align*}
using the definition of $\hat{g}$ and $\hat{D}$. Then, we have
\begin{align*}
\mathbb{E}\Big[&L(\hat{g},\overline{D}_{\hat{g}}) - L_n(\hat{g},\overline{D}_{\hat{g}}) \Big]\\
= &  \mathbb{E}\Big[\mathbb{E}[\overline{D}_{\hat{g}}(\hat{g}(U))-\overline{D}_{\hat{g}}(g^\star(U))] -\frac{1}{n}\sum \limits_{i=1}^n \overline{D}_{\hat{g}}(\hat{g}(U_i))-\overline{D}_{\hat{g}}(g^\star(Y_i))\Big]\\
\leq & \mathbb{E}\Big[\sup \limits_{\substack{g\in \mathcal{G} \\D\in \mathcal{D}}} \mathbb{E}[D(g(U)]-\frac{1}{n}\sum \limits_{i=1}^n D(g(U_i))\Big] \\
& + \mathbb{E}\Big[\sup \limits_{D\in \mathcal{D}}\mathbb{E}[D(g^\star(U)]-\frac{1}{n}\sum \limits_{i=1}^n D(g^\star(Y_i))\Big].
\end{align*}
The first term of the precedent inequality can be controlled by Proposition \ref{prop:meanvsiid}. As $Y_i\sim \mathcal{U}([0,1]^d)$, doing the same derivation as in the proof of Proposition \ref{prop:meanvsiid} for a fixed $g^\star \in \mathcal{H}^{\beta+1}_1$, we can also bound the second term by the same quantity. Then, doing the same thing for $\mathbb{E}\Big[L_n(\overline{g},\hat{D}_{\overline{g}})- L(\overline{g},\hat{D}_{\overline{g}}) \Big]$, we get the result.
\end{proof}

\subsubsection{Proof of Proposition \ref{prop:appvscov}}\label{sec:prop:appvscov}
This result is already well known (Theorem 4.2 in \citeproofs{devore1989optimal}) but we give a new proof in the case $\eta$ non integer using wavelet decomposition. The case $\eta$ integer can also be proven the same way but at the cost of losing a factor $log(\epsilon)^2$ coming from the fact that $\mathcal{B}^{\eta}_{\infty,\infty}$ does not inject in $\mathcal{H}^\eta$ but $\mathcal{B}^{\eta,2}_{\infty,\infty}$ does.
\begin{proof}[Proof of Proposition \ref{prop:appvscov}]
Let $f\in \mathcal{H}^\eta_1$ and $g\in \mathcal{F}$, as $\mathcal{H}^0 \xhookrightarrow{} \mathcal{B}^0_{\infty,\infty}$ we have that there exists $C_0>0$ such that
\begin{align*}
        \|f-g\|_\infty & \geq C_0^{-1}\|f-g\|_{\mathcal{B}^{0}_{\infty,\infty}}= C_0^{-1}\sup \limits_{j\geq 0} 2^{jd/2}\sum \limits_{l=1}^{2^d} \sup \limits_{z\in \{0,...,2^j-1\}^d} |\alpha_g(j,l,z)-\alpha_f(j,l,z)|.
\end{align*}
Therefore, $\forall f\in \mathcal{H}^\eta_1$, there exists $g\in \mathcal{F}$ such that 
$$
\sum \limits_{l=1}^{2^d}2^{jd/2}|\alpha_g(j,l,z)-\alpha_f(j,l,z)|\leq C_0\|f-g\|_\infty\leq C_0 \epsilon,
$$
for all $j\geq 0$ and $z\in \{0,...,2^j-1\}^d$. On the other hand, as $\mathcal{B}^{\eta}_{\infty,\infty}=\mathcal{H}^\eta$, there exists $ C_\eta>0$ such that, if 
$ f$ verifies
$$\sum  \limits_{l=1}^{2^d}|\alpha_f(j,l,z)| \leq C_\eta 2^{-j(\eta+d/2)},\ \forall j\geq 0,z\in \{0,...,2^j-1\}^d,$$ then $f \in \mathcal{H}^\eta_1$.
Therefore, for all $k\in \{0,...,\lfloor 2^{-j\eta-(d+2)}C_\eta(C_0\epsilon)^{-1}\rfloor\}^{2^d}$, as 
$$
\sum  \limits_{l=1}^{2^d}k_l2^{2-jd/2}C_0\epsilon\leq \sum \limits_{l=1}^{2^d} C_\eta2^{-j(\eta+d/2)-d} \leq C_\eta 2^{-j(\eta+d/2)},
$$
then there exists $g\in \mathcal{F}$ such that for all $ j\geq 0,z\in \{0,...,2^j-1\}^d$, 
$$
2^{jd/2}\sum \limits_{l=1}^{2^d}|\alpha_g(j,l,z)-k_l2^{2-jd/2}C_0\epsilon|\leq C_0 \epsilon.
$$
Furthermore, we have
$$
2^{jd/2}\sum \limits_{l=1}^{2^d}|\alpha_g(j,l,z)-k_l2^{2-jd/2}C_0\epsilon|\leq C_0 \epsilon$$
implies $$ \ 2^{jd/2}|\alpha_g(j,l,z)-k_l2^{2-jd/2}C_0\epsilon|\leq C_0 \epsilon , \forall l \in \{1,...,2^d\},$$
which implies $$\ 2^{jd/2}|\alpha_g(j,l,z)-(k_l+1)2^{2-jd/2}C_0\epsilon|> C_0 \epsilon , \forall l \in \{1,...,2^d\}.
$$
Then for each sequence  $$\alpha \in \bigotimes \limits_{j=1}^{\log_2(\lfloor (C_\eta^{-1}C_0\epsilon)^{-\frac{1}{\eta}}\rfloor)-d/\eta} \bigotimes \limits_{z\in\{0,...,2^j-1\}^d} \{0,...,\lfloor 2^{-j\eta-(d+2)}C_\eta(C_0\epsilon)^{-1}\rfloor\}^{2^d},$$ there exists a different $ g \in \mathcal{F}$ such that
$$
\|g-\sum \limits_{j=1}^{\log_2(\lfloor (C_\eta^{-1}C_0\epsilon)^{-\frac{1}{\eta}}\rfloor)}\sum \limits_{l=1}^{2^d} \sum \limits_{z\in\{0,...,2^j-1\}^d} \alpha_{jlz}C_0 2^{2-jd/2}\epsilon\ \psi^{per}_{jlk}\|_\infty\leq \epsilon.
$$
 We stop at the frequency $j=\log_2(\lfloor (C_\eta^{-1}C_0\epsilon)^{-\frac{1}{\eta}}\rfloor)-d/\eta$, as
for $f\in \mathcal{H}^\eta_1$ and $j\geq \log_2(\lfloor (C_\eta^{-1}C_0\epsilon)^{-\frac{1}{\eta}}\rfloor)-d/\eta$ we have 
$$2^{jd/2}\sum  \limits_{l=1}^{2^d}|\alpha_f(j,l,z)|\leq 2^{jd/2}C_\eta 2^{-j(\eta+d/2)} = C_\eta 2^{d-j\eta}=C_0\epsilon,$$
so $\mathcal{F}$ does not have to cover these frequencies.
Then for $\theta \in (0,\frac{C_0\epsilon}{3})$, the $\theta$-covering number of $\mathcal{F}$ verifies
\begin{align*}
    |\mathcal{F}_\theta| & \geq \prod \limits_{j=0}^{\log_2(\lfloor (C_\eta^{-1}C_0\epsilon)^{-\frac{1}{\eta}}\rfloor)-(d+3)/\eta}\prod  \limits_{z\in \{0,...,2^j-1\}^d} \lfloor 2^{-j\eta-(d+2)}C_\eta(C_0\epsilon)^{-1}\rfloor^{2^d}\\
    & \geq \prod   \limits_{z\in \{0,...,2^{\log_2(\lfloor (C_\eta^{-1}C_0\epsilon)^{-\frac{1}{\eta}}\rfloor)-(d+3)/\eta}\}^d}2\\
    & = \prod   \limits_{z\in \{0,...,\lfloor (C_\eta^{-1}C_0\epsilon)^{-\frac{1}{\eta}}\rfloor 2^{-(d+3)/\eta}\}^d}2\\
    & \geq C 2^{\left(\lfloor (C_\eta^{-1}C_0\epsilon)^{-\frac{1}{\eta}}\rfloor\right)^d}.
\end{align*}
We then deduce that 
$$
\log(|\mathcal{F}_\theta|)\geq C \epsilon^{-\frac{d}{\eta}}. 
$$
\end{proof}

\subsection{Proofs of the results on the properties of $\hat{\mathcal{F}}^{\eta,\delta}_{per}$ (Section \ref{sec:wavelets})}
\subsubsection{Proof of Lemma \ref{lemma:deltaG}}\label{sec:lemma:deltaG}
\begin{proof}
Let $g\in \mathcal{B}^{\eta}_{\infty,\infty}(K)$ and 
$D\in \mathcal{H}^\gamma_1$, using a Taylor expansion we have 
    \begin{align*}
      \mathbb{E}&\Big[D(g(U))-D(g^\star(U))\Big]\\
      & = \mathbb{E}\Big[\int_0^1 \langle\nabla D(g^\star(U)+t(g(U)-g^\star(U))),g(U)-g^\star(U) \rangle dt \Big]\\
      & = \int_0^1 \sum \limits_{i=1}^p \mathbb{E}\Big[ \partial_i D(g^\star(U)+t(g(U)-g^\star(U)))(g(U)-g^\star(U))_i  \Big] dt. 
    \end{align*}
Let us fix a $t\in [0,1]$ and write the functions in their (periodized) wavelets expansions:
    $$
    D_i^t(u):=\partial_i D(g^\star(u)+t(g(u)-g^\star(u))) = \sum \limits_{j=0}^\infty \sum \limits_{l=1}^{2^d}\sum \limits_{z\in \{0,...,2^j-1\}^d}\alpha_{D^t_i}(j,l,z)\psi_{jlz}^{per}(u),
    $$
$$
g(u)_i = \sum \limits_{j=0}^\infty \sum \limits_{l=1}^{2^d}\sum \limits_{z\in \{0,...,2^j-1\}^d}\alpha_g(j,l,w)_{i}\psi_{jlz}^{per}(u),
$$
and 
$$g^\star(u)_i = \sum \limits_{j=0}^\infty \sum \limits_{l=1}^{2^d}\sum \limits_{z\in \{0,...,2^j-1\}^d}\alpha_{g^\star}(j,l,w)_{i}\psi_{jlz}^{per}(u).$$
As $g-g^\star\in \mathcal{H}^{\eta}_K \subset \mathcal{B}^{\eta}_{\infty,\infty}(C)$ and $D_{i}^t\in  \mathcal{H}^{\gamma-1}_C \subset \mathcal{B}^{\gamma-1}_{\infty,\infty}(C)$, we have that for all $i\in\{1,...p\}$ and $j\geq 0,$ 
$$
\alpha_{D^t_i}(j,l,z)(\alpha_{g}(j,l,w)_{i}-\alpha_{g^\star}(j,l,w)_{i})\leq C2^{-j(\gamma-1+\eta+d)}.
$$ 
Then, using the wavelet expansions we get
\begin{align*}
      \mathbb{E} & \Big[ \partial_i D(g^\star(U)+t(g(U)-g^\star(U)))(g(U)-g^\star(U))_i  \Big] \\
       = &\sum \limits_{j=0}^\infty \sum \limits_{l=1}^{2^d}\sum \limits_{z\in \{0,...,2^j-1\}^d}\alpha_{D^t_i}(j,l,z)(\alpha_{g}(j,l,w)_{i}-\alpha_{g^\star}(j,l,w)_{i})\\
        \leq &\sum \limits_{j=0}^{\log_2(\lfloor \epsilon^{-1}\rfloor)} \sum \limits_{l=1}^{2^d}\sum \limits_{z\in \{0,...,2^j-1\}^d}\alpha_{D^t_i}(j,l,z)(\alpha_{g}(j,l,w)_{i}-\alpha_{g^\star}(j,l,w)_{i})\\
        & +  \sum \limits_{j=\log_2(\lfloor \epsilon^{-1}\rfloor)+1}^\infty \sum \limits_{l=1}^{2^d}\sum \limits_{z\in \{0,...,2^j-1\}^d}C2^{-j(\gamma-1+\eta+d)}\\
     \leq &\sum \limits_{j=0}^{\log_2(\lfloor \epsilon^{-1}\rfloor)} \sum \limits_{l=1}^{2^d}\sum \limits_{z\in \{0,...,2^j-1\}^d}\alpha_{D^t_i}(j,l,z)(\alpha_{g}(j,l,w)_{i}-\alpha_{g^\star}(j,l,w)_{i}) + C \epsilon^{\eta+\gamma-1}.
\end{align*}
Applying twice the Cauchy-Schwarz inequality we get
\begin{align*}
\sum \limits_{j=0}^{\log_2(\lfloor \epsilon^{-1}\rfloor)} & \sum \limits_{l=1}^{2^d}\sum \limits_{z\in \{0,...,2^j-1\}^d}\alpha_{D^t_i}(j,l,z)(\alpha_{g}(j,l,w)_{i}-\alpha_{g^\star}(j,l,w)_{i})\\
    \leq & \sum \limits_{j=0}^{\log_2(\lfloor \epsilon^{-1}\rfloor)} \left( \sum \limits_{l=1}^{2^d} \sum \limits_{z\in \{0,...,2^j-1\}^d} |\alpha_{D^t_i}(j,l,z)|^2\right)^{1/2}\\
    & \times \left( \sum \limits_{l=1}^{2^d} \sum \limits_{z\in \{0,...,2^j-1\}^d}|\alpha_{g}(j,l,w)_{i}-\alpha_{g^\star}(j,l,w)_{i}|^2\right)^{1/2}\\
         \leq & \left(\sum \limits_{j=0}^{\log_2(\lfloor \epsilon^{-1}\rfloor)} 2^{2j(\gamma-1)} \sum \limits_{l=1}^{2^d} \sum \limits_{z\in \{0,...,2^j-1\}^d} |\alpha_{D^t_i}(j,l,z)|^2\right)^{1/2}\\
        & \times \left(\sum \limits_{j=0}^{\log_2(\lfloor \epsilon^{-1}\rfloor)} 2^{-2j(\gamma-1)}\sum \limits_{l=1}^{2^d} \sum \limits_{z\in \{0,...,2^j-1\}^d}|\alpha_{g}(j,l,w)_{i}-\alpha_{g^\star}(j,l,w)_{i}|^2\right)^{1/2}\\
       \leq & \left(\sum \limits_{j=0}^{\log_2(\lfloor \epsilon^{-1}\rfloor)} 2^{2j(\gamma-1)} \sum \limits_{l=1}^{2^d} \sum \limits_{z\in \{0,...,2^j-1\}^d} |C2^{-j(\gamma-1+d/2)}|^2\right)^{1/2}\\
       & \times \|\Gamma^{-\gamma+1}(g_i)-\Gamma^{-\gamma+1}(g^\star_i)\|_{\mathcal{B}^{0}_{2,2}}\\ 
       \leq & C \log_2(\lfloor \epsilon^{-1}\rfloor)^{1/2} \|\Gamma^{-\gamma+1}(g_i)-\Gamma^{-\gamma+1}(g^\star_i)\|_{L^2},
\end{align*}
as $\forall s\geq 0$, $\mathcal{B}^{s}_{2,2}=\mathcal{W}^{s,2}$ the $L^2$ Sobolev space of regularity $s$.
\end{proof}

\subsubsection{Proof of Proposition \ref{prop:Fdelta}}\label{sec:prop:Fdelta}

\begin{proof}
For  $f \in\mathcal{F}_{per}^{\eta,\delta}$ we have 
$f\in \mathcal{H}^{\eta}_{C\log(\delta^{-1})^{2}}$ (we do not distinguish between $\eta$ integer or not) so using Lemma \ref{lemma:deltaG} with $\epsilon=\delta$, for any $f^\star\in \mathcal{H}^\eta_K$ we have \begin{align*}
    d_{\mathcal{H}_1^\gamma}(f_{\# U},f^\star_{\# U}) & \leq C\log(\delta^{-1})^{C_2}\left( \sum \limits_{i=1}^p \|\Gamma^{-(\gamma-1)}(f_i)-\Gamma^{-(\gamma-1)}(f^\star_i)\|_{L^2}+ C\delta^{\eta-1+\gamma}\right).
\end{align*}
Then, for $f$ such that $\forall j \leq \log_2(\delta^{-1})$ and $i\in \{1,...,p\}$, $\alpha_{f_i}(j,l,z)=\alpha_{f^\star_i}(j,l,z)$, we have
\begin{align*}
     \|\Gamma^{-(\gamma-1)} &(f_i)-\Gamma^{-(\gamma-1)}(f^\star_i)\|_{L^2}^2\\
     & \leq C  \|\Gamma^{-(\gamma-1)}(f_i)-\Gamma^{-(\gamma-1)}(f^\star_i)\|_{\mathcal{B}^{0}_{2,2}}^2\\
     & = C \sum \limits_{j=0}^\infty \sum \limits_{l=1}^{2^d} \sum \limits_{z\in \{0,...,2^j-1\}^d} 2^{-2j(\gamma-1)}|\alpha_{f_i}(j,l,z)-\alpha_{f^\star_i}(j,l,z)|^2\\
     & = C \sum \limits_{j\geq \log_2(\delta^{-1})}^\infty \sum \limits_{l=1}^{2^d} \sum \limits_{z\in \{0,...,2^j-1\}^d} 2^{-2j(\gamma-1)}|\alpha_{f^\star_i}(j,l,z)|^2\\
     & \leq C \sum \limits_{j\geq \log_2(\delta^{-1})}^\infty  2^{-2j(\gamma+\eta-1)}\\
     & \leq C 2^{-\log_2(\delta^{-1})2(\gamma-1+\eta)}=C\delta^{2(\gamma-1+\eta)}.
\end{align*}
Let us now bound the covering number for a certain $\epsilon\in (0,1)$. Define 
\begin{align*}
(\mathcal{F}_{per}^{\eta,\delta})_\epsilon=\{ & f \in \mathcal{F}_{per}^{\eta,\delta}|\ \forall i,j,l,z,\exists k\in[-KC_\eta\epsilon^{-1}2^{-j\eta+d}(1+j)^2,...,KC_\eta\epsilon^{-1}2^{-j\eta+d}(1+j)^2],\\
& \langle f_i,\psi_{j,l,z}^{per}\rangle_{L^2}=\frac{k\epsilon}{
(1+j)^22^{jd/2+d}} \}.
\end{align*}
Then for $f\in \mathcal{F}_{per}^{\eta,\delta}$ and $f_\epsilon\in (\mathcal{F}_{per}^{\eta,\delta})_\epsilon$ such that 
$$
\langle (f_{\epsilon})_i,\psi_{j,l,z}^{per}\rangle_{L^2}=\frac{\epsilon} {(1+j)^22^{jd/2+d}}\lfloor \langle f_i,\psi_{j,l,z}^{per}\rangle_{L^2}\epsilon^{-1}(1+j)^22^{jd/2+d} \rfloor,
$$ we have 
\begin{align*}
    \|f_i-(f_{\epsilon})_i\|_\infty & \leq C \|f_i-(f_{\epsilon})_i\|_{\mathcal{B}^{0,2}_{\infty,\infty}}\\
    & = C\sup \limits_{j\in \mathbb{N}_0} 2^{jd/2}\sum \limits_{l=1}^{2^d} \sup \limits_{z\in \{0,...,2^j-1\}^d}(1+j)^2|\alpha_{f_i}(j,l,z)-\alpha_{(f_{\epsilon})_i}(j,l,z)|\\
    & \leq C \sup \limits_{j\in \mathbb{N}_0} 2^{jd/2}\sum \limits_{l=1}^{2^d} \sup \limits_{z\in \{0,...,2^j-1\}^d}(1+j)^2\frac{\epsilon} {(1+j)^22^{jd/2+d}}=C\epsilon\\
\end{align*}
so $(\mathcal{F}_{per}^{\eta,\delta})_\epsilon$ is an $C\epsilon$ covering of $\mathcal{F}_{per}^{\eta,\delta}$. Furthermore
\begin{align*}
    \big|(\mathcal{F}_{per}^{\eta,\delta})_\epsilon\big| & = \prod \limits_{j=0}^{\log_2(\lfloor \delta^{-1}\rfloor)} \prod \limits_{l=1}^{2^d} \prod \limits_{z\in \{0,...,2^j-1\}^d} 2^{-j\eta+d+1}(1+j)^2C_\eta K\epsilon^{-1}\\
    & \leq C \prod \limits_{j=0}^{\log_2(\lfloor \delta^{-1}\rfloor)} \prod \limits_{l=1}^{2^d} \prod \limits_{z\in \{0,...,2^j-1\}^d}  \epsilon^{-1} = C \prod \limits_{j=0}^{\log_2(\lfloor \delta^{-1}\rfloor)} ( \epsilon^{-1})^{2^{jd+d}}\\
    & \leq C ( \epsilon^{-1})^{\log_2(\lfloor \delta^{-1}\rfloor)2^{\log_2(\lfloor \delta^{-1}\rfloor)d+d}}  = C ( \epsilon^{-1})^{\log_2(\lfloor \delta^{-1}\rfloor)2^d\delta^{-d}},
\end{align*}
so 
$$\log(\big|(\mathcal{F}_{per}^{\eta,\delta})_\epsilon\big|) \leq C \delta^{-d}\log(\lfloor \delta^{-1}\rfloor)\log(\epsilon^{-1}).$$
\end{proof}

\subsubsection{Proof of Proposition \ref{prop:propertiesofhatper}}\label{sec:prop:propertiesofhatper}
\begin{proof}

For $g\in \hat{\mathcal{F}}^{\eta,\delta}_{per}$ and $D\in \mathcal{H}^\gamma_1$, we have 
\begin{align*}
     \mathbb{E}[D(g(U))-  D(g^\star(U))] & = \mathbb{E}[\langle \int_0^1\nabla D(g(U)+t(g^\star(U)-g(U))dt,g^\star(U)-g(U)\rangle]\\
     & \leq C \sup \limits_{f\in \mathcal{H}^{\gamma-1}_1} \mathbb{E}[\langle f(U),g^\star(U)-g(U)\rangle].
\end{align*}
Then the rest of the proof proceeds as the proof of  Proposition \ref{prop:propertiesofhat} Section \ref{sec:prop:propertiesofhat}.
    
\end{proof}

\section{Proofs of the bounds of Section \ref{sec:tractable}}

\subsection{Proof of the bound in the general low dimensional case (Section~\ref{sec:theoreticalgan})}
In this section, we give the proofs of the results of Section \ref{sec:tractable} in the order in which they are stated in the paper.
\subsubsection{Proof of Theorem \ref{theo:minimaxfirstgan}}\label{sec:theo:minimaxfirstgan}
\begin{proof}

Using Proposition \ref{prop:propertiesofhatper} for $\eta=\beta+1$, $\delta=n^{-\frac{1}{2\beta+d}}$ and $\epsilon=1/n$ we have that
$$
\Delta_\mathcal{G}\leq  C\log(n)^{C_2} n ^{-\frac{\beta+\gamma}{2\beta+d}}\ \  \text{and }\ \ \log(|\mathcal{G}_{1/n}|)\leq Cn^{\frac{d}{2\beta+d}}\log(n)^2.
$$

Let us now take care of $\Delta_\mathcal{D}$. Write $\overline{g}\in \argmin \limits_{g\in \mathcal{G}} L(g,D^\star_g)$, then for $g\in \mathcal{G}$, as $D^\star_{g,\overline{g}}\in \mathcal{D}$, we have
\begin{align*}
 \sup_{D\in \mathcal{D}} L(g,D^\star_g)- L(g,D)  \leq & L(g,D^\star_g)- L(g,D^\star_{g,\overline{g}})\\
  = & \mathbb{E}\Big[D_g^\star(g(U))-  D_g^\star(g^\star(U)) -D^\star_{g,\overline{g}}(g(U))+D^\star_{g,\overline{g}}(g^\star(U))\Big]\\
   = &  \mathbb{E}\Big[D_g^\star(g(U))-  D_g^\star(\overline{g}(U)) -D^\star_{g,\overline{g}}(g(U))+D^\star_{g,\overline{g}}(\overline{g}(U))\Big]\\
   & + \mathbb{E}\Big[D_g^\star(\overline{g}(U))-  D_g^\star(g^\star(U)) -D^\star_{g,\overline{g}}(\overline{g}(U))+D^\star_{g,\overline{g}}(g^\star(U))\Big]\\
  \leq& \mathbb{E}\Big[D_g^\star(g(U))-  D_g^\star(\overline{g}(U)) -D^\star_{g,\overline{g}}(g(U))+D^\star_{g,\overline{g}}(\overline{g}(U))\Big]+ 2\Delta_\mathcal{G}\\
   \leq & 2\Delta_\mathcal{G} \leq C\log(n)^{C_2} n ^{-\frac{\beta+\gamma}{2\beta+d}}. 
\end{align*}
On the other hand, by definition of $\mathcal{D}$, we have that $|\mathcal{D}_{1/n}|\leq |\mathcal{D}|\leq |\mathcal{G}_{1/n}|^2$.

Using Theorem \ref{theo:boundexpecterror} we have
\begin{align*}  
\mathbb{E}&\Big[L(\hat{g},D^\star_{\hat{g}})\Big]\\
\leq & C \min \limits_{\delta \in [0,1]}   \sqrt{\frac{(\delta+1/n)^2\log(n|\mathcal{G}_{1/n}| |\mathcal{D}_{1/n}|)}{n}}+\frac{1}{\sqrt{n}}(1+\delta^{(1-\frac{d}{2\gamma})}+\log(\delta^{-1})\mathds{1}_{\{2\gamma= d\}})\\
& +\Delta_\mathcal{G}  +\Delta_\mathcal{D}\\
\leq & C \min \limits_{\delta \in [0,1]}  \sqrt{\frac{(\delta+1/n)^2n^{\frac{d}{2\beta+d}}}{n}} +\frac{1}{\sqrt{n}}(1+\delta^{(1-\frac{d}{2\gamma})}+\log(\delta^{-1})\mathds{1}_{\{2\gamma= d\}})\\
& +\log(n)^{C_2}n^{-\frac{\beta+\gamma}{2\beta+d}}.
\end{align*}
Now taking $\delta=n^{-\frac{\gamma}{2\beta+d}}$ we have $$\sqrt{\frac{(\delta+1/n)^2n^{\frac{d}{2\beta+d}}}{n}} \leq 2n^{-\frac{\gamma}{2\beta+d}}n^{\frac{d/2}{2\beta+d}}n^{-\frac{\beta+d/2}{2\beta+d}}=2n^{-\frac{\beta+\gamma}{2\beta+d}}$$
and 
$$\frac{1}{\sqrt{n}}(1+\delta^{(1-\frac{d}{2\gamma})})= n^{-\frac{1}{2}}+n^{-\frac{\gamma}{2\beta+d}\frac{\gamma-d/2}{\gamma}-\frac{\beta+d/2}{2\beta+d}}=n^{-\frac{1}{2}}+n^{-\frac{\beta+\gamma}{2\beta+d}},$$
 so we finally obtain
$$\mathbb{E}\Big[L(\hat{g},D^\star_{\hat{g}})\Big]\leq C\log(n)^{C_2}\left(n^{-\frac{\beta+\gamma}{2\beta+d}}\vee n^{-\frac{1}{2}}\right).$$
\end{proof}

\subsection{Proofs of the bound in the full dimensional case (Section~\ref{sec:fulldim})}

\subsubsection{Proof of Corollary \ref{coro:boundexpectederrorfulldim}}\label{sec:coro:boundexpectederrorfulldim}
\begin{proof}
Following the proof of Proposition \ref{prop:meanvsiid}, for $$\Theta=\{Df\in \mathcal{H}^{\gamma}_{C} \ | \ D\in \mathcal{D},f\in \mathcal{F}\}$$  and $$X_\theta=\frac{1}{\sqrt{n}} \sum \limits_{i=1}^n D(U_i)f(U_i)-\int D(x)f(x)d\lambda^p(x),$$ with $U_i$ i.i.d. uniform random variables on $B^p(0,K)$, we get using Dudley's inequality  $\forall \delta \in [0,2K]$,
\begin{align*}
\mathbb{E}\Big[\sup \limits_{\substack{f\in \mathcal{F} \\D\in \mathcal{D}}}& \int D(x)f(x)d\lambda^p(x)-\sum \limits_{i=1}^n D(U_i)f(U_i) \Big]\\
\leq & 2 \mathbb{E}\Big[\sup \limits_{\substack{f\in \mathcal{G},\ D\in \mathcal{D}\\ \|Df \|_\infty\leq \delta}} \int D(x)f(x)d\lambda^p(x)-\frac{1}{n}\sum \limits_{i=1}^n D(U_i)f(U_i)\Big]\\
& +\frac{16}{\sqrt{n}}\int_{\delta/4}^{2K}\sqrt{\log(|\Theta_\epsilon|)}d\epsilon.
\end{align*}
Doing as in the proof of Proposition \ref{prop:Fdelta}, we can show that $\log(|(\mathcal{B}^{\gamma}_{\infty,\infty})_\epsilon|)\leq \epsilon^{-\frac{d}{\gamma}}\log(\epsilon^{-1})$. The rest of the proof follows the same arguments than the end of the proof of Proposition \ref{prop:meanvsiid} and the proof of Theorem \ref{theo:boundexpecterror} (see Section \ref{sec:theo:boundexpecterror}).
\end{proof}

\subsubsection{Proof of Proposition \ref{prop:propertiesofhat} }\label{sec:prop:propertiesofhat}

First, let us give the approximation error and covering number of the class $\mathcal{F}^{\eta,\delta}$.
\begin{proposition}\label{prop:Fdeltap}
    We have 
    $$\Delta_{\mathcal{F}^{\eta,\delta}}:=\sup \limits_{h^\star\in \mathcal{H}^{\eta}_1}\inf \limits_{ f\in \mathcal{F}^{\eta,\delta}} \sup \limits_{D\in \mathcal{H}^{\gamma}_1} \int_{\mathbb{R}^p}D(x)(f(x)-h^\star(x))d\lambda^p(x) \leq  C\delta^{\eta+\gamma},$$
    and $\forall \epsilon\in (0,1)$,
    $$\log(|\mathcal{F}^{\eta,\delta}|_\epsilon)\leq C\delta^{-p}\log(\lfloor \delta^{-1}\rfloor)\log(\epsilon^{-1}).$$
\end{proposition}

\begin{proof}
Let $h^\star\in \mathcal{H}^{\eta}_1$ and $f\in \mathcal{F}^{\eta,\delta}$ be such that for all $j\leq \log_2(\delta^{-1})$, $\langle h^\star,\psi_{j,l,w}\rangle=\langle f,\psi_{j,l,w}\rangle$. Then for $D\in \mathcal{H}^{\gamma}_1$ we have
    \begin{align*}
    \int_{\mathbb{R}^p}D(x)(f(x)-h^\star(x))d\lambda^p(x) & = \sum \limits_{j=0}^\infty \sum \limits_{l=1}^{2^p} \sum \limits_{w \in \mathbb{Z}^p} \alpha_D(j,l,w)(\alpha_f(j,l,w)-\alpha_{h^\star}(j,l,w))\\
    & = \sum \limits_{j\geq log_2(\delta^{-1}) }^\infty \sum \limits_{l=1}^{2^p} \sum \limits_{w \in \mathbb{Z}^p} \alpha_D(j,l,w)\alpha_{h^\star}(j,l,w)\\
    & \leq  \sum \limits_{j\geq log_2(\delta^{-1}) }^\infty \sum \limits_{l=1}^{2^p} \sum \limits_{w \in \mathbb{Z}^p} K2^{-j(p+\eta+\gamma)}\mathds{1}_{\{supp(\psi_{jlw})\cap B^p(0,K)\neq\varnothing\}}\\
    & \leq  \sum \limits_{j\geq log_2(\delta^{-1}) }^\infty  C2^{-j(\eta+\gamma)}\\
    & = C \delta^{\eta+\gamma}.
    \end{align*}
For the bound on the covering number see the proof of Proposition \ref{prop:Fdelta}, Section \ref{sec:prop:Fdelta}.
\end{proof}
We can now give the proof of Proposition \ref{prop:propertiesofhat}.
\begin{proof}[Proof of Proposition \ref{prop:propertiesofhat}]
Let $h^\star\in \mathcal{H}^{\eta}_1$ and define $f\in \hat{\mathcal{F}}^{\eta,\delta}$ by
$$
f=\sum \limits_{j=0}^{\log_2(\delta^{-1})} \sum \limits_{l=1}^{2^p} \sum \limits_{w\in \{-K2^{j},...,K2^j\}^p} \alpha_{h^\star}(j,l,w) \hat{\psi}_{jlw},
$$
for $\alpha_{h^\star}(j,l,w)=\langle h^\star,\psi_{j,l,w}\rangle$. Then, from Proposition \ref{prop:approxhat} there exist coefficients $(\alpha_{f}(j,l,w))$ such that 
$$f=\sum \limits_{j=0}^{\infty} \sum \limits_{l=1}^{2^p} \sum \limits_{w\in \{-K2^{j},...,K2^j\}^p} \alpha_{f}(j,l,w) \psi_{jlw} ,$$
with $\forall j\leq \log_2(\delta^{-1})$,
$$|\alpha_f(j,l,w)-\hat{\alpha}_f(j,l,w)|\leq CL^{-1}(2^{-j(\lfloor \beta \rfloor +2+p/2)}\delta^{-(\lfloor \beta \rfloor +2-\eta)}\log_2(\delta^{-1})+2^{-j(\eta+p/2)})$$
and $\forall j> \log_2(\delta^{-1})$,
$$|\alpha_f(j,l,w)|\leq CL^{-1}2^{-j(\lfloor \beta \rfloor +2+p/2)}\delta^{-(\lfloor \beta \rfloor +2-\eta)}\log_2(\delta^{-1}).$$
As $L\geq \delta^{-(\lfloor \beta \rfloor +2+\gamma)}$, we have in particular that $$|\alpha_f(j,l,w)-\alpha_{h^\star}(j,l,w)|\leq C2^{-j(\eta+p/2)}\delta^{\eta+\gamma}, \text{ for } j\leq \log_2(\delta^{-1})$$
and
$$|\alpha_f(j,l,w)|\leq C2^{-j(\lfloor \beta \rfloor +2+p/2)}\delta^{\eta+\gamma}, \text{ for } j> \log_2(\delta^{-1}).$$
Therefore, for $D\in \mathcal{H}^{\gamma}_1$ we have
    \begin{align*}
    \int_{\mathbb{R}^p} &  D(x)(f(x)-h^\star(x))d\lambda^p(x)  =  \sum \limits_{j=0}^\infty \sum \limits_{l=1}^{2^p} \sum \limits_{w \in \mathbb{Z}^p} \alpha_D(j,l,w)(\alpha_f(j,l,w)-\alpha_{h^\star}(j,l,w))\\
     = & \sum \limits_{j=0}^{log_2(\delta^{-1})} \sum \limits_{l=1}^{2^p} \sum \limits_{w \in \mathbb{Z}^p} |\alpha_D(j,l,w)(\alpha_f(j,l,w)-\alpha_{h^\star}(j,l,w))|\\
     & +\sum \limits_{j>log_2(\delta^{-1})}^\infty \sum \limits_{l=1}^{2^p} \sum \limits_{w \in \mathbb{Z}^p} |\alpha_D(j,l,w)(\alpha_f(j,l,w)-\alpha_{h^\star}(j,l,w))|\\
     \leq & C \sum \limits_{j\leq log_2(\delta^{-1}) }^\infty \sum \limits_{l=1}^{2^p} \sum \limits_{w \in \mathbb{Z}^p} 2^{-j(\gamma+p/2)}2^{-j(\eta+p/2)}\delta^{\eta+\gamma}\mathds{1}_{\{supp(\psi_{jlw})\cap B^p(0,K)\neq\varnothing\}}\\
     & +\sum \limits_{j>log_2(\delta^{-1})}^\infty \sum \limits_{l=1}^{2^p} \sum \limits_{w \in \mathbb{Z}^p} |\alpha_D(j,l,w)|(|\alpha_f(j,l,w)|+|\alpha_{h^\star}(j,l,w)|)\\
     \leq & C \delta^{\eta+\gamma} +C\sum \limits_{j>log_2(\delta^{-1})}^\infty \sum \limits_{l=1}^{2^p} \sum \limits_{w \in \mathbb{Z}^p} 2^{-j(\gamma+p/2)}\\
     & \quad \quad \quad \quad (2^{-j(\lfloor \beta \rfloor +2+p/2)}\delta^{\eta+\gamma}+2^{-j(\eta+p/2)})\mathds{1}_{\{supp(\psi_{jlw})\cap B^p(0,K)\neq\varnothing\}}\\
     \leq & C \delta^{\eta+\gamma}.
    \end{align*}
Let us now prove the bound on the covering number. Let 
\begin{align*}
\overline{\mathcal{F}}=\{&f\in \mathcal{B}^\eta_{\infty,\infty}(B^p(0,K),\mathbb{R},K)|\\
& \ |\alpha_f(j,l,w)|\leq CL^{-1}2^{-j(\lfloor \beta \rfloor +2+p/2)}\delta^{-(\lfloor \beta \rfloor +2-\eta)}\log_2(\delta^{-1}), \forall j>log_2(\delta^{-1})\}.
\end{align*}
Then from Proposition \ref{prop:approxhat}, we have that $\forall \epsilon\in (0,1)$,
$$
\log(|(\hat{\mathcal{F}}^{\eta,\delta})_\epsilon|)\leq C\log(|(\overline{\mathcal{F}})_\epsilon|).
$$
Let $\epsilon >L^{-1}\delta^{\eta-1}$ and $(\mathcal{F}^{\eta,\delta})_\epsilon$ an $\epsilon$-covering of $\mathcal{F}^{\eta,\delta}$. For all $f\in \overline{\mathcal{F}}$, writing  $$\tilde{f}=\sum \limits_{j=0}^{log_2(\delta^{-1})} \sum \limits_{l=1}^{2^p} \sum \limits_{w \in \mathbb{Z}^p} \alpha_f(j,l,w)\psi_{jlw},
$$ 
we have $\tilde{f}  \in \mathcal{F}^{\eta,\delta}$. Then there exist $f_\epsilon \in (\mathcal{F}^{\eta,\delta})_\epsilon$ such that 
\begin{align*}
    \|\tilde{f}-f_\epsilon\|_\infty \leq \epsilon.
\end{align*}
Then,
\begin{align*}
    \|f-f_\epsilon\|_\infty &  \leq \epsilon + \|f-\tilde{f}\|_\infty\leq\epsilon + C \|f-\tilde{f}\|_{\mathcal{B}^{0,2}_{\infty,\infty}}\\
    & =\epsilon +C \sup \limits_{j\geq log_2(\delta^{-1})}(1+j)^2 2^{jp/2}\sum \limits_{l=1}^{2^d} \sup \limits_{w\in \mathbb{Z}^p} |\alpha_f(j,l,w)|\\
    & \leq\epsilon + C  \sup \limits_{j\geq log_2(\delta^{-1})}(1+j)^2 2^{jp/2}\sum \limits_{l=1}^{2^d}  L^{-1}2^{-j(\lfloor \beta \rfloor +2+p/2)}\delta^{-(\lfloor \beta \rfloor +2-\eta)}\log_2(\delta^{-1})\\
    & \leq\epsilon + C\log(\delta^{-1})^{C_2}L^{-1}\delta^{\eta}.
\end{align*}
Therefore we deduce that $(\mathcal{F}^{\eta,\delta})_\epsilon$ is a $C\epsilon$-covering of $|\overline{\mathcal{F}}|$ so
$$
 \log(|(\overline{\mathcal{F}})_\epsilon|)\leq C\log(|(\mathcal{F}^{\eta,\delta})_\epsilon|).
$$
Finally, using Proposition \ref{prop:Fdeltap} we get the result.

\end{proof}

\subsubsection{Proof of Corollary \ref{coro:ineqfulldim}}\label{sec:coro:ineqfulldim}

\begin{proof}
    Let $D\in \mathcal{B}^{\gamma,2}_{\infty,\infty}(1)$, then applying Proposition \ref{prop:Hölder} for $h_1=D$, $h_2=f-f^\star$, $\tau=\gamma-\alpha$, $t=1$, $r=\frac{\gamma}{\beta+\gamma}$ and $q=\frac{\beta+\alpha}{\beta+\gamma}$, we get 
\begin{align*}
\int_{\mathbb{R}^p}D(x)(f(x)-f^\star(x))d\lambda^p(x)\leq &\left(\int_{\mathbb{R}^p}\tilde{\Gamma}^{\gamma-\alpha}(D)(x)(f(x)-f^\star(x))d\lambda^p(x)\right)^{\frac{\beta+\gamma}{\beta+\alpha}}\\
& \times \left(\int_{\mathbb{R}^p}\tilde{\Gamma}^{\gamma}(D)(x)\Gamma^{\beta}(f-f^\star)(x)d\lambda^p(x)\right)^{\frac{\alpha-\gamma}{\beta+\alpha}}.
\end{align*}

Furthermore,
\begin{align*}
    \int_{\mathbb{R}^p}& \tilde{\Gamma}^{\gamma}(D)(x)\Gamma^{\beta}(f-f^\star)(x)d\lambda^p(x)\\
    & = \sum \limits_{j=0}^\infty \sum \limits_{l=1}^{2^p} \sum \limits_{w \in \mathbb{Z}^p} 2^{j(\beta+\gamma)}|\alpha_D(j,l,w)(\alpha_f(j,l,w)-\alpha_{f^\star}(j,l,w))|\\
    & \leq   \sum \limits_{j=0}^\infty \sum \limits_{l=1}^{2^p} \sum \limits_{w \in \mathbb{Z}^p} 2^{-jp}(1+j)^{-2}C\mathds{1}_{\{supp(\psi_{jlw})\cap B^p(0,K)\neq\varnothing\}}\\
    & \leq C  \sum \limits_{j=0}^\infty (1+j)^{-2}.
\end{align*}
\end{proof}

\subsubsection{Proof of Lemma \ref{lemma:petitlemme}}\label{sec:lemma:petitlemme}
\begin{proof}
From Proposition \ref{prop:approxhat} we have that $\hat{\mathcal{F}}^{\beta,n^{-\frac{1}{2\beta+p}}}\subset \mathcal{B}^{\beta}_{\infty,\infty}(\mathbb{R}^p,\mathbb{R},C)$. Then, for $f\in \mathcal{F}, D^\star \in \mathcal{H}_1^{\tilde{\beta}}$ and $D\in \mathcal{D}$,  we have 
    \begin{align*}
    \int_{\mathbb{R}^p}& (D^\star(x)-D(x))(f(x)-f^\star(x))d\lambda^p(x)\\
     = &  \sum \limits_{j=0 }^\infty \sum \limits_{l=1}^{2^p} \sum \limits_{w \in \mathbb{Z}^p} (\alpha_{D^\star}(j,l,w)-\alpha_{D}(j,l,w))(\alpha_{f}(j,l,w)-\alpha_{f^\star}(j,l,w))\\
     \leq &  \sum \limits_{j=0 }^\infty \sum \limits_{l=1}^{2^p} \sum \limits_{w \in \mathbb{Z}^p} C2^{-jp/2}2^{-j\beta}(\alpha_{D^\star}(j,l,w)-\alpha_{D}(j,l,w)).
    \end{align*}  
The map $h:\mathbb{R}^p\rightarrow \mathbb{R}$ having wavelets coefficients $\alpha_{h}(j,l,w)=2^{-jp/2}$, belongs to $\mathcal{B}^{0}_{\infty,\infty}$ which gives the result.
\end{proof}

\subsubsection{Proof of Proposition \ref{coro:minimaxfulldim}}\label{sec:coro:minimaxfulldim}
\begin{proof}
We bound $\Delta_{\mathcal{D}}$ following  the proof of Proposition \ref{prop:Fdeltap}. For $f\in \mathcal{F}$ and $\overline{D}_f\in \mathcal{D}$ such that $\langle \overline{D}_f,\psi_{j,l,w}\rangle=\langle D_f^\star,\psi_{j,l,w}\rangle$ for all $ j\leq \log_2(n^{\frac{1}{2\beta+p}})$, we have
    \begin{align*}
    \int_{\mathbb{R}^p}& (\overline{D}_f(x)-D_f^\star(x))(f(x)-f^\star(x))d\lambda^p(x)\\
    & =  \sum \limits_{j\geq \log_2(n^{\frac{1}{2\beta+p}}) }^\infty \sum \limits_{l=1}^{2^p} \sum \limits_{w \in \mathbb{Z}^p} (\alpha_{D_f^\star}(j,l,w)(\alpha_{f}(j,l,w)-\alpha_{f^\star}(j,l,w))\\
    & \leq C n^{-\frac{\beta+\tilde{\beta}}{2\beta+p}}.
    \end{align*} 
Furthermore, from Proposition \ref{prop:propertiesofhat} we have
$$\Delta_{\mathcal{G}}\leq Cn^{-\frac{\beta+\tilde{\beta}}{2\beta+p}} \text{ and } \log(|\mathcal{G}_{1/n}||\mathcal{D}_{1/n}|)\leq \log(n)^{2}n^{\frac{p}{2\beta+p}}.$$
Then, we get the result using Corollary \ref{coro:boundexpectederrorfulldim}  with $\delta=n^{-\frac{\tilde{\beta}}{2\beta+p}}$.    
\end{proof}

\subsubsection{Proof of Proposition \ref{prop:minimaxfulldimgamma}}\label{sec:prop:minimaxfulldimgamma}
\begin{proof}
If $\beta\geq \tilde{\beta}$ and $\gamma \in [\tilde{\beta},\beta]$ we have 
$$\mathbb{E}\Big[d_{\mathcal{H}_1^{\gamma}}(\hat{f},f^\star)\Big]\leq \mathbb{E}\Big[d_{\mathcal{H}_1^{\tilde{\beta}}}(\hat{f},f^\star)\Big]\leq C\log(n)^{C_2}n^{-\frac{1}{2}},$$
using Proposition \ref{coro:minimaxfulldim}. Otherwise, for $D\in \mathcal{H}^\gamma_1$, we have 
    \begin{align*}
        \int_{\mathbb{R}^p}&D(x)  (f(x)-f^\star(x))d\lambda^p(x)  = \sum \limits_{j=0}^\infty \sum \limits_{l=1}^{2^p} \sum \limits_{w \in \mathbb{Z}^p} \alpha_D(j,l,w)(\alpha_f(j,l,w)-\alpha_{f^\star}(j,l,w))\\
         \leq & \sum \limits_{j=0}^{\log(\lfloor n \rfloor)} \sum \limits_{l=1}^{2^p} \sum \limits_{w \in \mathbb{Z}^p} \alpha_D(j,l,w)(\alpha_f(j,l,w)-\alpha_{f^\star}(j,l,w))\\
         & + \sum \limits_{j=\log(\lfloor n \rfloor)+1}^\infty \sum \limits_{l=1}^{2^p} \sum \limits_{w \in \mathbb{Z}^p} 2^{-j(p+\beta+\gamma)}K\mathds{1}_{\{supp(\psi_{jlw})\cap B^p(0,K)\neq\varnothing\}}\\
         \leq & (1+\log(n))^2\sum \limits_{j=0}^{\log(\lfloor n \rfloor)}(1+j)^{-2} \sum \limits_{l=1}^{2^p} \sum \limits_{w \in \mathbb{Z}^p} |\alpha_D(j,l,w)(\alpha_f(j,l,w)-\alpha_{f^\star}(j,l,w))|\\
         & + \sum \limits_{j=\log(\lfloor n \rfloor)+1}^\infty  2^{-j(\beta+\gamma)}C\\
         \leq& (1+\log(n))^2 \sup \limits_{D \in \mathcal{B}^{\gamma,2}_{\infty,\infty}(1)}  \int_{\mathbb{R}^p}D(x)  (f(x)-f^\star(x))d\lambda^p(x)  + C/n.
    \end{align*}
Then using Corollary \ref{coro:ineqfulldim} with $\alpha=\tilde{\beta}$ we have
\begin{align*}
    \sup \limits_{D \in \mathcal{B}^{\gamma,2}_{\infty,\infty}(1)} & \int_{\mathbb{R}^p}D(x)  (f(x)-f^\star(x))d\lambda^p(x)\\
    & \leq  C \sup \limits_{D \in \mathcal{B}^{\tilde{\beta},2}_{\infty,\infty}(1)}\left(\int_{\mathbb{R}^p}D(x)(f(x)-f^\star(x))d\lambda^p(x)\right)^\frac{\beta+\gamma}{\beta+\tilde{\beta}}\\
    & \leq C\log(n)^{C_2}\left(n^{-\frac{\beta+\tilde{\beta}}{2\beta+p}}\right)^\frac{\beta+\gamma}{\beta+\tilde{\beta}} \\
    & \leq C\log(n)^{C_2}n^{-\frac{\beta+\gamma}{2\beta+p}},
\end{align*}
using Proposition \ref{coro:minimaxfulldim}.
\end{proof}

\section{Proof of the interpolation inequality in the manifold case (Section~\ref{sec:interpolation})}
\subsection{Definition of the $(g_{\# U},g^\star_{\# U})$ compatibility}\label{sec:compatibility}
We first define the notion of approximate Jacobian (Definition 2.10 in \citeproofs{federer1959}).
\begin{definition}\label{def:approx}
    Let $\mathcal{M}$ be a $d$-dimensional submanifold of $\mathbb{R}^p$ and a map $T:\mathbb{R}^p\rightarrow \mathcal{M}$ of regularity $C^1$. For $x\in \mathbb{R}^p$, using the matrix of $\nabla T(x)$ with respect to orthonormal bases
of $\mathbb{R}^p$ and $\mathcal{T}_{T(x)}(\mathcal{M})$, the approximate Jacobian of $T$, noted $\text{ap}_d(\nabla T(x))$, is equal to the square root of the sum of the squares
of the determinants of the $d$ by $d$  minors of this matrix.
    \end{definition}

We will use multiple times the next proposition.
\begin{proposition}\label{approx}
(Theorem 3.2.22, \citeproofs{federer2014geometric}) For $t>0$, $T:\mathcal{M}^t\rightarrow \mathcal{M}$ of regularity $C^1$ and $D:\mathcal{M}^t\rightarrow \mathbb{R}$, we have
$$
\int_{\mathcal{M}^t}  D(x)\text{ap}_d(\nabla T(x))d\lambda^{p}(x)=\int_{\mathcal{M}} \int_{T^{-1}(\{z \})} D(x)d\lambda^{p-d}(x)d\lambda_{\mathcal{M}}(z),
$$
for $\lambda_{\mathcal{M}}$ the volume measure on the submanifold $\mathcal{M}$.
\end{proposition}

Writing $\mathcal{M}_g:=g(\mathbb{T}^d)$ for the image of the torus by a map $g$, we can now  define the notion of $(g_{\# U},g^\star_{\# U})$ compatibility.

\begin{definition}\label{defi:compatibility} For $g,g^\star\in \mathcal{H}^{\beta+1}_K(\mathbb{T}^d,\mathbb{R}^p)$ and $t\in (0,r_{g^\star}/2)$ such that\\ $\mathbb{H}(\mathcal{M}_g,\mathcal{M}_{g^\star})<t$. We say that a map $T:\mathcal{M}_{g^\star}^{t}\rightarrow \mathcal{M}_{g^\star}$ is $(g_{\# U},g^\star_{\# U})_{K_T}$-compatible with a radius $t$ if:
\begin{itemize}
    \item[i)] $T \in \mathcal{H}^{\beta+1}_{K_T}(\mathcal{M}_{g^\star}^{t},\mathcal{M}_{g^\star})$ and the restriction $T_{|_{\mathcal{M}_g}}:\mathcal{M}_{g}\rightarrow \mathcal{M}_{g^\star}$ is a diffeomorphism such that $(T\circ g)_{\# U}$ admits a density $f_{T} \in \mathcal{H}^{\beta}_{K_T}(\mathcal{M}_{g^\star},\mathbb{R})$ with respect to the volume measure on $\mathcal{M}_{g^\star}$.
    \item[ii)] For all $x\in \mathcal{M}_{g^\star}$, 
 $T^{-1}(\{x\})-x\subset E_x$ a $p-d$ dimensional subspace of $\mathbb{R}^p$ and 
 $$(T^{-1}(\{x\})-x)\cap B^p(x,t)=B_{E_x}(0,t).$$
    \item[iii)] For any function $f\in \mathcal{H}^{\beta+1}_1(\mathcal{M}_{g^\star}^{t},\mathbb{R})$, the map $x\mapsto \int_{T^{-1}(\{x\})}f(y)d\lambda^{p-d}_{E_x}(y)$ belongs to $\mathcal{H}^{\beta+1}_{K_T}(\mathcal{M}_{g^\star},\mathbb{R})$.
    \item[iv)] For $T_g^{-1}:\mathcal{M}_{g^\star}\rightarrow \mathcal{M}_g$ the inverse application such that $T_g^{-1}\circ T_{|_{\mathcal{M}_g}} = \text{Id}_{|_{\mathcal{M}_g}}$, $T_g^{-1}\circ T\in \mathcal{H}^{\beta+1}_{K_T}(\mathcal{M}_{g^\star}^{t},\mathcal{M}_{g})$. Furthermore, the approximate Jacobian of $T$ (see Definition \ref{def:approx})  $x\mapsto \text{ap}_d(\nabla T(x))$ is bounded below by $K_T^{-1}$.
\end{itemize}
\end{definition}

Requiring this compatibility seems to be quite a strong property but we will show in Proposition \ref{prop:compatibility} that actually, as long as the $d_{\mathcal{H}^{\beta+1}_1}$ distance between the measures is relatively small, there always exists a $(g_{\# U},g^\star_{\# U})$-compatible map.

The idea behind Definition \ref{defi:compatibility} is that we would like to project $\mathcal{M}_g$ onto $\mathcal{M}_{g^\star}$ with a $C^{\beta+1}$-diffeomorphism. From \citepproofs{Leobacher}, we know that the orthogonal projection on a manifold of regularity $\beta+1$ is only of regularity $\beta$ so unfortunately, we cannot use it here. Therefore, we define the notion of $(g_{\# U},g^\star_{\# U})$-compatible map as a "pseudo" projection of regularity $\beta+1$. This notion regroup technical properties we will need in the proof of Theorem \ref{theo:theineq}. They are essentially properties that the orthogonal projection onto $\mathcal{M}_{g^\star}$ would verify if $\mathcal{M}_{g^\star}$ was of regularity $\beta+2$ and $\mathcal{M}_{g}$, $\mathcal{M}_{g^\star}$ are close enough. Indeed, to have the notion of compatibility, we will need the sets $\mathcal{M}_g$ and $\mathcal{M}_{g^\star}$ to be close in Hausdorff distance. First, let us show that the Hausdorff distance between the image of $g$ and $g^\star$ is controlled by the $d_{\mathcal{H}^{\beta+1}_1}$ norm.

\begin{proposition}\label{prop:hausdorff}
    Let $g,g^\star \in \mathcal{H}^{1}_K(\mathbb{T}^d,\mathbb{R}^p)$, then for all $ t>0$, if
    $$
     \sup \limits_{D \in \mathcal{H}^{\beta+1}_1}\mathbb{E}_{U\sim \mathcal{U}([0,1]^d)}[D(g(U))-D(g^\star(U))]\leq t^{(d+1)(2\beta+1)},
     $$
     then 
     $$
     W_1(g_{\# U},g^\star_{\# U}) 
     \leq Ct^{d+1} \quad \text{ and } \quad \mathbb{H}(\mathcal{M}_g,\mathcal{M}_{g^\star})< Ct.
     $$ 
\end{proposition}

\begin{proof}
    Suppose first that $\beta$ is not an integer. Then,
    using Corollary \ref{coro:ineq without reg} with $\theta=1$, $\theta_1=1/2$ and $\theta_2=\beta+1$, we have 
    \begin{align*}      
    W_1&(g_{\# U},g^\star_{\# U})\\
    & \leq 2K \sup \limits_{D \in \mathcal{H}^{1}_1}\mathbb{E}_{U\sim \mathcal{U}([0,1]^d)}[D(g(U))-D(g^\star(U))]\\
    & \leq C \sup \limits_{D \in \mathcal{H}^{1/2}_1}\left(\mathbb{E}[D(g(U))-D(g^\star(U))]\right)^\frac{2\beta}{2\beta+1}\sup \limits_{D \in \mathcal{H}^{\beta+1}_1}\left(\mathbb{E}[D(g(U))-D(g^\star(U))]\right)^\frac{1}{2\beta+1}\\
     & \leq C \sup \limits_{D \in \mathcal{H}^{\beta+1}_1}\left(\mathbb{E}[D(g(U))-D(g^\star(U))]\right)^\frac{1}{2\beta+1}.
    \end{align*}
If now $\beta$ is an integer, then using Corollary \ref{coro:ineq without reg} with $\theta=1$, $\theta_1=\frac{\beta-1/2}{2\beta}$ and $\theta_2=\beta+3/2$, we get the same result.
    Then by hypothesis we obtain
    \begin{align*}       
    W_1(g_{\# U},g^\star_{\# U}) 
    & \leq Ct^{d+1}.
    \end{align*}
Suppose that $\max_{x\in \mathcal{M}_g} d(x,\mathcal{M}_{g^\star}) \geq \delta$. For $u\in[0,1]^d$ such that $d(g(u),\mathcal{M}_{g^\star})\geq \delta$, using the dual formulation of the $1$-Wasserstein distance, we have that
\begin{align*}
    W_1(g_{\# U},g^\star_{\# U})&  = \sup \limits_{D \in \text{Lip}_1}\mathbb{E}_{U\sim \mathcal{U}([0,1]^d)}[D(g(U))-D(g^\star(U))]\\
    & \geq \int_{[0,1]^d} d(g(z),\mathcal{M}_{g^\star})d\lambda^d(z)> \int_{B^d(u,\frac{\delta}{2K})} \frac{\delta}{2}d\lambda^d(z)\geq C \delta^{d+1}.
\end{align*}    
Therefore we can deduce that $$\mathbb{H}(\mathcal{M}_g,\mathcal{M}_{g^\star})\leq C t.$$  \end{proof}
In this proposition, the exponents are not optimal but are sufficient for our purpose. Using this result, let show that if the distance $d_{\mathcal{H}^{\beta+1}_1}(g_{\# U},g^\star_{\# U})$ is small enough, then there exists a $(g_{\# U},g^\star_{\# U})$-compatible map.

\begin{proposition}\label{prop:compatibility} 
    Let $g,g^\star \in \mathcal{H}^{\beta+1}_K(\mathbb{T}^d,\mathbb{R}^p)$ that verify the $K$-manifold regularity condition with $g^\star$ that verifies the $K$-density regularity condition. There exist constants $C,C_{2}>0$ such that if 
        $$
     \sup \limits_{D \in \mathcal{H}^{\beta+1}_1}\mathbb{E}_{U\sim \mathcal{U}([0,1]^d)}[D(g(U))-D(g^\star(U))]\leq C^{-1},
     $$
 then $g$ verifies the $C_2$-density 
 regularity condition and there exists a $(g_{\# U},g^\star_{\# U})_{K_T}$-compatible map with radius $t\geq C_2^{-1}$and  $K_T\leq C_2$. 
\end{proposition}
The proof of Proposition \ref{prop:compatibility} can be found in the following Section.

\subsection{Proof of Proposition \ref{prop:compatibility}}\label{sec:prop:compatibility}

 The proof of proposition \ref{prop:compatibility} breaks down as follow: 
we first approximate $\mathcal{M}_{g^\star}$ by a manifold $\mathcal{M}_{g_s}$ of regularity $\beta+2$ using convolution and show that its reach is bounded below. Using that $g_{\# U}$ is close to $g^\star_{\# U}$ for the $d_{\mathcal{H}^{\beta+1}_1}$ distance, we show that $g$ is a local diffeomorphism. We then show that $\pi_{s|_{\mathcal{M}_{g^\star}}}:\mathcal{M}_{g^\star}\rightarrow \mathcal{M}_{g_s}$ the restriction to $\mathcal{M}_{g^\star}$ of $\pi_s$ the orthogonal projection onto $\mathcal{M}_{g_s}$, is a $C^{\beta+1}$ diffeomorphism. Finally we define our map T as $T(x)=(\pi_{s|_{\mathcal{M}_{g^\star}}})^{-1} \circ \pi_s (x)$ and show that it verifies all the conditions to be $(g_{\# U},g^\star_{\# U})$-compatible.

\subsubsection{Approximation of $\mathcal{M}_{g^\star}$ by a manifold $\mathcal{M}_{g_s}$ of regularity $\beta+2$}

Let $s\in (0,r_{g^\star}/2)$ and  $R_s:\mathbb{T}^d\rightarrow \mathbb{R}$ defined by $R_s(u)=0 \vee ( 1-\|u\|/s)/C_s$ with $C_s$ such that $\|R_s\|_{L^1}=1$. Then define $g_s:\mathbb{T}^d \rightarrow \mathbb{R}^p$ by
$$
g_s :=g^\star \ast R_s.
$$
We have that $g_s$ belongs to  $ \mathcal{H}^{\beta+1}_{K}(\mathbb{T}^d,\mathbb{R}^p)\cap \mathcal{H}^{\beta+2}_{Ks^{-1}}(\mathbb{T}^d,\mathbb{R}^p)$, verifies that $$\|g^\star-g_s\|_\infty\leq Ks$$ and $\|g^\star-g_s\|_{\mathcal{H}^{\lfloor \beta \rfloor+1}}\leq Cs$. Write $C_0$ for the smallest $C>0$ such that
$$\|g^\star-g_s\|_{\mathcal{H}^{\lfloor \beta \rfloor+1}}\leq C_0s.$$
We do not fix the value of $s$ yet because we will lower it all along the proof to suit our needs. Let us show that $g_s$ is a diffeomorphism.

\begin{proposition}\label{eq:justepourle}
    There exists a constant $C>0$ such that if $s\in (0,C^{-1})$, the map $g_s$ is a diffeomorphism on its image.
\end{proposition}

\begin{proof}
For $u\in \mathbb{T}^d$ and $w\in \mathbb{R}^d$, we have 
$$\|\nabla g_s(u) \frac{w}{\|w\|}\|\geq \|\nabla g^\star(u) \frac{w}{\|w\|}\|-C_0s\geq K^{-1}-C_0s\geq \frac{1}{2K} $$
for $s\leq (4KC_0)^{-1}$, so $g_s$ is a local diffeomorphism. Suppose there exists $u,v\in \mathbb{T}^d$, such that $g_s(u)=g_s(v)$ and $u\neq v$. Then as 
\begin{align*}
\|&g_s(u)-g_s(v)\| \\
& = \|\nabla g_s(u)(v-u)+\int_0^1 \nabla^2 g_s(u+t(v-u))(v-u)^2(1-t)dt\|\\
& \geq \|\nabla g^\star(u)(v-u)\| -\|\int_0^1 \nabla^2 g^\star(u+t(v-u))(v-u)^2(1-t)dt\| -2C_0s\|u-v\|\\
&\geq  K^{-1}\|u-v\|-\frac{K}{2}\|u-v\|^2-2C_0s\|u-v\|\\
& =\frac{K}{2}\|u-v\|(2K^{-2}-4\frac{C_0}{K}s-\|u-v\|),
\end{align*}
we deduce that $\|u-v\|\geq 2K^{-2}-4\frac{C_0}{K}s$. Then for $s\leq (16C_0K^{3})^{-1}$, we have $$\|g^\star(u)-g^\star(v)\|\leq \|g_s(u)-g_s(v)\|+2C_0s\leq (8K^{3})^{-1}.$$ 
As $\|u-v\|\geq K^{-2}$, applying Lemma \ref{lemma:reachcurve}, we obtain that $r_{g^\star}\leq (16K^{3})^{-1}$ which is impossible as $g^\star$ verifies the $K$-manifold regularity condition. Therefore, $g_s$ is injective and so is a diffeomorphism.
\end{proof}

Let us show that the submanifold $\mathcal{M}_{g_s}=g_s([0,1]^d)$ as a stricly positive  reach.

\begin{lemma}\label{lemma:reachgs}
        There exists a constant $C>0$ such that if $s\in (0,C^{-1})$, we have that the reach $r_{g_s}$ of $\mathcal{M}_{g_s}$, verifies $$r_{g_s}\geq C^{-1}.$$
\end{lemma}

\begin{proof}
Let $u,v\in \mathbb{T}^d$, recall from the proof of proposition \ref{eq:justepourle} that for $s\leq (4KC_0)^{-1}$, we have $\|\nabla g_s(v)(u-v)\|\geq \frac{\|u-v\|}{2K}$. We are going to use characterization of the reach for submanifolds of Lemma \ref{lemma:reachformanifolds}. Using a Taylor expansions on $g_s$ we have
\begin{align}\label{align:reachbound}
    & \frac{\|g_s(u)-g_s(v)\|^2}{2d(g_s(u)-g_s(v),\mathcal{T}_{g_s(v)}(\mathcal{M}_{g_s}))}  =     \frac{\|\nabla g_s(v)(u-v)+O(\|u-v\|^2)\|^2}{2d(\nabla g_s(v)(u-v)+O(\|u-v\|^2),\mathcal{T}_{g_s(v)}(\mathcal{M}_{g_s}))}\nonumber\\
     &\geq     \frac{\|\nabla g_s(v)(u-v)\|^2 +2\langle \nabla g_s(v)(u-v),O(\|u-v\|^2) \rangle+\|O(\|u-v\|^2)\|^2}{2d(\nabla g_s(v)(u-v),\mathcal{T}_{g_s(v)}(\mathcal{M}_{g_s}))+2d(O(\|u-v\|^2)\|,\mathcal{T}_{g_s(v)}(\mathcal{M}_{g_s}))}\nonumber\\
    & \geq \frac{\frac{K^{-2}}{4}\|u-v\|^2-2K^2\|u-v\|^3-K^2\|u-v\|^4}{2K\|u-v\|^2}\nonumber\\
    & \geq \frac{1}{8K^3}(1-K^4(8\|u-v\|+4\|u-v\|^2)),
\end{align}
where we used that $d(\nabla g_s(v)(u-v),\mathcal{T}_{g_s(v)}(\mathcal{M}_{g_s}))=0$. Then $\|u-v\|\leq (32K^4)^{-1}$ implies that $$\frac{\|g_s(u)-g_s(v)\|^2}{2d(g_s(u)-g_s(v),\mathcal{T}_{g_s(v)}(\mathcal{M}_{g_s}))}> (16K^3)^{-1}.$$ 
Therefore, if $r_{g_s}\leq (16K^3)^{-1}$ we have 
$$
r_{g_s}=\min \limits_{\|u-v\| \geq (32K^4)^{-1} } \frac{\|g_s(u)-g_s(v)\|^2}{2d(g_s(u)-g_s(v),\mathcal{T}_{g_s(v)}(\mathcal{M}_{g_s}))}.
$$
Let $$(u^\star,v^\star)\in \argmin \limits_{\|u-v\| \geq (32K^4)^{-1} } \frac{\|g_s(u)-g_s(v)\|^2}{2d(g_s(u)-g_s(v),\mathcal{T}_{g_s(v)}(\mathcal{M}_{g_s}))},$$ then 
\begin{align*}
\|g^\star(u^\star)-g^\star(v^\star)\|& \leq \|g_s(u^\star)-g_s(v^\star)\|+2C_0s\\
& \leq \frac{\|g_s(u^\star)-g_s(v^\star)\|^2}{d(g_s(u^\star)-g_s(v^\star),\mathcal{T}_{g_s(v^\star)}(\mathcal{M}_{g_s}))} +2C_0s\\
&= 2r_{g_s}+2C_0s.
\end{align*}
Then, if  $r_{g_s}\vee C_0s < (2^8K^5)^{-1}$, 
we have $\|g^\star(u^\star)-g^\star(v^\star)\|<(2^6K^5)^{-1}$ and as $\|u^\star-v^\star\|\geq (32K^4)^{-1}$, applying
Lemma \ref{lemma:reachcurve}, we obtain $r_{g^\star}\leq (2^7K^5)^{-1}$
which is impossible as $r_{g^\star}\geq K^{-1}$. Therefore, taking $s<(C_02^8K^5)^{-1} $ we have $r_{g_s}\geq (2^8K^5)^{-1}$.
\end{proof}

\subsubsection{Properties of the projection onto $\mathcal{M}_{g_s}$}

Define $\pi_s:\mathcal{M}_{g_s}^{r_{g_s}}\rightarrow \mathcal{M}_{g_s}$ as the orthogonal projection.
Let us first show that $\pi_s$ verifies the point $ii)$ of Definition \ref{defi:compatibility} on a neighborhood of $\mathcal{M}_{g_s}$.

\begin{proposition}\label{prop:projectconstant}
    There exists a constant $C>0$ such that the projection $\pi_s$ verifies the point $ii)$ of Definition \ref{defi:compatibility} on $\mathcal{M}_{g_s}^{C^{-1}}$.
\end{proposition}

\begin{proof}
Let $t\in (0,r_{gs}/2)$ and $x\in \mathcal{M}_{g_s}^{t}$, $h\in \mathcal{T}_{\pi_s(x)}(M_{g_s})^\perp$ with $\|h\|\leq t$ , $u=g_s^{-1}\circ \pi_s(x)$ and $v\in [0,1]^d$. Define $\gamma:[0,1]\rightarrow \mathbb{T}^d$ by
$$
\gamma(r)=g_s^{-1}\circ \pi_s(x+rh).
$$
Then, we have that $\gamma$ is $4Kh$-Lipschitz as $\pi_s$ is $2$-Lipschitz on $\mathcal{M}_{g_s}^{t}$ (see formula \eqref{grad} and equation \eqref{eq:shapeop}) and $g_s^{-1}$ is $2K$-Lipschitz. We would like to show that $\pi_s(x+h)=\pi_s(x)$ i.e $$\|x+h-g_s(u+v)\|> \|x+h-g_s(u)\|.$$ Then as $\gamma$ is $4Kh$-Lipschitz it is sufficient to prove it for $v\in B^d(0,4Kh)$. We have
\begin{align*}
    \|x+h-g_s(u+v)\|^2  = & \|x+h-g_s(u)\|^2 + 2\langle x+h-g_s(u),-\nabla g_s(u)v +O(\|v\|^2) \rangle\\
    & + \| -\nabla g_s(u)v +O(\|v\|^2)\|^2.
\end{align*}
On the ohter hand, $x+h-g_s(u)=x+h-\pi_s(x)\in \text{Im}(\nabla g_s(u))^\perp$ the orthogonal of the image of $\nabla g_s(u)$,  so
\begin{align*}
    \|x+h-g_s(u+v)\|^2  = &\|x+h-g_s(u)\|^2 + 2\langle x+h-g_s(u),O(\|v\|^2) \rangle\\
    & + \| \nabla g_s(u)v \|^2 -2\langle \nabla g_s(u)v, O(\|v\|^2)\rangle +O(\|v\|^4)\\
     \geq &\|x+h-g_s(u)\|^2 - 2\|x+h-g_s(u)\|K\|v\|^2\\
     & + \frac{K^{-2}}{4}\|v \|^2 -2K\|v\| K\|v\|^2 -K^2\|v\|^4\\
      \geq &\|x+h-g_s(u)\|^2  + \|v \|^2(\frac{K^{-2}}{4}-4tK-2K^2\|v\|-K^2\|v\|^2)\\
     \geq &\|x+h-g_s(u)\|^2  + \|v \|^2(\frac{K^{-2}}{4}-K^3(8t+4Kt^2))
\end{align*}
where we used that $\|v\|\leq 4Kh\leq 4Kt$.
Then for $t=K^{-5}/2^6 $ we have $\|x+h-g(u+v)\|> \|x+h-g(u)\|$, so $\pi_s$ verifies $ii)$ for a radius $t$.
\end{proof}

\begin{lemma}
    The map $\pi_{s|_{\mathcal{M}_{g^\star}}}:\mathcal{M}_{g^\star}\rightarrow \mathcal{M}_{g_s}$, i.e. the restriction of the projection $\pi_s$ to the manifold $\mathcal{M}_{g^\star}$, is injective.
\end{lemma}

\begin{proof}
Suppose there exist $u,v\in \mathbb{T}^d$ different such that $\pi_s(g^\star(u))=\pi_s(g^\star(v))$. We have 
\begin{align*}
    \|g_s(u)-g_s(v)\| & \leq \|g^\star(u)-g^\star(v)\|+2C_0s\\
    & \leq \|g^\star(u)-\pi_s(g^\star(u))\|+ \|g^\star(v)-\pi_s(g^\star(u))\| +2C_0s\leq 4C_0s
\end{align*}
Suppose first that $\|u-v\|\geq (32K^4)^{-1}$ and let 
$$
(x,y)\in  \argmin \limits_{\|a-b\|\geq (32K^4)^{-1}} \|g_s(a)-g_s(b)\|.
$$
As detailed in the proof of Lemma \ref{lemma:reachgs}, we have that $g_s(x)-g_s(y)\in \mathcal{T}_{g_s(y)}(\mathcal{M}_{g_s})^\perp$ and  
\begin{align*}
r_{g_s} & =\min \limits_{\|a-b\| \geq (32K^4)^{-1} } \frac{\|g_s(a)-g_s(b)\|^2}{2d(g_s(a)-g_s(b),\mathcal{T}_{g_s(b)}(\mathcal{M}_{g_s}))} \leq \frac{\|g_s(x)-g_s(y)\|^2}{2d(g_s(x)-g_s(y),\mathcal{T}_{g_s(y)}(\mathcal{M}_{g_s}))}\\
& = \frac{1}{2} \|g_s(x)-g_s(y)\| \leq \frac{1}{2}\|g_s(u)-g_s(v)\| \leq 2C_0s.
\end{align*}
As from Lemma \ref{lemma:reachgs}
 we have $r_{g_s}\geq C^{-1}$, supposing $s>0$ is small enough, this is impossible and so $\|u-v\|\leq (32K^4)^{-1}$.

Let $F:[0,\infty)\rightarrow \mathbb{R}^p$ defined by $F(t)=g^\star(u+t(u-v))$
then
\begin{align}\label{eq:bornepourinjec}
|d(F(t)&,\mathcal{M}_{g_s})-d(g^\star(u)+t(g^\star(u)-g^\star(v)),\mathcal{M}_{g_s})| \nonumber\\
& \leq
    \|F(t)-(g^\star(u)+t(g^\star(u)-g^\star(v)))\|\nonumber\\
    & \leq t\|\nabla g^\star(u)(u-v)-(g^\star(u)-g^\star(v))\| + K\|u-v\|^2t^2\nonumber\\
    & \leq Kt\|u-v\|^2 + K\|u-v\|^2t^2 = K\|u-v\|^2(t+t^2).
    \end{align}
We have 
\begin{align*}
\langle & g^\star(v)-\pi_s(g^\star(v)),g^\star(v)-g^\star(u)\rangle\\
& = \langle g^\star(v)-g^\star(u)+g^\star(u)-\pi_s(g^\star(u)),g^\star(v)-g^\star(u)\rangle\\
& = \|g^\star(v)-g^\star(u)\|^2- \langle g^\star(u)-\pi_s(g^\star(u)),g^\star(u)-g^\star(v)\rangle,
    \end{align*}
so without loss of generality we can suppose that $\langle g^\star(u)-\pi_s(g^\star(u)),g^\star(u)-g^\star(v)\rangle\geq 0$ (otherwise we would have 
$\langle g^\star(v)-\pi_s(g^\star(v)),g^\star(v)-g^\star(u)\rangle\geq 0$). Let $C_1>0$ be the constant from Proposition \ref{prop:projectconstant} such that $\pi_s$ verifies the point $ii)$ of Definition \ref{defi:compatibility} on $\mathcal{M}_{g_s}^{C_1^{-1}}$. Supposing $s\leq (2C_0C_1)^{-1}$ we have $\mathbb{H}(\mathcal{M}_{g_s},\mathcal{M}_{g^\star})\leq C_0s<C_1^{-1}/2$, then for $t\in [0,\frac{C_1^{-1}}{2\|g^\star(u)-g^\star(v)\|}]$, using Proposition \ref{prop:projectconstant} we have $\pi_s(g^\star(u)+t(g^\star(u)-g^\star(v)))=\pi_s(g^\star(u))$. Then 
\begin{align*}
   & d(g^\star(u)+t(g^\star(u)-g^\star(v)),\mathcal{M}_{g_s})^2=\|g^\star(u)+t(g^\star(u)-g^\star(v))-\pi_s(g^\star(u))\|^2\\ & =\|g^\star(u)-\pi_s(g^\star(u))\|^2+2t\langle g^\star(u)-\pi_s(g^\star(u)),g^\star(u)-g^\star(v)\rangle+t^2\|g^\star(u)-g^\star(v)\|^2\\
    & \geq t^2\|g^\star(u)-g^\star(v)\|^2.
\end{align*}
Using \eqref{eq:bornepourinjec} for $t=\frac{2C_0s}{\|g^\star(u)-g^\star(v)\|}$, we have 
\begin{align*}
    d(F(\frac{2C_0s}{\|g^\star(u)-g^\star(v)\|}),\mathcal{M}_{g_s})  \geq & d(g^\star(u)+\frac{2C_0s}{\|g^\star(u)-g^\star(v)\|}(g^\star(u)-g^\star(v)),\mathcal{M}_{g_s})\\
    & - K\|u-v\|^2\left(\frac{2C_0s}{\|g^\star(u)-g^\star(v)\|}+\frac{(2C_0s)^2}{\|g^\star(u)-g^\star(v)\|^2}\right).
\end{align*}
Then as $\|g^\star(u)-g^\star(v)\|\leq  \|g^\star(u)-\pi_s(g^\star(u))\|+ \|g^\star(v)-\pi_s(g^\star(v))\| \leq 2C_0s$, we have $\frac{2C_0s}{\|g^\star(u)-g^\star(v)\|}\leq \left(\frac{2C_0s}{\|g^\star(u)-g^\star(v)\|}\right)^2$ so 
\begin{align*}
    d(F(\frac{2C_0s}{\|g^\star(u)-g^\star(v)\|}),\mathcal{M}_{g_s})  \geq & \frac{2C_0s}{\|g^\star(u)-g^\star(v)\|}\|g^\star(u)-g^\star(v)\|\\
    & -2K\|u-v\|^2\left(\frac{2C_0s}{\|g^\star(u)-g^\star(v)\|}\right)^2\\
     \geq & 2C_0s-8KC_0^2s^2\left(\frac{\|u-v\|}{\|g^\star(u)-g^\star(v)\|}\right)^2\\
     \geq & 2C_0s(1-4KC_0s\left(\frac{\|u-v\|}{K^{-1}\|u-v\|-K\|u-v\|^2}\right)^2) \\
     \geq & 2C_0s(1-16K^3C_0s)
    \end{align*}
    where we used that $\|u-v\| \leq (32K^4)^{-1}$ implies that $\|u-v\|-K^2\|u-v\|^2\geq \|u-v\|/2$.
    Then taking $s\leq (2^6K^3C_0)^{-1}$ we have $d(F(\frac{2C_0s}{\|g^\star(u)-g^\star(v)\|},\mathcal{M}_{g_s})>C_0s$ which is impossible by definition of $g_s$. 
\end{proof}

From \citepproofs{Leobacher} we know that the canonical projection $\pi$ onto a submanifold $\mathcal{M}$ with strictly positive reach $r$, verifies
\begin{equation}\label{grad}
\nabla \pi(x)=\left(\text{Id}_{\mathcal{T}_{\pi(x)}(\mathcal{M})}-\|x-\pi(x)\|L_{\pi(x),v} \right)^{-1}P_{\mathcal{T}_{\pi(x)}(\mathcal{M})},
\end{equation}
for $L_{\pi(x),v}$ the shape operator in the direction $v=\frac{x-\pi(x)}{\|x-\pi(x)\|}$. From Corollary 3 in \citeproofs{Leobacher}, we also have that 
\begin{equation}\label{eq:shapeop}
\|\left(\text{Id}_{\mathcal{T}_{\pi(x)}(\mathcal{M})}-\|x-\pi(x)\|L_{\pi(x),v} \right)^{-1}\|\leq (1-\|x-\pi(x)\|/r)^{-1}.
\end{equation}
Furthermore if $\mathcal{M}$ is of regularity $\beta+2$, there exists a constant $C>0$ depending on the $\mathcal{H}^{\beta+2}$ norm of the charts such that
\begin{equation}\label{gradborn}
\|\pi_{|_{\mathcal{M}^{r/2}}}\|_{\mathcal{H}^{\beta+1}}\leq C
\end{equation}
Using these results, we obtain in the next proposition that the restriction of
$\pi_s$ to the manifold $\mathcal{M}_{g^\star}$ is a diffeomorphism.
\begin{proposition}\label{prop:pisdiffeo}
    There exists a constant $C>0$ such that if $s\leq C^{-1}$, then
 the map $\pi_{s|_{\mathcal{M}_{g^\star}}}$ is a diffeomorphism.
\end{proposition}

\begin{proof}
From \eqref{gradborn} we have
\begin{equation}\label{eq:gradborn}
\|\pi_{s|_{\mathcal{M}_{g_s}^{r_{g_s}/2}}}\|_{\mathcal{H}^{\beta+1}}\leq C s^{-1}.
\end{equation}

To show that $\pi_{s|_{\mathcal{M}_{g^\star}}}$ is a local diffeomorphism, let us show that if there existed $u$ such that $\lambda_{\min}(\nabla (\pi_{s}\circ g^\star)(u)^\top \nabla (\pi_{s}\circ g^\star)(u))$ was too small, then the Hausdorff distance between $\mathcal{M}_{g_s}$ and $\mathcal{M}_{g^\star}$ would be too big. Let $u\in [0,1]^d$ and $h>0$, then for $v\in \mathbb{R}^d$ such that $\|v\|=1$ and $$\langle \pi_{s}(g^\star(u))-g^\star(u), \nabla(\pi_{s}\circ g^\star - g^\star)(u)v\rangle\geq 0,$$ we have
\begin{align*}
    \|\pi_{s} & (g^\star(u+hv))-g^\star(u+hv)\|  \geq  \|\pi_{s}(g^\star(u))-g^\star(u)+ \nabla (\pi_{s}\circ g^\star - g^\star)(u)hv\| - 2Kh^2\\
     =& \big(\|\pi_{s}(g^\star(u))-g^\star(u)\|^2 +2\langle \pi_{s}(g^\star(u))-g^\star(u), \nabla (\pi_{s}\circ g^\star - g^\star)(u)hv\rangle\\
     & +\|\nabla (\pi_{s}\circ g^\star - g^\star)(u)hv \|^2\big)^{1/2} - 2Kh^2\\
    \geq & h\|\nabla (\pi_{s}\circ g^\star - g^\star)(u)v \| - 2Kh^2\\
    \geq & h(\|\nabla g^\star(u)v \|-\|\nabla (\pi_{s}\circ g^\star)(u)v \|) - 2Kh^2\\
     \geq &h(K^{-1}-\|\nabla (\pi_{s}\circ g^\star)(u)v \|) - 2Kh^2.
\end{align*}
As $\mathbb{H}(\mathcal{M}_{g_s},\mathcal{M}_{g^\star})\leq Ks$ we deduce with $h=(K^{-1}-\|\nabla (\pi_{s}\circ g^\star)(u)v \|)/4K$ that 
$$\frac{(K^{-1}-\|\nabla (\pi_{s}\circ g^\star)(u)v \|)^2}{8K}\leq Ks$$
so
\begin{align}\label{lambdamin}
    \lambda_{\min}(\nabla (\pi_{s}\circ g^\star)(u) ^\top \nabla (\pi_{s}\circ g^\star)(u))^{1/2} \geq K^{-1}-\sqrt{8s}K.
\end{align}

In particular $\pi_s\circ g^\star$ is an immersion and therefore $\pi_{s|_{\mathcal{M}_{g^\star}}}:\mathcal{M}_{g^\star}\rightarrow \mathcal{M}_{g_s}$ is a local diffeomorphism. As $\mathcal{M}_{g^\star}$ is compact without boundary, $\pi_s(\mathcal{M}_{g^\star})$ is a compact boundaryless submanifold of  $\mathcal{M}_{g_s}$  so being of same dimension, there are equal. We can conclude that $\pi_{s|_{\mathcal{M}_{g^\star}}}$ is a diffeomorphism.
\end{proof}

\subsubsection{Density regularity of $g$}
As for all $u\in \mathbb{T}^d$ we have $
    \lambda_{\min}((\nabla g^\star(u) )^\top \nabla g^\star(u))^{1/2}\geq 1/K$, let us show that if $\inf \limits_{u\in[0,1]^d} \lambda_{\min}((\nabla g(u) )^\top \nabla g(u))$ was to small, then the distane $d_{\mathcal{H}^{\beta+1}_1}(g_{\# U},g^\star_{\# U})$ would be too big.
\begin{proposition}\label{prop:lambdamin}
    There exist constants $C,\delta>0$ such that for all $\tau>\delta$, if $$\inf \limits_{u\in \mathbb{T}^d} \inf \limits_{w\in \mathbb{R}^d}\|\nabla g(u)\frac{w}{\|w\|}\|\leq K^{-\tau},$$ then 
    $$
d_{\mathcal{H}^{\beta+1}_1}(g_{\# U},g^\star_{\# U})\geq CK^{-\delta(d+1)(2\beta+1)}.
$$
\end{proposition}

\begin{proof}
    Let $\tau> 3$  and suppose that there exists $u_0\in \mathbb{T}^d$ and $w_0\in \mathbb{R}^d$ with $\|w_0\|=1$ such that $\|\nabla g(u_0)w_0\|\leq K^{-\tau}$. For $\delta\in (3,\tau)$, define $D_\delta:\mathcal{M}^{K^{-2}/2}_{g^\star}\rightarrow \mathbb{R}$ by
$$D_\delta(x)=(\frac{K^{-\delta}}{8}\wedge(\frac{K^{-\delta}}{4}-\|\pi(x)-\pi(g(u_0))\|))\vee 0$$
for $\pi$ the canonical projection onto $\mathcal{M}_{g^\star}$. As $D_\delta$ is $2$-Lipschitz (see \eqref{grad}), we have
\begin{align}\label{align:w1}
    W_1(g_{\# U},g^\star_{\# U})  \geq & \frac{1}{2}\mathbb{E}[D_\delta(g(U))-D_\delta(g^\star(U))]\nonumber\\
     = &\frac{1}{2}\mathbb{E}[D_\delta(g(U))\mathds{1}_{\{\|\pi(g(U))-\pi(g(u_0))\|\leq \frac{K^{-\delta}}{4}\}}]\\
     & -\frac{1}{2}\mathbb{E}[D_\delta(g^\star(U))\mathds{1}_{\{\|\pi(g^\star(U))-\pi(g(u_0))\|\leq \frac{K^{-\delta}}{4}\}}]\nonumber\\
     \geq & \frac{1}{2}\mathbb{E}[ \frac{K^{-\delta}}{8}\mathds{1}_{\{\|\pi(g(U))-\pi(g(u_0))\|\leq \frac{K^{-\delta}}{8}\}}-\frac{K^{-\delta}}{4}\mathds{1}_{\{\|\pi(g^\star(U))-\pi(g(u_0))\|\leq \frac{K^{-\delta}}{4}\}}]
\end{align}
Let us first bound below the quantity $\mathbb{E}[\mathds{1}_{\{\|\pi(g(U))-\pi(g(u_0))\|\leq \frac{K^{-\delta}}{8}\}}]$. Let $v\in \mathbb{R}^d$ and write $v_0=\langle v,w_0\rangle w_0$, $v_0^\perp=v-v_0$. Using a Taylor expansion on the path $\zeta:[0,1]\rightarrow \mathbb{R}^d$ defined by $$\zeta(t)= u_0+4(\frac{t}{2}\wedge\frac{1}{2}) v_0+4((\frac{t}{2}-\frac{1}{2})\vee0)v^{\perp}_0,$$ we have
\begin{align*}
    & \|g(u_0+v)-g(u_0)\|\\
    & =\| \int_0^1\nabla g(\zeta(t))\dot{\zeta}(t)dt\|\\
    & =\|\int_0^1 \left(\nabla g(u_0+tv_0)v_0+\nabla g(u_0+v_0+tv_0^\perp)v^{\perp}_0\right) dt\|\\
    & = \|\int_0^1 \left(\nabla g(u_0)v_0 +\int_0^1\nabla^2 g(u_0+stv_0)tv_0^2)ds+\nabla g(u_0+v_0+tv_0^\perp)v^{\perp}_0\right) dt\|\\
    & \leq K^{-\tau}\|v_0\|+ \frac{K}{2}\|v_0\|^2+K\|v_0^\perp\| .
\end{align*}
Therefore, if $\|v_0\|\leq K^{-\frac{\delta+1}{2}}/2^6$ and $\|v_0^\perp\|\leq K^{-(\delta+1)}/2^5 $ then 
$$\|\pi(g(u_0+v))-\pi(g(u_0))\|\leq 2\|g(u_0+v)-g(u_0)\|\leq \frac{K^{-\delta}}{8}.$$
Writing $C_{d-1}$ the volume of the $d-1$-dimensional ball, we can conclude that
$$\mathbb{E}[\mathds{1}_{\{\|\pi(g(U))-\pi(g(u_0))\|\leq \frac{K^{-\delta}}{8}\}}]\geq C_{d-1}\left(\frac{K^{-(\delta+1)}}{2^{5}}\right)^{d-1}\frac{K^{-\frac{\delta+1}{2}}}{2^{6}}.$$

Let us now bound the quantity $\mathbb{E}[\mathds{1}_{\{\|\pi(g^\star(U))-\pi(g(u_0))\|\leq \frac{K^{-\delta}}{4}\}}]$. For $u\in \mathbb{T}^d$, we have
\begin{align}\label{eq:petitcalcul}
\|\pi(g^\star(u))- \pi \circ g(u_0)\|  = &
    \|g^\star(u)-g^\star\circ g^{\star -1}\circ \pi \circ g(u_0)\|\nonumber\\
     \geq  &\|\nabla g^\star (g^{\star -1}\circ \pi \circ g(u_0))(u-g^{\star -1}\circ \pi \circ g(v_0))\|\nonumber\\
     & -K(\|u-g^{\star -1}\circ \pi \circ g(u_0)\|^2)\nonumber\\
    \geq & K^{-1}\|u-g^{\star -1}\circ \pi \circ g(u_0)\|-K\|u-g^{\star -1}\circ \pi \circ g(u_0)\|^2.
\end{align}
As $r_{g^\star}\geq K^{-1}$, from Lemma \ref{lemma:reachcurve} we have 
$$
\|g^\star(u)-g^\star(v)\|\leq K^{-3}/4\ \text{  imply }\ \|u-v\|\leq K^{-2}/2,
$$
then,
\begin{equation}\label{eq:calcul2}
\|g^\star(u)-g^\star\circ g^{\star -1}\circ \pi \circ g(u_0)\| \leq \frac{K^{-\delta}}{4}\ \text{ imply }\ \|u-g^{\star -1}\circ \pi \circ g(u_0)\|\leq K^{-2}/2.
\end{equation}
Solving $K^{-1}x-Kx^2\leq \frac{K^{-\delta}}{4}$, we find $x\notin (\frac{K^{-1}-\sqrt{K^{-2}-K^{-\delta+1}}}{2K},\frac{K^{-1}+\sqrt{K^{-2}-K^{-\delta+1}}}{2K})$ so from~\eqref{eq:petitcalcul},
 $$ \|g^\star(u)-\pi \circ g(u_0)\|\leq \frac{K^{-\delta}}{4}$$
 implies $$\|u-g^{\star -1}\circ \pi \circ g(u_0)\| \notin (\frac{K^{-1}-\sqrt{K^{-2}-K^{-\delta+1}}}{2K},\frac{K^{-1}+\sqrt{K^{-2}-K^{-\delta+1}}}{2K}).$$
From \eqref{eq:calcul2}, we have that $\|g^\star(u)-\pi \circ g(u_0)\|\leq \frac{K^{-\delta}}{4}$ implies
$$
\|u-g^{\star -1}\circ \pi \circ g(u_0)\|\leq K^{-2}/2< \frac{K^{-1}+\sqrt{K^{-2}-K^{-\delta+1}}}{2K},
$$
so finally $ \|g^\star(u)-\pi \circ g(u_0)\|\leq \frac{K^{-\delta}}{4}$ implies that
$$
\|u-g^{\star -1}\circ \pi \circ g(u_0)\| \leq \frac{K^{-1}-\sqrt{K^{-2}-K^{-\delta+1}}}{2K}\leq \frac{K^{-\delta+2}}{2K}=\frac{K^{-\delta+1}}{2}.
$$
Writing $C_{d}$ the volume of the $d$-dimensional ball, we can conclude that
$$\mathbb{E}[\mathds{1}_{\{\|\pi(g^\star(U))-\pi(g(u_0))\|\leq \frac{K^{-\delta}}{4}\}}]\leq C_{d}\left(\frac{K^{-\delta+1}}{2}\right)^d.
$$
Finally, from \eqref{align:w1} we have
\begin{align*}
    W_1(g_{\# U},g^\star_{\# U})& \geq \frac{1}{2}\mathbb{E}[ \frac{K^{-\delta}}{8}\mathds{1}_{\{\|\pi(g(U))-\pi(g(u_0))\|\leq \frac{K^{-\delta}}{8}\}}-\frac{K^{-\delta}}{4}\mathds{1}_{\{\|\pi(g^\star(U))-\pi(g(u_0))\|\leq \frac{K^{-\delta}}{4}\}}]\\
    & \geq \frac{1}{2}\left( \frac{K^{-\delta}}{8} C_{d-1}\left(\frac{K^{-(\delta+1)}}{2^{5}}\right)^{d-1}\frac{K^{-\frac{\delta+1}{2}}}{2^{6}}-\frac{K^{-\delta}}{4}C_{d}\left(\frac{K^{-\delta+1}}{2}\right)^d\right)\\
    & =\frac{C_{d-1}K^{-\delta(d+1)}}{2^{5(d+1)}}(K^{\delta/2-d+1/2}-\frac{C_d2^{4d+2}K^d}{C_{d-1}}).
\end{align*}
Then for $\delta>3$ such that $K^{\delta/2-d+1/2}-\frac{C_d2^{4d+2}K^d}{C_{d-1}}\geq \frac{2^{5(d+1)}}{C_{d-1}}$ we have 
$$
W_1(g_{\# U},g^\star_{\# U})\geq K^{-\delta(d+1)}
$$
so from Proposition \ref{prop:hausdorff} we have 
$$
d_{\mathcal{H}^{\beta+1}_1}(g_{\# U},g^\star_{\# U})\geq CK^{-\delta(d+1)(2\beta+1)}.
$$
\end{proof}

Recalling that in the assumptions of Proposition \ref{prop:compatibility} we suppose that $
d_{\mathcal{H}^{\beta+1}_1}(g_{\# U},g^\star_{\# U})\leq C^{-1}$, then supposing $C>0$ large enough, using proposition \ref{prop:lambdamin}, we obtain 
\begin{equation}\label{eq:lambdaming}
    \inf \limits_{u\in[0,1]^d} \lambda_{\min}((\nabla g(u) )^\top \nabla g(u))^{1/2}\geq C^{-1}.
    \end{equation}

\subsubsection{The $(g_{\# U},g^\star_{\# U})$-compatible map}

We can now define the map $T:\mathcal{M}_{g^\star}^{r_{g_s}/4}\rightarrow \mathcal{M}_{g^\star}$ by 
$$T(x)=(\pi_{s|_{\mathcal{M}_{g^\star}}})^{-1} \circ \pi_s (x),$$
which is well defined because $\mathbb{H}(\mathcal{M}_{g_s},\mathcal{M}_{g^\star})\leq Ks< r_{g_s}/4$. Let us now state that $T$ is the map we are interested in:

\begin{proposition}
    The map $T:\mathcal{M}_{g^\star}^{t}\rightarrow \mathcal{M}_{g^\star}$ defined by 
$T(x)=(\pi_{s|_{\mathcal{M}_{g^\star}}})^{-1} \circ \pi_s (x)$, is $(g_{\# U},g^\star_{\# U})_{K_T}$-compatible with radius $t\geq C^{-1}$ and  $K_T\leq C$.
\end{proposition}

\begin{proof}
Let us start by proving the point $i)$. We have 
\begin{align}\label{gradT}
    \nabla T(x)=  \nabla (\pi_{s|_{\mathcal{M}_{g^\star}}})^{-1} (\pi_s(x))\circ \nabla \pi_s(x).
\end{align}
From \eqref{lambdamin}, 
as $g^\star$ is $K$-Lipschitz, we have that for $x\in \mathcal{M}_{g^\star}$ and $h\in \mathcal{T}_{x}(\mathcal{M}_{g^\star})$, 
\begin{equation}\label{ineqpis}
\|\nabla \pi_{s|_{\mathcal{M}_{g^\star}}}(x)h\|\geq K^{-1}(K^{-1}-\sqrt{8s}K)\|h\|.
\end{equation}
Then, as $\pi_s \in \mathcal{H}^{\beta+1}_{Cs^{-C_2}}(\mathcal{M}_{g_s}^{r_{g_s}/4},\mathcal{M}_{g_s})$, from \eqref{gradT} and the Faa di Bruno formula, we deduce that $T\in \mathcal{H}^{\beta+1}_{Cs^{-C_2}}(\mathcal{M}_{g_s}^{r_{g_s}/4},\mathcal{M}_{g^\star})$.

Let us now show that the density of $(T\circ g)_{\# U}$ with respect to the volume measure on $\mathcal{M}_{g^\star}$, is of regularity $\beta$. As $\mathbb{H}(\mathcal{M}_{g_s},\mathcal{M}_{g})\leq \mathbb{H}(\mathcal{M}_{g},\mathcal{M}_{g^\star}) +\mathbb{H}(\mathcal{M}_{g_s},\mathcal{M}_{g^\star})\leq 2Ks $ we show the same way as we did with $g^\star$ that $\pi_{s}\circ g$ is a  diffeomorphism and therefore, $T \circ g$ is also a diffeomorphism. Let us write $f_T$ the density of $(T\circ g)_{\# U}$ with respect to the volume measure on $\mathcal{M}_{g^\star}$. We have 
$$
f_T(x)=\text{det}\big(\nabla(T\circ g)((T\circ g)^{-1}(x))^\top \nabla(T\circ g)((T\circ g)^{-1}(x))\big)^{-1/2}.
$$
Using \eqref{eq:lambdaming} and
doing the same derivations as with $g^\star$ in the proof of Proposition \ref{prop:pisdiffeo}, we can show that there exists a constant $C>0$ such that for all $u\in \mathbb{T}^d$,
$$
\lambda_{\min}(\nabla (\pi_{s}\circ g)(u)^\top \nabla (\pi_{s}\circ g)(u))^{1/2} \geq C^{-1}-\sqrt{16s}K.
$$
Furthermore, for all $u\in \mathbb{T}^d$, as
 $\forall h \in \mathcal{T}_{\pi_s(g(u))}(\mathcal{M}_{g_s})$, 
$$\|\nabla  \pi_{s|_{\mathcal{M}_{g^\star}}}(T\circ g(u))h\|\leq 2\|h\|,$$
we have
\begin{equation}\label{boundlambda}
    \lambda_{\min}\big(\nabla(T\circ g)((T\circ g)^{-1}(x))^\top \nabla(T\circ g)((T\circ g)^{-1}(x))\big)^{1/2}\geq \frac{1}{2}(C^{-1}-\sqrt{16s}K).
\end{equation}
Then for $s>0$ small enough, using the Faa di Bruno formula we have that $f_T \in \mathcal{H}^\beta_{Cs^{-C_2}}(\mathcal{M}_{g^\star},\mathbb{R})$. Therefore point $i)$ is verified.

Proposition \ref{prop:projectconstant} asserts that there exists a constant $C>0$ such that $\pi_s$ verifies $ii)$ for a radius $C^{-1} $. Then, as $\mathbb{H}(\mathcal{M}_{g^\star},\mathcal{M}_{g_s})\leq Ks \leq C^{-1}/2$ for $s>0$ small enough,  we deduce that $T$ verifies $ii)$ for a radius $t=C^{-1}/4$. 

Let us prove point $iii)$. Let  $f\in \mathcal{H}^{\beta+1}_{Cs^{-1}}(\mathcal{M}_{g_s},(\mathbb{R}^p)^{p-d})$ such that $(f_1(x),...,f_{p-d}(x))$ is an orthonormal basis of $\mathcal{T}_x(\mathcal{M}_{g_s})^\top$. For $x\in \mathcal{M}_{g_s}$, define  $R_x:\mathbb{R}^{p-d}\rightarrow \mathbb{R}^p$ a linear map such that $R_x(e_i)=f_i(x)$ for $(e_1,...,e_{p-d})$ the canonical basis of $\mathbb{R}^{p-d}$. Then for $h\in \mathcal{H}^{\beta+1}_1(\mathcal{M}_{g^\star}^{t},\mathbb{R})$ and $x\in \mathcal{M}_{g^\star}$, we have
\begin{align*}
    \int_{T^{-1}(\{x\})}h(y)d\lambda^{p-d}_{E_x}(y) & = \int_{\mathcal{T}_{\pi_s(x)}(\mathcal{M}_{g_s})^\perp\cap B^d(x,t)}h(y)d\lambda^{p-d}_{\mathcal{T}_{\pi_s(x)}(\mathcal{M}_{g_s})^\perp}(y)\\
    & = \int_{\mathbb{R}^{p-d}\cap B^d(0,t)}h(x+R_{\pi_s(x)}(z))d\lambda^{p-d}(z).
\end{align*}
 Then as $\forall z \in \mathbb{R}^{p-d}$, the map $x\mapsto x +R_{\pi_s(x)}(z)$ belongs to $\mathcal{H}^{\beta+1}_{Cs^{-1}}(\mathcal{M}_{g^\star},\mathbb{R}^p)$, we have that $T$ verifies $iii)$.

Let us prove $iv)$. We have 
$$
T_g^{-1}\circ T = (\pi_{s|_{\mathcal{M}_{g}}})^{-1} \circ \pi_s (x)
$$
so as $\mathbb{H}(\mathcal{M}_{g_s},\mathcal{M}_{g})\leq 2Ks$ we show the same we showed that $T\in \mathcal{H}^{\beta+1}_{Cs^{-C_2}}(\mathcal{M}_{g_s}^{t},\mathcal{M}_{g^\star})$, that $T_g^{-1}\circ T  \in \mathcal{H}^{\beta+1}_{Cs^{-C_2}}(\mathcal{M}_{g_s}^{t},\mathcal{M}_{g})$. 

Let us now lower bound the approximate Jacobian of $T$. For $x \in\mathcal{M}_{g^\star}^{t}$, using \eqref{ineqpis} we have that 
$$\inf \limits_{z\in \mathcal{T}_{T(x)}(\mathcal{M}_g^{\star})} \|\nabla \pi_s(x)\frac{z}{\|z\|}\|\geq \frac{1}{K}(K^{-1}-\sqrt{8s}K),$$
so as $\pi_s$ is $2$-Lipschitz on $\mathcal{M}_{g^\star},$ we deduce that
$$\inf \limits_{z\in \mathcal{T}_{T(x)}(\mathcal{M}_g^{\star})} \|\nabla T(x)\frac{z}{\|z\|}\|\geq  \frac{1}{2K}(K^{-1}-\sqrt{8s}K).$$ 
Therefore, we deduce that the approximate Jacobian of $T$ verifies  
$$
\text{ap}_d(\nabla T(x))\geq  \left( \frac{1}{2K}(K^{-1}-\sqrt{16s}K)\right)^d,
$$
so $T$ verifies the point $iv)$ for $s>0$ small enough.
\end{proof}

\subsection{Proof of Lemma \ref{lemma:firstterm}}\label{sec:lemma:firstterm}
The proof breaks down as follow. We first prove that we only need to treat the case $\gamma=1$. Then, we build a function $H\in \mathcal{H}^{\beta+1}_C(A,\mathbb{R})$ for $A$ a neighborhood of $\mathcal{M}_g\cup \mathcal{M}_{g^\star}$, and prove that $d_{\mathcal{H}_1^1}(g_{\# U},(T\circ g)_{\# U})\leq \mathbb{E}[H(g(U))]^{\frac{\beta+1}{2\beta+1}}$. Finally, we extend $H$ to the whole $\mathbb{R}^p$ space.
\subsubsection{Proof that we only need to treat the case $\gamma=1$} Let us first show an interpolation inequality on the distance $d_{\mathcal{H}_1^\gamma}$.

\begin{lemma}\label{lemma:wefocuson}
   For $\gamma\in(1,\beta+1)$ we have for all $\epsilon \in (0,1)$,
        \begin{align*}
        d_{\mathcal{H}_1^\gamma}(g_{\# U},(T\circ g)_{\# U})  \leq & C\log(\epsilon^{-1})^2 d_{\mathcal{H}_1^1}(g_{\# U},(T\circ g)_{\# U})^{\frac{\beta+1-\gamma}{\beta}}\\
        &\times d_{\mathcal{H}_1^{\beta+1}}(g_{\# U},(T\circ g)_{\# U})^{\frac{\gamma-1}{\beta}} + C\log(\epsilon^{-1})^2\epsilon. 
    \end{align*}
\end{lemma}

\begin{proof} The proof follows the same arguments than the proof of Corollary \ref{coro:ineq without reg} except we are first going to give additional regularity to the discriminator so that its regularization belongs to Hölder spaces.
        Let $\epsilon>0$ and $f^\epsilon_g,f^\epsilon_{T}$ the densities of the $\epsilon$-envelopes of $g_{\# U},(T\circ g)_{\# U}$ (see Definition \ref{defi:envelope}) respectively. For $D\in \mathcal{H}^\gamma_1$ we have 
\begin{align*}
    \mathbb{E}_{U\sim \mathcal{U}([0,1]^d)}[D(g(U))-D(T(g(U)))] &
\leq \int_{\mathbb{R}^p} D(x)(f_g^\epsilon(x)-f^\epsilon_{T}(x))d\lambda^p(x) +2\epsilon.
\end{align*}

Using Proposition \ref{prop:logforweakregularity} we have that
\begin{align*}
\int_{\mathbb{R}^p}& D(x)(f_g^\epsilon(x)-f^\epsilon_{T}(x))d\lambda^p(x)\\
& \leq C\log(\epsilon^{-1})^2\int_\mathcal{X}\tilde{\Gamma}^{0,-2}_\epsilon(D)(x)(f_g^\epsilon(x)-f^\epsilon_{T}(x))d\lambda_{\mathcal{X}}(x) + C\epsilon\\
    & \leq C\log(\epsilon^{-1})^2\left(\mathbb{E}[\tilde{\Gamma}^{0,-2}_\epsilon(D)(g(U))-\tilde{\Gamma}^{0,-2}_\epsilon(D)(T(g(U)))]  + C\epsilon\right).
\end{align*}
Let us write $F=\tilde{\Gamma}^{0,-2}_\epsilon(D)$ and let $\xi>0$, then 
$$\mathbb{E}[F(g(U))-F(T(g(U)))]\leq \int_{\mathbb{R}^p} F(x)(f_g^\xi(x)-f^\xi_{T}(x))d\lambda^p(x) +2\xi.$$

Using Proposition \ref{prop:Hölder} for $h_1= F$, $h_2=f_g^\xi-f^\xi_{T}$, $s_1=\gamma$,  $\tau=\gamma-1$, $b_1=-2$, $s_2=b_2=0$, $t=r=1$
and $q=\frac{\beta+1-\gamma}{\beta}$ we have 
\begin{align*}
\int_{\mathbb{R}^p}  F(x)(f_g^\xi(x)-f^\xi_{T}(x))d\lambda^p(x)  \leq &  \left(\int \tilde{\Gamma}^{\gamma-1}(F)(f_g^\xi(x)-f^\xi_{T}(x))\right)^{\frac{\beta+1-\gamma}{\beta}}\\
& \times \left(\int \tilde{\Gamma}^{\gamma-(\beta+1)}(F)(f_g^\xi(x)-f^\xi_{T}(x))\right)^{\frac{\gamma-1}{\beta}}\\
     \leq & \mathbb{E}[ \tilde{\Gamma}^{\gamma-1}(F)(g(U))-\tilde{\Gamma}^{\gamma-1}(F)(T(g(U)))]^{\frac{\beta+1-\gamma}{\beta}}\\
     & \times \mathbb{E}[ \tilde{\Gamma}^{\gamma-(\beta+1)}(F)(g(U))-\tilde{\Gamma}^{\gamma-(\beta+1)}(F)(T(g(U)))]^{\frac{\gamma-1}{\beta}}\\
    & + C(\xi^{\frac{\beta+1-\gamma}{\beta}}+\xi^{\frac{\gamma-1}{\beta}}).
\end{align*}
As $\tilde{\Gamma}^{\gamma-1}(F)\in \mathcal{B}^{1,2}_{\infty,\infty}(C)\subset \mathcal{H}^1_C$ and $\tilde{\Gamma}^{\gamma-(\beta+1)}(F)\in \mathcal{B}^{\beta+1,2}_{\infty,\infty}(C)\subset \mathcal{H}^{\beta+1}_C$, letting $\xi$ tend to 0 we get the result.
\end{proof}

Now if we prove,
\begin{align*}
        d_{\mathcal{H}_1^{1}}(g_{\# U},(T\circ g)_{\# U}) & \leq C\log(\epsilon^{-1})^2 \left(d_{\mathcal{H}_1^{\beta+1}}(g_{\# U},(T\circ g)_{\# U})^{\frac{\beta+1}{2\beta+1}} + C\epsilon\right), 
    \end{align*}
    then as $\frac{\beta+1-\gamma}{\beta}\frac{\beta+1}{2\beta+1}+\frac{\gamma-1}{\beta}=\frac{\beta+\gamma}{2\beta+1}$, using lemma \ref{lemma:wefocuson} we will have that
    \begin{align*}
        d_{\mathcal{H}_1^{\gamma}}(g_{\# U},(T\circ g)_{\# U}) & \leq C\log(\epsilon^{-1})^4 \left(d_{\mathcal{H}_1^{\beta+1}}(g_{\# U},(T\circ g)_{\# U})^{\frac{\beta+\gamma}{2\beta+1}} + C\epsilon\right). 
    \end{align*}
Therefore, we can conclude that we only need to treat the case $\gamma=1$.

\subsubsection{Construction of a good potential of regularity $\beta+1$}

Let us write $$\omega_g=\inf \limits_{u\in \mathbb{T}^d} \inf \limits_{w\in \mathbb{R}^d}\|\nabla g(u)\frac{w}{\|w\|}\|\wedge 1.$$
From Proposition \ref{prop:lambdamin}, we have that $\omega_g\geq C^{-1}$ so $g$ is a local diffeomorphism and as it is injective, it is a diffeomorphism. Let us show that $g$ is a $C^{\beta+1}$ diffeomorphism.

\begin{lemma}\label{lemma:gmoins1inversible}
    We have $g^{-1}\in \mathcal{H}^{\beta+1 }_{C}(\mathcal{M}_g,\mathbb{T}^d).$
\end{lemma}

\begin{proof}
    As $\nabla g^{-1}(x)=(\nabla g)^{-1}(g^{-1}(x))$ and $\omega_g\geq C$, we deduce from the Faa di Bruno formula that $g^{-1}\in \mathcal{H}^{\lfloor \beta \rfloor+1}_{C}(\mathcal{M}_g,\mathbb{T}^d).$ Furthermore, we have $\nabla ^{\lfloor \beta \rfloor+1} (g^{-1})(x)=F((\nabla g)^{-1},\nabla^2 g,..., \nabla ^{\lfloor \beta \rfloor+1} g)(g^{-1}(x))$ for $F$ multilinear that can be computed from the Faa di Bruno formula. Therefore we have $u\mapsto F((\nabla g)^{-1},\nabla^2 g,..., \nabla ^{\lfloor \beta \rfloor+1} g)(u)$ belongs to $\mathcal{H}^{\beta -\lfloor \beta \rfloor}_{C}$. Let us now show that there exists a constant $C_0>0$ such that $g^{-1}$ is $C_0$-Lipschitz for the Euclidean norm. Let $x,y\in \mathcal{M}_g$, if $\|x-y\|\geq C_0^{-1}$ then 
    $$
    \|g^{-1}(x)-g^{-1}(y)\|\leq 1 \leq C_0\|x-y\|.
    $$
    If now $\|x-y\|\leq C_0^{-1}$ and suppose that  $\|g^{-1}(x)-g^{-1}(y)\|\geq (4K)^{-1}\omega_g$. Then taking $C_0^{-1}\leq \frac{1}{4}K^{-1}w_g^2$ we have from Lemma \ref{lemma:reachcurve} that
    \begin{align*}
   r_g & \leq \frac{1}{2}\|g(g^{-1}(x))-g(g^{-1}(y))\|=\frac{1}{2}\|x-y\|\leq \frac{1}{8}K^{-1}w_g^2
\end{align*}
which is impossible as $r_g\geq 1/K$ by hypothesis. Therefore we have $\|g^{-1}(x)-g^{-1}(y)\|< (4K)^{-1}\omega_g$, so
\begin{align*}
   \|x-y\| & = \|g(g^{-1}(x))-g(g^{-1}(y))\|\geq \omega_g\|g^{-1}(x)-g^{-1}(y)\|-K\|g^{-1}(x)-g^{-1}(y)\|^2\\
   & \geq \frac{1}{2}\omega_g\|g^{-1}(x)-g^{-1}(y)\|,
\end{align*}
so we conclude that $g^{-1}$ is $C_0$-Lipschitz.
Finally we have 
\begin{align*}
    |\nabla ^{\lfloor \beta \rfloor+1} &g^{-1}(x)-\nabla ^{\lfloor \beta \rfloor+1} g^{-1}(y)| \\
    & =|F((\nabla g)^{-1},\nabla^2 g,..., \nabla ^{\lfloor \beta \rfloor+1} g)(g^{-1}(x))-F((\nabla g)^{-1},\nabla^2 g,..., \nabla ^{\lfloor \beta \rfloor+1} g)(g^{-1}(y))|\\
    & \leq C \|g^{-1}(x)-g^{-1}(y)\|^{\beta-\lfloor \beta \rfloor}\\
    & \leq C \|x-y\|^{\beta-\lfloor \beta \rfloor}.
\end{align*}
\end{proof}
Applying Proposition \ref{prop:compatibility}, we have that there exists a map T being $(g_{\# U},g^\star_{\# U})_{K_T}$-compatible with radius $t\geq C^{-1}$ and $K_T\leq C$. For the lightness of the derivations, let us write $$X(u)=g(u)-T(g(u)).$$
Then 
    for $D\in \mathcal{H}^1_1$ we have 
\begin{equation}\label{eq:objectivtheo}
    \mathbb{E}_{U\sim \mathcal{U}([0,1]^d)}\Big[D(g(U))-D(T(g(U)))\Big]\leq C \mathbb{E}_{U\sim \mathcal{U}([0,1]^d)}\left[\|X(U)\|\right] .
     \end{equation}

Let $\xi>0$, and define $L:\mathbb{T}^d\rightarrow \mathbb{R}^p$  by \begin{align}\label{eq:L}
    L(u)=\frac{X(u)}{\|X(u)\|}\big(1-(1-\|X(u)\|/\xi)\mathds{1}_{\{\|X(u)\|< \xi\}}\big).
\end{align}
Let us write $K_X=\|X\|_{\mathcal{H}^{\beta+1}}\vee 1\leq C$ and $\delta_u=\|X(u)\|$ for $u\in \mathbb{T}^d$.
We have that 
$L\in \mathcal{H}^{0}_{1}(\mathbb{T}^d,\mathbb{R}^p)$, let us show that $L$ has some additional Hölder regularity.

\begin{proposition}\label{prop:HölderregL}
For $L:\mathbb{T}^d\rightarrow \mathbb{R}^p$ defined in \eqref{eq:L}, we have for all $\alpha \in [0,1]$ and $u,v\in \mathbb{T}^d$,
$$    \|L(u)-L(v)\|\leq 8K_X\min(\delta_u,\delta_v)^{-\alpha}\|u-v\|^{\alpha}.$$
\end{proposition}

\begin{proof}
Suppose first that $\|u-v\|\geq \min(\delta_u,\delta_v)/(4K_X)$. Then
\begin{align*}
    \|L(u)-L(v)\| &\leq 2=2\|u-v\|^{-\alpha}\|u-v\|^{\alpha}\\
    & \leq 8K_X\min(\delta_u,\delta_v)^{-\alpha}\|u-v\|^{\alpha}.
\end{align*}
Now if $\|u-v\|\leq \min(\delta_u,\delta_v)/(4K_X)$, then for $u_s=u+s(v-u)$ with $s\in [0,1]$, we have 
\begin{align*}
    \|X(u_s)\|& \geq  \|X(u)\|-sK_X\|u-v\|\\
    & \geq \delta_u- \frac{1}{4}\min(\delta_u,\delta_v)\\
    & \geq \frac{1}{2}\min(\delta_u,\delta_v).
\end{align*}
Furthermore, the differential of $L$ is defined almost everywhere and values
\begin{align*}
    \nabla L(u)= & \frac{\mathds{1}_{\{\|X(u)\|\geq \xi\}}}{\|X(u)\|}\Big(\nabla X(u)  - \frac{X(u)}{\|X(u)\|} \frac{(X(u))^\top}{\|X(u)\|}\nabla X(u) \Big)+ \frac{\mathds{1}_{\{\|X(u)\|<\xi\}}}{\xi}\nabla X(u),
\end{align*}
so $\|\nabla L(u)\|\leq K_X\|X(u)\|^{-1}= K_X\delta_u^{-1}$.
Then
\begin{align*}
    \|L(u)-L(v)\| & =\|\int_0^1\nabla L(u+s(v-u))(v-u)ds\|\\
    & \leq \|v-u\|\int_0^1\|\nabla L(u+s(v-u))\|ds\leq 2K_X\min(\delta_u,\delta_v)^{-1}\|v-u\|\\
    & \leq 2K_X\min(\delta_u,\delta_v)^{-\alpha}\|u-v\|^{\alpha}.
\end{align*} 
\end{proof}

Note that we have
\begin{align*}
\mathbb{E}_{U\sim \mathcal{U}([0,1]^d)}\left[\|X(U)\|\right]& = \mathbb{E}_{U\sim \mathcal{U}([0,1]^d)}\left[\left\langle \frac{X(U)}{\|X(U)\|},X(U)) \right\rangle\right]\\
& \leq \mathbb{E}_{U\sim \mathcal{U}([0,1]^d)}\left[\left\langle L(U),X(U)) \right\rangle\right]+\xi,
\end{align*}
so we are now going to focus on finding a function $H\in \mathcal{H}^{\beta+1}_{C}( \mathbb{R}^p,\mathbb{R})$ such that
    $$\mathbb{E}_{U\sim \mathcal{U}([0,1]^d)}\left[\left\langle L(U),X(U)\right\rangle\right]\leq C \mathbb{E}_{U\sim \mathcal{U}([0,1]^d)}\left[H(g(U))-H(T\circ g(U))\right]^{\frac{\beta+1}{2\beta+1}},$$
up to logarithmic terms. Let us define
\begin{equation}\label{eq:A}
A=\{x \in \mathcal{M}_{g^\star}^t | \ \|x-T(x)\|\leq \|T_g^{-1}\circ T(x)-T(x)\|\}
\end{equation}
and
$H:A\rightarrow \mathbb{R}$ by
\begin{equation}\label{eq:H}
H(x)=\Big\langle \tilde{\Gamma}^{-\beta,-2}(L)^{}\circ g^{-1} \circ T_g^{-1} \circ T(x)\ ,\ x-T(x)\Big\rangle,
\end{equation}
for $T_g^{-1}:\mathcal{M}_{g^\star}\rightarrow \mathcal{M}_g$ the inverse application from Definition \ref{defi:compatibility}. We have the following result on $H$.

\begin{proposition} The application $H$ belongs to $\mathcal{H}^{\beta+1}_{C}(A,\mathbb{R})$ and for all $\epsilon\in (0,1)$ we have
    $$\mathbb{E}_{U}\left[\left\langle L(U),X(U)\right\rangle\right]\leq C \log(\epsilon^{-1})^2\mathbb{E}_{U}\left[H(g(U))-H(T\circ g(U))\right]^{\frac{\beta+1}{2\beta+1}}+C\epsilon.$$
\end{proposition}

\begin{proof}
Let $$\tilde{\Gamma}^{0,-2}_\epsilon(L_i)(u) =  \sum \limits_{j=0}^{\log(\lfloor \epsilon^{-2}\rfloor)} \sum \limits_{l=1}^{2^d} \sum \limits_{z\in \{0,...,2^j-1\}^d}(1+j)^{-2}\alpha_{L_i}(j,l,z)S(j,l,w)_i\psi^{per}_{jlz}(u)$$
with $S(j,l,w)_i=\frac{\alpha_{L_i}(j,l,z)\alpha_{X_i}(j,l,z)}{|\alpha_{L_i}(j,l,z)\alpha_{X_i}(j,l,z)|}$.
Then using Proposition \ref{prop:logforweakregularity}, we have
$$\mathbb{E}_{U\sim \mathcal{U}([0,1]^d)} [\langle L(U),X(U) \rangle]\leq C\log(\epsilon^{-1})^2\mathbb{E}_{U\sim \mathcal{U}([0,1]^d)} [\langle \tilde{\Gamma}^{0,-2}_\epsilon(L)(U),X(U) \rangle]+C\epsilon.$$

Applying Proposition \ref{prop:Hölder} for $h_1=\tilde{\Gamma}^{0,-2}_\epsilon(L)$, $h_2=X$, $s_1=0$, $b_1=-2$, $s_2=\beta+1$, $b_2=0$, $\tau=-\beta$, $t=1$, $r=0$ and $q=\frac{2\beta+1}{\beta+1}$ we get
\begin{align*}
\mathbb{E}_{U} [\langle \tilde{\Gamma}^{0,-2}_\epsilon(L)(U),X(U) \rangle] & \leq \Big\langle \tilde{\Gamma}^{-\beta}(\tilde{\Gamma}^{0,-2}_\epsilon(L)),X\Big\rangle_{L^2}^{\frac{\beta+1}{2\beta+1}}\Big\langle \tilde{\Gamma}^0(\tilde{\Gamma}^{0,-2}_\epsilon(L)),\Gamma^{\beta+1}(X)\Big\rangle_{L^2}^{\frac{\beta }{2\beta+1}}\\
& \leq \left(\Big\langle \tilde{\Gamma}^{-\beta,-2}(L),X\Big\rangle_{L^2}+C\epsilon^{2\beta+1}\right)^{\frac{\beta+1}{2\beta+1}}\Big\langle \tilde{\Gamma}^{0,-2}_\epsilon(L),\Gamma^{\beta+1}(X)\Big\rangle_{L^2}^{\frac{\beta}{2\beta+1}},
\end{align*}
where we used that $$\Big\langle \tilde{\Gamma}^{-\beta}(\tilde{\Gamma}^{0,-2}_\epsilon(L)),X\Big\rangle_{L^2}=\Big\langle \tilde{\Gamma}_\epsilon^{-\beta,-2}(L),X\Big\rangle_{L^2}\leq \Big\langle \tilde{\Gamma}^{-\beta,-2}(L),X\Big\rangle_{L^2}+C\epsilon^{2\beta+1}.$$
We have that $\Gamma^{\beta+1}(X)\in \mathcal{B}^{0}_{\infty,\infty}(C)$ and $\tilde{\Gamma}^{0,-2}_\epsilon(L)\in \mathcal{B}^{0,2}_{\infty,\infty}(C)$ so
\begin{align*}
 \sum \limits_{i=1}^p & \sum \limits_{j=0}^\infty \sum \limits_{l=1}^{2^d} \sum \limits_{z\in \{0,...,2^j-1\}^d}\alpha_{\tilde{\Gamma}^{0,-2}_\epsilon(L_i)}(j,l,z) \alpha_{\Gamma^{\beta+1}(X_i)}(j,l,z)\\
 & \leq   \sum \limits_{i=1}^p \sum \limits_{j=0}^\infty \sum \limits_{l=1}^{2^d} \sum \limits_{z\in \{0,...,2^j-1\}^d}(1+j)^{-2}C2^{-jd}\\
 & \leq C  \sum \limits_{i=1}^p \sum \limits_{j=0}^\infty (1+j)^{-2}\leq C.
\end{align*}
Let us show that $\nabla^{\lfloor \beta \rfloor+1}\tilde{\Gamma}^{-\beta,-2}_\epsilon(L)$ exists and has some Hölder regularity. Let $u\in \mathbb{T}^d$ be such that $\delta_u>\xi$ and define $L_u:B^d(u,\frac{\delta_u}{4K_X})\rightarrow \mathbb{R}^p$ being equal to $L$. We then extend $L_u$ to the whole cube by 
\begin{equation}\label{eq:extensionprojec}
    \overline{L}_u(v)=L_u(\pi_{B_u}(v))(0\vee(1-3\|v-\pi_{B_u}(v)\|)),
\end{equation}
for $\pi_{B_u}$ the projection on $B^d(u,\frac{\delta_u}{4K_X})$ (recall that it is the periodised Euclidean norm we are using). For all $v\in B^d(u,\frac{\delta_u}{4K_X})$, we have $\delta_v\geq \delta_u/2$ so from Proposition \ref{prop:HölderregL} with $\alpha=1-(\beta-\lfloor \beta \rfloor )$ we get
$$L_u\in \mathcal{H}^{0}_{1}(B^d(u,\frac{\delta_u}{4K_X}),\mathbb{R}^p)\cap \mathcal{H}^{1-(\beta-\lfloor \beta \rfloor )}_{16K_X\delta_u^{-(1-(\beta-\lfloor \beta \rfloor ))}}(B^d(u,\frac{\delta_u}{4K_X}),\mathbb{R}^p).$$ Therefore, we also have that
$$\overline{L}_u \in \mathcal{H}^{0}_{1}(\mathbb{T}^d,\mathbb{R}^p)\cap \mathcal{H}^{1-(\beta-\lfloor \beta \rfloor )}_{16K_X\delta_u^{-(1-(\beta-\lfloor \beta \rfloor ))}}(\mathbb{T}^d,\mathbb{R}^p).$$

Let us look at the size of the support of the periodised wavelet. Recall from \eqref{eq:periowav} that if $2^j> N$, for $s\in [0,1]$ and $r\in \{0,...,2^j-1\}$ , the one-dimensional periodised scaling and wavelet functions  equal
\begin{equation*}
    \phi^{per}_j(s-2^{-j}z)=\phi(2^js-z)+\phi(2^j(s+1)-z)
\end{equation*}
and
\begin{equation*}
    \psi^{per}_j(s-2^{-j}z)=\psi(2^j(s-1)-z)+\psi(2^js-z)+\psi(2^j(s+1)-z).
\end{equation*}
For $l_i\in \{0,1\}$, as $supp(\psi_{l_i}(2^j\cdot-r))\subset[(r-N)2^{-j},(r+N)2^{-j}]$ we have:\\ 
if $z_i\geq 2^j-N$,
\begin{align*}
    [0,1]\cap supp(\psi_{l_i}(2^j\cdot-z_i)) & \subset [(z_i-N)2^{-j},1]\\
    [0,1]\cap supp(\psi_{l_i}(2^j(\cdot+1)-z_i)) & \subset [0,(z_i+N)2^{-j}-1]
\end{align*}
and if $z_i\leq N$,
\begin{align*}
[0,1]  \cap supp(\psi_{l_i}(2^j(\cdot-1)-z_i)) & \subset [1+(z_i-N)2^{-j},1]\\
[0,1]  \cap supp(\psi_{l_i}(2^j\cdot-z_i)) & \subset [0,(z_i+N)2^{-j}].
\end{align*}

Then $supp(\psi_{l_i}^{per}(2^j\cdot-z_i))\subset[(z_i-N)2^{-j},(z_i+N)2^{-j}]/\mathbb{Z}$ so finally we can conclude that 
$$supp(\psi_{j,l,z}^{per})\subset B^d(z2^{-j},Nd2^{-j}),$$
recalling that this is the ball for the periodised Euclidean distance.

Let us write $(\alpha_{(\overline{L}_u)_i}(j,l,z))_{(j,l,z)}$ the wavelet coefficients of $(\overline{L}_u)_i$ for $i=1,...,p$. As $\text{support}(\psi^{per}_{jlz})\subset B^d(2^{-j}z,dN2^{-j})$, we have that for any $(j,l,z)$ such that \\ $j>\lfloor\log_2(dN4K_X\delta_u^{-1})\rfloor+1$ and $\text{support}(\psi^{per}_{jlz})\cap B^d(u,\frac{\delta_u}{8K_X})\neq \varnothing$ then $$\alpha_{(\overline{L}_u)_i}(j,l,z)=\alpha_{L_i}(j,l,z).$$ Let us note $$\xi_u=\lfloor\log_2(dN4K_X\delta_u^{-1})\rfloor+1$$
and define the map $\tilde{\Gamma}^{-\beta,-2}(\overline{L}_u)_i:\mathbb{T}^d\rightarrow \mathbb{R}$ by
$$
\tilde{\Gamma}^{-\beta,-2}(\overline{L}_u)_i(v)= \sum \limits_{j=0}^\infty \sum \limits_{l=1}^{2^d} \sum \limits_{z\in \{0,...,2^j-1\}^d}2^{-j\beta} (1+j)^{-2}S(j,l,z)_i\alpha_{(\overline{L}_u)_i}(j,l,z)\psi^{per}_{j,l,z}(v),
$$
the $\beta,2$ wavelet regularization of $\overline{L}_{u_i}$. Then $$\tilde{\Gamma}^{-\beta,-2}(\overline{L}_u) \in \mathcal{H}^{\beta}_{C}(\mathbb{T}^d,\mathbb{R}^p)\cap \mathcal{H}^{\lfloor \beta \rfloor +1}_{CK_X\delta_u^{-(1-(\beta-\lfloor \beta \rfloor ))}}(\mathbb{T}^d,\mathbb{R}^p).$$
Let $\bold{i}_1,...,\bold{i}_{\lfloor \beta \rfloor+1} \in \{1,...,d\}$, we have
\begin{align}\label{align:split}
    |  \partial_{\bold{i}_1,...,\bold{i}_{\lfloor \beta \rfloor+1}}^{\lfloor \beta \rfloor+1} & \tilde{\Gamma}^{-\beta,-2}(L)_i(u)|   = |\partial_{\bold{i}_1,...,\bold{i}_{\lfloor \beta \rfloor+1}}^{\lfloor \beta \rfloor+1}\sum \limits_{j=0}^\infty \sum \limits_{l=1}^{2^d} \sum \limits_{z\in \{0,...,2^j-1\}^d} \alpha_{\tilde{\Gamma}^{-\beta,-2}(L_i)}(j,l,z)  \psi_{j,l,z}^{per}(u)|\nonumber\\
     \leq & |\partial_{\bold{i}_1,...,\bold{i}_{\lfloor \beta \rfloor+1}}^{\lfloor \beta \rfloor+1}\sum \limits_{j=0}^{\xi_u} \sum \limits_{l=1}^{2^d} \sum \limits_{z\in \{0,...,2^j-1\}^d}(\alpha_{\tilde{\Gamma}^{-\beta,-2}(L_i)}(j,l,z)-\alpha_{\tilde{\Gamma}^{-\beta,-2}(\overline{L}_u)_i}(j,l,z))  \psi_{j,l,z}^{per}(u)|\nonumber\\
    & +|\partial_{\bold{i}_1,...,\bold{i}_{\lfloor \beta \rfloor+1}}^{\lfloor \beta \rfloor+1}\sum \limits_{j=0}^{\xi_u} \sum \limits_{l=1}^{2^d} \sum \limits_{z\in \{0,...,2^j-1\}^d}\alpha_{\tilde{\Gamma}^{-\beta,-2}(\overline{L}_u)_i}(j,l,z)  \psi_{j,l,z}^{per}(u)\nonumber \\
    & +\partial_{\bold{i}_1,...,\bold{i}_{\lfloor \beta \rfloor+1}}^{\lfloor \beta \rfloor+1} \sum \limits_{j=\xi_u+1}^{\infty} \sum \limits_{l=1}^{2^d} \sum \limits_{z\in \{0,...,2^j-1\}^d} \alpha_{\tilde{\Gamma}^{-\beta,-2}(L_i)}(j,l,z)  \psi_{j,l,z}^{per}(u)|.
    \end{align}
For $k\in \{1,...,d\}$, let us write $\theta_k=\sum \limits_{r=1}^{\lfloor \beta \rfloor+1}\mathds{1}_{\{i_r=k\}}$. For the first term of \eqref{align:split} we have 
\begin{align*}
    |\partial_{\bold{i}_1,...,\bold{i}_{\lfloor \beta \rfloor+1}}^{\lfloor \beta \rfloor+1}\sum \limits_{j=0}^{\xi_u} & \sum \limits_{l=1}^{2^d} \sum \limits_{z\in \{0,...,2^j-1\}^d}(\alpha_{\tilde{\Gamma}^{-\beta,-2}(L_i)}(j,l,z)-\alpha_{\tilde{\Gamma}^{-\beta,-2}(\overline{L}_u)_i}(j,l,z))  \psi_{j,l,z}^{per}(u)|\\
     = & |\sum \limits_{j=0}^{\xi_u} \sum \limits_{l=1}^{2^d} \sum \limits_{z\in \{0,...,2^j-1\}^d}(\alpha_{\tilde{\Gamma}^{-\beta,-2}(L_i)}(j,l,z)-\alpha_{\tilde{\Gamma}^{-\beta,-2}(\overline{L}_u)_i}(j,l,z)) \\
     & \times 2^{j(\lfloor \beta \rfloor+1+d/2)} \prod \limits_{k=1}^d \sum \limits_{m\in \mathbb{Z}} \partial^{\theta_k} \psi_{l_k}(2^{j}(u_k-m)-z_k)|\\
    \leq & C \sum \limits_{j=0}^{\xi_u} \sum \limits_{l=1}^{2^d}\sum \limits_{z\in \{0,...,2^j-1\}^d}2^{j(\lfloor \beta \rfloor+1-\beta)}|\prod \limits_{k=1}^d \sum \limits_{m\in \mathbb{Z}} \partial^{\theta_k} \psi_{l_k}(2^{j}(u_k-m)-z_k)|\\
     \leq & C \sum \limits_{j=0}^{\xi_u} \sum \limits_{z\in \{0,...,2^j-1\}^d}2^{j(\lfloor \beta \rfloor+1-\beta)}\mathds{1}_{\{u\in supp(\psi_{jlz}^{per})\}}\\
      \leq & C \sum \limits_{j=0}^{\xi_u} 2^{j(\lfloor \beta \rfloor+1-\beta)}\leq C2^{\xi_u(\lfloor \beta \rfloor+1-\beta)}\leq CK_X\delta_u^{-(\lfloor \beta \rfloor+1-\beta)}.
\end{align*}
For the second term we have
    \begin{align*}
    |\partial_{\bold{i}_1,...,\bold{i}_{\lfloor \beta \rfloor+1}}^{\lfloor \beta \rfloor+1}\sum \limits_{j=0}^{\xi_u} & \sum \limits_{l=1}^{2^d} \sum \limits_{z\in \{0,...,2^j-1\}^d}\alpha_{\tilde{\Gamma}^{-\beta,-2}(\overline{L}_u)_i}(j,l,z)  \psi_{j,l,z}^{per}(u) \\
    & +\partial_{\bold{i}_1,...,\bold{i}_{\lfloor \beta \rfloor+1}}^{\lfloor \beta \rfloor+1} \sum \limits_{j=\xi_u+1}^{\infty} \sum \limits_{l=1}^{2^d} \sum \limits_{z\in \{0,...,2^j-1\}^d} \alpha_{\tilde{\Gamma}^{-\beta,-2}(L_i)}(j,l,z)  \psi_{j,l,z}^{per}(u)|\\
    & = |\partial_{\bold{i}_1,...,\bold{i}_{\lfloor \beta \rfloor+1}}^{\lfloor \beta \rfloor+1}\sum \limits_{j=0}^{\infty} \sum \limits_{l=1}^{2^d} \sum \limits_{z\in \{0,...,2^j-1\}^d}\alpha_{\tilde{\Gamma}^{-\beta,-2}(\overline{L}_u)_i}(j,l,z)  \psi_{j,l,z}^{per}(u)|\\
    & \leq \|\partial_{\bold{i}_1,...,\bold{i}_{\lfloor \beta \rfloor+1}}^{\lfloor \beta \rfloor+1}\tilde{\Gamma}^{-\beta,-2}(\overline{L}_u)_i\|_\infty \leq CK_X\delta_u^{-(\lfloor \beta \rfloor+1-\beta)}.
\end{align*}
The same way, let us show that for $u,v\in \mathbb{T}^d$ we have
\begin{equation}\label{Hölder}
|\partial_{\bold{i}_1,...,\bold{i}_{\lfloor \beta \rfloor+1}}^{\lfloor \beta \rfloor+1} \tilde{\Gamma}^{-\beta,-2}(L)_i(u)-\partial_{\bold{i}_1,...,\bold{i}_{\lfloor \beta \rfloor+1}}^{\lfloor \beta \rfloor+1} \tilde{\Gamma}^{-\beta,-2}(L)_i(v)|\leq CK_X(\delta_u^{-1} \vee \delta_v^{-1})\|u-v\|^{\beta-\lfloor \beta \rfloor}.
\end{equation}
Suppose that $\|u-v\|\leq \frac{\delta_u\vee \delta_v}{8K_X}$, then define $L_{u,v}$ coinciding with $L$ on $B^d(\argmax \limits_{w\in \{u,v\}} \delta_w,\frac{ \delta_u\vee \delta_v}{8K_X})$. Otherwise, if $\|u-v\|> \frac{\delta_u\vee \delta_v}{8K_X}$ define $L_{u,v}$ coinciding with $L$ on $B^d(u,\frac{\delta_u}{16K_X})\cup B^d( v,\frac{\delta_v}{16K_X})$. In both cases we have $L_{u,v}\in \mathcal{H}^1_{16K_X\delta_u^{-1}\vee \delta_v^{-1}}$. In the first case extend $L_{u,v}$ by projection on the ball as we did with $L_{u}$ in \eqref{eq:extensionprojec}, in the second case we extend it with the sum of the projections:
\begin{align*}
    \overline{L}_{u,v}(w)=&L_{u,v}(\pi_{B_u}(w))(0\vee(1-24K_X\delta_u^{-1}\|w-\pi_{B_u}(w)\|))\\
    &+L_{u,v}(\pi_{B_v}(w))(0\vee(1-24K_X\delta_v^{-1}\|w-\pi_{B_v}(w)\|)).
\end{align*}
 In both cases we also have $\overline{L}_{u,v}\in \mathcal{H}^1_{C\delta_u^{-1}\vee \delta_v^{-1}}$. Then by splitting the wavelet coefficients of $\partial_{\bold{i}_1,...,\bold{i}_{\lfloor \beta \rfloor+1}}^{\lfloor \beta \rfloor+1} \tilde{\Gamma}^{-\beta,-2}(L)_i(u)-\partial_{\bold{i}_1,...,\bold{i}_{\lfloor \beta \rfloor+1}}^{\lfloor \beta \rfloor+1} \tilde{\Gamma}^{-\beta,-2}(L)_i(v)$ as in \eqref{align:split}, we obtain \eqref{Hölder}. 
 
 Let $A$ and 
$H:A\rightarrow \mathbb{R}$ defined in \eqref{eq:A} and \eqref{eq:H}. We have $g^{-1} \circ T_g^{-1} \circ T\in \mathcal{H}^{\beta+1}_{C}(A,\mathbb{T}^d)$, $Id-T \in \mathcal{H}^{\beta+1}_{K_T}(\mathbb{R}^p,\mathbb{R}^p)$ and $\tilde{\Gamma}^{-\beta,-2}(L)^{} \in \mathcal{H}_{C}^{\beta}(\mathbb{T}^d,\mathbb{R}^p)$, so we get from the Faa di Bruno formula that  $H\in \mathcal{H}^{\beta}_{C}(A,\mathbb{R})$. Let $\bold{i}_1,...,\bold{i}_{\lfloor \beta \rfloor+1} \in \{1,...,p\}$ and write $P(\{\bold{i}_1,...,\bold{i}_{\lfloor \beta \rfloor+1}\})$ the set of subsets of $\{\bold{i}_1,...,\bold{i}_{\lfloor \beta \rfloor+1}\}$. We have

\begin{align*}
    & \partial_{\bold{i}_1,...,\bold{i}_{\lfloor \beta \rfloor+1}}^{\lfloor \beta \rfloor+1} H(x)\\
    & = \sum \limits_{S\in P(\{\bold{i}_1,...,\bold{i}_{\lfloor \beta \rfloor+1}\})}\Big\langle \partial^{|S|}_{i\in S} (\tilde{\Gamma}^{-\beta,-2}(L)^{}\circ g^{-1} \circ T_g^{-1} \circ T)(x)\ ,\ \partial^{\lfloor \beta \rfloor+1-|S|}_{i\notin S}(\text{Id}-T)(x)\Big\rangle.
\end{align*}
For all $S\in P(\{\bold{i}_1,...,\bold{i}_{\lfloor \beta \rfloor+1}\})$ with $|S|<\lfloor \beta \rfloor+1$, we have that $$x\mapsto \Big\langle \partial^{|S|}_{i\in S} (\tilde{\Gamma}^{-\beta,-2}(L)^{}\circ g^{-1} \circ T_g^{-1} \circ T)(x)\ ,\ \partial^{\lfloor \beta \rfloor+1-|S|}_{i\notin S}(\text{Id}-T)(x)\Big\rangle \in \mathcal{H}^{\beta-\lfloor \beta \rfloor}_{C}(A,\mathbb{R}).
$$
From the previous calculations, we know that 
\begin{align*}
|  \partial_{\bold{i}_1,...,\bold{i}_{\lfloor \beta \rfloor+1}}^{\lfloor \beta \rfloor+1} &\tilde{\Gamma}^{-\beta,-2}(L)_i(g^{-1} \circ T_g^{-1} \circ T(x))| \\
& \leq CK_X\delta_{g^{-1} \circ T_g^{-1} \circ T(x)}^{-(\lfloor \beta \rfloor+1-\beta)}\\
 & =CK_X\| T_g^{-1} \circ T(x)-T( T_g^{-1} \circ T(x))\|^{-(\lfloor \beta \rfloor+1-\beta)}\\
 & = CK_X\| T_g^{-1} \circ T(x)-T(x)\|^{-(\lfloor \beta \rfloor+1-\beta)}\\
  & \leq  CK_X\|x-T(x)\|^{-(\lfloor \beta \rfloor+1-\beta)}
\end{align*}
so 
$$|\Big\langle\partial_{\bold{i}_1,...,\bold{i}_{\lfloor \beta \rfloor+1}}^{\lfloor \beta \rfloor+1} (\tilde{\Gamma}^{-\beta,-2}(L)^{}\circ g^{-1} \circ T_g^{-1} \circ T)(x)\ ,(\text{Id}-T)(x)\Big\rangle|\leq CK_X^{C_3}\|x-T(x)\|^{\beta-\lfloor \beta \rfloor}.
$$

Let $x,y\in A$ with $\|x-T(x)\|\leq \|y-T(y)\|$, from \eqref{Hölder} we get
\begin{align*}
|&\Big\langle\partial_{\bold{i}_1,...,\bold{i}_{\lfloor \beta \rfloor+1}}^{\lfloor \beta \rfloor+1}  (\tilde{\Gamma}^{-\beta,-2}(L)^{}  \circ  g^{-1} \circ T_g^{-1} \circ T)(x)\ ,(\text{Id}-T)(x)\Big\rangle\\
& -\Big\langle\partial_{\bold{i}_1,...,\bold{i}_{\lfloor \beta \rfloor+1}}^{\lfloor \beta \rfloor+1} (\tilde{\Gamma}^{-\beta,-2}(L)^{}\circ g^{-1} \circ T_g^{-1} \circ T)(y)\ ,(\text{Id}-T)(y)\Big\rangle|\\
     \leq  &  |\Big\langle\partial_{\bold{i}_1,...,\bold{i}_{\lfloor \beta \rfloor+1}}^{\lfloor \beta \rfloor+1} (\tilde{\Gamma}^{-\beta,-2}(L)^{}  \circ g^{-1} \circ T_g^{-1} \circ T)(x)\\
     & -\partial_{\bold{i}_1,...,\bold{i}_{\lfloor \beta \rfloor+1}}^{\lfloor \beta \rfloor+1} (\tilde{\Gamma}^{-\beta,-2}(L)^{}\circ g^{-1} \circ T_g^{-1} \circ T)(y)\ ,(\text{Id}-T)(x)\Big\rangle|\\
    & + |\Big\langle\partial_{\bold{i}_1,...,\bold{i}_{\lfloor \beta \rfloor+1}}^{\lfloor \beta \rfloor+1} (\tilde{\Gamma}^{-\beta,-2}(L)^{}\circ g^{-1} \circ T_g^{-1} \circ T)(y)\ , (\text{Id}-T)(x)-(\text{Id}-T)(y)\Big\rangle|\\
     \leq & CK_X^{C_3}\min(\|x-T(x)\|,\|y-T(y)\|)^{-1}\|x-y\|^{\beta-\lfloor \beta \rfloor}\|x-T(x)\| \\
    & +\|T_{g}^{-1} \circ T(y)-T(y)\|^{-(1-(\beta-\lfloor \beta \rfloor))}\|(\text{Id}-T)(x)-(\text{Id}-T)(y)\|\\
    \leq & CK_X^{C_3}\|x-y\|^{\beta-\lfloor \beta \rfloor}+\|y-T(y)\|^{-(1-(\beta-\lfloor \beta \rfloor))}\|(\text{Id}-T)(x)-(\text{Id}-T)(y)\|\\
    \leq & C\|x-y\|^{\beta-\lfloor \beta \rfloor}\\
    &+\|y-T(y)\|^{-(1-(\beta-\lfloor \beta \rfloor))}2\|y-T(y)\|^{1-(\beta-\lfloor \beta \rfloor)}\|(\text{Id}-T)(x)-(\text{Id}-T)(y)\|^{\beta-\lfloor \beta \rfloor}\\
     \leq & C\|x-y\|^{\beta-\lfloor \beta \rfloor}.
\end{align*}
Therefore we have finally that $H\in \mathcal{H}^{\beta+1}_{C}(A,\mathbb{R})$.
\end{proof}  

\subsubsection{Extension of the potential to the whole $\mathbb{R}^p$ space}

Let us now prove that we can extend $H$ to the whole $\mathbb{R}^p$ using Whitney's extension Theorem (Theorem \ref{whitney}). 

\begin{proposition}\label{prop:extensionH}
    The map $H\in \mathcal{H}^{\beta+1}_{C}(A,\mathbb{R})$ can be extended into a map in $\mathcal{H}^{\beta+1}_{C}(\mathbb{R}^p,\mathbb{R})$.
\end{proposition}

\begin{proof}

Let $x_1,...,x_k\in A$ and define 
$$
P_i(x)=\sum \limits_{l=0}^{\lfloor \beta \rfloor +1} \nabla^l H(x_i)\frac{(x-x_i)^l}{l!},
$$
with the notations $\nabla^l H(x_i)$ the $l$-differential of $H$ and $$ \nabla^l H(x_i)(x-x_i)^l=\nabla^l H(x_i)\big(x-x_i,x-x_i,...,x-x_i).$$
Then the $P_i$ check conditions a) and b) for $M=\|H\|_{\mathcal{H}^{\beta+1}(A,\mathbb{R})}$ . Let show that they also verify condition c). First, for $\alpha=\lfloor \beta \rfloor +1$ we have $$
\|\nabla^{\lfloor \beta \rfloor +1} P_i(x_i)-\nabla^{\lfloor \beta \rfloor +1} P_j(x_i)\| =\|\nabla^{\lfloor \beta \rfloor +1} H(x_i)-\nabla^{\lfloor \beta \rfloor +1} H(x_j)\|\leq C\|x_i-x_j\|^{\beta-\lfloor \beta \rfloor}.
$$
Let $\alpha \in \{0,...,\lfloor \beta \rfloor\}$ and $\gamma:[0,1]\rightarrow A$ continuous and differentiable almost everywhere such that $\gamma(0)=x_j$ and $\gamma(1)=x_i$. 
Then, from Lemma \ref{lemma:extensionlemme} we get
\begin{align*}
    \|\nabla^\alpha &  P_i(x_i) - \nabla^\alpha P_j(x_i)\|=\|\nabla^\alpha H(x_i) - \sum \limits_{l=0}^{\lfloor\beta\rfloor+1-\alpha}\nabla^{l+\alpha} H(x_j)\frac{(x_i-x_j)^{l}}{l!}\|
    \leq C \|\dot{\gamma}\|^{\beta+1-\alpha}_{L^1}.
\end{align*}
Let us now define a path from $x_i$ to $x_j$. Suppose first that $\|x_i-x_j\|\leq t/4$ (for $t>0$ the radius of definition of $T:\mathcal{M}_{g^\star}^t\rightarrow \mathcal{M}_{g^\star}$), and define $P_A:\mathcal{M}_{g^\star}^t\rightarrow A$ by 
\begin{equation*}
P_A(x) = \left\{
\begin{array}{ll}
  x & \text{if } x \in A   \medskip\\
    \frac{x-T(x)}{\|x-T(x)\|}\|T_g^{-1} \circ T(x)-T(x)\|+T(x)& \text{if } x \notin A
\end{array}
\right.
\end{equation*} 
For $x\notin A$ we have 
\begin{align*}
    \nabla P_A(x) = &\frac{\|T_g^{-1} \circ T(x)-T(x)\|}{\|x-T(x)\|}\left(\nabla (\text{Id}-T)(x)-\frac{x-T(x)}{\|x-T(x)\|} \frac{(x-T(x))^\top}{\|x-T(x)\|} \nabla(\text{Id}-T)(x)\right)\\
    & + \frac{x-T(x)}{\|x-T(x)\|} \frac{(x-T(x))^\top}{\|x-T(x)\|} \nabla(T_g^{-1}\circ T -T)(x)+\nabla T(x),
\end{align*}
so $\|\nabla P_A(x)\|\leq 4K_T^2$. Let $\gamma:[0,1]\rightarrow A$ defined by
$$\gamma(s)=P_A(x_i+s(x_j-x_i)),$$
which is well defined as for $s\in [0,1]$,
$$d(x_i+s(x_j-x_i),\mathcal{M}_{g^\star})\leq \|x_i+s(x_j-x_i)-x_i\| +d(x_i,\mathcal{M}_{g^\star}) \leq t/4+t/2.$$
We have $\|\dot{\gamma}(s)\|\leq 4K_T^2\|x_i-x_j\|$ so we deduce that
\begin{align*}
    \|& \nabla^\alpha   P_i(x_i) - \nabla^\alpha P_j(x_i)\|\leq C \|x_i-x_j\|^{\beta+1-\alpha}.
\end{align*}
Suppose now that $\|x_i-x_j\|\geq t/4$ and define 
\begin{equation*}
\gamma(s) = \left\{
\begin{array}{ll}
  x_i + 3s(\pi_{g^\star}(x_i)-x_i) & \text{if } s\in [0,1/3]   \medskip\\
    g^\star\Big(g^{\star-1}(\pi_{g^\star}(x_i))+3(s-1/3)(g^{\star-1}(\pi_{g^\star}(x_j))-g^{\star-1}(\pi_{g^\star}(x_i))\Big) & \text{if } s\in (1/3,2/3)   \medskip\\ 
    \pi_{g^\star}(x_j) + 3(s-2/3)(x_j-\pi_{g^\star}(x_j)) & \text{if } s\in [2/3,1],
\end{array}
\right.
\end{equation*} 
the path from $x_i$ to $x_j$ that starts by projecting $x_i$ onto $\mathcal{M}_{g^\star}$, then circulates along $\mathcal{M}_{g^\star}$ until it reaches the projection of $x_j$, and finally go to $x_j$ in a straight line.

For $s\in [0,1/3]\cup [2/3,1]$ we have 
$$\|\dot{\gamma}(s)\|\leq 3 \max(\|\pi_{g^\star}(x_i)-x_i\|,\|\pi_{g^\star}(x_j)-x_j\|)\leq 3t\leq 12 \|x_i-x_j\|$$
and for $s\in(1/3,2/3)$
$$\|\dot{\gamma}(s)\|\leq 3K^2\|\pi_{g^\star}(x_i)-\pi_{g^\star}(x_j)\|\leq 3K^3\leq 12K^3t^{-1}\|x_i-x_j\|.$$
We then get as in the case where $\|x_i-x_j\|\leq t/4$, that 
\begin{align*}
    \|\nabla^\alpha   P_i(x_i) - \nabla^\alpha P_j(x_i)\| \leq &  C \|x_i-x_j\|^{\beta+1-\alpha}.
\end{align*}
Therefore using the sharp form of Whitney's extension Theorem \ref{whitney}, $H$ can be extended to a map in $\mathcal{H}^{\beta+1}_{C}(\mathbb{R}^p,\mathbb{R})$.
\end{proof}

We conclude the proof of Lemma \ref{lemma:firstterm} by noticing that 
\begin{align*}
    \mathbb{E}_{U} [\langle\tilde{\Gamma}^{-\beta,-2}(L)^{}(U),g(U)-T(g(U)) \rangle ] &  = \mathbb{E}_{U} [H(g(U))-H(T(g(U))) ]\\
    & \leq C \sup \limits_{f \in \mathcal{H}^{\beta+1}_1,\ f\circ T=0}\mathbb{E}_{U}[f(g(U))-f(T(g(U)))].
\end{align*}
\subsection{Proof of Lemma \ref {lemma:secondterm}}\label{sec:lemma:secondterm}

\begin{proof}
By the point $i)$ of definition of the  $(g_{\# U},g^\star_{\# U})$ compatibility (Defintion \ref{defi:compatibility}), we have that $T\circ g_{\# U}$ admits a density $f_{T} \in \mathcal{H}^{\beta}_{K_T}(\mathcal{M}_{g^\star},\mathbb{R})$ with respect to the volume measure on $\mathcal{M}_{g^\star}$.
Define the $t$-envelope of $f_T$ and $f_{g^\star}$ as in Definition \ref{defi:envelope} but in the case where the densities have some regularity: 
\begin{equation*}
f_T^{t}(x) = 
  \frac{f_T(T(x))\Theta(4\|x-T(x)\|^2/t^2)|\text{ap}_d(\nabla T(x))|}{\int_{\mathbb{R}^{p-d}} \Theta(4\|y\|^2/t^2)d\lambda^{p-d}(y)} \end{equation*}
  and
  \begin{equation*}
  f_{g^\star}^{t}(x) = 
  \frac{f_{g^\star}(T(x))\Theta(4\|x-T(x)\|^2/t^2)|\text{ap}_d(\nabla T(x))|}{\int_{\mathbb{R}^{p-d}} \Theta(4\|y\|^2/t^2)d\lambda^{p-d}(y)},
\end{equation*} 
for $\text{ap}_d$ the approximate jacobian from Definition \ref{def:approx} and $\Theta\in\mathcal{H}_C^{\beta}(\mathbb{R},\mathbb{R}_+)$ such that $\Theta(x)=\Theta(-x)$, $\Theta(0)=1$ and  $\Theta_{|(1,\infty)}=0$.

As $\nabla T$ belongs to $ \mathcal{H}_{K_T}^\beta(\mathbb{R}^p,L(\mathbb{R}^p,\mathbb{R}^p))$ and $\text{ap}_d(\nabla T(x))$ is bounded below by $C^{-1}$ ($T$ verifies point (iv) of Definition \ref{defi:compatibility}), then from the Faa di Bruno formula, we deduce that $\text{ap}_d(\nabla T)$ belongs to $\mathcal{H}_{C}^\beta(\mathcal{M}_{g^\star}^t,\mathbb{R})$. Let $D\in \mathcal{H}^{\gamma}_{C}(\mathbb{R}^p,\mathbb{R})$ and define $\overline{D}:\mathcal{M}_{g^\star}^{t}\rightarrow \mathbb{R} $ by 
$$\overline{D}(x)=D(T(x))\kappa(2\|x-T(x)\|^2/t^2),
$$
for $\kappa\in \mathcal{H}^{\gamma}_C(\mathbb{R},\mathbb{R})$ such that $\kappa_{|(-\infty,1/2)}=1$ and $\kappa_{|(1,\infty)}=0$. Then, extending $\overline{D}$ by 0 to the whole space $\mathbb{R}^p$, we have that $\overline{D} \in \mathcal{H}^{\gamma}_{C}(\mathbb{R}^p,\mathbb{R})$. From Proposition \ref{approx}, we have 

\begin{align*}
   \mathbb{E}_{U} &[D(T(g(U)))-D(g^\star(U))] = \int_{\mathcal{M}_{g^\star}} D(x)(f_T(x)-f_{g^\star}(x))d\lambda_{\mathcal{M}_{g^\star}}(x)\\
     = & \int_{\mathcal{M}_{g^\star}} \frac{D(x)(f_T(x)-f_{g^\star}(x))}{\int_{\mathbb{R}^{p-d}} \Theta(4\|y\|^2/t^2)d\lambda^{p-d}(y)}\int_{T^{-1}(\{x \})}\Theta(4\|z-T(z)\|^2/t^2)d\lambda^{p-d}(z)d\lambda_{\mathcal{M}_{g^\star}}(x)\\
         =& \int_{\mathcal{M}_{g^\star}} \int_{T^{-1}(\{x \})}\\
         & \frac{D(T(z))(f_T(T(z))-f_{g^\star}(T(z)))\Theta(4\|z-T(z)\|^2/t^2)}{\int_{\mathbb{R}^{p-d}} \Theta(4\|y\|^2/t^2)d\lambda^{p-d}(y)}d\lambda^{p-d}(z)d\lambda_{\mathcal{M}_{g^\star}}(x)\\
    & = \int_{\mathcal{M}_{g^\star}^t}  D(T(x))(f_T^{t}(T(x))-f_{g^\star}^{t}(T(x)))d\lambda^p(x)\\
    & = \int_{\mathbb{R}^p} \overline{D}(x)(f_T^{t}(x)-f_{g^\star}^{t}(x))d\lambda^p(x).
\end{align*}

Let $$\tilde{\Gamma}^{0,-2}_\epsilon(\overline{D})(x) =  \sum \limits_{j=0}^{\log(\lfloor \epsilon^{-2}\rfloor)} \sum \limits_{l=1}^{2^d} \sum \limits_{w\in \mathbb{Z}^p}(1+j)^{-2}S(j,l,w)\alpha_{\overline{D}}(j,l,w)\psi_{jlw}(x)$$
for $S(j,l,w)=\frac{\alpha_{\overline{D}}(j,l,w)(\alpha_{f_T^{t}}(j,l,w)-\alpha_{f_{g^\star}^{t}}(j,l,w))}{|\alpha_{\overline{D}}(j,l,w)(\alpha_{f_T^{t}}(j,l,w)-\alpha_{f_{g^\star}^{t}}(j,l,w))|}$ . Then using Proposition \ref{prop:logforweakregularity}, we have
\begin{align*}
\int_{\mathbb{R}^p} &\overline{D}(x)(f_T^{t}(x)-f_{g^\star}^{t}(x))d\lambda^p(x)\\
& \leq C\log(\epsilon^{-1})^2\int_{\mathbb{R}^p} \tilde{\Gamma}^{0,-2}_\epsilon(\overline{D})(f_T^{t}(x)-f_{g^\star}^{t}(x))d\lambda^p(x)+C\epsilon.
\end{align*}
Applying Proposition \ref{prop:Hölder} for $h_1=\tilde{\Gamma}^{0,-2}_\epsilon(\overline{D})$, $h_2=f_T^{t}-f_{g^\star}^{t}$, $s_1=\gamma$, $b_1=2$, $s_2=\beta$, $b_2=0$, $\tau=-(\beta+1-\gamma)$, $t=1$, $r=\frac{\gamma}{\beta+\gamma}$ and $q=\frac{2\beta+1}{\beta+\gamma}$ we get
\begin{align*}
\int_{\mathbb{R}^p} \tilde{\Gamma}^{0,-2}_\epsilon(\overline{D})(x)(f_T^{t}(x)-f_{g^\star}^{t}(x))d\lambda^p(x) 
    \leq & \Big\langle \tilde{\Gamma}^{\gamma-(\beta+1),-2}_\epsilon(\overline{D}),f_T^{t}-f_{g^\star}^{t}\Big\rangle_{L^2(\mathbb{R}^p)}^{\frac{\beta+\gamma}{2\beta+1}}\\
    &\Big\langle \tilde{\Gamma}^{\gamma,-2}_\epsilon(\overline{D}),\Gamma^{\beta}(f_T^{t}-f_{g^\star}^{t})\Big\rangle_{L^2(\mathbb{R}^p)}^{\frac{\beta +1-\gamma}{2\beta+1}}.
\end{align*}
Like in the proof of lemma \ref{lemma:firstterm}, we have 
\begin{align*}
    \sum \limits_{j=0}^\infty \sum \limits_{l=1}^{2^p}  \sum \limits_{w \in \mathbb{Z}^p} & 2^{j(\beta+\gamma)}|\alpha_{\tilde{\Gamma}^{0,-2}_\epsilon(\overline{D})}(j,l,w)| |\alpha_{f_T}(j,l,w)-\alpha_{f_{g^\star}}(j,l,w)|\\
    & \leq   \sum \limits_{j=0}^\infty \sum \limits_{l=1}^{2^p} \sum \limits_{w \in \mathbb{Z}^p}(1+j)^{-2}C2^{-jp}\mathds{1}_{\{supp(\psi_{jlw})\cap B^p(0,K)\neq \varnothing\}}\\
 & \leq C   \sum \limits_{j=0}^\infty (1+j)^{-2}\leq C.
\end{align*}
Define $F_D:\mathcal{M}_{g^\star}\rightarrow \mathbb{R}$ by 
$$
F_D(x)=\int_{T^{-1}(\{x \})}\tilde{\Gamma}^{\gamma-(\beta+1),-2}_\epsilon(\overline{D})(z)\ \frac{\Theta(4\|z-T(z)\|^2/t^2)}{\int_{\mathbb{R}^{p-d}} \Theta(4\|y\|^2/t^2)d\lambda^{p-d}(y)}d\lambda_{T^{-1}(\{x \})}(z).
$$
Then 
\begin{align*}
   & \int_{\mathbb{R}^p}   \tilde{\Gamma}^{\gamma-(\beta+1),-2}_\epsilon(\overline{D})(x)(f_T^{t}(x)-f_{g^\star}^{t}(x))d\lambda^p(x) \\
 = &\int_{\mathcal{M}_{g^\star}} \int_{T^{-1}(\{x \})} \tilde{\Gamma}^{\gamma-(\beta+1),-2}_\epsilon(\overline{D})(z) (f_T^{t}(z)-f_{g^\star}^{t}(z))|\text{ap}_d(\nabla T(z))|^{-1}d\lambda^{p-d}(z)d\lambda_{\mathcal{M}_{g^\star}}(x)\\
 = &\int_{\mathcal{M}_{g^\star}} \int_{T^{-1}(\{x \})} \tilde{\Gamma}^{\gamma-(\beta+1),-2}_\epsilon(\overline{D})(z)\\
 & \frac{(f_T(T(z))-f_{g^\star}(T(z)))\Theta(4\|z-T(z)\|^2/t^2)}{\int_{\mathbb{R}^{p-d}} \Theta(4\|y\|^2/t^2)d\lambda^{p-d}(y)}d\lambda^{p-d}(z)d\lambda_{\mathcal{M}_{g^\star}}(x)\\
 = &\int_{\mathcal{M}_{g^\star}} F_D(x)(f_T(x)-f_{g^\star}(x))d\lambda_{\mathcal{M}_{g^\star}}(x).
\end{align*}
As $T$ verifies point iii) of Definition \ref{defi:compatibility}, we have $F_D\in \mathcal{H}^{\beta+1}_{C}(\mathcal{M}_{g^\star},\mathbb{R})$ so we can conclude that 
\begin{align*}
    \mathbb{E}_{U\sim \mathcal{U}([0,1]^d)}&[D(T(g(U)))-D(g^\star(U))]\\
    & \leq C \sup \limits_{f \in \mathcal{H}^{\beta+1}_1(\mathcal{M}_{g^\star},\mathbb{R})}\mathbb{E}_{U\sim \mathcal{U}([0,1]^d)}[f(T(g(U)))-f(g^\star(U))]^{\frac{\beta+1}{2\beta+1}}.
\end{align*}
\end{proof}

\subsection{Additional technical lemmas}
\subsubsection{Results on the reach of submanifolds}
Let us give the characterization of the reach for submanifolds.

\begin{lemma}[\citeproofs{federer1959}, Theorem 4.18]\label{lemma:reachformanifolds}
 The reach $r\in [0,\infty)$ of a submanifold $\mathcal{M}$ verifies 
$$
r=\inf \limits_{q\neq p\in \mathcal{M}} \frac{\|q-p\|^2}{2d(q-p,\mathcal{T}_p(\mathcal{M}))}.$$
\end{lemma}
Using this result, let us show an upper bound on the reach.

\begin{lemma}\label{lemma:reachcurve}
    Let $g\in \mathcal{H}^{\beta+1}_{K}(\mathbb{T}^d,\mathbb{R})$ injective that verifies the $K_2$-regularity density condition and such that $\mathcal{M}_g=g([0,1]^d)$ is a submanifold. Then if there exist 
$\delta \in (0,(2K_2K)^{-1})$ and  $u^\star,v^\star\in \mathbb{T}^d$ such that $$\|u^\star-v^\star\|\geq \delta\quad \text{ and } \quad \|g(u^\star)-g(v^\star)\|< K_2^{-1}\delta-K\delta^2,$$ then $$r_g\leq \|g(u^\star)-g(v^\star)\|/2.$$
\end{lemma}

\begin{proof}
For $u,v\in \mathbb{T}^d$ we have 
\begin{align*}
    \|g(u)-g(v)\| & =\|\nabla g(v)(u-v)-O(\|u-v\|^2)\|\\
        & \geq K_2^{-1}\|u-v\|-K\|u-v\|^2
        \end{align*}
        so $\|u-v\|=\delta$ implies that $\|g(u)-g(v)\|\geq K_2^{-1}\delta-K\delta^2$. Let 
$$
(x,y)\in  \argmin \limits_{\|u-v\|\geq \delta} \|g(u)-g(v)\|,
$$
then  as $\|g(x)-g(y)\|\leq\|g(u^\star)-g(v^\star)\|<K_2^{-1}\delta-K\delta^2 $, we have  $\|x-y\|> \delta$. Defining $F(u,v)=\|g(u)-g(v)\|^2$, we have that $\nabla F(x,y)=0$ so we deduce that $g(x)-g(y)\in \mathcal{T}_{g(y)}(\mathcal{M}_{g})^\perp.$ Then using Lemma \ref{lemma:reachformanifolds}, we have
\begin{align*}
    r_g & =\inf \limits_{q\neq p\in \mathcal{M}_{g}} \frac{\|q-p\|^2}{2d(q-p,\mathcal{T}_p(\mathcal{M}_{g}))}\leq \frac{\|g(x)-g(y)\|^2}{2d(g(x)-g(y),\mathcal{T}_{g(y)}(\mathcal{M}_{g}))}\\
    & =\frac{\|g(x)-g(y)\|}{2} \leq \|g(u^\star)-g(v^\star)\|/2.
\end{align*}
\end{proof}

\subsubsection{Extension results}
To extend maps that are defined only on a subspace of $\mathbb{R}^p$, we will use the following version of Whitney's extension Theorem (Theorem A in \citeproofs{fefferman2005sharp} and "Whitney's extension Theorem" in \citeproofs{fefferman2009extension}).

\begin{theorem}\label{whitney} Given $\beta,p \geq 1$, there exists $k\in \mathbb{N}_0$ depending only on $p$ and $\beta$, for which the following holds. Let $f:E\rightarrow \mathbb{R}$ with $E$ an arbitrary subset of $\mathbb{R}^p$. Suppose that for any $k$ distinct points $x_1,...,x_k\in E$, there exists $\lfloor \beta \rfloor +1$ degree polynomials $P_1,...,P_k$ on $\mathbb{R}^p$, satisfying
\begin{itemize}
\item[a)] $P_i(x_i)=f(x_i)$ for $i=1,...,k$;
\item[b)] $\|\nabla^\alpha P_i(x_i)\|\leq M $ for $i=1,...,k$ and $|\alpha|\leq \lfloor \beta \rfloor +1$ and 
\item[c)] $\|\nabla^\alpha( P_i-P_j)(x_i)\|\leq M\|x_i-x_j\|^{\beta+1-\alpha} $ for $i,j=0,...,k$ and $|\alpha|\leq \lfloor \beta \rfloor +1$ with $M$
 independent of $x_1,...,x_k$
\end{itemize}
Then $f$ extends to  $\mathcal{H}_{C}^{\beta+1}(\mathbb{R}^p,\mathbb{R})$.
\end{theorem}

Let us now prove a bound on the growth of smooth functions defined on arbitrary subset of $\mathbb{R}^p$.
\begin{lemma}\label{lemma:extensionlemme}
    Let $E\subset \mathbb{R}^p$, $x_1,x_2\in E$ and $\gamma:[0,1]\rightarrow E$ continuous and differentiable almost everywhere with $\gamma(0)=x_2$ and $\gamma(1)=x_1$. Then for all $\eta>0$, $k\in \mathbb{N}$ and map $F\in  \mathcal{H}^\eta_1(E,\mathbb{R}^k)$, we have 
    $$\|F(x_1)-\sum \limits_{i=0}^{\lfloor \eta \rfloor } \nabla^iF(x_2)\frac{(x_1-x_2)^i}{i!}\|\leq \|\dot{\gamma}\|_{L^1}^\eta.$$
\end{lemma}
\begin{proof}
This result is proven by induction. The case $\lfloor \eta \rfloor=0$ is straightforward. Let us prove the case  $\lfloor \eta \rfloor\geq 1$ supposing that the result is true for $\lfloor \eta \rfloor-1$. We have 
\begin{align*}
    \|&F(x_1)-\sum \limits_{i=0}^{\lfloor \eta \rfloor } \nabla^iF(x_2)\frac{(x_1-x_2)^i}{i!}\|\\
    & =\|\int_0^1 \left(\nabla F(\gamma(t))-\sum \limits_{i=1}^{\lfloor \eta \rfloor } \nabla^iF(x_2)\frac{(\gamma(t)-x_2)^{i-1}}{(i-1)!}\right)\dot{\gamma}(t)dt\|\\
    &\leq \|\nabla F(\gamma(\cdot))-\sum \limits_{i=0}^{\lfloor \eta \rfloor-1 } \nabla^{i+1}F(x_2)\frac{(\gamma(\cdot)-x_2)^{i}}{i!}\|_\infty\|\dot{\gamma}\|_{L^1}\\
    & \leq \|\dot{\gamma}\|^\eta_{L^1},
\end{align*}
where we applied the induction hypothesis to the function
$\nabla F$.
\end{proof}

\section{Proofs of the bound in the manifold case (Section~\ref{sec:tractablehbeta})}

\subsection{Numerical regularity condition}\label{sec:numericcondi}
Let us define an easy to compute condition that will ensure that $\hat{g}$ verifies the manifold regularity condition.
\begin{definition}\label{defi:chinumerical}For $\chi>K$, we say that a map $g \in \mathcal{H}^{2}_{\chi}(\mathbb{T}^d,\mathbb{R}^p)\cap \mathcal{H}^{1}_K(\mathbb{T}^d,\mathbb{R}^p)$ verifies the $\chi$-numerical regularity condition if for all $z_i,z_j\in \{0,...,[2^5\sqrt{d}K^2\chi^2 ]\}^d$ with $ \|\frac{z_i}{2^5\sqrt{d}K^2\chi^2}-\frac{z_j}{2^5\sqrt{d}K^2\chi^2}\|\geq \frac{1}{4\chi^2}$, we have $\|g(\frac{z_i}{2^5\sqrt{d}K^2\chi^2})-g(\frac{z_j}{2^5\sqrt{d}K^2\chi^2})\|>\frac{1}{8K\chi^2}$.
\end{definition}
This condition ensures that if two points are far in the torus, then their images by $g$ are not too close. The following proposition asserts that the numerical and density regularity conditions together, imply the manifold regularity condition.

\begin{proposition}\label{prop:numericalcondi}
    For $\chi> K$, let $g \in \mathcal{H}^{1}_K(\mathbb{T}^d,\mathbb{R}^p)\cap \mathcal{H}^{2}_\chi(\mathbb{T}^d,\mathbb{R}^p)$ verify the $\chi$-numerical regularity and $\chi$-density regularity conditions. Then $g$ is injective and its image $\mathcal{M}_g$ has a reach lower bounded by $(C\chi^5)^{-1}$.
\end{proposition}
The proof of Proposition \ref{prop:numericalcondi} can be found in Section \ref{sec:prop:numericalcondi}.
Let us now choose the value of the parameter $\chi$ that we use for our class of generators. Let $C_{2,\beta+1}>0$ such that $\hat{\mathcal{F}}^{\beta+1,n^{-\frac{1}{2\tilde{\beta}+d}}}\subset \mathcal{H}^2_{C_{2,\beta+1}}$. Note that by Lemma \ref{lemma:inclusions}, for $\beta+1>2$ we have $C_{2,\beta+1}\leq C$ and for $\beta+1=2$ we have $C_{2,\beta+1}\leq C\log(n)^2$. Let $C_3>0$ be the constant given by Proposition \ref{prop:lambdamin} such that if $
d_{\mathcal{H}^{\beta+1}_1}(g_{\# U},g^\star_{\# U})\leq C_3^{-1}$, then $g$ verifies the $C_3^{-1}$-density regularity condition.
We choose
$$\chi:=C_{2,\beta+1}\vee C_3\vee K.$$
For $n$ large enough, as with high probability we will have that $d_{\mathcal{H}^{\beta+1}_1}(\hat{g}_{\# U},g^\star_{\# U})\leq \chi^{-1}$ (see Lemma \ref{lemma:hatgisclose}), then $\hat{g}$ will verify the assumptions of Proposition \ref{prop:numericalcondi} and therefore it will verify the $(C\chi^{5}\vee\|\hat{g}\|_{\mathcal{H}^{\beta+1}})$-manifold regularity condition.

Note that in practice, the numerical regularity condition of Definition \ref{defi:chinumerical} can be checked by adding to the loss of the GAN estimator \eqref{WGANS} the term 
\begin{align*}
\sum \limits_{\substack{z_i,z_j\in \{0,...,[2^5\sqrt{d}K^2\chi^2 ]\}^d\\ \|\frac{z_i}{2^5\sqrt{d}K^2\chi^2}-\frac{z_j}{2^5\sqrt{d}K^2\chi^2}\|\geq \frac{1}{4\chi^2}}}(&(8K\chi^2)^{-2}-\|g(\frac{z_i}{2^5\sqrt{d}K^2\chi^2})-g(\frac{z_j}{2^5\sqrt{d}K^2\chi^2})\|^2)\\
&\times \mathds{1}_{\{\|g(\frac{z_i}{2^5\sqrt{d}K^2\chi^2})-g(\frac{z_j}{2^5\sqrt{d}K^2\chi^2})\|\leq \frac{1}{8K\chi^2}\}},
\end{align*}
which is easily differentiable with respect to $g$.

\subsection{Proof of proposition \ref{prop:numericalcondi}}\label{sec:prop:numericalcondi}

\begin{proof}
    Let us first show that $g$ is injective. As $g$ verifies the $\chi$-density regularity condition and belongs to $\mathcal{H}^2_\chi$, we have for all $ u,v\in \mathbb{T}^d,$
    $$\|g(u)-g(v)\|\geq \chi^{-1}\|u-v\|-\chi\|u-v\|^2.$$
    Then for $\|u-v\|\leq \frac{1}{2\chi^2}$ we have,
    $$\|g(u)-g(v)\|\geq \frac{1}{2\chi}\|u-v\|.$$
    Suppose there exists $u,v\in \mathbb{T}^d$ different such that $g(u)=g(v)$, then from the previous derivation we have $\|u-v\|>\frac{1}{2\chi^2}$. Define $z_u,z_v\in \{0,...,[2^5\sqrt{d}K^2\chi^2 ]\}^d$ such that $$\|u-\frac{z_u}{2^5\sqrt{d}K^2\chi^2}\|+\|v-\frac{z_v}{2^5\sqrt{d}K^2\chi^2}\|\leq \frac{1}{16K^2\chi^2}.$$
    Then, we have 
    \begin{align*}\|\frac{z_u}{2^5\sqrt{d}K^2\chi^2}-\frac{z_v}{2^5\sqrt{d}K^2\chi^2}\|&\geq \|u-v\| -(\|u-\frac{z_u}{2^5\sqrt{d}K^2\chi^2}\|+\|v-\frac{z_v}{2^5\sqrt{d}K^2\chi^2}\|)\\
    &\geq \frac{1}{2\chi^2} -\frac{1}{16K^2\chi^2}\geq  \frac{1}{4\chi^2}\end{align*}
    and
    \begin{align*}
    \|g(\frac{z_u}{2^5\sqrt{d}K^2\chi^2})-g(\frac{z_v}{2^5\sqrt{d}K^2\chi^2})\|\leq& \|g(u)-g(v)\| \\
    & +K(\|u-\frac{z_u}{2^5\sqrt{d}K^2\chi^2}\|+\|v-\frac{z_v}{2^5\sqrt{d}K^2\chi^2}\|)\\
    \leq &\frac{1}{16K\chi^2},
    \end{align*}
    which is impossible by definition of the $\chi$-numerical regularity condition.
Therefore, we have that $g$ is injective and being a local diffeomorphism, it is a diffeomorphism so in particular $\mathcal{M}_g$ is a submanifold. 

Let us now show that its reach is greater than $(C\chi^{5})^{-1}$. We are going to use the characterisation of the reach for submanifold (Lemma \ref{lemma:reachformanifolds}). Doing the same derivation as \eqref{align:reachbound}, we obtain that for all $u,v\in \mathbb{T}^d$,
\begin{align*}
    \frac{\|g(u)-g(v)\|^2}{2d(g(u)-g(v),\mathcal{T}_{g(v)}(\mathcal{M}_{g}))} & \geq \frac{1}{\chi}(\chi^{-2}-\chi^2(4\|u-v\|+\|u-v\|^2)),
\end{align*}
so $\|u-v\|\leq \frac{1}{2^6\chi^4}$ implies that 
\begin{align*}
    \frac{\|g(u)-g(v)\|^2}{2d(g(u)-g(v),\mathcal{T}_{g(v)}(\mathcal{M}_{g}))} & \geq \frac{1}{2\chi^3}.
\end{align*}
Let now consider the case $\|u-v\|\geq \frac{1}{2^6\chi^4}$.  We have 
\begin{align*}
 \frac{\|g(u)-g(v)\|^2}{2d(g(u)-g(v),\mathcal{T}_{g(v)}(\mathcal{M}_{g}))} \geq \frac{\|g(u)-g(v)\|}{2},
\end{align*}
So if $\|u-v\|\in [\frac{1}{2^6\chi^4},\frac{1}{2\chi^2}]$ then 
\begin{align*}
\|g(u)-g(v)\|& \geq \chi^{-1}\|u-v\|-\chi\|u-v\|^2 \\
& =\|u-v\|(\chi^{-1}-\chi\|u-v\|) \geq \frac{1}{2^7\chi^5},
\end{align*}
and if $\|u-v\|>\frac{1}{2\chi^2}$ 
\begin{align*}
\|g(u)-g(v)\| & \geq \|g(z_u)-g(z_v)\|-K(\|u-\frac{z_u}{2^5\sqrt{d}K^2\chi^2}\|+\|v-\frac{z_v}{2^5\sqrt{d}K^2\chi^2}\|)\\
& \geq \frac{1}{8K\chi^2}-\frac{1}{16K^2\chi^2}\geq\frac{1}{16K\chi^2}.
\end{align*}
Finally, using Lemma \ref{lemma:reachformanifolds} we get the result.
\end{proof}
 
\subsection{Proof of Theorem \ref{theo:boundexpecterror2}}\label{sec:theo:boundexpecterror2}
\begin{proof}Recalling the notations of Definition \ref{defi:L} for $\gamma=\tilde{\beta}+1$, define $\overline{g}\in \argmin_{g\in \mathcal{G}}L(g,D^\star_{g})$.
Doing as in the proof of Theorem \ref{theo:boundexpecterror}
 we obtain that 
$$ L(\hat{g},D^\star_{\hat{g}})= \Delta^{\hat{g}}_{\mathcal{D}}+L(\hat{g},\overline{D}_{\hat{g}})$$
and 
 $$
 L(\hat{g},\overline{D}_{\hat{g}}) \leq L(\hat{g},\overline{D}_{\hat{g}}) - L_n(\hat{g},\overline{D}_{\hat{g}}) + L_n(\overline{g},\hat{D}_{\overline{g}})- L(\overline{g},\hat{D}_{\overline{g}}) +\Delta_\mathcal{G}.
 $$
 For $g\in \mathcal{G},$ write
 $$\mathcal{D}_g:=\{D\circ g\ |\ D\in \mathcal{D}\}.$$
Recalling the proof of Proposition \ref{prop:meanvsiid} we have for all $\delta>0$
\begin{align*}
     & \mathbb{P}\Big(\sup \limits_{\substack{g\in \mathcal{G},\ D\in \mathcal{D}\\ \|D\circ g\|_\infty\leq \delta}} \mathbb{E}[D(g(U))]-\frac{1}{n}\sum \limits_{i=1}^n D(g(U_i))\geq \alpha \Big)\\
     & \leq \sum \limits_{g_n \in \mathcal{G}_{1/n}} \mathbb{P}\Big( \sup \limits_{\substack{\ D\in \mathcal{D}\\ \|D\circ g_n\|_\infty\leq \delta+\frac{1}{n}}}\mathbb{E}[D(g_n(U))]-\frac{1}{n}\sum \limits_{i=1}^n D(g_n(U_i))\geq \alpha -\frac{2}{n} \Big)\\
     & \leq \sum \limits_{g_n \in \mathcal{G}_{1/n}} \mathbb{P}\Big( \sup \limits_{\substack{\exists\  f_n\in (\mathcal{D}_{g_n})_{1/n}\\ \|f_n\|_\infty\leq \delta+\frac{2}{n}}}\mathbb{E}[f_n(U))]-\frac{1}{n}\sum \limits_{i=1}^n f_n(U_i))\geq \alpha -\frac{4}{n} \Big)\\
     & \leq \sum \limits_{g_n \in \mathcal{G}_{1/n}} \sum \limits_{\substack{  f_n\in \mathcal{D}_{g_n 1/n}\\ \|f_n\|_\infty\leq \delta+\frac{2}{n}}}\mathbb{P}\Big( \mathbb{E}[f_n(U))]-\frac{1}{n}\sum \limits_{i=1}^n f_n(U_i))\geq \alpha -\frac{4}{n} \Big)\\
    & \leq |\mathcal{G}_{1/n}| |(\mathcal{D}_{\mathcal{G}})_{1/n}| \exp\left(-\frac{n(\alpha -\frac{4}{n})^2}{2(\delta+\frac{2}{n})^2}\right).
\end{align*}
Then solving with respect to $\alpha$ the equation $|\mathcal{G}_{1/n}||(\mathcal{D}_{\mathcal{G}})_{1/n}| \exp\left(-\frac{n(\alpha -\frac{4}{n})^2}{2(2K)^2}\right)=1/n$, we obtain that with probability at least $1/n$

\begin{align*}
\sup \limits_{\substack{g\in \mathcal{G},\ D\in \mathcal{D}\\ \|D\circ g\|_\infty\leq \delta}}\mathbb{E}[D(g(U))]-\frac{1}{n}\sum \limits_{i=1}^n D(g(U_i))&  \leq  C  \sqrt{\frac{(\delta+1/n)^2\log(n|\mathcal{G}_{1/n}| |(\mathcal{D}_{\mathcal{G}})_{1/n}|)}{n}}+\frac{1}{n}.
\end{align*}
Now doing as in the proof of Theorem \ref{theo:boundexpecterror}, we obtain that

\begin{align}\label{eq:coveringdefin}
\mathbb{E}\Big[L(\hat{g},\overline{D}_{\hat{g}}))\Big]&  \leq  C \min \limits_{\delta \in [0,1]}  \sqrt{\frac{(\delta+1/n)^2\log(n|\mathcal{G}_{1/n}| |(\mathcal{D}_{\mathcal{G}})_{1/n}|)}{n}}\nonumber\\
& +\frac{1}{\sqrt{n}}(1+\delta^{(1-\frac{d}{2(\beta+1)})}\mathds{1}_{\{\beta+1<d/2\}}+\log(\delta^{-1})\mathds{1}_{\{\beta+1\geq d/2\}})  +\Delta_\mathcal{G}.
\end{align}
\end{proof}

\subsection{Proof of Lemma \ref{lemma:coveringd}}\label{sec:lemma:coveringd}

\begin{proof} Supposing that $\delta^{-1}K$ is an integer, for $(k_1,...,k_d)\in \{0,...,2\delta^{-1}K-1\}^d$ let  
$$ u^{k_1,...,k_d}=(k_1\frac{\delta}{2K},...,k_d\frac{\delta}{2K}) 
$$
and
$$
A_{k_1,...,k_d} = [u^{k_1,...,k_d}_1,u^{k_1,...,k_d}_1+ \frac{\delta}{2K}]\times ...\times [u^{k_1,...,k_d}_d,u^{k_1,...,k_d}_d+ \frac{\delta}{2K}] \subset [0,1]^d .
$$
For $D\in \mathcal{F}^{\tilde{\beta}+1,\delta}$, we have by definition of $\mathcal{D}$ that  there exist weights\\ $\alpha_D(j,l,w) \in [-2^{-j(\tilde{\beta}+1+p/2)} , 2^{-j(\tilde{\beta}+1+p/2)}]$ such that 
$$
D(x)= \sum \limits_{j=0}^{\log_2(\lfloor \delta^{-1} \rfloor)} \sum \limits_{l=1}^{2^p}\sum \limits_{w\in \mathbb{Z}^p} \alpha_D(j,l,w)\hat{\psi}_{jlw}(x)
$$
As $supp(\hat{\psi}_{jlw})\subset \bigotimes \limits_{i=1}^p [2^{-j}(w_i-C),2^{-j}(w_i+C)]$ and $g$ is $K$-Lipschitz, we have for all $i\in \{1,...,p\}$ and $j\leq \log_2(\lfloor \delta^{-1} \rfloor)$,
$$
    supp(\hat{\psi}_{jlw})
\cap g_i(A_{k_1,...,k_d}) \neq \varnothing$$
implies$$
[2^{-j}(w_i-C),2^{-j}(w_i+C)]\cap [g_i(u^{k_1,...,k_d})-\delta/2,g_i(u^{k_1,...,k_d})+\delta/2]\neq \varnothing,
$$
which implies
$$2^j(g_i(u^{k_1,...,k_d})-\delta/2)-C\leq w_i \leq 2^j(g_i(u^{k_1,...,k_d})+\delta/2)+C
$$
and as $|2^{j-1}\delta|<1$ they are at most $2+2C$ integers $w_i\in \mathbb{Z}$ such that $supp(\hat{\psi}_{jlw})
\cap g_i(A_{k_1,...,k_d}) \neq \varnothing$. Therefore, $D_{|g(A_{k_1,...,k_d})}$ the restriction of $D$ to $g(A_{k_1,...,k_d})$ can be represented by $\log(\delta^{-1})(2^p(2+2C))^p$ weights and so $D_{|g([0,1]^d)}$ can be represented by $(2K\delta^{-1})^d\log(\delta^{-1})(2^p(2+2C))^p$ weights. 

We can then conclude that $\forall \epsilon \in (0,1),$
$$\log(|(\mathcal{N}_g)_\epsilon|)\leq \log((C^{-1}\epsilon)^{-(2K\delta^{-1})^d\log(\delta^{-1})(2^p(2+C))^p})\leq C\delta^{-d}\log(\delta^{-1})\log(\epsilon^{-1}).$$
\end{proof}

\subsection{Proof of Lemma \ref{lemma:deltaD}}\label{sec:lemma:deltaD}
\begin{proof}
Let $\pi_g$ be an optimal transport plan for the Euclidean distance between $g_{\# U}$ and $g^\star_{\# U}$.  We have
\begin{align*}
    \mathbb{E}[&D_g^\star(g(U))-D_g^\star(g^\star(U))-(\overline{D}_g(g(U))-\overline{D}_g(g^\star(U)))]\\
    & = \int_{\mathbb{R}^p\times \mathbb{R}^p} \left(D_g^\star(x)-D_g^\star(y)-(\overline{D}_g(x)-\overline{D}_g(y))\right)d\pi_g(x,y)\\
    & = \int_{\mathbb{R}^p\times \mathbb{R}^p}  \int_0^1\Big\langle \nabla D_g^\star(y+t(x-y))-\nabla \overline{D}_g(y+t(x-y)),x-y\Big\rangle dtd\pi_g(x,y)\\
    & \leq \|\nabla D_g^\star-\nabla \overline{D}_g\|_\infty \int_{\mathbb{R}^p\times \mathbb{R}^p}  \|x-y\| d\pi_g(x,y)\\
    &  = C \|\nabla D_g^\star-\nabla \overline{D}_g\|_\infty \sup \limits_{f\in \text{Lip}_1}\mathbb{E}[  f(g(U))-f(g^\star(U))].
\end{align*}
Writing $$\mathcal{K}=\|g\|_{\mathcal{H}^{\tilde{\beta}+1}},$$ 
we have $\mathcal{K}\leq C\log(n)^2$ from Proposition \ref{prop:reguofG}. Using Theorem \ref{theo:theineq} with $\epsilon=1/n$ we obtain
\begin{align*}
\sup \limits_{f\in \text{Lip}_1}& \mathbb{E}[  f(g(U))-f(g^\star(U))]\\
&\leq (2K+1)\sup \limits_{f\in \mathcal{H}^1_1}\mathbb{E}[  f(g(U))-f(g^\star(U))]\\
& \leq C(\mathcal{K}\chi)^{C_2}\left(\log(n)^4\sup \limits_{f\in \mathcal{H}^{\tilde{\beta}+1}_1}\mathbb{E}[  f(g(U))-f(g^\star(U))]^{\frac{\tilde{\beta}+1}{\tilde{\beta}+\tilde{\beta}+1}}+\frac{1}{n}\right)
\end{align*}    
\end{proof}

\subsection{Proof of Theorem \ref{theo:boundhbeta}}\label{sec:theo:boundhbeta}
The proof breaks down as follow. First, we determine the smoothness of the functions in the class of generators $\mathcal{G}$ and show that with high probability, $\hat{g}$ verifies the density and manifold regularity conditions. Then, we bound the term of approximation error of the generator class. Finally we prove that the GAN estimator attains the minimax rate for the distance $d_{\mathcal{H}^{\tilde{\beta}+1}_1}$ using Theorem \ref{theo:boundexpecterror2}.

\subsubsection{Regularity of $\hat{g}$}
Let us first determine the regularity of the functions in the class~$\mathcal{G}$.
\begin{proposition}\label{prop:reguofG}
    For all $\eta\in (0,\beta+1]$ and $\delta\in(0,1)$, we have that if $\eta$ is not an integer then
    $$\hat{\mathcal{F}}^{\eta,\delta}_{per}\subset\mathcal{H}^{\eta}_C,$$
    and if $\eta$ is an integer
    $$\hat{\mathcal{F}}^{\eta,\delta}_{per}\subset\mathcal{H}^{\eta}_{C\log(\delta^{-1})^2}.$$
\end{proposition}
\begin{proof}
Let $f\in \hat{\mathcal{F}}^{\eta,\delta}_{per}$, then from Proposition \ref{prop:approxhatper} there exist coefficients $(\alpha_{f}(j,l,z))$ such that 
$$f=\sum \limits_{j=0}^{\infty} \sum \limits_{l=1}^{2^d} \sum \limits_{z\in \{0,...,2^j-1\}^d} \alpha_{f}(j,l,z) \psi_{jlz}^{per} ,$$
with $$|\alpha_f(j,l,w)|\leq (C_\eta K+C\delta^{\eta+\gamma}) 2^{-j(\eta+d/2)}, \text{ for } j\leq \log_2(\delta^{-1})$$
and
$$|\alpha_f(j,l,w)|\leq C2^{-j(\lfloor \beta \rfloor +2+d/2)}\delta^{\eta+\gamma}, \text{ for } j> \log_2(\delta^{-1}).$$
Then if $\eta$ is not an integer, we have $\|f\|_{\mathcal{H}^{\eta}}\leq C$ as $\mathcal{B}^{\eta}_{\infty,\infty}\subset \mathcal{H}^{\eta}$ . Now if $\eta$ is an integer we have 
\begin{align*}   \|f\|_{\mathcal{H}^{\eta}} & \leq C\|f\|_{\mathcal{B}^{\eta,2}_{\infty,\infty}}=C\sup \limits_{j  \in \mathbb{N}_0} 2^{j(\eta+d/2)}\sum \limits_{l=1}^{2^d} \sup \limits_{z\in \{0,...,2^j-1\}^d}(1+j)^2|\alpha_{f}(j,l,z)|\\
& \leq C\sup \limits_{j  \leq\log_2(\delta^{-1})} 2^{j(\eta+d/2)}\sum \limits_{l=1}^{2^d} \sup \limits_{z\in \{0,...,2^j-1\}^d}(1+j)^2|\alpha_{f}(j,l,z)|+C\delta^{\lfloor \beta \rfloor +2-\eta}\log(\delta^{-1})^2\\
& \leq C\sup \limits_{j  \leq\log(\delta^{-1}))}(1+j)^2+C\delta^{\lfloor \beta \rfloor +2-\eta}\log(\delta^{-1})^2\leq  C\log(\delta^{-1})^2.
\end{align*}
\end{proof}
Using Proposition \ref{prop:reguofG}, we have that if $\beta$ is not an integer then $\hat{\mathcal{F}}^{\beta+1,n^{-\frac{1}{2\beta+d}}}_{per}\subset\mathcal{H}^{\beta+1}_C,$
and if $\beta$ is an integer $\hat{\mathcal{F}}^{\beta+1,n^{-\frac{1}{2\beta+d}}}_{per}\subset\mathcal{H}^{\beta+1}_{C\log(n)^2}.$ 

Using this result, let us give a first (sub-optimal) bound on $d_{\mathcal{H}^{\tilde{\beta}+1}_1}(\hat{g}_{\# U},g^\star_{\# U})$ that will allow us to use the different results that require the distance $d_{\mathcal{H}^{\tilde{\beta}+1}_1}(\hat{g}_{\# U},g^\star_{\# U})$ to be small. 
\begin{lemma}\label{lemma:hatgisclose}
For all $g^\star \in \mathcal{H}^{\beta+1}_K(\mathbb{T}^d,\mathbb{R}^p)$, the GAN estimator $\hat{g}$ of \eqref{WGANS} with $\mathcal{G}$ and $\mathcal{D}$ defined in \eqref{eq:G} , verifies that    with probability at least $1/n$ we have 
    $$d_{\mathcal{H}^{\tilde{\beta}+1}_1}(\hat{g}_{\# U},g^\star_{\# U})\leq Cn^{-\frac{\tilde{\beta}}{2{\beta}+d}}.$$
\end{lemma}
\begin{proof}
We are going to use the bound given by Theorem \ref{theo:boundexpecterror2}.
Using Lemma \ref{lemma:coveringd} we have that $\log( |(\mathcal{D}_{\mathcal{G}})_{1/n}|)\leq C\log(n)^2 n^{\frac{d}{2\tilde{\beta}+d}}$, and from Proposition \ref{prop:propertiesofhatper} we have $|\mathcal{G}_{1/n}|\leq C\log(n)^{2} n^{\frac{d}{2\beta+d}}$.
On the other hand, from Proposition \ref{prop:deltagmodel3} we have that $\Delta_{\mathcal{G}}\leq C\log(n)^{C_2}n^{-\frac{\beta+\tilde{\beta}+1}{2\tilde{\beta}+d}}$ and 
\begin{align*}
    \mathbb{E}[\Delta^{\hat{g}}_\mathcal{D}]& \leq \Delta_\mathcal{D}=\sup \limits_{g\in \mathcal{G}} \bigl\{ d_{\mathcal{H}_1^{\tilde{\beta}+1}}(g_{\# U},g^\star_{\# U})-\Big( \sup\limits_{D\in \mathcal{D}}\mathbb{E}_{U\sim \mathcal{U}([0,1]^d)}[D(g(U))-D(g^\star(U))]\Big)\bigl\}\\
    & \leq \sup_{D^\star\in \mathcal{H}_1^{\tilde{\beta}+1}(B^p(0,K),\mathbb{R})} \inf\limits_{D\in \mathcal{D}} \|D-D^\star\|_\infty\\
    & \leq \sup_{D^\star\in \mathcal{H}_1^{\tilde{\beta}+1}(B^p(0,K),\mathbb{R})} \inf\limits_{D\in \mathcal{D}} \|D-D^\star\|_{\mathcal{B}^{0,2}_{\infty,\infty}}\\
    & \leq  C \log(n)^2 n^{-\frac{\tilde{\beta}}{2\tilde{\beta}+d}}.
\end{align*}
Finally, taking $\delta=1$ and putting everything in the bound of Theorem \ref{theo:boundexpecterror2}, we get the result.
\end{proof}
As discussed Section \ref{sec:numericcondi}, using Lemma \ref{lemma:hatgisclose} and Proposition \ref{prop:numericalcondi} for $n$ large enough, we have that $\hat{g}$ verifies the $(C\chi^{5}\vee\|\hat{g}\|_{\mathcal{H}^{\beta+1}})$-manifold regularity condition.

\subsubsection{Approximation error of the class $\mathcal{G}$}\label{sec:approerrofdd} From Proposition \ref{prop:propertiesofhatper}, we know that
\begin{equation}\label{eq:neededfodeltag}
\inf_{g\in \hat{\mathcal{F}}^{\beta+1,n^{-\frac{1}{2\beta+d}}}_{per}} d_{\mathcal{H}_1^{\tilde{\beta}+1}}(g_{\# U},g^\star_{\# U})\leq C\log(n)^{C_2}n^{-\frac{\beta+\tilde{\beta}+1}{2\tilde{\beta}+d}}.
\end{equation}
Therefore it suffices to prove that there exists $\overline{g}\in \hat{\mathcal{F}}^{\beta+1,n^{-\frac{1}{2\beta+d}}}_{per}$ that verifies both \eqref{eq:neededfodeltag} and the $\chi$-numerical regularity condition. We show in the proof of the next result that the map being the low wavelet frequency approximation of $g^\star$ by neural networks, verifies these conditions.

\begin{proposition}\label{prop:deltagmodel3}
    For all $g^\star \in \mathcal{H}^{\beta+1}_K(\mathbb{T}^d,\mathbb{R}^p)$ that verifies the $K$-density regularity condition of Definition \ref{defi:manifold regularity}, the class $\mathcal{G}$ defined in \eqref{eq:G} verifies
    $$\Delta_{\mathcal{G}}=\inf_{g\in \mathcal{G}}d_{\mathcal{H}_1^{\tilde{\beta}+1}}(g_{\# U},g^\star_{\# U})\leq C\log(n)^{C_2}n^{-\frac{\beta+\tilde{\beta}+1}{2\tilde{\beta}+d}}. $$
\end{proposition}

\begin{proof}
    Let 
    $$\overline{g}(u)=\sum \limits_{j=0}^{n^{\frac{1}{2\beta+d}}} \sum \limits_{l=1}^{2^d}\sum \limits_{z\in \{0,...,2^j-1\}^d} \alpha_{g^\star}(j,l,z)\hat{\psi}_{jlz}^{per}(u),$$
for $\alpha_{g^\star}(j,l,z)=\langle g^\star,\psi_{jlz}^{per}\rangle $. Then, from Proposition \ref{prop:approxhatper} there exist coefficients $(\alpha_{\overline{g}}(j,l,z))$ such that 
$$\overline{g}=\sum \limits_{j=0}^{\infty} \sum \limits_{l=1}^{2^d} \sum \limits_{z\in \{0,...,2^j-1\}^d} \alpha_{\overline{g}}(j,l,z) \psi_{jlz}^{per} ,$$
with $$|\alpha_{\overline{g}}(j,l,w)-\alpha_{g^\star}(j,l,w)|\leq Cn^{-\frac{\beta+\tilde{\beta}+2}{2\beta+d}} 2^{-j(\beta+1+d/2)}, \text{ for } j\leq \log_2(n^{\frac{1}{2\beta+d}})$$
and
$$|\alpha_{\overline{g}}(j,l,w)|\leq C2^{-j(\lfloor \beta \rfloor +2+d/2)}n^{-\frac{\beta+\tilde{\beta}+2}{2\beta+d}}, \text{ for } j> \log_2(n^{\frac{1}{2\beta+d}}).$$
Let $u,v\in\mathbb{T}^d$ with $\|u-v\|\geq (4\chi^2)^{-1}$, we have
\begin{align*}
    &\| \overline{g}(u)-\overline{g}(v)\| \geq  \| g^\star(u)-g^\star(v)\|-2\| \overline{g}-g^\star\|_\infty\\
    \geq&  \| g^\star(u)-g^\star(v)\|-2C\| \overline{g}-g^\star\|_{\mathcal{B}^{0,2}_{\infty,\infty}}\\
    \geq &\| g^\star(u)-g^\star(v)\|-C\sup \limits_{j\geq 0} 2^{jd/2}(1+j)^2\sum \limits_{l=1}^{2^d} \sup \limits_{z\in \{0,...,2^j-1\}^d}|\alpha_{\overline{g}}(j,l,w)-\alpha_{g^\star}(j,l,w)|\\
     \geq &\| g^\star(u)-g^\star(v)\|-C\Big(\sup \limits_{j  \leq\log_2(n^{\frac{1}{2\beta+d}})} (1+j)^2\sum \limits_{l=1}^{2^d} \sup \limits_{z\in \{0,...,2^j-1\}^d}n^{-\frac{\beta+\tilde{\beta}+2}{2\beta+d}} 2^{-j(\beta+1)}\\
     & +\sup \limits_{j  >\log_2(n^{\frac{1}{2\beta+d}})} (1+j)^2\sum \limits_{l=1}^{2^d} \sup \limits_{z\in \{0,...,2^j-1\}^d}(CK+n^{-\frac{\beta+\tilde{\beta}+2}{2\beta+d}}) 2^{-j(\beta+1)}\Big)\\
    & \geq \| g^\star(u)-g^\star(v)\|-C\log(n)^2n^{-\frac{\beta+1}{2\beta+d}}.
\end{align*}
Using Lemma \ref{lemma:reachcurve} with $\delta=(4\chi^2)^{-1}$, we have that if
$$\| g^\star(u)-g^\star(v)\|\leq K^{-1}(4\chi^2)^{-1}-K(4\chi^2)^{-2},$$
then 
$$r_{g^\star}\leq \frac{\| g^\star(u)-g^\star(v)\|}{2}.$$
Furtermore,
\begin{align*}
K^{-1}(4\chi^2)^{-1}-K(4\chi^2)^{-2}&=(4\chi^2)^{-1}(K^{-1}-K(4\chi^2)^{-1})\\
& \geq (4\chi^2)^{-1}(K^{-1}-(4K)^{-1})=\frac{3}{4}(4K\chi^2)^{-1},
\end{align*}
as $\chi\geq K$. Then if $\| g^\star(u)-g^\star(v)\|\leq \frac{3}{4}(4K\chi^2)^{-1}$, we would have $r_{g^\star}\leq \frac{3}{8}(4K\chi^2)^{-1}<K^{-1}$, which is impossible as $g^\star$ verifies the $K$-manifold regularity condition. Therefore we deduce that 
\begin{align*}
    \| \overline{g}(u)-\overline{g}(v)\|& \geq \| g^\star(u)-g^\star(v)\|-C\log(n)^2n^{-\frac{\beta+1}{2\beta+d}}\\
    & \geq \frac{3}{4}(4K\chi^2)^{-1}-C\log(n)^2n^{-\frac{\beta+1}{2\beta+d}}> (8K\chi^2)^{-1},
\end{align*}
for n large enough.

We conclude that $\overline{g}$ verifies the $\chi$-numerical condition and thus belongs to $\mathcal{G}$. Finally we can show like in the Proof of Proposition \ref{prop:Fdelta} Section \ref{sec:prop:Fdelta}, that 
$$d_{\mathcal{H}_1^{\tilde{\beta}+1}}(\overline{g}_{\# U},g^\star_{\# U})\leq C\log(n)^{C_2}n^{-\frac{\beta+\tilde{\beta}+1}{2\tilde{\beta}+d}}.$$
\end{proof}

\subsubsection{Putting everything together}
We now give the proof of Theorem \ref{theo:boundhbeta}.

\begin{proof}[Proof of theorem \ref{theo:boundhbeta}]
By stability of Besov spaces for the differentiation (proposition 4.3.19, \citetproofs{giné_nickl_2015}) we have
\begin{align*}
    \|\nabla D_g^\star-\nabla \overline{D}_g\|_\infty & \leq C \|\nabla D_g^\star-\nabla \overline{D}_g\|_{\mathcal{B}^{0,2}_{\infty,\infty}} \leq C \|D_g^\star-\overline{D}_g\|_{\mathcal{B}^{1,2}_{\infty,\infty}}\\
        & = C \sup \limits_{j\geq \log_2(n^{-\frac{1}{2\tilde{\beta}+d}})} 2^{j(1+p/2)}(1+j)^2\sum \limits_{l=1}^{2^d} \sup \limits_{w\in \mathbb{Z}^p} |\alpha_{D_g^\star}(j,l,z)|\\
     & \leq C \sup \limits_{j\geq \log_2(n^{-\frac{1}{2\tilde{\beta}+d}})}  2^{-j\tilde{\beta}}(1+j)^2\\
     & =Cn^{-\frac{\tilde{\beta}}{2\tilde{\beta}+d}}\log(n)^2.
\end{align*}

Therefore, by Lemma \ref{lemma:deltaD}, we have 
    $$\Delta_{\mathcal{D}}^g\leq C\log(n)^{C_3}n^{-\frac{\tilde{\beta}}{2\tilde{\beta}+d}}\Big( d_{\mathcal{H}_1^{\tilde{\beta}+1}}(g_{\# U},g^\star_{\# U})^{\frac{\tilde{\beta}+1}{2\tilde{\beta}+1}}+\frac{1}{n}\Big).$$

Suppose first that $\beta+1<d/2$.
Then, using the bounds on the approximation errors and covering numbers derived previously, for $\delta=n^{-\frac{\beta+1}{2\beta+d}}$ we have
\begin{align*}
\Delta_{\mathcal{G}}& +\min \limits_{\delta \in [0,1]}\Biggl\{\sqrt{\frac{(\delta+1/n)^2\log(n|\mathcal{G}_{1/n}| |(\mathcal{D}_{\mathcal{G}})_{1/n}|)}{n}}
 +\frac{1}{\sqrt{n}}(1+\delta^{(1-\frac{d}{2(\beta+1)})})  \Biggl\}\\
 & \leq C \log(n)^{C_2}n^{-\frac{2\beta+1}{2\beta+d}}.
\end{align*}

Then, using Theorem \ref{theo:boundexpecterror2} and Lemma \ref{lemma:deltaD}, we have 
\begin{align*} 
 \mathbb{E}\Big[d_{\mathcal{H}_1^{\beta+1}}(\hat{g}_{\# U},g^\star_{\# U})\Big]
\leq &\mathbb{E}[\Delta^{\hat{g}}_\mathcal{D}] + C \log(n)^{C_2}n^{-\frac{2\beta+1}{2\beta+d}}\\
     \leq &C\log(n)^{C_3}n^{-\frac{\beta}{2\beta+d}}\left(\mathbb{E}\Big[d_{\mathcal{H}_1^{\beta+1}}(\hat{g}_{\# U},g^\star_{\# U})\Big]^{\frac{\beta+1}{2\beta+1}}+n^{-\frac{\beta+1}{2\beta+d}}\right).
\end{align*}
Now if $\mathbb{E}\Big[d_{\mathcal{H}_1^{\beta+1}}(\hat{g}_{\# U},g^\star_{\# U})\Big]\leq n^{-\frac{2\beta+1}{2\beta+d}}$ we have the result, otherwise $$n^{-\frac{\beta+1}{2\beta+d}}\leq \mathbb{E}\Big[d_{\mathcal{H}_1^{\beta+1}}(\hat{g}_{\# U},g^\star_{\# U})\Big]^{\frac{\beta+1}{2\beta+1}}$$ so
\begin{align*}
\mathbb{E}\Big[d_{\mathcal{H}_1^{\beta+1}}(\hat{g}_{\# U},g^\star_{\# U})\Big] & 
     \leq C\log(n)^{C_3}n^{-\frac{\beta}{2\beta+d}}\left(2\mathbb{E}\Big[d_{\mathcal{H}_1^{\beta+1}}(\hat{g}_{\# U},g^\star_{\# U})\Big]^{\frac{\beta+1}{2\beta+1}}\right),
\end{align*}
so
\begin{align*}   \mathbb{E}\Big[d_{\mathcal{H}_1^{\beta+1}}(\hat{g}_{\# U},g^\star_{\# U})\Big]^{\frac{\beta}{2\beta+1}} & 
     \leq C\log(n)^{C_3}n^{-\frac{\beta}{2\beta+d}}
\end{align*}
which gives 
\begin{align*}
\mathbb{E}\Big[d_{\mathcal{H}_1^{\beta+1}}(\hat{g}_{\# U},g^\star_{\# U})\Big] & 
     \leq C\log(n)^{C_3}n^{-\frac{2\beta+1}{2\beta+d}}.
\end{align*}
Let us treat now the case $\beta+1\geq d/2$. For $\delta=n^{-\frac{d/2}{2\tilde{\beta}+d}}$, we obtain
\begin{align*}
\Delta_{\mathcal{G}}& +\min \limits_{\delta \in [0,1]}\Biggl\{\sqrt{\frac{(\delta+1/n)^2\log(n|\mathcal{G}_{1/n}| |(\mathcal{D}_{\mathcal{G}})_{1/n}|)}{n}}
 +\frac{1}{\sqrt{n}}(1+\log(\delta^{-1}))  \Biggl\}\\
 & \leq C\log(n)^{C_2}n^{-\frac{1}{2}},
\end{align*}
and as for the case $\beta+1<d/2$, using Theorem \ref{theo:boundexpecterror2} and Lemma \ref{lemma:deltaD}, we have 
\begin{align*}
    \mathbb{E}\Big[d_{\mathcal{H}_1^{d/2}}(\hat{g}_{\# U},g^\star_{\# U})\Big] & \leq C\log(n)^{C_3}n^{-\frac{\tilde{\beta}}{2\tilde{\beta}+d}}\left( \mathbb{E}\Big[d_{\mathcal{H}_1^{d/2}}(\hat{g}_{\# U},g^\star_{\# U})\Big]^{\frac{\tilde{\beta}+1}{2\tilde{\beta}+1}}+ n^{-\frac{\tilde{\beta}+1}{2\tilde{\beta}+d}}\right).
\end{align*}
Now if $\mathbb{E}\Big[d_{\mathcal{H}_1^{d/2}}(\hat{g}_{\# U},g^\star_{\# U})\Big]\leq n^{-\frac{1}{2}}$ we have the result, otherwise
\begin{align*}
\mathbb{E}\Big[d_{\mathcal{H}_1^{d/2}}(\hat{g}_{\# U},g^\star_{\# U})\Big] & \leq C\log(n)^{C_3}n^{-\frac{\tilde{\beta}}{2\tilde{\beta}+d}}\mathbb{E}\Big[d_{\mathcal{H}_1^{d/2}}(\hat{g}_{\# U},g^\star_{\# U})\Big]^{\frac{\tilde{\beta}+1}{2\tilde{\beta}+1}}.
\end{align*}
We conclude the same way by

\begin{align*}
\mathbb{E}\Big[d_{\mathcal{H}_1^{d/2}}(\hat{g}_{\# U},g^\star_{\# U})\Big]^{\frac{\tilde{\beta}}{2\tilde{\beta}+1}} & 
     \leq C\log(n)^{C_3}n^{-\frac{\tilde{\beta}}{2\tilde{\beta}+d}}
\end{align*}
which gives  
\begin{align*}
\mathbb{E}\Big[d_{\mathcal{H}_1^{d/2}}(\hat{g}_{\# U},g^\star_{\# U})\Big] & 
     \leq C\log(n)^{C_3}n^{-\frac{2\tilde{\beta}+1}{2\tilde{\beta}+d}}=C\log(n)^{C_3}n^{-\frac{1}{2}}.
\end{align*}
\end{proof}

\bibliographystyleproofs{plainnat}
\bibliographyproofs{adversarial_training}

%% file: main.bbl
\begin{thebibliography}{34}
\providecommand{\natexlab}[1]{#1}
\providecommand{\url}[1]{\texttt{#1}}
\expandafter\ifx\csname urlstyle\endcsname\relax
  \providecommand{\doi}[1]{doi: #1}\else
  \providecommand{\doi}{doi: \begingroup \urlstyle{rm}\Url}\fi

\bibitem[Aamari and Levrard(2019)]{aamari2017nonasymptotic}
Eddie Aamari and Cl{\'e}ment Levrard.
\newblock {Nonasymptotic rates for manifold, tangent space and curvature estimation}.
\newblock \emph{The Annals of Statistics}, 47\penalty0 (1):\penalty0 177 -- 204, 2019.

\bibitem[Arjovsky et~al.(2017)Arjovsky, Chintala, and Bottou]{arjovsky2017wasserstein}
M.~Arjovsky, S.~Chintala, and L.~Bottou.
\newblock Wasserstein generative adversarial networks.
\newblock In D.~Precup and Y.W. Teh, editors, \emph{Proceedings of the 34th International Conference on Machine Learning}, volume~70, pages 214--223. PMLR, 2017.

\bibitem[Belomestny et~al.(2023)Belomestny, Naumov, Puchkin, and Samsonov]{belomestny2022simultaneous}
Denis Belomestny, Alexey Naumov, Nikita Puchkin, and Sergey Samsonov.
\newblock Simultaneous approximation of a smooth function and its derivatives by deep neural networks with piecewise-polynomial activations.
\newblock \emph{Neural Networks}, 161:\penalty0 242--253, 2023.

\bibitem[Chae(2022)]{chae2022rates}
Minwoo Chae.
\newblock Rates of convergence for nonparametric estimation of singular distributions using generative adversarial networks.
\newblock \emph{arXiv preprint arXiv:2202.02890}, 2022.

\bibitem[Chen et~al.(2020)Chen, Liao, Zha, and Zhao]{chen2020distribution}
Minshuo Chen, Wenjing Liao, Hongyuan Zha, and Tuo Zhao.
\newblock Distribution approximation and statistical estimation guarantees of generative adversarial networks.
\newblock \emph{arXiv preprint arXiv:2002.03938}, 2020.

\bibitem[Daubechies(1988)]{daubechies1988orthonormal}
Ingrid Daubechies.
\newblock Orthonormal bases of compactly supported wavelets.
\newblock \emph{Communications on pure and applied mathematics}, 41\penalty0 (7):\penalty0 909--996, 1988.

\bibitem[Daubechies and Lagarias(1991)]{daubechies1991two}
Ingrid Daubechies and Jeffrey~C Lagarias.
\newblock Two-scale difference equations. i. existence and global regularity of solutions.
\newblock \emph{SIAM Journal on Mathematical Analysis}, 22\penalty0 (5):\penalty0 1388--1410, 1991.

\bibitem[Daubechies et~al.(2023)Daubechies, DeVore, Dym, Faigenbaum-Golovin, Kovalsky, Lin, Park, Petrova, and Sober]{Daubechies23}
Ingrid Daubechies, Ronald DeVore, Nadav Dym, Shira Faigenbaum-Golovin, Shahar~Z. Kovalsky, Kung-Chin Lin, Josiah Park, Guergana Petrova, and Barak Sober.
\newblock Neural network approximation of refinable functions.
\newblock \emph{IEEE Transactions on Information Theory}, 69\penalty0 (1):\penalty0 482--495, 2023.

\bibitem[De~Ryck et~al.(2021)De~Ryck, Lanthaler, and Mishra]{De_Ryck_2021}
Tim De~Ryck, Samuel Lanthaler, and Siddhartha Mishra.
\newblock On the approximation of functions by tanh neural networks.
\newblock \emph{Neural Networks}, 143:\penalty0 732–750, November 2021.
\newblock ISSN 0893-6080.
\newblock URL \url{http://dx.doi.org/10.1016/j.neunet.2021.08.015}.

\bibitem[Divol(2022)]{divol2021measure}
Vincent Divol.
\newblock Measure estimation on manifolds: an optimal transport approach.
\newblock \emph{Probability Theory and Related Fields}, 183\penalty0 (1):\penalty0 581--647, 2022.

\bibitem[Federer(1959)]{federer1959}
Herbert Federer.
\newblock Curvature measures.
\newblock \emph{Transactions of the American Mathematical Society}, 93\penalty0 (3):\penalty0 418--491, 1959.

\bibitem[Genovese et~al.(2012)Genovese, Perone-Pacifico, Verdinelli, and Wasserman]{Genovese_2012}
Christopher~R. Genovese, Marco Perone-Pacifico, Isabella Verdinelli, and Larry Wasserman.
\newblock Manifold estimation and singular deconvolution under hausdorff loss.
\newblock \emph{The Annals of Statistics}, 40\penalty0 (2), apr 2012.

\bibitem[Giné and Nickl(2015)]{giné_nickl_2015}
Evarist Giné and Richard Nickl.
\newblock \emph{Mathematical Foundations of Infinite-Dimensional Statistical Models}.
\newblock Cambridge Series in Statistical and Probabilistic Mathematics. Cambridge University Press, 2015.

\bibitem[Gonz{\'a}lez-Prieto et~al.(2021)Gonz{\'a}lez-Prieto, Mozo, Talavera, and G{\'o}mez-Canaval]{gonzalez2021dynamics}
{\'A}ngel Gonz{\'a}lez-Prieto, Alberto Mozo, Edgar Talavera, and Sandra G{\'o}mez-Canaval.
\newblock Dynamics of fourier modes in torus generative adversarial networks.
\newblock \emph{Mathematics}, 9\penalty0 (4):\penalty0 325, 2021.

\bibitem[Goodfellow et~al.(2014)Goodfellow, Pouget-Abadie, Mirza, Xu, Warde-Farley, Ozair, Courville, and Bengio]{GANs}
I.J. Goodfellow, J.~Pouget-Abadie, M.~Mirza, B.~Xu, D.~Warde-Farley, S.~Ozair, A.~Courville, and Y.~Bengio.
\newblock Generative adversarial nets.
\newblock In Z.~Ghahramani, M.~Welling, C.~Cortes, N.D. Lawrence, and K.Q. Weinberger, editors, \emph{Advances in Neural Information Processing Systems}, volume~27, pages 2672--2680. Curran Associates, Inc., 2014.

\bibitem[Gulrajani et~al.(2017)Gulrajani, Ahmed, Arjovsky, Dumoulin, and Courville]{gulrajani2017improved}
I.~Gulrajani, F.~Ahmed, M.~Arjovsky, V.~Dumoulin, and A.C. Courville.
\newblock Improved training of {W}asserstein {GAN}s.
\newblock In I.~Guyon, U.~von Luxburg, S.~Bengio, H.~Wallach, R.~Fergus, S.~Vishwanathan, and R.~Garnett, editors, \emph{Advances in Neural Information Processing Systems}, volume~30, pages 5767--5777. Curran Associates, Inc., 2017.

\bibitem[Haroske(2006)]{haroske2006envelopes}
Dorothee~D Haroske.
\newblock \emph{Envelopes and sharp embeddings of function spaces}.
\newblock CRC Press, 2006.

\bibitem[Karras et~al.(2021)Karras, Aittala, Laine, H{\"{a}}rk{\"{o}}nen, Hellsten, Lehtinen, and Aila]{karras2021}
Tero Karras, Miika Aittala, Samuli Laine, Erik H{\"{a}}rk{\"{o}}nen, Janne Hellsten, Jaakko Lehtinen, and Timo Aila.
\newblock Alias-free generative adversarial networks.
\newblock \emph{CoRR}, abs/2106.12423, 2021.

\bibitem[Li et~al.(2017)Li, Chang, Cheng, Yang, and Póczos]{li2017mmd}
Chun-Liang Li, Wei-Cheng Chang, Yu~Cheng, Yiming Yang, and Barnabás Póczos.
\newblock Mmd gan: Towards deeper understanding of moment matching network, 2017.

\bibitem[Liang(2018)]{liang2018generative}
Tengyuan Liang.
\newblock How well generative adversarial networks learn distributions: A nonparametric view, 2018.

\bibitem[Liang(2021)]{liang2021generative}
Tengyuan Liang.
\newblock How well generative adversarial networks learn distributions.
\newblock \emph{The Journal of Machine Learning Research}, 22\penalty0 (1):\penalty0 10366--10406, 2021.

\bibitem[Liu et~al.(2019)Liu, Zhou, Liu, Dong, Wang, and Wang]{liu2019wasserstein}
Yufei Liu, Yuan Zhou, Xin Liu, Fang Dong, Chang Wang, and Zihong Wang.
\newblock Wasserstein gan-based small-sample augmentation for new-generation artificial intelligence: A case study of cancer-staging data in biology.
\newblock \emph{Engineering}, 5\penalty0 (1):\penalty0 156--163, 2019.

\bibitem[Luo and Lu(2018)]{luo2018eeg}
Yun Luo and Bao-Liang Lu.
\newblock Eeg data augmentation for emotion recognition using a conditional wasserstein gan.
\newblock In \emph{2018 40th annual international conference of the IEEE engineering in medicine and biology society (EMBC)}, pages 2535--2538. IEEE, 2018.

\bibitem[Mroueh et~al.(2018)Mroueh, Li, Sercu, Raj, and Cheng]{mroueh2017sobolev}
Youssef Mroueh, Chun-Liang Li, Tom Sercu, Anant Raj, and Yu~Cheng.
\newblock Sobolev {GAN}.
\newblock In \emph{International Conference on Learning Representations}, 2018.

\bibitem[M{\"u}ller(1997)]{IPMsMuller}
A.~M{\"u}ller.
\newblock Integral probability metrics and their generating classes of functions.
\newblock \emph{Advances in Applied Probability}, 29:\penalty0 429--443, 1997.

\bibitem[Puchkin et~al.(2024)Puchkin, Samsonov, Belomestny, Moulines, and Naumov]{belomestny2023rates}
Nikita Puchkin, Sergey Samsonov, Denis Belomestny, Eric Moulines, and Alexey Naumov.
\newblock Rates of convergence for density estimation with generative adversarial networks.
\newblock \emph{Journal of Machine Learning Research}, 25\penalty0 (29):\penalty0 1--47, 2024.

\bibitem[Schreuder et~al.(2021)Schreuder, Brunel, and Dalalyan]{schreuder2021statistical}
Nicolas Schreuder, Victor-Emmanuel Brunel, and Arnak Dalalyan.
\newblock Statistical guarantees for generative models without domination.
\newblock In \emph{Algorithmic Learning Theory}, pages 1051--1071. PMLR, 2021.

\bibitem[Singh et~al.(2018)Singh, Uppal, Li, Li, Zaheer, and Poczos]{singh2018nonparametric}
Shashank Singh, Ananya Uppal, Boyue Li, Chun-Liang Li, Manzil Zaheer, and Barnabas Poczos.
\newblock Nonparametric density estimation under adversarial losses.
\newblock In S.~Bengio, H.~Wallach, H.~Larochelle, K.~Grauman, N.~Cesa-Bianchi, and R.~Garnett, editors, \emph{Advances in Neural Information Processing Systems}, volume~31. Curran Associates, Inc., 2018.

\bibitem[Stanczuk et~al.(2021)Stanczuk, Etmann, Kreusser, and Sch{\"o}nlieb]{stanczuk2021wasserstein}
Jan Stanczuk, Christian Etmann, Lisa~Maria Kreusser, and Carola-Bibiane Sch{\"o}nlieb.
\newblock Wasserstein gans work because they fail (to approximate the wasserstein distance).
\newblock \emph{arXiv preprint arXiv:2103.01678}, 2021.

\bibitem[Tang and Yang(2023)]{tang2022minimax}
Rong Tang and Yun Yang.
\newblock Minimax rate of distribution estimation on unknown submanifolds under adversarial losses.
\newblock \emph{The Annals of Statistics}, 51\penalty0 (3):\penalty0 1282--1308, 2023.

\bibitem[Tsybakov(2004)]{tsybakov2004introduction}
Alexandre~B Tsybakov.
\newblock Introduction to nonparametric estimation, 2009.
\newblock \emph{URL https://doi. org/10.1007/b13794. Revised and extended from the}, 9\penalty0 (10), 2004.

\bibitem[Villani(2009)]{villani2009optimal}
C{\'e}dric Villani.
\newblock \emph{Optimal transport: old and new}, volume 338.
\newblock Springer, 2009.

\bibitem[Vondrick et~al.(2016)Vondrick, Pirsiavash, and Torralba]{vondrick2016generating}
C.~Vondrick, H.~Pirsiavash, and A.~Torralba.
\newblock Generating videos with scene dynamics.
\newblock In D.~Lee, M.~Sugiyama, U.~von Luxburg, I.~Guyon, and R.~Garnett, editors, \emph{Advances in Neural Information Processing Systems}, volume~29, pages 613--621. Curran Associates, Inc., 2016.

\bibitem[Yu et~al.(2017)Yu, Zhang, Wang, and Yu]{SeqGANs}
L.~Yu, W.~Zhang, J.~Wang, and Y.~Yu.
\newblock Seq{GAN}: {S}equence generative adversarial nets with policy gradient.
\newblock In \emph{Proceedings of the Thirty-First AAAI Conference on Artificial Intelligence}, pages 2852--2858. AAAI Press, 2017.

\end{thebibliography}


\begin{thebibliography}{15}
\providecommand{\natexlab}[1]{#1}
\providecommand{\url}[1]{\texttt{#1}}
\expandafter\ifx\csname urlstyle\endcsname\relax
  \providecommand{\doi}[1]{doi: #1}\else
  \providecommand{\doi}{doi: \begingroup \urlstyle{rm}\Url}\fi

\bibitem[Belomestny et~al.(2023)Belomestny, Naumov, Puchkin, and Samsonov]{belomestny2022simultaneous}
Denis Belomestny, Alexey Naumov, Nikita Puchkin, and Sergey Samsonov.
\newblock Simultaneous approximation of a smooth function and its derivatives by deep neural networks with piecewise-polynomial activations.
\newblock \emph{Neural Networks}, 161:\penalty0 242--253, 2023.

\bibitem[Daubechies(1988)]{daubechies1988orthonormal}
Ingrid Daubechies.
\newblock Orthonormal bases of compactly supported wavelets.
\newblock \emph{Communications on pure and applied mathematics}, 41\penalty0 (7):\penalty0 909--996, 1988.

\bibitem[Daubechies and Lagarias(1991)]{daubechies1991two}
Ingrid Daubechies and Jeffrey~C Lagarias.
\newblock Two-scale difference equations. i. existence and global regularity of solutions.
\newblock \emph{SIAM Journal on Mathematical Analysis}, 22\penalty0 (5):\penalty0 1388--1410, 1991.

\bibitem[Daubechies and Lagarias(1992)]{daubechies1992two}
Ingrid Daubechies and Jeffrey~C Lagarias.
\newblock Two-scale difference equations ii. local regularity, infinite products of matrices and fractals.
\newblock \emph{SIAM Journal on Mathematical Analysis}, 23\penalty0 (4):\penalty0 1031--1079, 1992.

\bibitem[DeVore et~al.(1989)DeVore, Howard, and Micchelli]{devore1989optimal}
Ronald~A DeVore, Ralph Howard, and Charles Micchelli.
\newblock Optimal nonlinear approximation.
\newblock \emph{Manuscripta mathematica}, 63:\penalty0 469--478, 1989.

\bibitem[Donoho et~al.(1998)Donoho, Vetterli, DeVore, and Daubechies]{donoho1998data}
David~L. Donoho, Martin Vetterli, Ronald~A. DeVore, and Ingrid Daubechies.
\newblock Data compression and harmonic analysis.
\newblock \emph{IEEE transactions on information theory}, 44\penalty0 (6):\penalty0 2435--2476, 1998.

\bibitem[Federer(1959)]{federer1959}
Herbert Federer.
\newblock Curvature measures.
\newblock \emph{Transactions of the American Mathematical Society}, 93\penalty0 (3):\penalty0 418--491, 1959.

\bibitem[Federer(2014)]{federer2014geometric}
Herbert Federer.
\newblock \emph{Geometric measure theory}.
\newblock Springer, 2014.

\bibitem[Fefferman(2009)]{fefferman2009extension}
Charles Fefferman.
\newblock Extension of c\^{}m,$\omega$-smooth functions by linear operators.
\newblock 2009.

\bibitem[Fefferman(2005)]{fefferman2005sharp}
Charles~L Fefferman.
\newblock A sharp form of whitney's extension theorem.
\newblock \emph{Annals of mathematics}, 161\penalty0 (1):\penalty0 509--577, 2005.

\bibitem[Giné and Nickl(2015)]{giné_nickl_2015}
Evarist Giné and Richard Nickl.
\newblock \emph{Mathematical Foundations of Infinite-Dimensional Statistical Models}.
\newblock Cambridge Series in Statistical and Probabilistic Mathematics. Cambridge University Press, 2015.

\bibitem[Leobacher and Steinicke(2018)]{Leobacher}
Gunther Leobacher and Alexander Steinicke.
\newblock Existence, uniqueness and regularity of the projection onto differentiable manifolds, 2018.

\bibitem[Puchkin et~al.(2024)Puchkin, Samsonov, Belomestny, Moulines, and Naumov]{belomestny2023rates}
Nikita Puchkin, Sergey Samsonov, Denis Belomestny, Eric Moulines, and Alexey Naumov.
\newblock Rates of convergence for density estimation with generative adversarial networks.
\newblock \emph{Journal of Machine Learning Research}, 25\penalty0 (29):\penalty0 1--47, 2024.

\bibitem[Vaart and Wellner(2023)]{vaart2023empirical}
AW~van~der Vaart and Jon~A Wellner.
\newblock Empirical processes.
\newblock In \emph{Weak Convergence and Empirical Processes: With Applications to Statistics}, pages 127--384. Springer, 2023.

\bibitem[van Handel(2016)]{Van2016prob}
R.~van Handel.
\newblock Probability in high dimension.
\newblock \emph{Lecture Notes, Princeton University}, 2016.

\end{thebibliography}
